\documentclass[11pt]{amsart}
\usepackage[dvipsnames]{xcolor}
\usepackage{amsmath, amscd}
\usepackage{amsfonts}
\usepackage{amssymb}
\usepackage{mathtools}
\usepackage{mathrsfs}
\usepackage{graphicx}
\usepackage{epstopdf}
\usepackage{alias cnt}
\usepackage{overpic}
\usepackage{enumitem,times}
\usepackage{epigraph}
\usepackage{tikz}
\usepackage[bookmarks=true,pdfstartview=FitH, pdfborder={0 0 0}, colorlinks=true,citecolor=blue, linkcolor=blue]{hyperref}
\usepackage{caption}
\usepackage{comment}

\usepackage{etoolbox}
\makeatletter
\patchcmd{\epigraph}{\@epitext{#1}}{\itshape\@epitext{#1}}{}{}
\makeatother

\newtheorem{theorem}{Theorem}[section]
\newaliascnt{conj}{theorem}
\newaliascnt{cor}{theorem}
\newaliascnt{lemma}{theorem}
\newaliascnt{fact}{theorem}
\newaliascnt{claim}{theorem}
\newaliascnt{prop}{theorem}
\newaliascnt{defn}{theorem}
\newaliascnt{assump}{theorem}
\newaliascnt{question}{theorem}
\newaliascnt{notation}{theorem}
\newaliascnt{convention}{theorem}

\newtheorem{cor}[cor]{Corollary}
\newtheorem{lemma}[lemma]{Lemma}

\newtheorem{claim}[claim]{Claim}
\newtheorem{prop}[prop]{Proposition}
\newtheorem{defn}[defn]{Definition}
\newtheorem{question}[question]{Question}

\numberwithin{equation}{subsection}
\numberwithin{theorem}{subsection}
\numberwithin{figure}{subsection}

\aliascntresetthe{conj} \aliascntresetthe{cor}
\aliascntresetthe{lemma} \aliascntresetthe{fact}
\aliascntresetthe{claim} \aliascntresetthe{prop}
\aliascntresetthe{defn} \aliascntresetthe{question}
\aliascntresetthe{notation} \aliascntresetthe{convention}
\aliascntresetthe{assump}

\makeatletter
\let\oldtheequation\theequation
\renewcommand\tagform@[1]{\maketag@@@{\ignorespaces#1\unskip\@@italiccorr}}
\renewcommand\theequation{(\oldtheequation)}
\makeatother

\theoremstyle{remark}
\newaliascnt{rmk}{theorem}

\newtheorem{remark}[rmk]{Remark}
\aliascntresetthe{rmk}

\theoremstyle{remark}
\newaliascnt{exam}{theorem}

\aliascntresetthe{exam}

\def\sek~{\S{}}

\DeclareMathOperator{\inward}{in}
\DeclareMathOperator{\outward}{out}

\DeclareMathOperator{\ind}{ind}

\newcommand{\D}{\mathbb{D}}
\newcommand{\R}{\mathbb{R}}

\newcommand{\ve}{\varepsilon}
\newcommand{\s}{\vskip.1in}
\newcommand{\n}{\noindent}
\newcommand{\p}{\partial}
\newcommand{\bdry}{\partial}

\newcommand{\be}{\begin{enumerate}}
\newcommand{\ee}{\end{enumerate}}
\newcommand{\op}{\operatorname}
\newcommand{\cg}{\color{green}}
\newcommand{\cbu  }{\color{blue}}

\makeatletter
\newcommand{\opm}{\mathbin{\mathpalette\make@circled\pm}}
\newcommand{\make@circled}[2]{%
  \ooalign{$\m@th#1\smallbigcirc{#1}$\cr\hidewidth$\m@th#1#2$\hidewidth\cr}%
}
\newcommand{\smallbigcirc}[1]{%
  \vcenter{\hbox{\scalebox{1}{$\m@th#1\bigcirc$}}}%
}
\makeatother

\usepackage{mleftright}
\newcommand{\set}[2]{\mleft\{\,#1~\middle|~#2\,\mright\}}

\begin{document}

\title{The Giroux correspondence in arbitrary dimensions}

\author{Joseph Breen}
\address{University of Iowa, Iowa City, IA 52240}
\email{joseph-breen-1@uiowa.edu} \urladdr{https://sites.google.com/view/joseph-breen/home}

\author{Ko Honda}
\address{University of California, Los Angeles, Los Angeles, CA 90095}
\email{honda@math.ucla.edu} \urladdr{http://www.math.ucla.edu/\char126 honda}

\author{Yang Huang}
\address{Somewhere On Earth}
\email{hymath@gmail.com} \urladdr{https://sites.google.com/site/yhuangmath}

\thanks{JB was partially supported by NSF Grant DMS-2003483 and NSF Grant DMS-2038103. KH was partially supported by NSF Grant DMS-2003483. YH was partially supported by the grant KAW 2016.0198 from the Knut and Alice Wallenberg Foundation.}

\begin{abstract}
We establish the Giroux correspondence in arbitrary dimensions.  As corollaries we (i) give an alternate proof of a result of Giroux-Pardon that states that any Weinstein domain is Weinstein homotopic to one which admits a Weinstein Lefschetz fibration and (ii) prove that any two Weinstein Lefschetz fibrations whose Weinstein domain structures are Weinstein homotopic are related by the Weinstein Lefschetz fibration moves, affirming a conjecture of Giroux-Pardon.
\end{abstract}

\maketitle

\tableofcontents

\section{Introduction} \label{sec:intro}

Around twenty years ago in the influential paper~\cite{Gi02}, Emmanuel Giroux formulated the equivalence of {\em contact structures} and {\em open book decompositions} with Weinstein pages up to stabilization.  For contact $3$-manifolds the Weinstein condition is superfluous and there is a precise equivalence between contact structures and open book decompositions up to stabilization.\footnote{Although, to the authors' knowledge the only publicly available account of the stabilization equivalence is in Williams' thesis~\cite{Wi18}.} In higher dimensions, the existence of an open book decomposition with Weinstein pages follows without much difficulty from the technology of Donaldson \cite{Don96}, Ibort, Martinez-Torres, and Presas \cite{IMP00}, and Mohsen \cite{Moh01,Moh19}, as explained by Giroux in \cite{Gi02}. On the other hand, the stabilization equivalence, due to Giroux-Mohsen \cite{GM} and as explained below, is a much less complete result.  

The aim of this paper is to apply the recent advances in convex hypersurface theory to ameliorate the situation. 

\subsection{Main result} \label{subsection: main result}

Let $(M,\xi)$ be a closed cooriented contact manifold of dimension $2n+1$, i.e., $\xi=\ker \alpha$ such that $\alpha\wedge (d\alpha)^n>0$. Here $n\geq 1$.  The $n=1$ case is often slightly different (and easier) and will be treated separately when a difference arises.

\begin{defn}
A contact manifold $(M,\xi)$ is {\em supported by an open book decomposition $(B,\pi: M\setminus B\to S^1)$} (abbreviated OBD) if 
\be
\item the {\em binding} $B^{2n-1}$ is a codimension $2$ closed contact submanifold,
\item $\pi:M\setminus B\to S^1$ is a fibration which agrees with the angular coordinate $\theta$ on a neighborhood $B\times D^2$ of $B=B\times\{0\}$, and
\item there exists a Reeb vector field $R_\alpha$ for $\xi$ which is everywhere transverse to all the pages $\pi^{-1}(\theta)$
$\Leftrightarrow$ all the {\em pages} $\pi^{-1}(\theta)$ are Liouville.
\ee
\end{defn}

In this paper we will view a page as either a Liouville domain or its completion and will use the same notation $(W,\lambda)$, since there is no substantial difference.

\begin{defn}
A Liouville domain/completion $(W,\lambda)$ is {\em Weinstein} or {\em $0$-Weinstein} (resp.\ {\em $1$-Weinstein; $2$-Weinstein}) if its Liouville vector field $X_\lambda$ is gradient-like for some function $g:W\to \R$ which only has Morse type (resp.\ Morse and birth-death type; Morse, birth-death, and swallowtail type) critical points. 
\end{defn}

We denote an $i$-Weinstein domain/completion by $(W,\lambda,g)$, or simply $(W,\lambda)$ if some $g$ is understood.

\begin{defn}
A supporting OBD is  
\begin{itemize}
\item {\em strongly Weinstein} if all of its pages are $1$-Weinstein; and 
\item {\em weakly Weinstein} if at least one page is Weinstein. 
\end{itemize}
For convenience we always assume that $\pi^{-1}(0)$ is Weinstein.
\end{defn}

\begin{defn}
Two strongly/weakly Weinstein OBDs $(B_t,\pi_t:M\setminus B_t\to S^1)$, $t=0,1$, for $(M,\xi)$ are:
\begin{itemize}
\item {\em strongly Weinstein homotopic} if there is a $1$-parameter family of OBDs $(B_t,\pi_t:M\setminus B_t\to S^1)$, $t\in[0,1]$, interpolating between $(B_0,\pi_0)$ and $(B_1,\pi_1)$, all of whose pages are $2$-Weinstein, and
\item {\em weakly Weinstein homotopic} if there is a $1$-parameter family of OBDs $(B_t,\pi_t)$, $t\in[0,1]$, interpolating between $(B_0,\pi_0)$ and $(B_1,\pi_1)$ such that $\pi_t^{-1}(0)$ is $1$-Weinstein for all $t\in[0,1]$.
\end{itemize}   
\end{defn}

We can view a strongly/weakly Weinstein OBD $(B,\pi)$ as a relative mapping torus of $(W,\phi)$, where $W$ is a Weinstein domain that is a slight retraction of $\pi^{-1}(0)$ and $\phi\in \op{Symp}(W,\bdry W)$. 

\begin{defn}
Let $L_0\subset W$ be a regular Lagrangian disk with Legendrian boundary $\subset \bdry W$.  Then either of $(W\cup h, \phi\circ \tau_L)$ or $(W\cup h, \tau_L\circ\phi)$ is called a {\em (positive) stabilization of $(W^{2n},\phi)$ along $L_0$}, where: 
\be
\item $h$ is a Weinstein $n$-handle with core Lagrangian disk $L_1$ such that $\bdry L_0=\bdry L_1$ and 
\item $\tau_{L}$ is the (positive symplectic) Dehn twist about $L=L_0\cup L_1$. 
\ee
\end{defn}

\begin{remark}
Given $\phi\in \op{Symp}(W,\bdry W)$, it is not clear whether $(W,\phi)$ gives rise to a strongly Weinstein OBD. However, one can verify that a positive stabilization of a strongly/weakly Weinstein OBD along a regular $L_0$ is strongly/weakly Weinstein.
\end{remark}

\begin{defn}
Two strongly/weakly Weinstein OBDs of $(M,\xi)$ are {\em strongly/weakly stably equivalent} if they are related by a sequence of positive stabilizations and destabilizations, conjugations, and strong/weak Weinstein homotopies.
\end{defn}

The main result of this paper is the following:

\begin{theorem}[Stabilization equivalence]\label{theorem: stabilization equivalence}  Any two strongly Weinstein OBDs of $(M,\xi)$ are strongly stably equivalent.
\end{theorem} 

What Giroux~\cite{Gi02} and Giroux-Mohsen~\cite{GM} had already shown was that:

\begin{theorem} \label{theorem: Giroux Mohsen} $\mbox{}$
	\begin{enumerate}
		\item Any $(M,\xi)$ is supported by a strongly Weinstein OBD.
		\item Any two strongly Weinstein OBDs of $(M,\xi)$ obtained by the ``Donaldson construction'' are strongly stably equivalent.
	\end{enumerate}
\end{theorem}

Strictly speaking, the details of \autoref{theorem: Giroux Mohsen}(2) have not appeared yet.  Also \autoref{theorem: Giroux Mohsen}(2) does not say that any two strongly Weinstein OBDs that support $(M,\xi)$ are strongly stably equivalent, making it very hard to use in practice. In particular, it is difficult to tell whether a give OBD has been obtained by the Donaldson construction because the latter depends on a ``sufficiently large'' integer $k$. 

We are still left with the following fundamental question:

\begin{question}
Is a weakly Weinstein OBD of $(M,\xi)$ weakly stably equivalent to a strongly Weinstein OBD?
\end{question}

\begin{remark}
In general we cannot expect an OBD with Liouville pages to be stable equivalent to a weakly or strongly Weinstein OBD since there exist $2n$-dimensional Liouville domains $X$ with nontrivial $H_k(X)$ for $k\gg n$ (see for example \cite[Theorem C]{MNW13}, which gives $X$ diffeomorphic to $[0,1]$ times an odd-dimensional closed manifold), the pages of a positive stabilization are obtained by attaching $n$-handles, and Weinstein domains have the type of a $\leq n$-dimensional CW complex.  (The authors learned about this from Otto van Koert.) 
\end{remark}

\begin{remark}
    There is also work of Licata-V\'ertesi~\cite{LV} in dimension $3$ which uses the same basic outline as ours. In fact V\'ertesi, as early as 2010, insisted that there should be a proof along the lines that were ultimately carried out in this paper (in those days we simply did not have enough convex (hyper)-surface technology).
\end{remark}

\subsection{Relative version} \label{subsection: relative version}

Next we describe a relative version of our main result.  We first review {\em Weinstein partial open book decompositions} (POBDs), more or less as presented in \cite{HH19}. 

\begin{defn}
Given a Weinstein domain $W$, a \emph{cornered Weinstein subdomain} $S \subset W$ is a (possibly empty) codimension $0$ submanifold with corners which satisfies the following properties:
\begin{itemize} 
\item[(CW1)] There exists a decomposition $\p S = \p_{\inward} S \cup \p_{\outward} S$ such that
\be
\item $\p_{\inward} S$ and $\p_{\outward} S$ are compact manifolds with smooth boundary that intersect along their boundaries;
\item $\p(\p_{\inward} S) = \p(\p_{\outward} S)$ is the codimension $1$ corner of $\p S$; and
\item $\p_{\outward} S = S\cap \p W$ and is a proper subset of each component of $\bdry W$.
\ee
\item[(CW2)] The Liouville vector field $X_{\lambda}$ on $W$ is inward-pointing along $\p_{\inward} S$ and outward-pointing near $\p_{\outward} S$.
\item[(CW3)] $W\setminus S$ is a Weinstein domain after smoothing.
\end{itemize}
The region $S$ (without reference to the space $W$) is called a {\em cornered Weinstein cobordism}.
\end{defn}

\begin{defn} \label{defn: POBD}
A {\em strongly Weinstein POBD} for a compact cooriented contact manifold $(M^{2n+1},\xi)$ with Weinstein convex boundary is a pair $(W\times[0,\tfrac{1}{2}], S\times[\tfrac{1}{2},1])$, where:
\be
\item $W_t:=W\times \{t\}$ is a $1$-Weinstein domain for each $t\in [0,\tfrac{1}{2}]$ and $S_t:=S\times\{t\}$ is a $1$-Weinstein cornered Weinstein cobordism for each $t\in[\tfrac{1}{2},1]$;
\item $S_{1/2}$ is a cornered Weinstein subdomain of $W_{1/2}$ and there is an embedding $\phi: S_1\hookrightarrow W_0$ that identifies $S_1$ with a cornered Weinstein subdomain of $W_0$;
\item there exists a contact form $\alpha$ on the glued manifold which restricts to the $1$-Weinstein Liouville forms on each $W_t$ and $S_t$ and whose Reeb vector field $R_\alpha$ is given by $\bdry_t$ on $(\bdry W)\times [0,\tfrac{1}{2}]$ and $(\bdry S)\times [\tfrac{1}{2},1]$; and
\item $(M,\xi)$ is obtained by filling in $D^2 \times \p W$ with a ``standard" contact structure on a neighborhood of the contact submanifold $\bdry W$ and smoothing some corners.
\ee
We may also denote a strongly Weinstein POBD by $(W, \phi: S\to W)$ with the understanding that all the pages $W_t$ and $S_t$ are $1$-Weinstein.
\end{defn}

The definitions of Weinstein homotopic POBDs is analogous to the closed case, and the definition of a positive stabilization of a POBD is as follows:\footnote{In previous papers, they were called {\em partial} positive stabilizations.}  

\begin{defn}
Let $L_0\subset W$ be a regular Lagrangian disk with Legendrian boundary $\subset \bdry W$.  Then either of the following two POBDs, 
\begin{gather*}
    (W\cup h,   \tau_L\circ (\phi\cup\op{id}_h) : S\cup h \to W\cup h),\\
    (W\cup h, (\phi\cup \op{id}_h) \circ \tau_L: \tau^{-1}_L(S) \cup N(L_0)  \to W\cup h),
\end{gather*}
is called a {\em (positive) stabilization of $(W, \phi:S\to W)$ along $L_0$}, where: 
\be
\item $h$ is a Weinstein $n$-handle with core Lagrangian disk $L_1$ such that $\bdry L_0=\bdry L_1$, 
\item $N(L_0)$ is the Weinstein $n$-handle neighborhood with cocore $L_0$, attached to $W\setminus N(L_0)$, and 
\item $\tau_L$ is a (positive symplectic) Dehn twist about $L=L_0\cup L_1$. 
\ee
The two POBD stabilizations will be called the {\em(TB1) positive stabilization} and the {\em(TB2) positive stabilization}, respectively, for consistency with the language used in \autoref{subsection:trivial_bypasses_and_POBD_stabilization} (where we also explain how these stabilizations arise).
\end{defn}

As in the absolute case, a positive stabilization of a strongly/weakly Weinstein POBD along a regular $L_0$ is strongly Weinstein.

\begin{theorem}[Stabilization equivalence for POBDs] \label{theorem: relative version} 
Let $(M^{2n+1},\xi)$ be a compact cooriented contact manifold with Weinstein convex boundary. Then:
	\begin{enumerate}
		\item $(M,\xi)$ is supported by a strongly Weinstein POBD.
		\item Any two strongly Weinstein POBDs of $(M,\xi)$ are strongly stably equivalent.
	\end{enumerate}
\end{theorem}

Part (1) was proved in \cite{HH19}. The proof of Part (2) is identical to that of \autoref{theorem: stabilization equivalence} and is left to the reader.

\subsection{Some corollaries} \label{subsection: some corollaries}

We now state some corollaries of \autoref{theorem: stabilization equivalence}. The first is a result due to Giroux-Pardon, originally proved using Donaldson's technology:

\begin{cor}[Giroux-Pardon~\cite{GP17}] \label{cor: proof of GP}
Any Weinstein domain is $1$-Weinstein homotopic to one which admits a Weinstein Lefschetz fibration.
\end{cor}

\begin{defn}\label{def: WLF}
A {\em Weinstein Lefschetz fibration} $(p: X^{2n}\to D^2,\lambda)$ (abbreviated WLF) satisfies the following:
\be
\item $X$ is a $2n$-dimensional compact domain with corners and $\lambda$ is a Weinstein structure on the smoothing $X^{sm}$ of $X$;
\item $p$ is a smooth map which is a fibration except at a finite number $\op{Crit}(p)$ of critical points in $\op{int}(X)$;
\item there exist local holomorphic coordinates $z_1,\dots,z_n$ about each $x\in \op{Crit}(p)$ with respect to which $p(z_1,\dots,z_n)= p(x)+ z_1^2+\dots +z_n^2$ and $\lambda= i\sum_{j=1}^n(z_j d\overline{z}_j -\overline{z}_j dz_j)$ (and hence $d\lambda$ is K\"ahler with respect to these local holomorphic coordinates);
\item $\bdry X$ is the union of two faces $\bdry_v X= p^{-1}(\bdry D^2)$ and $\bdry_h X= \cup_{y\in D^2} \bdry( p^{-1}(y))$ meeting at a codimension two corner and $p|_{\bdry _vX}: \bdry_v X\to \bdry D^2$ and $p|_{\bdry_h X}:\bdry_h X\to D^2$ are fibrations;
\item $d\lambda$ is nondegenerate on each $\ker dp(x)$ for each $x\in X$; and 
\item the restriction of $\lambda$ to each nonsingular fiber is Weinstein and $\lambda|_{\bdry(p^{-1}(y))}$ is independent of $y\in D^2$.
\ee
A WLF often comes with a basepoint $*\in D^2$, which we take to be a regular value. We call $p^{-1}(*)$ the {\em distinguished fiber}. 
\end{defn} 

There is a corresponding combinatorial object called an {\em abstract WLF $(W^{\flat};\mathcal{L}=( L_1,\dots, L_m))$}, where $W^{\flat}$ is a ($2n-2$)-dimensional Weinstein domain and $L_1,\dots, L_m$ are exact parametrized Lagrangian spheres in $W^{\flat}$. A WLF  $p:X\to D^2$ can be constructed from an abstract WLF by attaching critical Weinstein handles along Legendrians $\tilde L_j\subset W^{\flat}\times S^1\subset \bdry (W^{\flat}\times D^2)$ near $2\pi j/m \in S^1$ obtained by lifting $L_j$. Here $W^{\flat}$ becomes the distinguished fiber of the WLF. We refer the reader to \cite[Section 6]{GP17} for more details on how to go back and forth between WLFs and abstract WLFs.

The following moves applied to $(W^{\flat}; \mathcal{L}=( L_1,\dots, L_m))$ are called the {\em abstract WLF moves}:
\be
\item (Deformation) Simultaneous Weinstein deformation of $W^{\flat}$ and exact Lagrangian isotopy of $\mathcal{L}$.
\item (Cyclic permutation) Replace $\mathcal{L}$ by $(L_2,\dots, L_m, L_1)$.
\item (Hurwitz move) Replace $\mathcal{L}$ by $(L_2, \tau_{L_2}L_1, L_3,\dots, L_m)$ or $(\tau^{-1}_{L_1}(L_2),L_1,L_3,\dots, L_m)$.
\item (Stabilization) Given a regular Lagrangian disk $D_0\subset W^{\flat}$ with Legendrian boundary, replace $W^{\flat}$ by $W^{\flat}\cup h$, where $h$ is a critical Weinstein ($n-1$)-handle with core Lagrangian disk $D_1$ such that $\bdry D_0=\bdry D_1$, and replace $(L_1,\dots, L_m)$ by $(L'=D_0\cup D_1, L_1,\dots, L_m)$.
\ee

Our second corollary was conjectured by Giroux-Pardon~\cite[p.967]{GP17}:

\begin{cor} \label{cor: conj of GP}
Any two abstract WLFs whose Weinstein domain structures are $1$-Weinstein homotopic are related by the abstract WLF moves.
\end{cor}

The proofs of both corollaries will be given in \autoref{section: proofs of Lefschetz}. The point is that, given a WLF $p:X \to D^2$, the restriction of $p$ to $\bdry X$ gives a strongly Weinstein OBD of $\bdry X$, where $p^{-1}(\op{int}(D^2))$ is a neighborhood of the binding and $p^{-1}(e^{i\theta})$, $e^{i\theta}\in \bdry D^2$, are the pages. From the point of view of an abstract WLF $(W^{\flat}; \mathcal{L} = (L_1, \dots, L_m))$, the page of the OBD is $W^{\flat}$ and the monodromy is given by positive Dehn twists around $L_1, \dots, L_m$. The abstract WLF moves can then be described in terms of the boundary OBD.

\subsection{Review and reformulation of \cite{HH19}} \label{subsection: review}

This paper is an extension of the following results from \cite{HH19}:  

\begin{theorem} \label{theorem:genericity}
	Any closed hypersurface in a contact manifold can be $C^0$-approxi\-mated by a Weinstein convex one.
\end{theorem}

\begin{theorem} \label{theorem:family genericity}
	Let $\xi$ be a contact structure on $\Sigma \times [0,1]$ such that the hypersurfaces $\Sigma \times \{0,1\}$ are Weinstein convex. Then, up to a boundary-relative contact isotopy, there exists a finite sequence $0<s_1<\dots<s_N<1$ such that the following hold:
	\begin{itemize}
		\item[(B1)] $\Sigma \times \{s\}$ is Weinstein convex if $s\neq s_i$ for any $1 \leq i \leq N$.
		\item[(B2)] For each $i$, there exists a small $\epsilon>0$ such that $\xi$ restricted to $\Sigma \times [s_i-\epsilon, s_i+\epsilon]$ is contactomorphic to a bypass attachment.
	\end{itemize}
\end{theorem}

We say that $\xi$ on $\Sigma\times[0,1]$ admits a {\em bypass decomposition}  if there exist $0<s_1<\dots<s_N<1$ for some $N$ such that (B1) and (B2) hold. Bypass attachments will be reviewed in \autoref{section: preliminaries}. 

The technical heart of \autoref{theorem: stabilization equivalence} is the $1$-parametric version of \autoref{theorem:family genericity} and the case-by-case analysis of all the possibilities that arise. As such, the following remark is essential:

\begin{remark} \label{rmk: parametric}
	A key feature of the proofs of Theorems~\ref{theorem:genericity} and \ref{theorem:family genericity} is that they can be done {\em parametrically}, i.e., in any number of parameters. 
\end{remark}

\begin{defn}$\mbox{}$
\be
\item A {\em contact handlebody over a Weinstein domain $(W,\lambda)$} is a contact manifold contactomorphic to $$(W\times [-C,C]_t, \xi|_{W\times[-C,C]}= \ker(dt+\lambda)),$$ 
	where $C>0$.  
\item A {\em generalized contact handlebody} is a contact manifold which is contactomorphic to 
$$(W\times [f_-(x),f_+(x)]_t, \xi|_{W\times[f_-(x),f_+(x)]}= \ker(dt+\lambda)),$$ 
	where $f_-,f_+: W\to \R$ are smooth functions with $f_-(x)<f_+(x)$ for all $x\in W$, and is foliated by $1$-Weinstein domains of graphical type (i.e, of the form $t=f(x)$).
\item A (generalized) contact handlebody with $W$ a flexible Weinstein domain is a {\em flexible (generalized) contact handlebody}. 

\ee
\end{defn}

\begin{remark}
	 A (generalized) contact handlebody has a natural isotopy class of smoothings.  We will often abuse notation and also call the smoothing a (generalized) contact handlebody.
\end{remark}

We briefly sketch the convex hypersurface theory proof of \autoref{theorem: Giroux Mohsen}(1), that is, the existence of supporting strongly Weinstein OBDs. Given a closed contact manifold $(M,\xi)$ of dimension $2n+1$, we first choose a self-indexing Morse function $f:M\to \R$ so that $\Sigma:= f^{-1}(n+\tfrac{1}{2})$ is a smooth hypersurface which divides $M$ into two components $M-\Sigma=H'_0\cup H'_1$.  Using Gromov's $h$-principle, we can realize some deformation retraction $H_i$ of $H'_i$, $i=0,1$, as a contact handlebody.

\begin{defn}[$\theta$-decomposition]
	A {\em $\theta$-decomposition}\footnote{Informally called a ``mushroom burger"; the $H_i$ are the buns and $\Sigma\times[0,1]$ is the patty.} of a closed contact manifold $(M,\xi)$ is a pair consisting of two decompositions $M= H_0\cup (\Sigma\times[0,1])\cup H_1$ and $\Delta$, where:
	\be\item $H_0$ and $H_1$ are contact handlebodies; 
	\item $\bdry H_0\simeq \Sigma\times\{0\}$ and $\bdry H_1\simeq -\Sigma\times\{1\}$ are Weinstein convex hypersurfaces;
	\item $\Delta$ is a bypass decomposition of $(\Sigma\times [0,1],\xi)$.
	\ee
\end{defn}

\autoref{theorem:family genericity} implies the existence of a $\theta$-decomposition $(M= H_0\cup (\Sigma\times[0,1])\cup H_1,\Delta)$.
Since a bypass is a smoothly canceling (but contact-topologically nontrivial) pair of contact handle attachments of indices $n$ and $n+1$, a $\theta$-decompo\-sition gives a contact handle decomposition of $(M, \xi)$, which in turn implies \autoref{theorem: Giroux Mohsen}(1).

\begin{remark}
	We will work with $\theta$-decompositions instead of {\em contact Morse functions $g:M\to \R$} even though a $\theta$-decomposition gives rise to a contact Morse function. This is because convex hypersurfaces transverse to the gradient of a contact Morse function are not necessarily Weinstein convex (i.e., the Weinstein data can be lost). 
\end{remark}

\n
{\em Organization of the paper.} In \autoref{section: preliminaries} we review and rework some preliminaries, including the theory of contact handles, Legendrian Kirby calculus, bypass attachments, and stabilizations of (P)OBDs.  Then in \autoref{section: outline of proof of stabilization equivalence} we outline the proof of \autoref{theorem: stabilization equivalence}, the main stabilization equivalence result.  The analysis of all the cases that arise in the $1$-parametric version of \autoref{theorem:family genericity} is carried out in \autoref{sec: sigma times interval}; as indicated earlier, this is the technical heart of the proof of \autoref{theorem: stabilization equivalence}.  The remaining steps of the proof of \autoref{theorem: stabilization equivalence} are given in \autoref{section: proof of theorem layer} and \autoref{section: proof of theorem: birth-death} and the proofs of Corollaries~\ref{cor: proof of GP} and \ref{cor: conj of GP} are given in \autoref{section: proofs of Lefschetz}.

\s\n
{\em Acknowledgements.} JB would like to thank Austin Christian for many valuable conversations over the course of 2022 and 2023. KH thanks Otto van Koert and Kevin Sackel for helpful correspondence and conversations. He is also grateful to Hirofumi Sasahira and Kyushu University for their hospitality during his sabbatical in AY 2022--2023. KH also thanks Matthias Scharitzer for helpful email and Vera V\'ertesi for discussions many many years ago.

\section{Preliminaries} \label{section: preliminaries}

In this section we review the theory of contact handles, Legendrian Kirby calculus, bypass attachments, and stabilizations of OBDs, primarily from the point of view of \cite{HH18,HH19}; see also \cite{Sac19,ding2009handle,casals2019legendrianfronts}. We will also introduce a new ingredient --- decorations on Legendrian surgery diagrams --- to more effectively describe contact ($n+1$)-handles. 

Recall that the ambient contact manifold $(M,\xi)$ has dimension $2n+1$, where $n\geq 1$.  The $n=1$ and $n>1$ cases will be treated in parallel in this paper.  Since the $n=1$ case is simpler, we will mostly discuss the $n> 1$ case and comment on the $n=1$ case when appropriate.

\subsection{Higher-dimensional Legendrian front projections}

We begin by describing the Legendrian surgery diagrams that we use when working with higher-dimensional Legendrians, i.e., when $\Lambda$ is a Legendrian and $\dim \Lambda > 1$. 

\begin{remark}
Although $\dim M=2n+1$, many of our diagrams represent ($n-1$)-dimensional Legendrians in a codimension-$2$ contact submanifold, like the dividing set of a convex hypersurface. Dimensions should be clear from context. 
\end{remark}

The nature of Legendrian front singularities and Legendrian Reidemeister moves in high dimensions is complicated, and in general cannot be comprehensively described by bootstrapping the behavior of Legendrian knots. In \cite{casals2019legendrianfronts}, Casals and Murphy work with a class of Legendrians in arbitrary dimensions that exhibit spherical symmetry, and thus, with specified conventions, can be diagrammatically represented by curves. Legendrian Reidemeister moves that respect this symmetry extend accordingly. For our purposes, which are generally local, we can afford the luxury of working with similar conventions as \cite{casals2019legendrianfronts}. 

As such, our diagrams are drawn using curves, although they implicitly describe Legendrian fronts in arbitrary dimensions roughly according the conventions established in Section 2.4 of \cite{casals2019legendrianfronts}. For convenience, we summarize them here; the reader is also encouraged to consult the reference in question.

Given a Legendrian front curve $\Lambda\subseteq \R^2_{(z,x_1)}$, we generate a Legendrian front $\tilde{\Lambda}^{n-1} \subseteq \R^{n}_{(z,x_1, \dots, x_{n-1})}$ representing a local Legendrian submanifold in $(\R^{2n-1}, \xi_{\mathrm{st}})$ as follows:
\begin{itemize}
    \item We include the curve and plane 
    \[
    \Lambda\subseteq \R^2_{(z,x_1)} \subseteq \R^{n}_{(z,x_1, \dots, x_{n-1})}
    \]
    in the natural way and consider $z$ the ``vertical'' direction. 
    
    \item Let $\Lambda_0\subseteq \Lambda$ be a curve corresponding to a single component of the corresponding Legendrian link. We extend $\Lambda_0$ to $\tilde{\Lambda}_0^{n-1}\subseteq \R^n$ by local $S^{n-2}$-symmetry around a vertical axis $\ell_{\Lambda_0}$ which depends on the component $\Lambda_0$; see, for example, \autoref{fig:spin_fronts}.  
    
    \item Isotropic ($n-2$)-dimensional spheres --- for example, those representing attaching spheres of $2n$-dimensional subcritical Weinstein ($n-1$)-handles --- are drawn as the union of two $2$-disks, or sometimes just one $2$-disk. The higher-dimensional projection is spun around a choice of vertical axis, which in the presence of two $2$-disks will coincide with the bisecting axis. This spun region is $S^{n-2} \times \D^2$, and represents the projection of the attaching region $S^{n-2} \times \D^{n+1}$ of the $2n$-dimensional subcritical Weinstein ($n-1$)-handle.
    
    \item When possible, the axis $\ell_{\Lambda_0}$ will be chosen to respect reflective symmetry of the Legendrian front $\Lambda_0$. In particular, local pieces of Legendrian fronts represented by parabolas are spun around the axis passing through their extreme point. See, for example, the dark blue Legendrian in \autoref{fig:spin_fronts}.
\end{itemize}

We further impose the following convention, distinct from \cite{casals2019legendrianfronts}: 

\begin{itemize}
    \item All of our cone singularities, that is, those obtained by a rotation of a transverse double point through an axis passing through the intersection point, are obtained as the result of the $\uplus$ operation (which includes handleslides), and so we do not adopt the thick dot convention of \cite{casals2019legendrianfronts} to distinguish cone singularities from genuine transverse intersections. That is, any regular double point in a Legendrian diagram that does not arise from a $\uplus$ operation is not spun to produce a cone singularity, but rather spun around an axis disjoint from the crossing. With the arguments below and in particular in Section \ref{sec: sigma times interval}, there should be no confusion. See, for example, the black Legendrian in \autoref{fig:spin_fronts}.
    
\end{itemize}

\begin{figure}[ht]
	\begin{overpic}[scale=0.43]{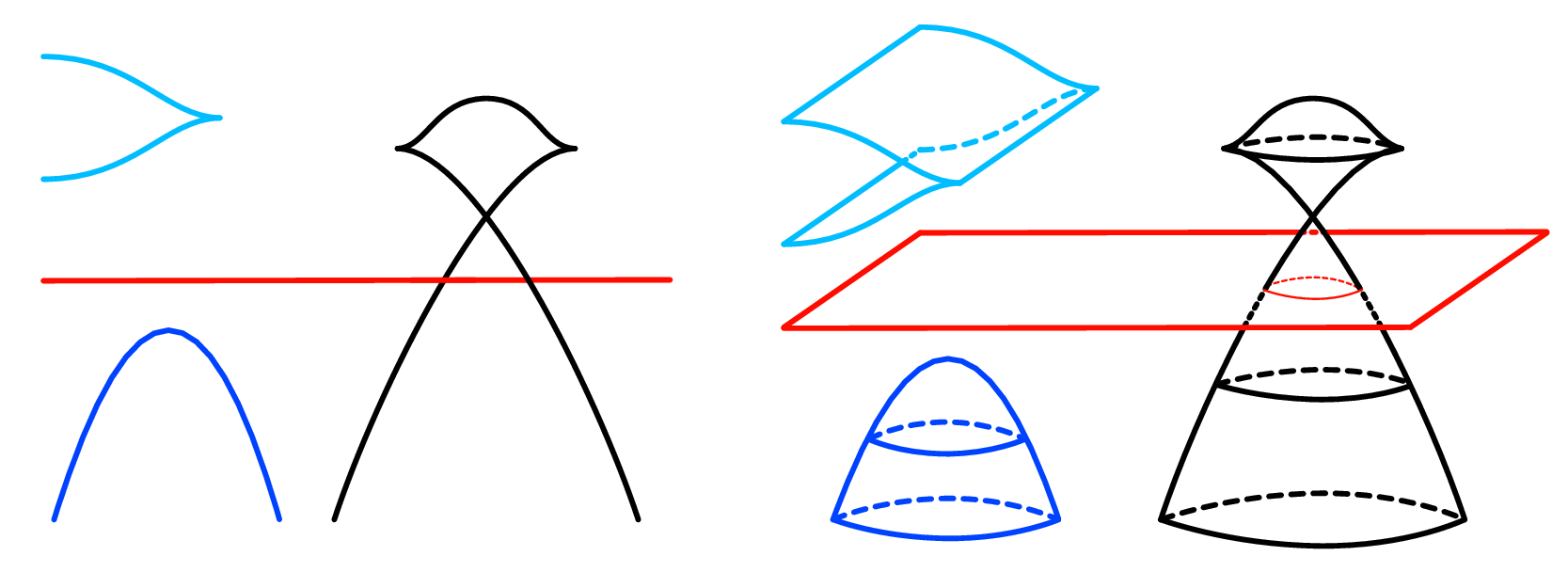}
	   
	\end{overpic}
	\caption{Legendrian front curves on the left that imply the (local) higher-dimensional Legendrian front projections on the right. From context, the black Legendrian arose from a $\uplus$ or handleslide operation and thus produces a cone singularity.}
	\label{fig:spin_fronts}
\end{figure}

\subsection{Contact \(n\)- and \((n+1)\)-handle attachments}

We review the theory of contact $n$-handles and ($n+1$)-handles, in particular highlighting the necessary attachment data. For more details, consult \cite{HH18} and \cite{Sac19}. 

In what follows, $(M^{2n+1}, \xi)$ is a contact manifold with convex boundary $\Sigma^{2n} = \partial M$. Fix an outward pointing contact vector field on $M$, let $\Gamma^{2n-1} \subseteq \Sigma$ be the corresponding dividing set of $\Sigma$, viewed as a contact submanifold of $(M, \xi)$, and let $R_{\pm}$ be the positive and negative regions of $\Sigma$, viewed as Liouville manifolds with ideal contact boundary $\Gamma$. We also write $\overline R_\pm = R_\pm \cup \Gamma$.  

The following two propositions from \cite{HH18} summarize the effect of attaching a critical contact handle (that is, either an $n$-handle or an $(n+1)$-handle) to $M$. 

\begin{remark}
For the sake of brevity, we will generally not distinguish between (finite-type) Liouville domains and their completions. 
\end{remark}

\begin{prop}[Contact $n$-handle attachment, Proposition 3.1.3 \cite{HH18}]\label{proposition:contact_n_handle}
Let $\Lambda^{n-1} \subseteq (\Gamma^{2n-1}, \xi|_{\Gamma})$ be an embedded Legendrian sphere in the dividing set. A contact $n$-handle can be attached to $M$ along $\Lambda$ to produce a new contact manifold $(M', \xi')$ with convex boundary $\Sigma' = R_+' \cup \Gamma' \cup R_-'$ such that:
\begin{enumerate}
    \item $R_{\pm}'$ is obtained from $R_{\pm}$ by attaching a Weinstein $n$-handle along $\Lambda \subseteq \partial R_{\pm}$.
    \item $\Gamma'$ is obtained from $\Gamma$ by performing a contact $(-1)$-surgery along $\Lambda$. 
\end{enumerate}
\end{prop}

\begin{prop}[Contact ($n+1$)-handle attachment, Proposition 3.2.6 \cite{HH18}]\label{proposition:contact_n+1_handle}
Let $\Lambda^{n-1} \subseteq (\Gamma^{2n-1}, \xi|_{\Gamma})$ be an embedded Legendrian sphere in the dividing set and let $D_{\pm}\subseteq R_{\pm}$ be regular, asymptotically cylindrical, open Lagrangian disks that asymptotically fill $\Lambda$. A contact ($n+1$)-handle can be attached to $M$ along $D_+ \cup \Lambda \cup D_-$ to produce a new contact manifold $(M', \xi')$ with convex boundary $\Sigma' = R_+' \cup \Gamma' \cup R_-'$ such that:
\begin{enumerate}
    \item $R_{\pm}'$ is obtained from $R_{\pm}$ by removing a standard neighborhood of $D_{\pm}$ from $R_{\pm}$.
    \item $\Gamma'$ is obtained from $\Gamma$ by performing a contact $(+1)$-surgery along $\Lambda$. 
\end{enumerate}
\end{prop}

\n
{\em Summary of attachment data.} To motivate our system of decorations on Legendrian diagrams, we highlight the attaching data for critical contact handles. 
\begin{itemize}
    \item A contact $n$-handle attachment to $(M^{2n+1}, \xi)$ is specified by an embedded Legendrian sphere $\Lambda^{n-1} \subseteq (\Gamma^{2n-1}, \xi|_{\Gamma})$ in the dividing set. 
    
    \item A contact ($n+1$)-handle attachment to $(M^{2n+1}, \xi)$ is specified by an embedded Legendrian sphere $\Lambda^{n-1} \subseteq (\Gamma^{2n-1}, \xi|_{\Gamma})$ in the dividing set, together with (regular, asymptotically cylindrical, open) Lagrangian disks $D_{\pm}^n\subseteq R_{\pm}$ that (asymptotically) fill $\Lambda$, i.e., informally $\partial D_{\pm} = \Lambda$. We call $\Lambda$ the {\em equator} of the handle attachment.
\end{itemize}
\noindent To keep track of and manipulate contact $n$-handles, Legendrians in the dividing set are sufficient. To keep track of all data involving contact ($n+1$)-handles via Legendrians in the dividing set, it is necessary to further describe Lagrangian disk fillings of Legendrians in the positive and negative regions. 

\s\n
{\em Decorated Legendrian diagrams.} We introduce a system of decorations to Legendrian surgery diagrams of dividing sets for the purpose of keeping track of Lagrangian disks with Legendrian boundary. As many of our arguments are completely local, our surgery diagrams will be local, drawn in the front projection, and typically enclosed by a rectangular frame.

Our decorated local Legendrian diagrams obey the following conventions:
\begin{enumerate}
    \item[(D0)] In the lower right corner of the frame we indicate the manifold represented by the empty diagram, and at the bottom center below the frame we indicate the manifold obtained after performing any indicated surgeries in the diagram.
    
    \item[(D1)] A $(-1)$-surgery coefficient (sometimes drawn outside the frame) means that a contact $n$-handle has been attached along the indicated Legendrian (as usual). 
    
    \item[(D2)] A Legendrian with a $\oplus$ node indicates that there is a preferred or specified Lagrangian disk in $R_+$ filling the Legendrian. The precise location of the node has no significance. Likewise, a Legendrian with a $\ominus$ node indicates a preferred or specified Lagrangian disk filling in $R_-$. A Legendrian may be decorated by any combination of $\oplus$ and $\ominus$ nodes; see any of the diagrams in \autoref{fig:decorations}. 
    
    \item[(D3)] A contact ($n+1$)-handle attachment is indicated by a Legendrian decorated by both a $\oplus$ and $\ominus$ node, together with a $(+1)$-surgery coefficient; see the top right diagram in \autoref{fig:decorations}.
    
    \item[(D4)] Let $\Lambda_0\subseteq \Gamma = \partial R_{\pm}$ be a Legendrian with a $(-1)$-surgery coefficient, indicating the attaching sphere of a contact $n$-handle to produce $\Gamma' = \partial R_{\pm}'$. Let $\Lambda$ be another Legendrian decorated by a $\oplus$ or $\ominus$ node, indicating a Lagrangian disk $D_{\pm}$ filling $\Lambda$ in $R_{\pm}'$. If $D_{\pm}$ intersects the cocore  of the Weinstein $n$-handle attached to $R_{\pm}$ along $\Lambda_0$ transversely and exactly once, then we indicate this with a solid node on $\Lambda_0$ and a dashed edge connecting the solid node to the $\oplus$ or $\ominus$ node on $\Lambda$; see the bottom diagram in \autoref{fig:decorations}. 
\end{enumerate}

\begin{remark}
We emphasize that the decorations described above are not sufficient to uniquely describe Lagrangian disk fillings of Legendrians; a Legendrian sphere could be filled by distinct disks, and our decorations cannot in general make the distinction. The nodes merely indicate that there is a chosen and preferred disk, and primarily offer a convenient way to keep track of what needs to be analyzed. 
\end{remark}

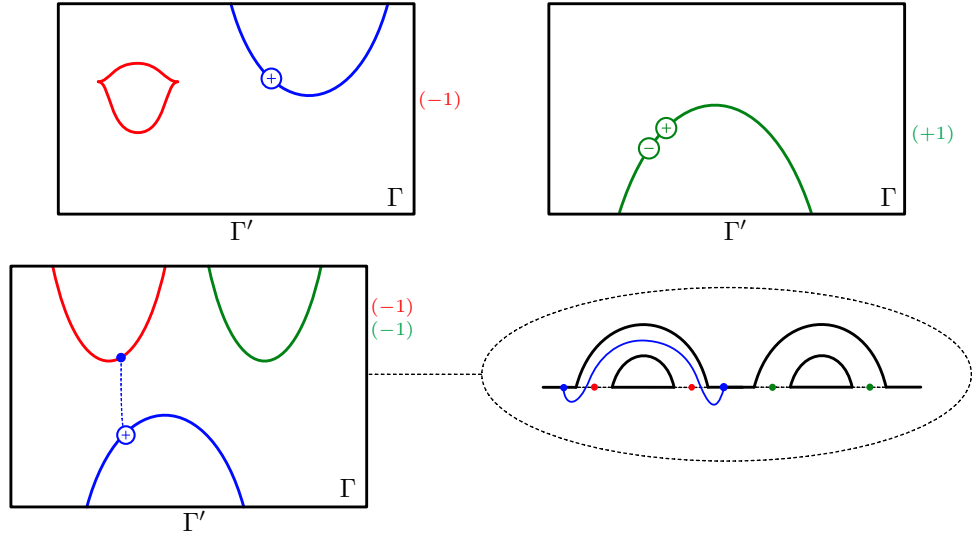
\begin{figure}[ht]
	\begin{overpic}[scale=0.65]{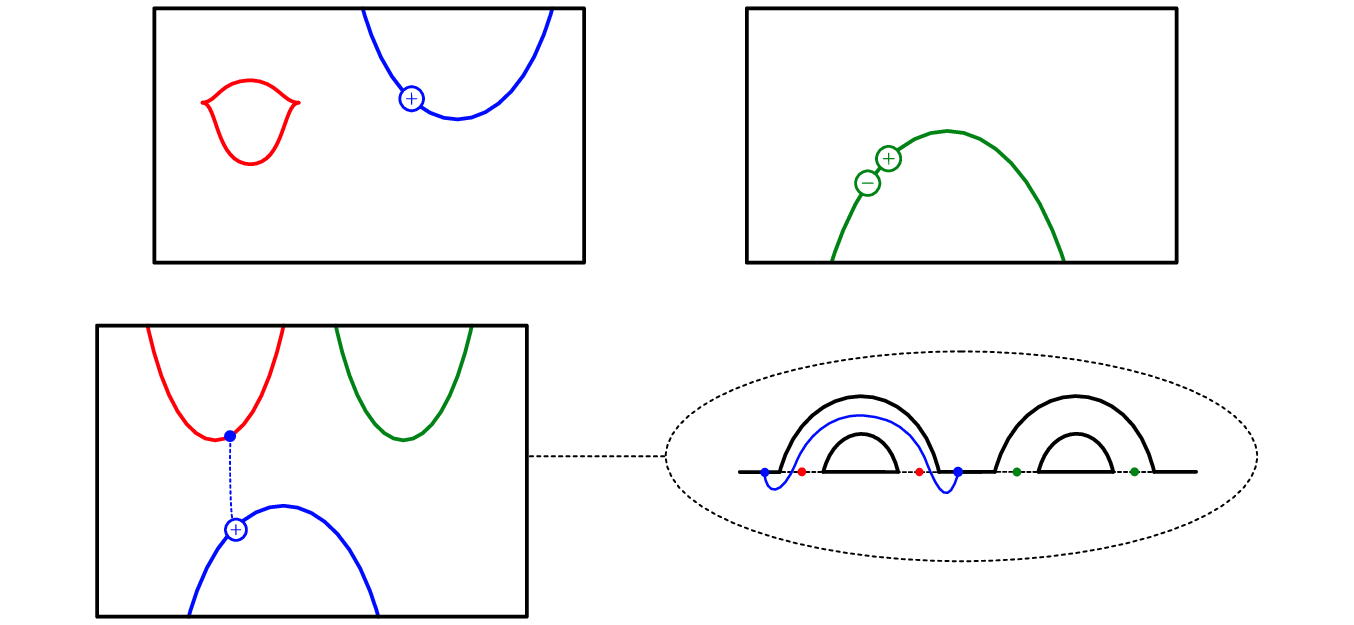}
	    \put(22.5,-1.5){\small $\Gamma'$}
	    \put(36.5,1.5){\small $\Gamma$}
	    \put(26.75,24.5){\small $\Gamma'$}
	    \put(40.75,27.5){\small $\Gamma$}
	    \put(70.75,24.5){\small $\Gamma'$}
	    \put(84.75,27.5){\small $\Gamma$}
	    \put(43.5,36.5){\tiny {\color{red}  $(-1)$}}
	    \put(39.25,18){\tiny {\color{red}  $(-1)$}}
	    \put(39.25,16){\tiny \textcolor{Green}{$(-1)$}}
	    \put(87.5,33.5){\tiny \textcolor{Green}{$(+1)$}}
	\end{overpic}
	\caption{Three different decorated local Legendrian surgery diagrams that exhibit the basic features of (D1)--(D4). In the top left, a contact $n$-handle is attached along the red Legendrian, while the blue Legendrian is filled by a Lagrangian disk in $R_+$. In the top right, the green Legendrian is filled by Lagrangians disks in both $R_{\pm}$ and a contact ($n+1$)-handle is attached along their union. The bottom diagram exhibits (D4). In particular, contact $n$-handles are attached along both the red and green Legendrians, and the blue Legendrian is filled by a disk in $R_+$ that intersects the cocore of the red Weinstein $n$-handle. To the right is an $n=1$ version for the sake of clarity.}
	\label{fig:decorations}
\end{figure}

\begin{remark}
An informal consequence of (D4) is that if a $(-1)$-surgered Legendrian does \textit{not} have any solid nodes, then the associated contact $n$-handle can be removed without affecting any Lagrangian disk fillings of other Legendrians. 
\end{remark}

\subsection{Handleslides}

Here we collect some handlesliding lemmas that will be frequently used. We begin with two that describe the effect of sliding {\em Legendrian-Lagrangian pairs} $(\Lambda; D)$, by which we mean a Legendrian $\Lambda \subseteq \Gamma$ and a Lagrangian disk in $R_{\pm}$ that fills it, across contact $n$- and ($n+1$)-handles. These are essentially reformulations of the Legendrian handleslide lemma (Lemma 5.2.3) from \cite{HH18}; see also Proposition 2.14 of \cite{casals2019legendrianfronts}.

In both lemmas, we let $(M^{2n+1},\xi)$ be a contact manifold with convex boundary $\Sigma = R_+\cup \Gamma \cup R_-$. Let $\Lambda, \Lambda_0\subseteq \Gamma$ be Legendrians in the dividing set that intersect $\xi|_{\Gamma}$-transversely at one point. When $n=1$, $\Lambda,\Lambda_0\simeq S^0$ and ``intersecting $\xi|_{\Gamma}$-transversely at one point" means intersecting at a common point $p$. Let $D_{\pm}\subseteq R_{\pm}$ be Lagrangian disk fillings of $\Lambda$, and $D_{0,\pm}\subseteq R_{\pm}$ be Lagrangian disk fillings of $\Lambda_0$.  The Legendrian sum operation $\Lambda\uplus \Lambda_0$ and the Legendrian boundary sum operation $D_\pm \uplus_b D_{0,\pm}$ are as defined in \cite[Section 4]{HH18}.  When $n=1$ we set $\Lambda\uplus \Lambda_0=\Lambda_0\uplus \Lambda=(\Lambda\cup \Lambda_0)\setminus \{p\}$ and $D_\pm \uplus_b D_{0,\pm}=D_\pm \cup D_{0,\pm}\cup \{p\}$. Finally let $\Lambda^{\ve,\Gamma}$ be the time-$\ve$ pushoff of $\Lambda$ with respect to some Reeb vector field on $\Gamma$ and let $D_\pm^{\ve,\Gamma}$ be the induced Hamiltonian pushoff of $D_\pm$.  We will often omit the superscript $\Gamma$ when it is understood from the context.

\begin{lemma}[Sliding over $n$-handles]\label{lemma:legendrian_handleslide_n_handle}
Let $(M', \xi')$ with convex boundary $\Sigma' = R_+' \cup \Gamma' \cup R_-'$ be obtained by attaching a contact $n$-handle $h_n$ to $M$ along $\Lambda_0$, so that $\Gamma'$ is obtained by ($-1$)-surgery along $\Lambda_0$. Let $K$ be the Legendrian core of $h_n$, which can be viewed as the Lagrangian core of the corresponding Weinstein handles attached to $R_{\pm}$.
\begin{enumerate}
    \item[(i)] Sliding $\Lambda^{-\ve}$ up over the ($-1$)-surgery along $\Lambda_0$ produces new Legendrian-Lagrangian pairs
    \[
    \left((\Lambda\uplus \Lambda_0)^{\ve,\Gamma} ; \, (D_{\pm}\uplus_b K)^{\ve,\Gamma} \right).
    \]
    \item[(ii)] Sliding $\Lambda^{\ve}$ down over the ($-1$)-surgery along $\Lambda_0$ produces new Legendrian-Lagrangian pairs 
    \[
    \left((\Lambda_0\uplus \Lambda)^{-\ve,\Gamma} ; \, (K\uplus_b D_{\pm})^{-\ve,\Gamma} \right).
    \]
\end{enumerate}
Here the time $\pm \ve$-pushoff is taken with respect to $\Gamma$ so that $(\Lambda\uplus \Lambda_0)^{\ve,\Gamma},(\Lambda_0\uplus \Lambda)^{-\ve,\Gamma} \subset \Gamma'$.
\end{lemma}

\begin{figure}[ht]
	\begin{overpic}[scale=0.53]{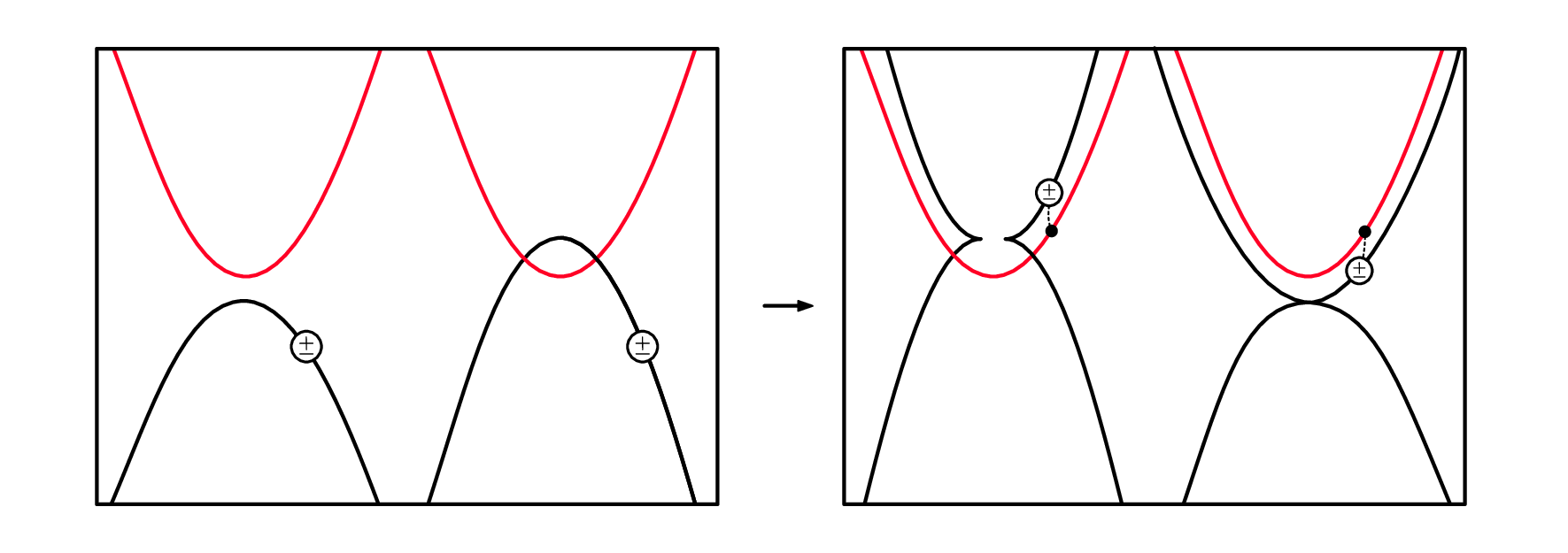}
	    \put(25,0.5){\small $\Gamma'$}
	    \put(10.5,8){\small $\Lambda^{-\ve}$}
	    \put(30,8){\small $(\Lambda')^{\ve}$}
	    \put(31,25){\small {\color{red}  $\Lambda_0'$}}
	    \put(11,25){\small {\color{red}  $\Lambda_0$}}
	    \put(46.5,28){\tiny {\color{red}  $(-1)$}}
	    \put(46.5,26){\tiny {\color{red}  $(-1)$}}
	    \put(73,0.5){\small $\Gamma'$}
	    \put(57.5,8){\small $(\Lambda\uplus \Lambda_0)^{\ve}$}
	    \put(78.5,8){\small $(\Lambda_0' \uplus \Lambda')^{-\ve}$}
	    \put(94,28){\tiny {\color{red}  $(-1)$}}
	    \put(94,26){\tiny {\color{red}  $(-1)$}}
	\end{overpic}
	\caption{The statement of \autoref{lemma:legendrian_handleslide_n_handle}, sliding across contact $n$-handles.}
	\label{fig:sliden}
\end{figure}

\begin{lemma}[Sliding over ($n+1$)-handles]\label{lemma:legendrian_handleslide_n+1_handle}
Let $(M', \xi')$ with convex boundary $\Sigma' = R_+' \cup \Gamma' \cup R_-'$ be obtained by attaching a contact ($n+1$)-handle $h_{n+1}$ to $M$ along $D_{0,+}\cup D_{0,-}$, so that $\Gamma'$ is obtained by ($+1$)-surgery along $\Lambda_0$. 
\begin{enumerate}
    \item[(i)] Sliding $\Lambda^{-\ve}$ up over the ($+1$)-surgery along $\Lambda_0$ produces new Legendrian-Lagrangian pairs 
    \[
    \left((\Lambda_0\uplus \Lambda)^{\ve} ; \, (D_{0,\pm} \uplus_b D_{\pm})^{\ve} \right).
    \]
    \item[(ii)] Sliding $\Lambda^{\ve}$ down over the ($+1$)-surgery along $\Lambda_0$ produces new Legendrian-Lagrangian pairs 
    \[
    \left((\Lambda\uplus \Lambda_0)^{-\ve} ; \, (D_{\pm} \uplus_b D_{0,\pm})^{-\ve} \right).
    \]
\end{enumerate}
\end{lemma}

\begin{figure}[ht]
	\begin{overpic}[scale=0.53]{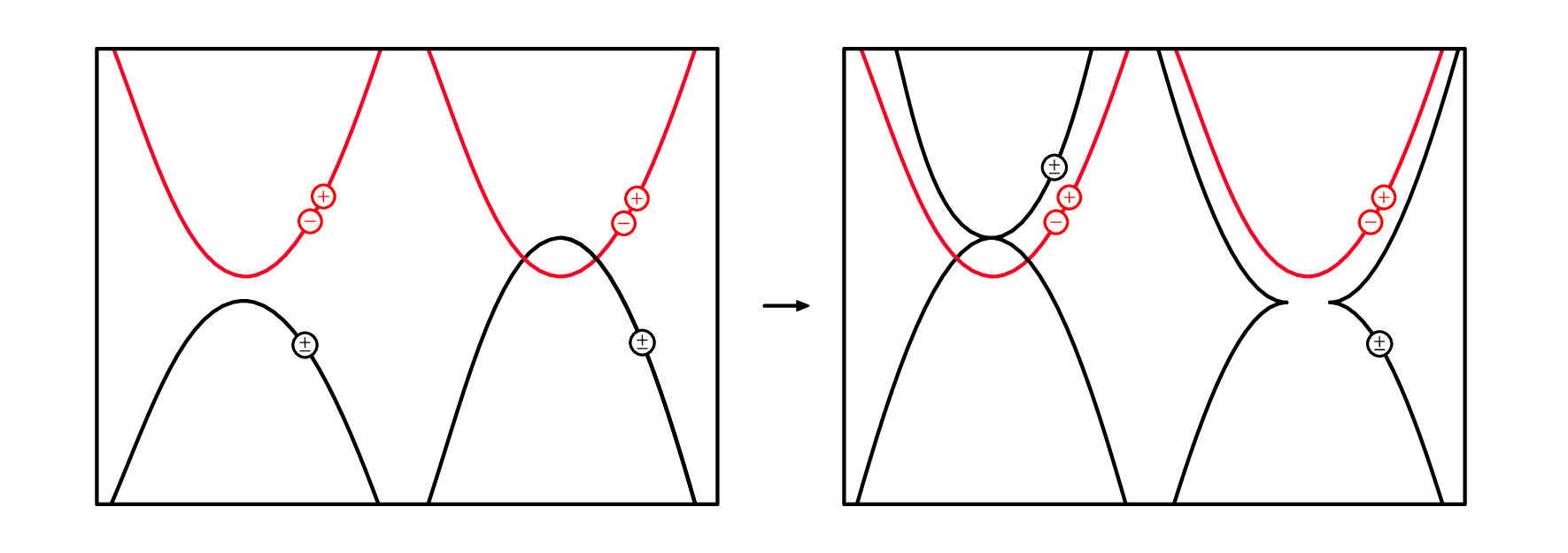}
	    \put(25,0.5){\small $\Gamma'$}
	    \put(10.5,8){\small $\Lambda^{-\ve}$}
	    \put(30,8){\small $(\Lambda')^{\ve}$}
	    \put(31,25){\small {\color{red}  $\Lambda_0'$}}
	    \put(11,25){\small {\color{red}  $\Lambda_0$}}
	    \put(46.5,28){\tiny {\color{red}  $(+1)$}}
	    \put(46.5,26){\tiny {\color{red}  $(+1)$}}
	    \put(73,0.5){\small $\Gamma'$}
	    \put(57.5,8){\small $(\Lambda_0\uplus \Lambda)^{\ve}$}
	    \put(78.5,8){\small $(\Lambda' \uplus \Lambda_0')^{-\ve}$}
	    \put(94,28){\tiny {\color{red}  $(+1)$}}
	    \put(94,26){\tiny {\color{red}  $(+1)$}}
	\end{overpic}
	\caption{The statement of \autoref{lemma:legendrian_handleslide_n+1_handle}, sliding across contact ($n+1$)-handles.}
	\label{fig:sliden1}
\end{figure}
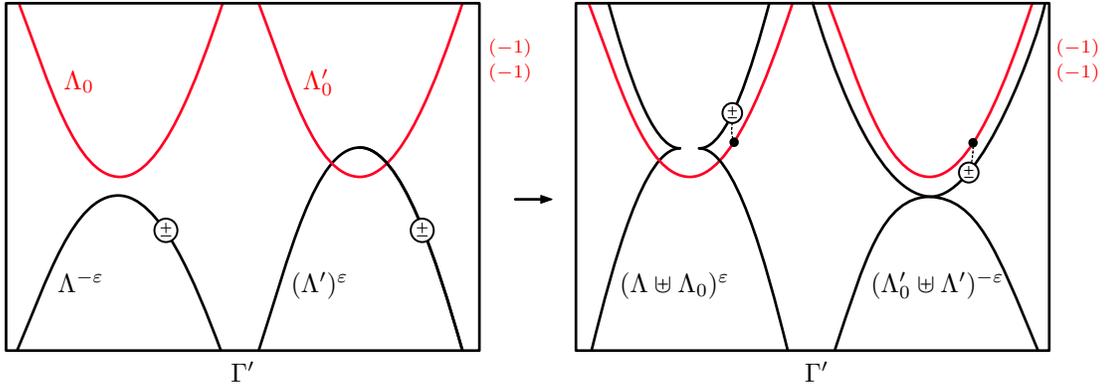

It is often useful to diagrammatically identify the Lagrangian disk that is the cocore of a Weinstein $n$-handle attached along a $(-1)$-surgered Legendrian. 

\begin{lemma}[Cocores in decorated Legendrian diagrams]\label{lemma:cocore_lemma}
Let $\Sigma^{2n} = R_+ \cup \Gamma \cup R_-$ be the convex boundary of $(M^{2n+1}, \xi)$. Let $h_n$ be a contact $n$-handle attached to $M$ along $\Lambda_0 \subseteq \Gamma$, and let $h_n^{\pm}$ denote the corresponding Weinstein $n$-handles attached to $R_{\pm}$. The belt sphere and cocore pair $(\Lambda; D_{\pm})$ of $h_n^{\pm}$, after applying suitable Hamiltonian isotopies, can be described in a decorated Legendrian diagram by any of the following in the $n>1$ case; see the right side of \autoref{fig:cocorelemma}.
\begin{enumerate}
    \item[(i)] The Legendrian $\Lambda^{\delta}$, decorated with a $\opm$-node and a dashed edge connected to a solid node on $\Lambda_0$. Here $\delta$ is a positive or negative number and $\Lambda^\delta$ is obtained by applying a time-$\delta$ Reeb flow in the surgered dividing set $\Gamma'$ to $\Lambda$ so that $\Lambda^\delta\cap h_n=\varnothing$.
    
    \item[(ii)] A standard $\mathrm{tb} = -1$ Legendrian $\Lambda'$, linked once with $\Lambda_0$ and decorated with a $\opm$-node.
\end{enumerate}
The contact isotopy on $\Gamma'$ taking $\Lambda$ to $\Lambda^\delta$ or $\Lambda'$ extends to a Hamiltonian isotopy of $D_\pm$.
\end{lemma}

The $n=1$ case is given on the left side of \autoref{fig:cocorelemma}.

\begin{proof}
The time-$\delta$ Reeb flow on $\Gamma'$ described in (i) extends to a Hamiltonian isotopy $\phi$ of $D_{\pm}$ in $R'_{\pm}$, obtained from $R_\pm$ by attaching $h_n^\pm$. The isotoped disk $\phi(D_{\pm})$ has nonzero intersection number with $D_{\pm}$, and hence the diagrammatic description (i) follows from decoration convention (D4). 

The second description (ii) is obtained by sliding $\Lambda^{\delta}$ with $\delta<0$, obtained originally via (i), up over the $(-1)$-surgery along $\Lambda$. 
\end{proof}

\begin{remark}
Observe that, by a sequence of Reidemeister moves, the cusps of the belt sphere in the bottom middle diagram on the right-hand side of \autoref{fig:cocorelemma} may also be positioned below the surgered strand.
\end{remark}

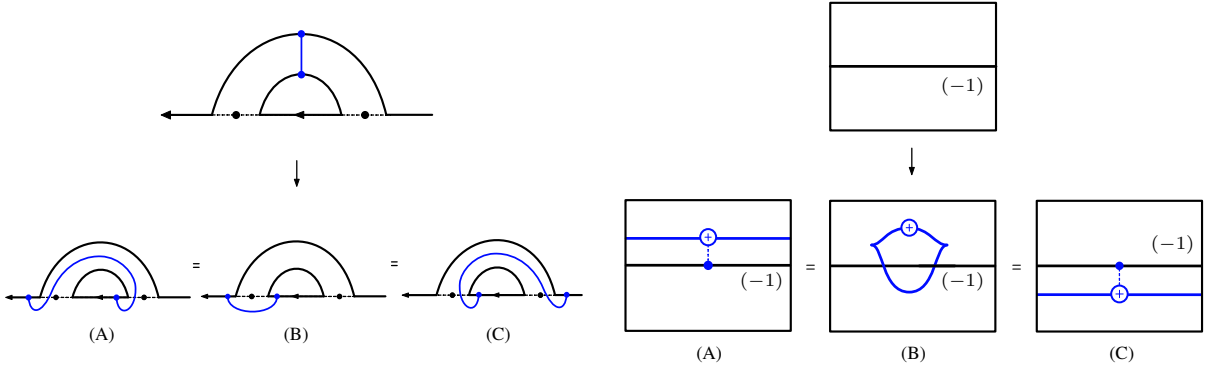
\begin{figure}[ht]
	\begin{overpic}[scale=0.54]{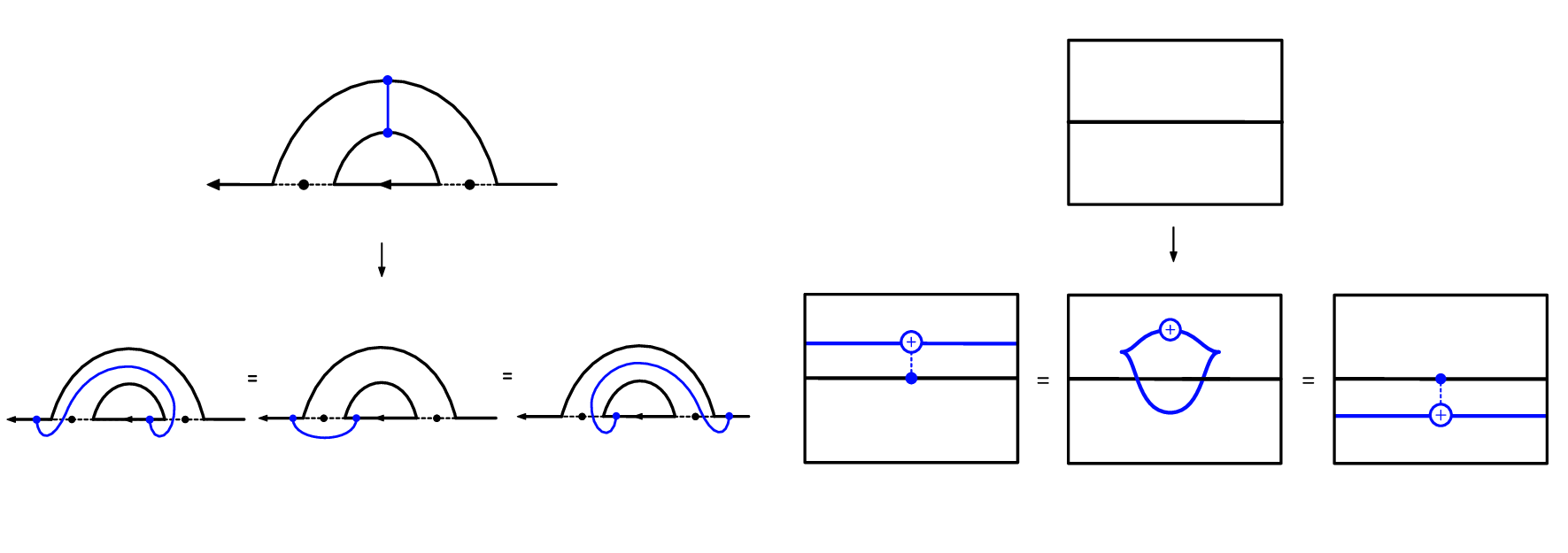}
	    \put(94.5,12.5){\tiny $(-1)$}
	    \put(77.25,9.5){\tiny $(-1)$}
	    \put(60.75,9.5){\tiny $(-1)$}
	    \put(77.25,25.5){\tiny $(-1)$}
	    \put(57.2,3.5){\tiny (A)} \put(74,3.5){\tiny (B)} \put(91.1,3.5){\tiny (C)}
	    \put(7.3,5){\tiny (A)} \put(23.3,5){\tiny (B)} \put(39.9,5){\tiny (C)}
	\end{overpic}
	\caption{On the right half, identifying isotopic pushoffs of a cocore of a Weinstein $n$-handle attached to $R_+$ as described by \autoref{lemma:cocore_lemma}. On the left half is the $n=1$ case. (A) is a positive Reeb pushoff, (B) is a nonintersecting pushoff, and (C) is a negative Reeb pushoff.  The same figures hold for handles attached to $R_-$.}
	\label{fig:cocorelemma}
\end{figure}

\subsection{Bypass attachments} A bypass attachment is a special type of smoothly canceling pair of contact $n$- and ($n+1$)-handles that may or may not be canceling at the level of contact structures. Here we review the basic theory from \cite{HH18}. We also introduce some new language and make use of decorations to clarify the diagrammatic description of a bypass. 

\begin{defn}
Let $\Sigma^{2n} = R_+ \cup \Gamma \cup R_-$ be the convex boundary of $(M^{2n+1}, \xi)$. A tuple of {\em bypass attachment data} is a quadruple $(\Lambda_{-}, \Lambda_+; D_{-}, D_+)$ where each of $(\Lambda_{\pm}; D_{\pm})$ is a Legendrian-Lagrangian pair in $R_{\pm}$ such that $\Lambda_-$ and $\Lambda_+$ intersect $\xi_{\Gamma}$-transversely at one point.  
\end{defn}

The following theorem is essentially Theorem 5.1.3 from \cite{HH18}.

\begin{theorem}[Bypass attachment]\label{theorem:bypass_attachment}
Let $\Sigma^{2n} = R_+ \cup \Gamma \cup R_-$ be the convex boundary of $(M^{2n+1}, \xi)$ with bypass attachment data $(\Lambda_{-}, \Lambda_+; D_{-}, D_+)$. Then a pair of smoothly canceling contact $n$- and ($n+1$)-handles, called a {\em bypass attachment,} can be attached according to either of the following models:
\begin{enumerate}
    \item[($R_+$)] Attach a contact $n$-handle along $\Lambda_- \uplus \Lambda_+$ and a contact ($n+1$)-handle along $\tilde{D}_- \cup \Lambda_+^{-\ve} \cup D_+^{-\ve}$, where $\tilde{D}_-$ is obtained by sliding $D_-^{\ve}$ down across the $(-1)$-surgery along $\Lambda_- \uplus \Lambda_+$. 
    \item[($R_-$)] Attach a contact $n$-handle along $\Lambda_- \uplus \Lambda_+$ and a contact ($n+1$)-handle along $D_-^{\ve} \cup \Lambda_-^{\ve} \cup \tilde{D}_+$, where $\tilde{D}_+$ is obtained by sliding $D_+^{-\ve}$ up across the $(-1)$-surgery along $\Lambda_- \uplus \Lambda_+$. 
\end{enumerate}
The two models, called the {\em $R_{\pm}$-centric models,} respectively, are identified by a handleslide of the ($n+1$)-handle across the $n$-handle; see \autoref{fig:bypass_attachment}. 
\end{theorem}

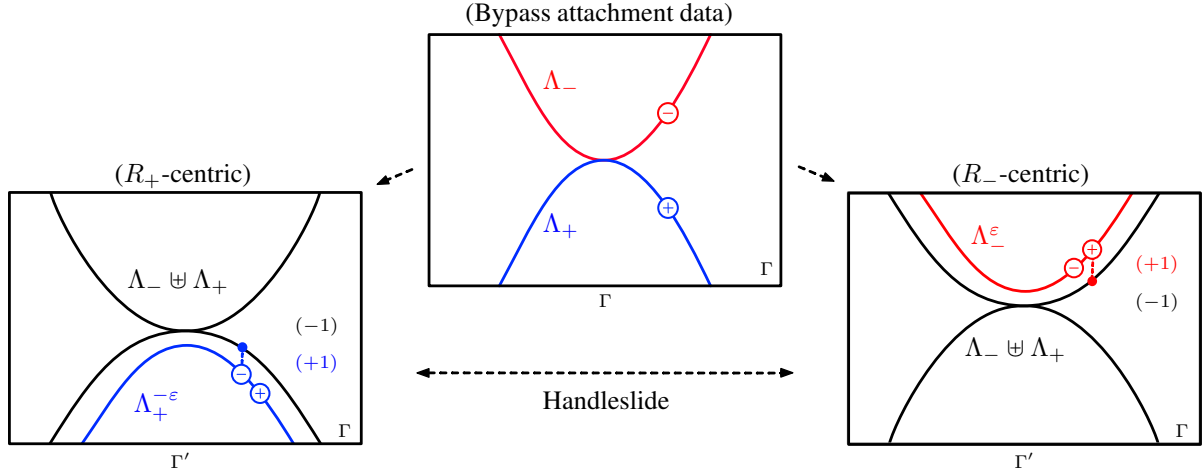
\begin{figure}[ht]
	\begin{overpic}[scale=0.55]{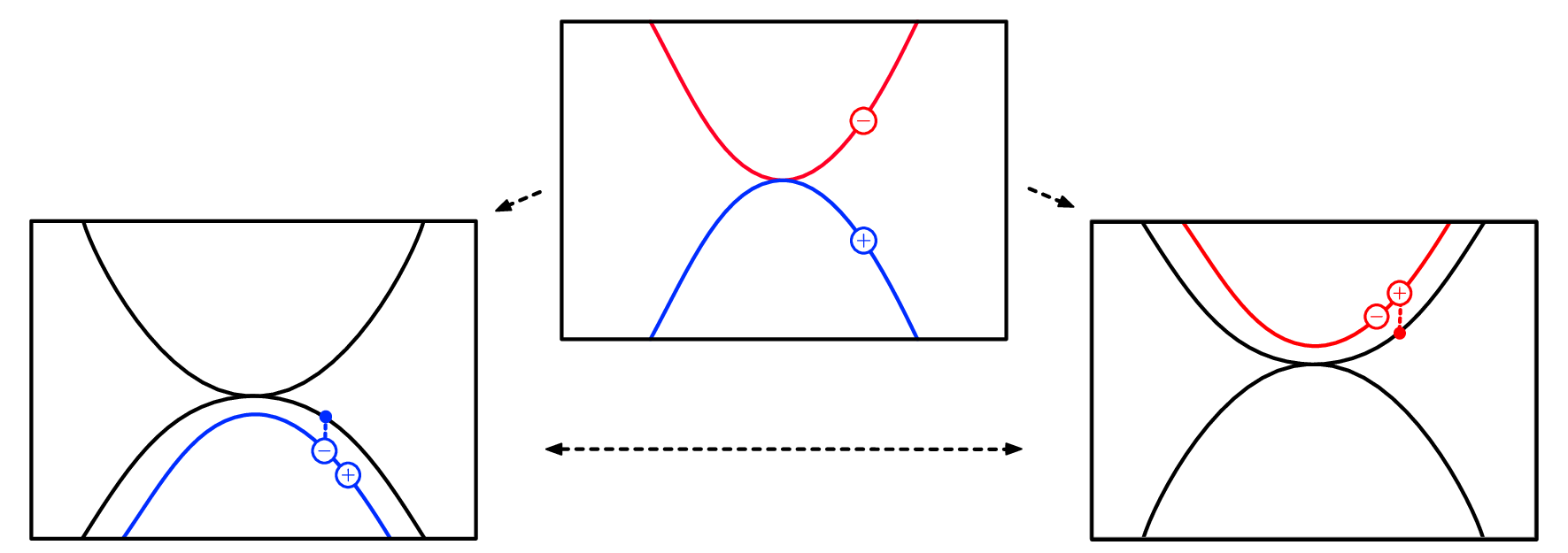}
	    \put(38.75,35){\small (Bypass attachment data)}
	    \put(49.5,12){\tiny $\Gamma$}
	    \put(62.5,14.5){\tiny $\Gamma$}
	    \put(45,29.5){\small {\color{red}  $\Lambda_-$}}
	    \put(45,18){\small \textcolor{blue}{$\Lambda_{+}$}}
	    \put(78,22){\small ($R_-$-centric)}
	    \put(83,-1){\tiny $\Gamma'$}
	    \put(96,1.5){\tiny $\Gamma$}
	    \put(92.75,12){\tiny $(-1)$}
	    \put(92.75,15){\tiny {\color{red}  $(+1)$}}
	    \put(79.5,17){\small {\color{red}  $\Lambda_-^{\ve}$}}
	    \put(79,8){\small $\Lambda_- \uplus \Lambda_+$}
	    \put(10.5,22){\small ($R_+$-centric)}
	    \put(15,-1){\tiny $\Gamma'$}
	    \put(28.5,1.5){\tiny $\Gamma$}
	    \put(25,10){\tiny $(-1)$}
	    \put(25,7){\tiny \textcolor{blue}{$(+1)$}}
	    \put(12,3.5){\small \textcolor{blue}{$\Lambda_+^{-\ve}$}}
	    \put(11.5,13.5){\small $\Lambda_- \uplus \Lambda_+$}
		\put(45,3.75){\small Handleslide}	
	\end{overpic}
	\caption{The two equivalent models of a bypass attachment according to \autoref{theorem:bypass_attachment}. In the lower figures, $\Lambda_-^{\ve}$ and $\Lambda_{+}^{-\ve}$ are isotopic via a handleslide across $\Lambda_- \uplus \Lambda_+$.}
	\label{fig:bypass_attachment}
\end{figure}

\begin{remark}
While the $R_{\pm}$-centric models are equivalent via a handleslide, we distinguish them in this paper for the purpose of decorated Legendrian calculations. It is often convenient to use one model over the other when manipulating multiple bypass attachments. Furthermore, we call them the $R_{\pm}$-centric models because they each give a natural description of $R_{\pm}'$, respectively, in terms of Weinstein handles attached to and removed from $R_{\pm}$. 
\end{remark}

\subsection{Trivial bypasses and stabilizations of POBDs}\label{subsection:trivial_bypasses_and_POBD_stabilization}

Informally, a trivial bypass is a bypass attachment which also cancels on the level of contact structures, not just smoothly. For convenience we review the definition from \cite{HH18}. 

\begin{defn}\label{def: trivial bypass data}
The bypass attachment data $(\Lambda_{\pm}; D_{\pm})$ is {\em trivial} if either of the following holds: 
\begin{itemize}
    \item[(TB1)] $(\Lambda_+; D_+)$ is a pair consisting of a standard Legendrian and standard Lagrangian disk such that $\Lambda_+$ is below $\Lambda_-$\footnote{More precisely, $\Lambda_+$ is below $\Lambda_-$ if in the local front projection given by \autoref{fig:trivial_data} the $z$-coordinates of $\Lambda_+$ are smaller than those of $\Lambda_-$, except at the unique point of intersection $\Lambda_-\cap \Lambda_+$.} (in the $n=1$ case a standard Lagrangian disk cuts off a half-disk $D\subset R_+$ with boundary $\delta\cup D_+$, where $\delta\subset \Gamma$ lies right below $\Lambda_+\cap \Lambda_-$).
    \item[(TB2)] $(\Lambda_-; D_-)$ is a pair consisting of a standard Legendrian and standard Lagrangian disk such that $\Lambda_-$ is above $\Lambda_+$ (in the $n=1$ case $D_-$ cuts off a half disk $D\subset R_-$ with boundary $\delta\cup D_-$, where $\delta\subset \Gamma$ lies right above $\Lambda_+\cap \Lambda_-$).
\end{itemize}

A bypass with trivial bypass data is itself called {\em trivial}.  See \autoref{fig:trivial_data}. 
\end{defn}

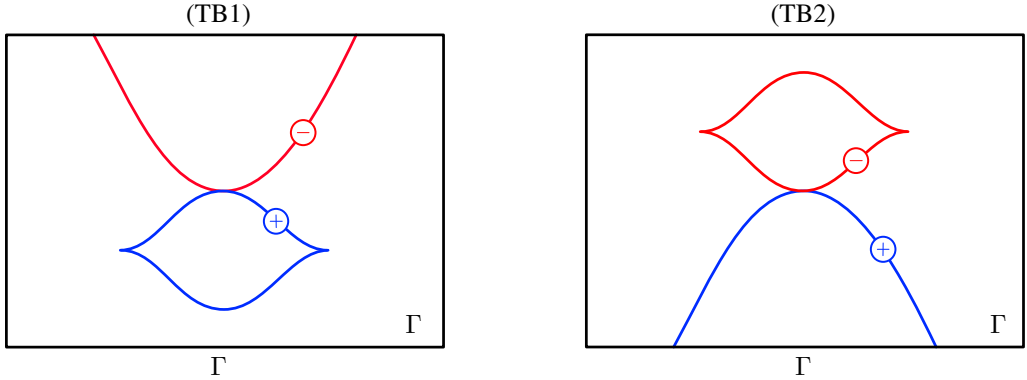
\begin{figure}[ht]
	\begin{overpic}[scale=0.54]{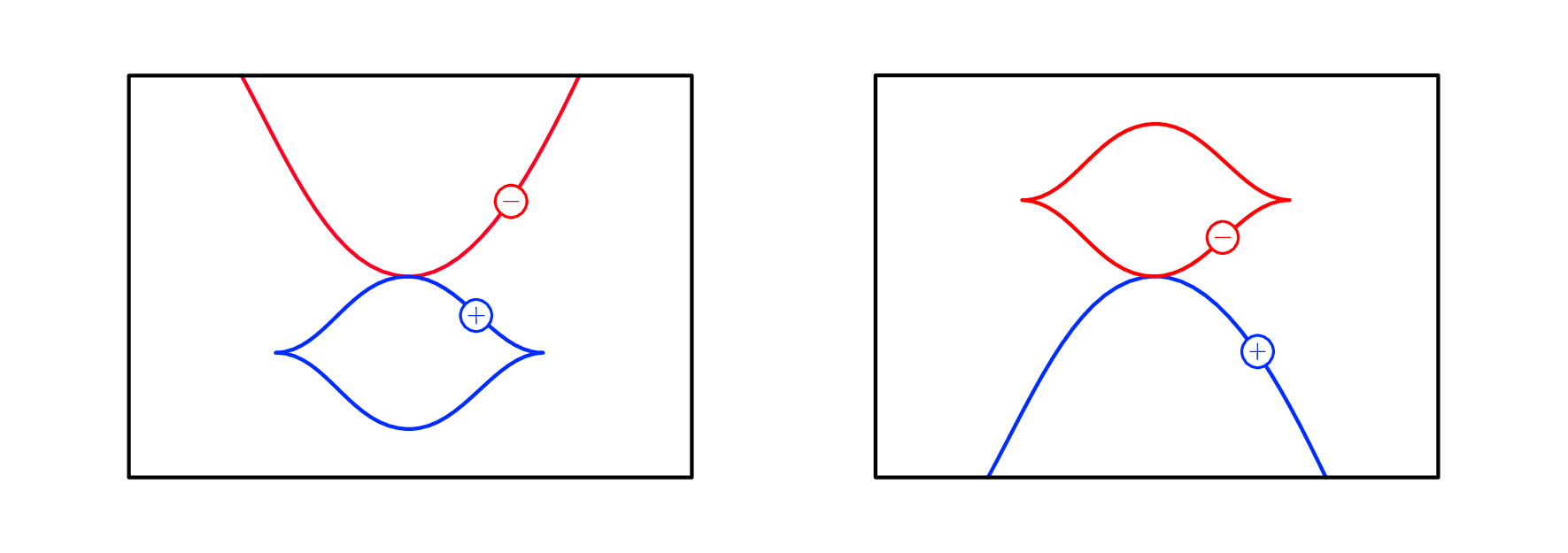}
	    \put(23,31.5){\small (TB1)}
	    \put(25,2.5){\small $\Gamma$}
	    \put(41,6){\small $\Gamma$}
	    \put(71,31.5){\small (TB2)}
	    \put(73,2.5){\small $\Gamma$}
	    \put(89,6){\small $\Gamma$}
	\end{overpic}
	\caption{The two types of trivial bypass data described in \autoref{def: trivial bypass data}.}
	\label{fig:trivial_data}
\end{figure}

\begin{prop}[Proposition 8.3.2 \cite{HH18}]
A trivial bypass attachment to $\Sigma \times \{0\}$ produces a vertically invariant contact structure on $\Sigma \times [0,1]$. That is, a trivial bypass does not change the isotopy class of the contact structure. 
\end{prop}

Of importance to this paper is the ability to identify trivial bypasses directly via pairs of contact handles. 

\begin{lemma}[Identifying trivial bypasses]\label{lemma:trivial_bypass_lemma}
Let $\Sigma^{2n} = R_+ \cup \Gamma \cup R_-$ be the convex boundary of $(M^{2n+1}, \xi)$. Let $h_n$ be a contact $n$-handle attached to $M$ along $\Lambda_n \subseteq \Gamma$, and let $h_{n+1}$ be a contact ($n+1$)-handle attached to $M \cup h_n$ along $D_+ \cup \Lambda_{n+1} \cup D_-$. Then the pair of handles $h_n \cup h_{n+1}$ forms a trivial bypass if either of the following conditions hold; see \autoref{fig:trivialbypasslemma}:
\begin{itemize}
    \item[(TB1)] $D_+$ is a positive Reeb shift of the cocore of the $n$-handle attached to $R_+$ along $\Lambda_n$, and $D_- \subseteq R_-$. 
    \item[(TB2)] $D_-$ is a negative Reeb shift of the cocore of the $n$-handle attached to $R_-$ along $\Lambda_n$, and $D_+ \subseteq R_+$.
\end{itemize}
\end{lemma}

\begin{figure}[ht]
	\begin{overpic}[scale=0.54]{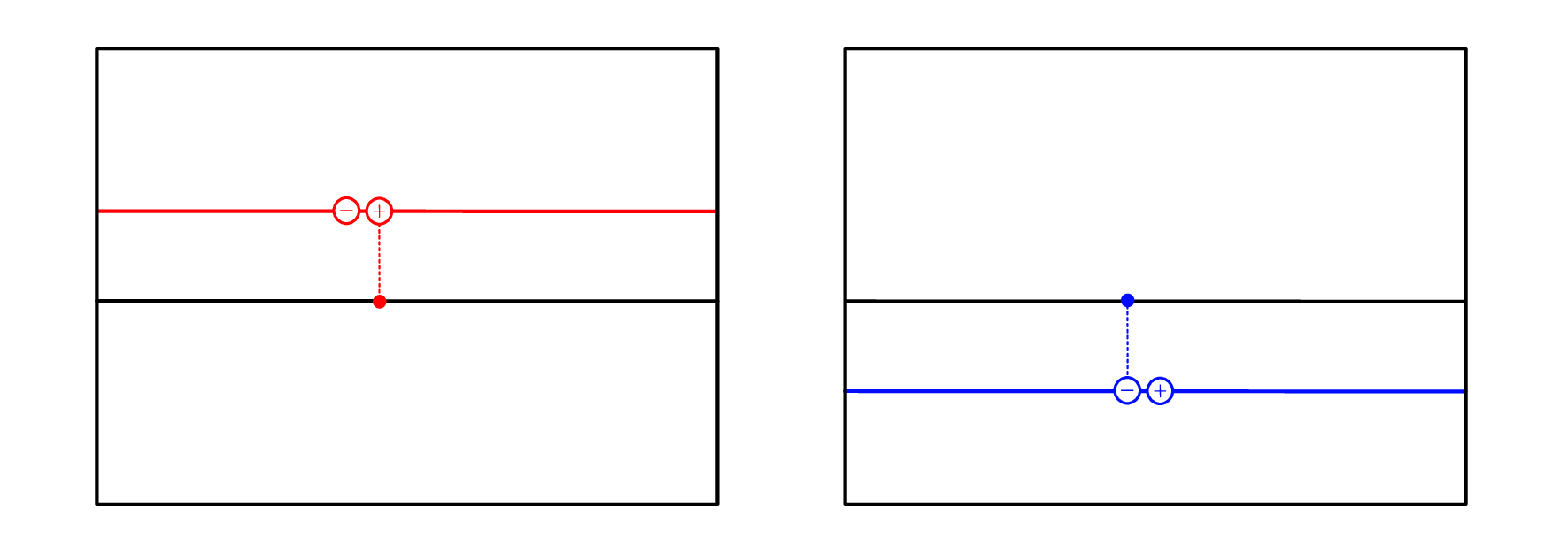}
	    \put(23,33){\small (TB1)}
	    \put(25,0.5){\small $\Gamma'$}
	    \put(43,4){\small $\Gamma$}
	    \put(34,13.5){\small $\Lambda_n$}
	    \put(34,23){\small {\color{red}  $\Lambda_{n+1}$}}
	    \put(46.5,15.5){\tiny $(-1)$}
	    \put(46.5,21.5){\tiny {\color{red}  $(+1)$}}
	    \put(71,33){\small (TB2)}
	    \put(73,0.5){\small $\Gamma'$}
	    \put(91,4){\small $\Gamma$}
	    \put(94,15.5){\tiny $(-1)$}
	    \put(94,10){\tiny \textcolor{blue}{$(+1)$}}
	    \put(82,17){\small $\Lambda_n$}
	    \put(82,8){\small \textcolor{blue}{$\Lambda_{n+1}$}}
	\end{overpic}
	\caption{The two trivial bypasses described in \autoref{lemma:trivial_bypass_lemma}. On the left, the positive Lagrangian disk filling of the red Legendrian is isotopic to the cocore of the Weinstein $n$-handle attached to $R_+$ along the black Legendrian. The same holds for the negative Legendrian disk filling of the blue Legendrian on the right. We again warn the reader that the dotted lines in the diagram do not uniquely determine the specific disks $D_+$ in (TB1) and $D_-$ in (TB2).}
	\label{fig:trivialbypasslemma}
\end{figure}
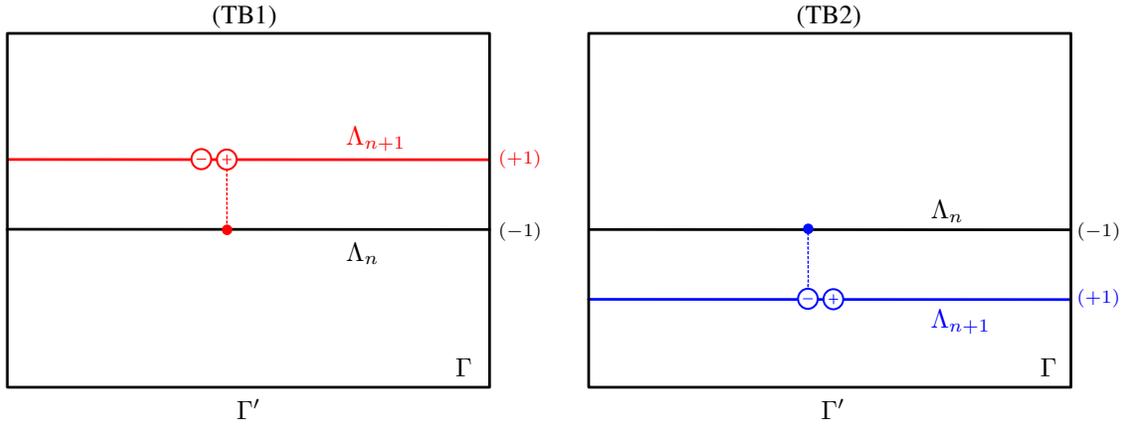

\begin{proof}[Proof of \autoref{lemma:trivial_bypass_lemma}.]
We prove the lemma for the (TB1) case; the (TB2) case is completely analogous. Let $h_n, h_{n+1}$ be a pair of contact handles attached to $(M,\xi)$ according to (TB1). This pair of handles arises from the bypass attachment data $(\tilde{\Lambda}_{\pm}; \tilde{D}_{\pm})$, where $\tilde{\Lambda}_- = \Lambda_{n+1}^{-\ve}$, $\tilde{D}_- = D_-^{-\ve}$, and  $(\tilde{\Lambda}_{+}; \tilde{D}_+)$ is a standard Legendrian unknot with standard Lagrangian disk filling, positioned below $\tilde\Lambda_-$ as in \autoref{fig:trivial_bypass_proof}. 
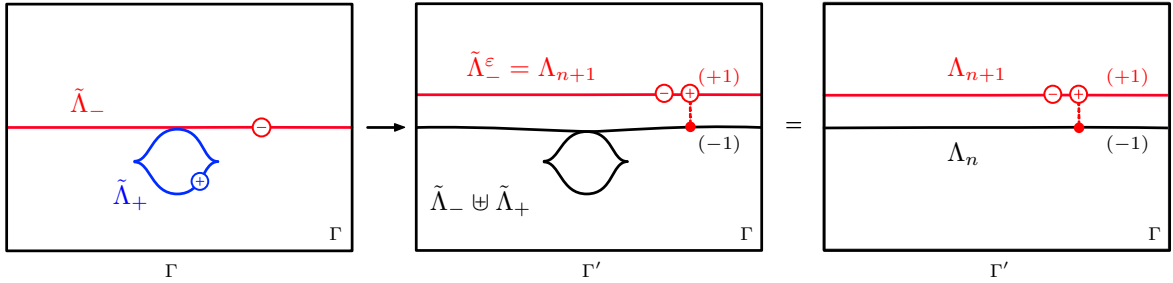
\begin{figure}[ht]
	\begin{overpic}[scale=0.54]{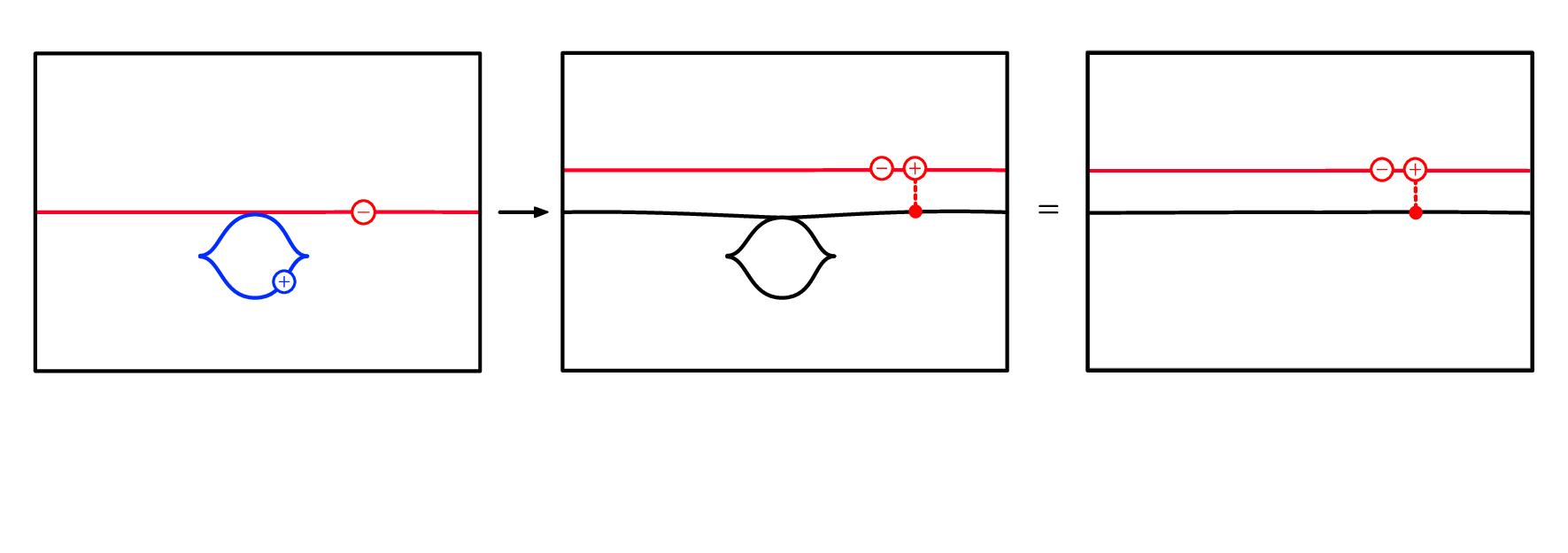}
	    \put(49.5,9.5){\tiny $\Gamma'$}
	    \put(62.5,12.5){\tiny $\Gamma$}
	    \put(40,26){\small {\color{red} $\tilde{\Lambda}_-^{\ve} = \Lambda_{n+1}$}}
	    \put(37,15){\small $\tilde{\Lambda}_- \uplus\tilde{\Lambda}_{+}$}
	    \put(59,20){\tiny $(-1)$}
	    \put(59,25.5){\tiny {\color{red}  $(+1)$}}
	    \put(83,9.5){\tiny $\Gamma'$}
	    \put(96,12.5){\tiny $\Gamma$}
	    \put(79.5,26){\small {\color{red}  $\Lambda_{n+1}$}}
	    \put(79.5,19){\small $\Lambda_n$}
	    \put(92.5,20){\tiny $(-1)$}
	    \put(92.5,25.5){\tiny {\color{red}  $(+1)$}}
	    \put(15.25,9.5){\tiny $\Gamma$}
	    \put(28.75,12.5){\tiny $\Gamma$}
	    \put(11,15.5){\small \textcolor{blue}{$\tilde{\Lambda}_+$}}
	    \put(7.5,23){\small {\color{red}  $\tilde{\Lambda}_-$}}
	\end{overpic}
	\vskip-.45in
	\caption{The pair of handles in (TB1) of \autoref{lemma:trivial_bypass_lemma} arising from a trivial bypass attachment. From the left to the middle is given by an $R_-$-centric bypass as in \autoref{theorem:bypass_attachment}, and from the middle to the right is given by a Kirby move.}
	\label{fig:trivial_bypass_proof}
\end{figure}
The resulting bypass is trivial by \autoref{def: trivial bypass data}. 
\end{proof}

\begin{remark}
\autoref{lemma:trivial_bypass_lemma} implies that a pair of decorated Legendrians as in \autoref{fig:trivialbypasslemma} can be erased at the level of the underlying contact manifold. That is, not only do the surgeries on $\Gamma^{2n-1}$ cancel, but the corresponding contact handles attached to $M^{2n+1}$ cancel at the level of contact structures.
\end{remark}

Now we discuss bypass attachments at the level of POBDs. Recall that if $\Sigma = R_+ \cup\Gamma \cup R_-$ is the convex boundary of the contact manifold $(M, \xi)$ supported by the POBD $(W, \phi:S \to W)$, then $R_+ = W \setminus S$ and $R_- = W \setminus \phi(S)$. A contact handle attachment, and thus a bypass attachment, can be interpreted in terms of its effect on a supporting POBD. Briefly, a contact $n$-handle has the effect of adding a Weinstein $n$-handle to $W$ with no additional monodromy, and a contact ($n+1$)-handle attached along $D_+ \cup D_-$ has the effect of enlarging the the cornered subdomain $S$ by a regular neighborhood of $D_+$ and modifying the monodromy so that the image of $D_+$ is $D_-$. For details see Section 8 of \cite{HH19}. 

Importantly, attaching a trivial bypass to the boundary of a POBD is equivalent to performing a positive stabilization of the POBD.

\begin{lemma}[Trivial bypasses are positive stabilizations]
Let $(W, \phi: S \to W)$ be a POBD supporting $(M, \xi)$ with convex boundary $\Sigma$. Let $(\Lambda_{\pm}; D_{\pm})\subseteq \Sigma$ be trivial bypass data.
\begin{enumerate}
    \item If the trivial bypass is trivial in the sense of (TB1) in \autoref{def: trivial bypass data}, then attaching the trivial bypass corresponds to a (TB1) positive stabilization of $(W, \phi: S \to W)$ along $D_-$.
    
    \item If the trivial bypass is trivial in the sense of (TB2) in \autoref{def: trivial bypass data}, then attaching the trivial bypass corresponds to a (TB2) positive stabilization of $(W, \phi: S \to W)$ along $D_+$.
\end{enumerate}
\end{lemma}

\begin{proof}
First we attach the contact $n$-handle of the bypass to $\Sigma$ along $\Lambda_- \uplus \Lambda_+$. As a contact $n$-handle is a Reeb-thickened Weinstein $n$-handle $h$, the resulting POBD is $(W \cup h, \phi: S \to W\cup h)$. Note that the cornered subdomain $S$ and the partially-defined monodromy $\phi$ remain unchanged. Next we attach the contact ($n+1$)-handle. Here we consider separately the trivial bypasses of type (TB1) and (TB2).

\s\n
{\em (TB1).} First assume the bypass data is trivial in the (TB1) sense, so that $(\Lambda_+;D_+)$ is standard and below $\Lambda_-$. With this assumption the $n$-handle $h$ is attached along $\Lambda_-\uplus \Lambda_+ \cong \Lambda_-$. We attach the contact ($n+1$)-handle according to the $R_+$-centric model. The positive hemisphere of the attaching sphere of the ($n+1$)-handle, that is, the intersection of the attaching sphere with $R_+$, is $D_+^{-\ve}$. The negative hemisphere arises from sliding $D_-^{\ve}$ down across the $n$-handle attached along $\Lambda_-$. By \autoref{lemma:cocore_lemma}, we may isotop (via Legendrian isotopy in the dividing set) the attaching sphere of the contact ($n+1$)-handle so that $D_+^{-\ve}$ becomes the cocore $K$ of $h$. After this isotopy, the image of the negative hemisphere is exactly $\tau_L(K)$, where $L$ is the union of $D_-$ and the core of $h$; see the top row of \autoref{fig:POBD_stabilization}. 

Consequently, the effect of this handle attachment to the POBD $(W\cup h, \phi: S \to W\cup h)$ is to add to $S$ a neighborhood of $K$, which we may take to be all of $h$, and to compose the monodromy with $\tau_L$. This yields a (TB1) positive stabilization of $(W, \phi:S \to W)$.

\s\n
{\em (TB2).} Next assume the bypass data is trivial in the (TB2) sense, so that $(\Lambda_-;D_-)$ is standard and above $\Lambda_+$. With this assumption the $n$-handle $h$ is attached along $\Lambda_-\uplus \Lambda_+ \cong \Lambda_+$. We attach the contact ($n+1$)-handle according to the $R_+$-centric model. The positive hemisphere of the attaching sphere of the ($n+1$)-handle, is $D_+^{-\ve}$. The negative hemisphere arises from sliding $D_-^{\ve}$ down across the $n$-handle attached along $\Lambda_+$; see the bottom row of \autoref{fig:POBD_stabilization}. We remark for the sake of symmetry with the (TB1) case that in the (TB2) case we can isotop the attaching sphere of the ($n+1$)-handle so that the negative hemisphere is the cocore $K$ of $h$ and the positive hemisphere is $\tau_L^{-1}(K)$, where $L$ is the union of $D_+$ and the core of $h$. However, at the moment we do not perform this last isotopy and work directly with the $R_+$-centric model of the bypass attachment. 

Consequently, the effect of this handle attachment to the POBD $(W\cup h, \phi: S \to W\cup h)$ is to add to $S$ a neighborhood of $D_+^{-\ve}$, and to compose the monodromy with $\tau_L$. Note that by assumption $D_+\subseteq W\setminus S$ and so $\tau_L$ commutes with $\phi \cup \mathrm{id}_h$. The yields a (TB2) positive stabilization of $(W,\phi:S \to W)$ along $D_+$.
\end{proof}

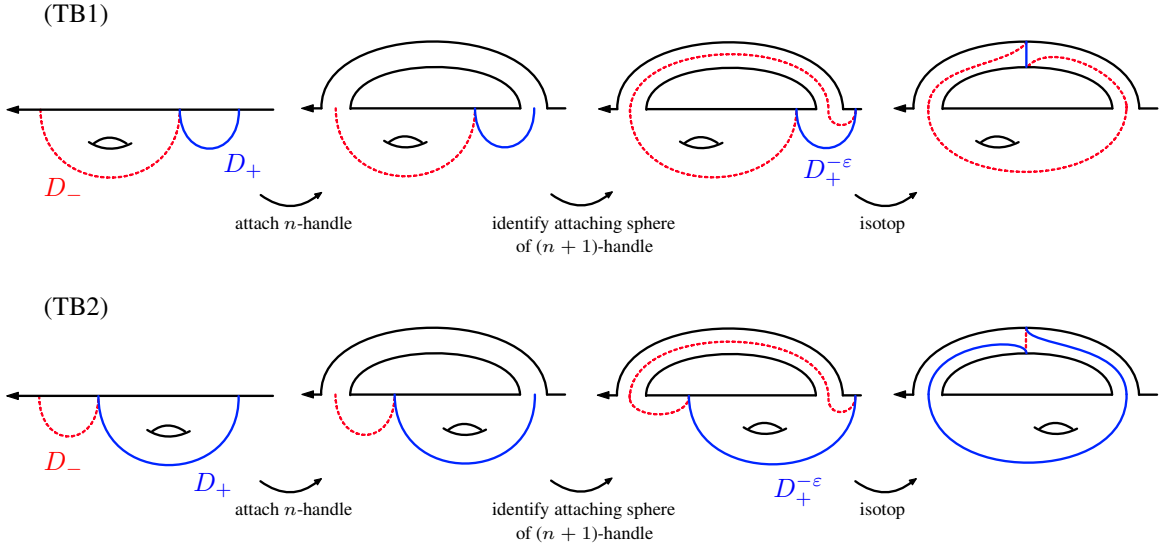
\begin{figure}[ht]
\vskip-0.8in
	\begin{overpic}[scale=0.44]{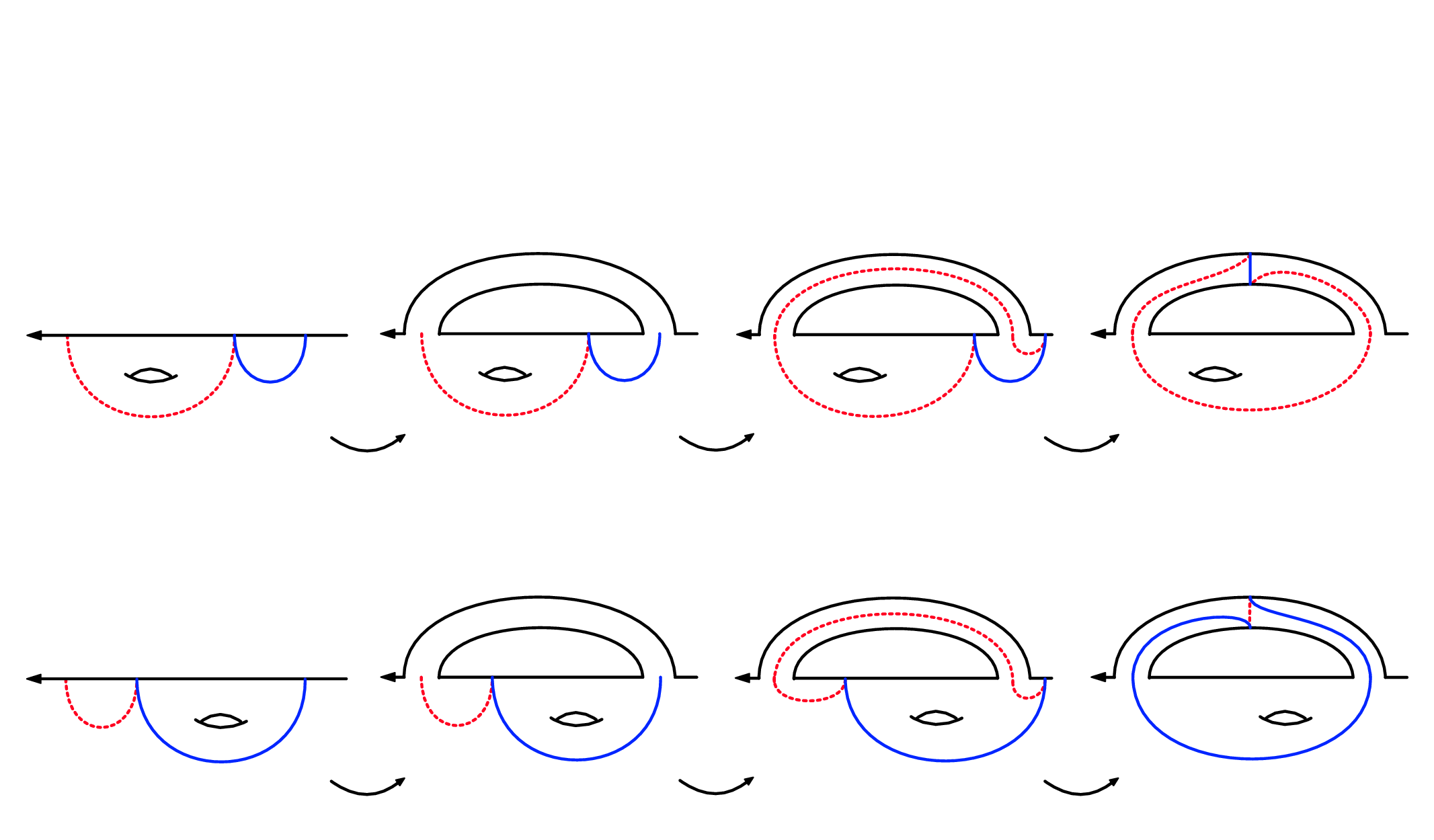}
	    \put(5,42.5){\small (TB1)}
	    \put(5,28){\small {\color{red}  $D_-$}}
	    \put(20,30){\small \textcolor{blue}{$D_+$}}
	    \put(21,25.25){\tiny attach $n$-handle}
	    \put(42.5,25.25){\tiny identify attaching sphere}
	    \put(44.5,23.25){\tiny of ($n+1$)-handle}
	    \put(73.25,25.25){\tiny isotop}
	    \put(68.5,29.5){\small \textcolor{blue}{$D_+^{-\ve}$}}
	    \put(5,18){\small (TB2)}
	    \put(5,5){\small {\color{red}  $D_-$}}
	    \put(17.5,3.25){\small \textcolor{blue}{$D_+$}}
	    \put(21,1.25){\tiny attach $n$-handle}
	    \put(42.5,1.25){\tiny identify attaching sphere}
	    \put(44.5,-0.75){\tiny of ($n+1$)-handle}
	    \put(73.25,1.25){\tiny isotop}
	    \put(66,2.75){\small \textcolor{blue}{$D_+^{-\ve}$}}
	\end{overpic}
	\caption{Positive stabilizations of POBDs arising from trivial bypass attachments in the case $n=1$. All figures depict $W\times [0,\frac{1}{2}]_t$ with the $t$-direction oriented out of the page, so that $W\times \{\frac{1}{2}\}$ is in front and $W\times \{0\}$ is in back. In each row, the first figure depicts trivial bypass data. The second figure depicts the contact $n$-handle attachment, interpreted as a Weinstein $n$-handle attachment to $W$. The third figure identifies the $R_+$-centric attaching sphere of the contact ($n+1$)-handle. In both the third and fourth figures, the new POBD monodromy after attaching the contact ($n+1$)-handle maps the solid blue arc to the dashed red arc.}
	\label{fig:POBD_stabilization}
\end{figure}

We leave it to the reader to further verify that, likewise, inserting a trivial bypass into the patty of a mushroom burger corresponds to positively stabilizing the OBD. 

\begin{remark}
The pair of contact handles described by the $\pi$-rotation of \autoref{fig:trivialbypasslemma} --- for example, involving a \textit{down}-shift of the positive cocore, or an \textit{up}-shift of the negative cocore --- corresponds to a \textit{negative} stabilization of the POBD, rendering the contact manifold overtwisted. 
\end{remark}

\subsection{How to detect stabilizations geometrically} \label{subsection: detecting stabilizations}

Above, we described the detection of stabilizations in terms of canceling contact handles. It will also be useful to detect stabilizations directly in a POBD, which we describe here.

Given a strongly Weinstein OBD $(W,\phi)$ for $(M,\xi)$, there exists a small neighborhood $N(B)$ of the binding $B$ and a decomposition $M\setminus N(B)= H_0\cup H_1$ such that $H_0=W\times[-2\epsilon,2\epsilon]$ with $\epsilon>0$ small is a contact handlebody and $H_1=M\setminus \op{int}(N(B)\cup H_0)$ is a generalized contact handlebody. (We may interchangeably use $H_1=M\setminus \op{int}(N(B)\cup H_0)$ or $H_1=M\setminus \op{int}(H_0)$, provided $N(B)$ is sufficiently small.)

Let $L_0\subset W$ be a regular Lagrangian disk with Legendrian boundary $\Lambda_0\subset \bdry W$. Since $L_0$ is regular, $W$ admits a decomposition into a Weinstein domain $W'$ and an $n$-handle $h$ attached along $\Lambda_0'\subset \bdry W'$ with core $L_0'$ and cocore $L_0$; see \autoref{fig: OBDstabilization}.

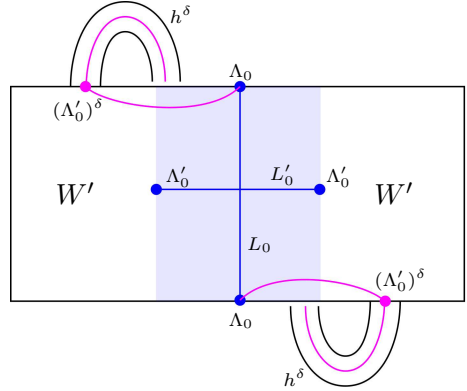
\begin{figure}[ht]
	\centering
	\begin{overpic}[scale=.5]{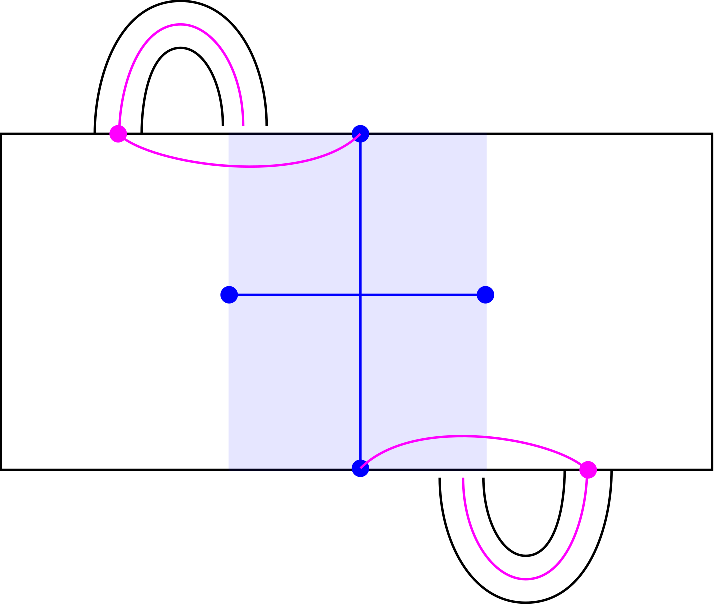}
	\put(10,40){$W'$} \put(80,40){$W'$}
	\put(34.1,45){\tiny $\Lambda_0'$} \put(69.4,45){\tiny $\Lambda_0'$} \put(57,45.3){\tiny $L_0'$}
	\put(48,13.3){\tiny $\Lambda_0$} \put(48,68){\tiny $\Lambda_0$} \put(52,30){\tiny $L_0$}
	\put(9.2,59.5){\tiny $(\Lambda_0')^\delta$} \put(80,22){\tiny $(\Lambda_0')^\delta$}
	\put(35,80){\tiny $h^\delta$} \put(60,0){\tiny $h^\delta$}
	\end{overpic}
	\caption{The shaded region is the $n$-handle $h$, with cocore $L_0$ and core $L_0'$. The purple arcs from $(\Lambda_0')^\delta$ to $\Lambda_0$ represent the trace of a Legendrian isotopy from $(\Lambda_0')^\delta$ to $\Lambda_0$.}
	\label{fig: OBDstabilization}
\end{figure}

Next let $\tilde L_t\subset h\times [-\epsilon,\epsilon]$, $t\in [-\epsilon,\epsilon]$, be a $1$-parameter family of pairwise disjoint $\xi$-Legendrians (this means Legendrians with respect to the contact structure $\xi$) with boundary $\bdry \tilde L_t=\tilde \Lambda_t$ such that:
\be
\item[(i)] $\tilde L_0=L_0\times\{0\}$;
\item[(ii)] $\tilde \Lambda_t \subset (\bdry W\cap h)\times \{t\}$ for $t\in [-\epsilon,\epsilon]$;
\item[(iii)] any two $\tilde \Lambda_{t_1},\tilde \Lambda_{t_2}$, $t_1,t_2\in[-\epsilon,\epsilon]$, are taken to each other by the characteristic foliation of $(\bdry W)\times[-\epsilon,\epsilon]$ and hence $\sqcup_{t\in[-\epsilon,\epsilon]}\tilde \Lambda_t$ is $\xi$-Legendrian;
\item[(iv)] $\sqcup_{t\in[-\epsilon,\epsilon]}\tilde L_t$ is a smooth disk with corners which sweeps out the core of a contact $(n+1)$-handle; and
\item[(v)] $\tilde \Lambda_{-\epsilon},\tilde \Lambda_{\epsilon}\subset (\bdry W'\cap \bdry W)\times \{\pm \epsilon\}$.
\ee
See \autoref{fig: OBDstabilization2}. Recall that if $\beta$ is the Liouville form on $W$ and $dt+\beta|_{\bdry W}$ is the restriction of the contact form on $(\bdry W)\times[-2\epsilon,2\epsilon]$, then the characteristic foliation on $(\bdry W)\times[-\epsilon,\epsilon]$ is directed by $\bdry_t - R_{\beta|_{\bdry W}}$, where $R$ denotes the Reeb vector field.  We leave it to the reader to verify the existence of the family $\{\tilde L_t\}_{t\in[-\epsilon,\epsilon]}$.

We then set $H_0'= H_0\setminus \sqcup_{t\in[-\epsilon,\epsilon]}N(\tilde L_t)$ and $H_1'=M\setminus \op{int}(H_0')$. Here $N(\tilde L_t)$, $t\in[-\epsilon,\epsilon]$, is a Weinstein $n$-handle which is a small neighborhood of $L_t$. 

\begin{figure}[ht]
	\centering
	\begin{overpic}[scale=.5]{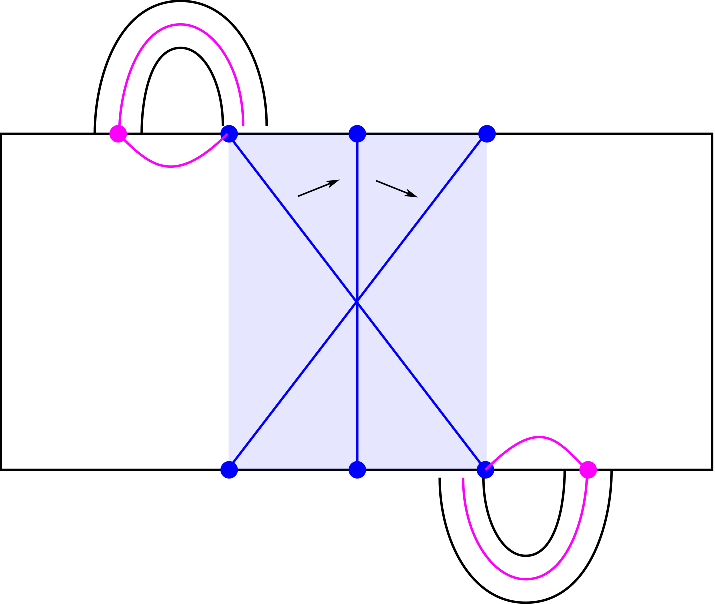}
	\put(35,50){\tiny $\tilde L_{-\epsilon}$} \put(61.6,50){\tiny $\tilde L_\epsilon$}
	\put(52,24){\tiny $L_0$}
	\end{overpic}
	\caption{The $1$-parameter family $\tilde L_t$ in blue.}
	\label{fig: OBDstabilization2}
\end{figure}

\begin{lemma} \label{lemma: stabilization}
The decomposition $M= H_0'\cup H_1'$ corresponds to the stabilization of $(W,\phi)$ along $L_0$.
\end{lemma}

\begin{proof}
We reformulate \autoref{lemma: stabilization} in terms of POBDs.

Let $(\Lambda_0')^\delta$ be the $\delta$-Reeb pushoff of $\Lambda_0'$ in $\bdry W'$ and $h^\delta$ an $n$-handle attached to $W$ along $(\Lambda_0')^\delta$, where $\delta>0$ is small. Also let $L$ be the Lagrangian sphere in $W\cup h^\delta$ obtained by gluing $L_0$, the core of $h^\delta$ with boundary on $(\Lambda_0')^\delta$, and a standard modification of the trace of the Legendrian isotopy from $(\Lambda_0')^\delta$ to $\Lambda_0$.
Then:

\begin{claim} \label{claim: stabilization}
The contact handlebody $H_0$ admits a POBD $(W\cup h^\delta, \tau_L)$.
\end{claim}

\begin{proof}[Proof of \autoref{claim: stabilization}]
The key is to view $H_0'$ as $W\times[-2\epsilon,-\epsilon]$ with a thickening of $h^\delta$ attached, where $\delta$ is taken to be a small number with a dependency on $\epsilon$. For this we slide the handle $N(\tilde L_\epsilon)$, viewed as attached to $W'\times \{\epsilon\}$, along the characteristic foliation of $\bdry W'\times[-2\epsilon,2\epsilon]$ in the negative $t$-direction.  (As far as the description of the POBD is concerned, there is no difference between isotopic handles $h$ and $N(\tilde L_\epsilon)$.) This has the effect of pushing $\tilde \Lambda_\epsilon=\bdry \tilde L_\epsilon$ in the positive Reeb direction by an amount $\delta$ which depends on the constant $\epsilon$.

The contact handlebody $H_0$ is obtained from $H_0'$ by attaching the canceling contact $(n+1)$-handle $\sqcup_{t\in[-\epsilon,\epsilon]}N(\tilde L_t)$.   
After a slight isotopy, we may assume that:
\be 
\item $\tilde L_{-\epsilon}$ is on the $R_+$ part of the boundary of the contact handlebody over $W\cup N(\tilde L_t)$; 
\item $\tilde L_\epsilon$ is on the $R_-$ part of the boundary of the contact handlebody over $W\cup N(\tilde L_t)$ and is represented by the purple arcs in \autoref{fig: OBDstabilization2}; and 
\item the projections of their boundaries to $W\cup N(\tilde L_t)$ coincide.
\ee
Hence the partially defined monodromy map sends $\tilde L_{-\epsilon}$ to $\tilde L_\epsilon$ and is isotopic to $\tau_L$.
\end{proof}

\autoref{lemma: stabilization} follows from \autoref{claim: stabilization}.
\end{proof}

\subsection{Folded Weinstein hypersurfaces}

A $C^0$-generic hypersurface in a contact manifold is Weinstein convex, which, informally, is equivalent to admitting a decomposition as the union of two positive and negative Weinstein domains. There is a more general class of folded Weinstein hypersurfaces \cite{HH19}, which, informally, is one that admits a decomposition into alternating positive and negative Weinstein cobordisms. Retrogradients of characteristic foliations, and hence bypass attachments, are described naturally by folded Weinstein hypersurfaces. In this subsection we review the concept as it will be used in the next section.

The following definition is from \cite{HH19}:

\begin{defn}
A {\em folded Weinstein hypersurface} $\Sigma \subseteq (M,\xi)$ is an oriented closed hypersurface such that the characteristic foliation $\Sigma_{\xi}$ satisfies the following properties:
\begin{itemize}
    \item[(FW1)] There are pairwise disjoint codimension-$1$ submanifolds $K_i \subseteq \Sigma$ called the {\em folding loci} that cut $\Sigma$ into $2m$ pieces, i.e., 
    \[
    \Sigma = W_1 \cup_{K_1} \cdots \cup_{K_{2m-1}} W_{2m}
    \]
    where $W_i$ are compact with $\partial W_i = K_i \cup K_{i-1}$. Here we assume $K_0 = K_{2m} = \varnothing$. 
    \item[(FW2)] The singular points of $\Sigma_{\xi}$ in each $W_i$ have the same sign, and the sign changes upon crossing $K_i$ or $K_{i-1}$. We assume the singularities in $W_1$ are positive, and that each $W_i$ has at least one singularity. 
    \item[(FW3)] There is a Morse function on each $W_i$ for which $\Sigma_{\xi}|_{W_i}$ is gradient-like and $K_i, K_{i-1}$ are regular level sets.
\end{itemize}
\end{defn}

\begin{remark}
We refer to $K_i$ as \textit{maximal} if $i$ is odd, and \textit{minimal} if $i$ is even. That is, maximal folding loci behave like dividing sets, in the sense that the oriented characteristic foliation flows from a positive to a negative region. The opposite is true of a minimal folding locus. 
\end{remark}

In \cite{HH19} the second and third authors construct a normalized contact neighborhood of a folded Weinstein hypersurface. In particular, there is a contact form which induces a Weinstein cobordism structure on each $W_i$. 

\subsection{Retrogradient and bypass diagrams}\label{subsection:retro_and_bypass}
If $\Sigma^{2n}$ is a convex hypersurface, then it naturally admits a (trivial) folded Weinstein structure as $\Sigma = R_+ \cup_{\Gamma} R_-$. Let $p^{\pm}$ be positive and negative, respectively, critical points of index $n$ of the characteristic foliation of $\Sigma$ and suppose that $W^{\mathrm{u}}(p^+) \cap W^{\mathrm{s}}(p^-) = \varnothing$, where $W^{\mathrm{u}}$ and $W^{\mathrm{s}}$ indicate the unstable and unstable manifolds. By exchanging the values of the critical points, we can shuffle the corresponding elementary Weinstein cobordisms to give $\Sigma$ a non-trivial folded Weinstein structure 
\[
\Sigma = W_1 \cup W_2 \cup_Y W_3 \cup W_4.
\]
See \autoref{fig:folded_weinstein}. 
\begin{figure}[ht]
\vskip-.15in
	\begin{overpic}[scale=0.43]{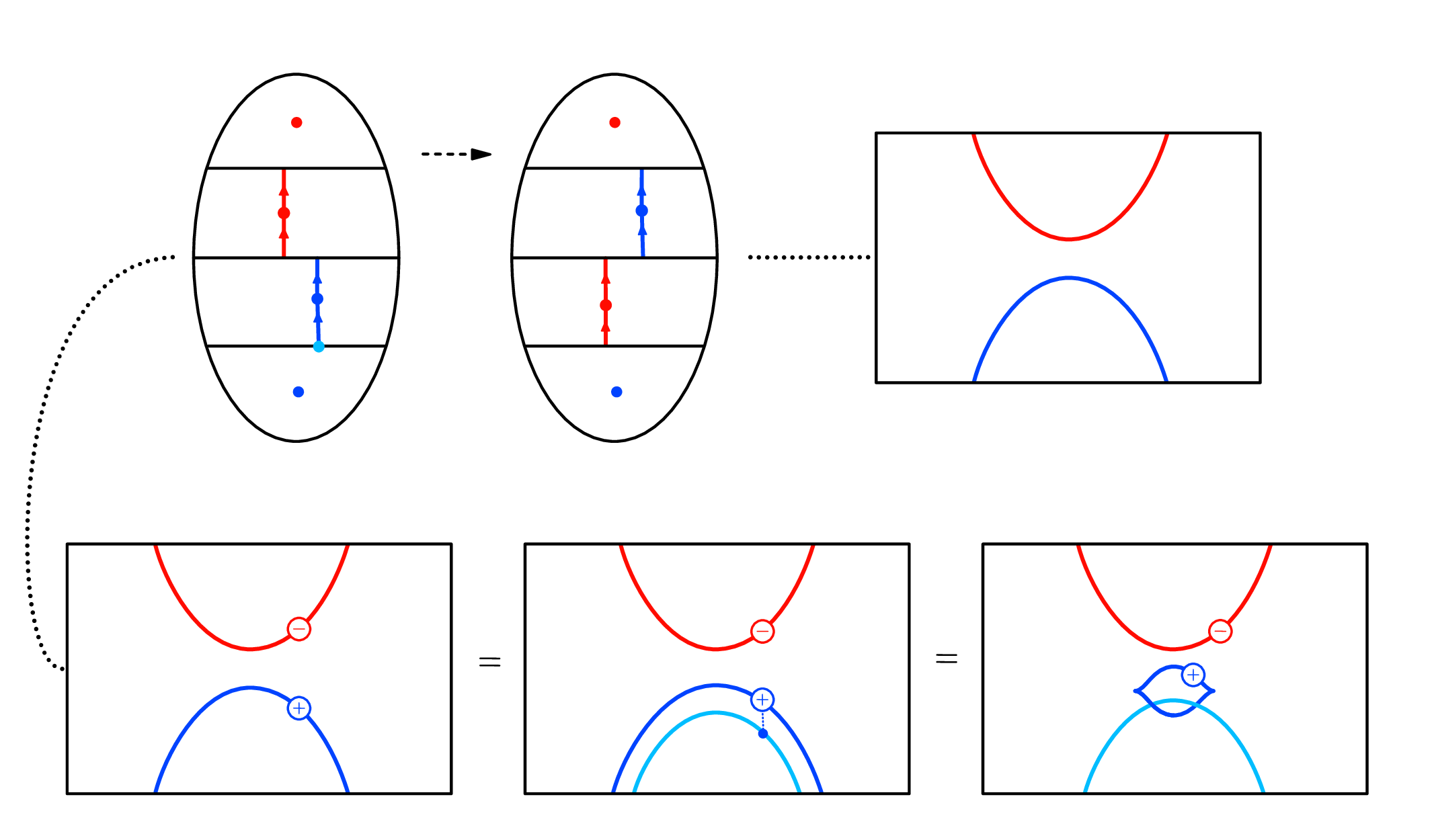}
	    \put(17,54.5){\tiny $\Sigma$ (convex)}
	    \put(36.5,54.5){\tiny $\Sigma$ (folded Weinstein)}
	    \put(28.75,40){\tiny $\Gamma$}
	    \put(50.5,40){\tiny $Y$}
	    \put(27.75,33.5){\tiny $G_+$}
	    \put(85,32.5){\small $Y$}
	    \put(74,29.5){\small $Y$}
	    \put(61.5,50.5){\small (local retrogradient diagram)}
	    \put(39,22.5){\small (local bypass diagrams)}
	    \put(28.75, 4.25){\small $\Gamma$}
	    \put(17.5, 0.5){\small $\Gamma$}
	    \put(59.5, 4.25){\small $G_+$}
	    \put(49.5, 0.5){\small $\Gamma$}
	    \put(91.5, 4.25){\small $G_+$}
	    \put(81, 0.5){\small $\Gamma$}
	    \put(59,8){\tiny \textcolor{Cerulean}{$(-1)$}}
	    \put(91,8){\tiny \textcolor{Cerulean}{$(-1)$}}
	\end{overpic}
	\vskip-.05in
	\caption{Shuffling cobordisms in a convex hypersurface to produce a non-trivial folded Weinstein structure. Various local surgery diagrams are drawn. The bottom three diagrams all depict the same information.}
	\label{fig:folded_weinstein}
\end{figure}
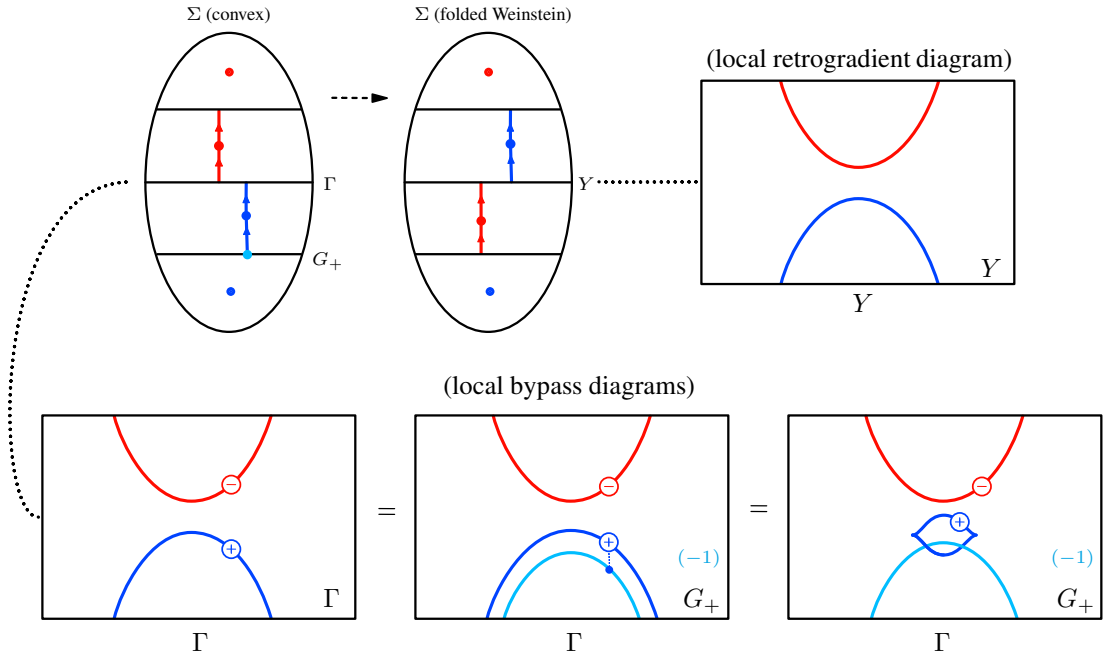

We will use this non-trivial folded Weinstein structure to identify and model retrogradient trajectories. 
To this end the following two types of Legendrian surgery diagrams are particularly relevant:

\begin{defn}$\mbox{}$
\begin{enumerate}
    \item A {\em (local) retrogradient diagram} is a (local) Legendrian surgery diagram that represents a minimal contact folding locus of a folded Weinstein hypersurface, where intersections of red and blue Legendrians represent retrogradient trajectories. 
    
    \item A {\em (local) bypass diagram} is a (local) Legendrian surgery diagram that describes the dividing set of a convex hypersurface, often with decorations representing bypass attachment data. 
\end{enumerate}
\end{defn}

\begin{remark}
We will often draw retrogradient diagrams parametrically; for example, see \autoref{fig:parametric_FW}. 
\end{remark}

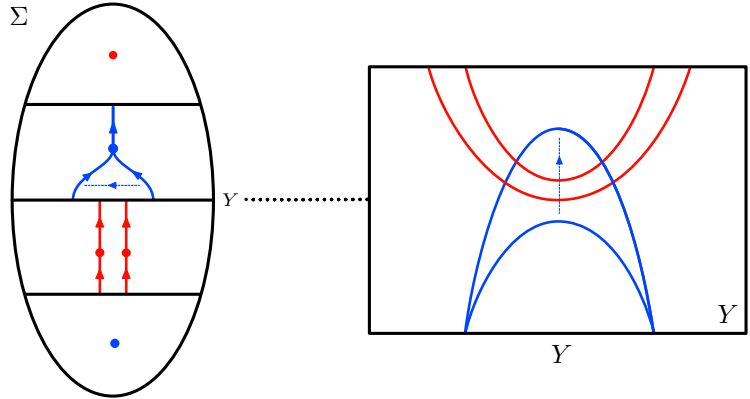
\begin{figure}[ht]
\vskip-.15in
	\begin{overpic}[scale=0.58]{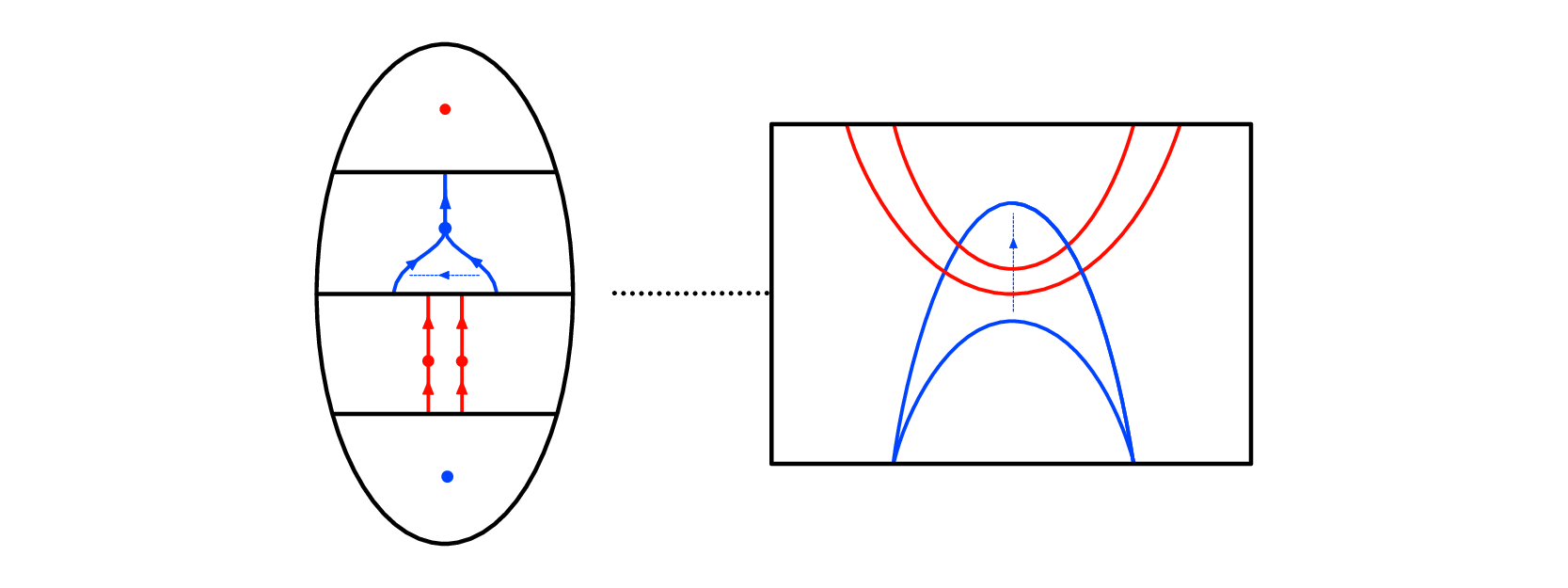}
	    \put(20,33){\small $\Sigma$}
	    \put(37.25,18.25){\tiny $Y$}
	    \put(77.5,8.75){\small $Y$}
	    \put(64,5.5){\small $Y$}
	\end{overpic}
	\vskip-.15in
	\caption{A $1$-parametric family of folded Weinstein hypersurfaces in which two retrogradients appear.}
	\label{fig:parametric_FW}
\end{figure}

\begin{remark}
It is not necessary for us to keep track of Lagrangian disk fillings of Legendrians in local retrogradient diagrams inside folded Weinstien hypersurfaces, so we typically omit decorations.  
\end{remark}

\subsection{The bypass-bifurcation correspondence}

Here we review material from Section 8 of \cite{HH19}. Every retrogradient of the characteristic foliation between nondegenerate critical points of index $n$ is modeled by a bypass attachment, in a way that will be given precisely by \autoref{prop:bypass_bifurcation_correspondence} below. 

We will often translate between folded Weinstein descriptions of retrogradients and bypass attachment data via this correspondence. In particular, we will perform chains of bypass attachments in succession, and as such it is important to keep track of the (un)stable disks of the critical points involved in retrogradients before and after a bypass. 
 
\begin{prop}[Proposition 8.3.2 of \cite{HH19}]\label{prop:bypass_bifurcation_correspondence}
Let $(\Sigma \times [-\delta, \delta], \xi)$ satisfy the property that $\Sigma_s := \Sigma \times \{s\}$ is convex for $s\neq 0$ and is folded Weinstein at $s=0$, at which point there is a retrogradient $p_- \rightsquigarrow p_+$ between nondegenerate critical points $p_{\pm}$ of index $n$. Then $(\Sigma \times [-\delta, \delta], \xi)$ is contactomorphic relative to the boundary to a bypass attachment along $\Sigma_{-\delta}$ with data $(\Lambda_{\pm}; D_{\pm})$, where 
\begin{itemize}
    \item $D_+$  can be identified with a small positive Reeb shift of the unstable disk $K_+$ of $p_+$, and 
    \item $D_-$ can be identified with a small negative Reeb shift of the stable disk $K_-$ of $p_-$.
\end{itemize}
Moreover, after the bypass, 
\begin{itemize}
    \item the unstable disk of $p_+$ in $\Sigma_{\delta}$ can be identified in the $R_+$-centric model as the cocore of the positive Weinstein $n$-handle, and 
    \item the stable disk of $p_-$ in $\Sigma_{\delta}$ can be identified in the $R_-$-centric model as the cocore of the negative Weinstein $n$-handle.
\end{itemize}
See \autoref{fig:bypass_bifurcation_correspondence}. 
\end{prop}

\begin{figure}[ht]
	\begin{overpic}[scale=0.455]{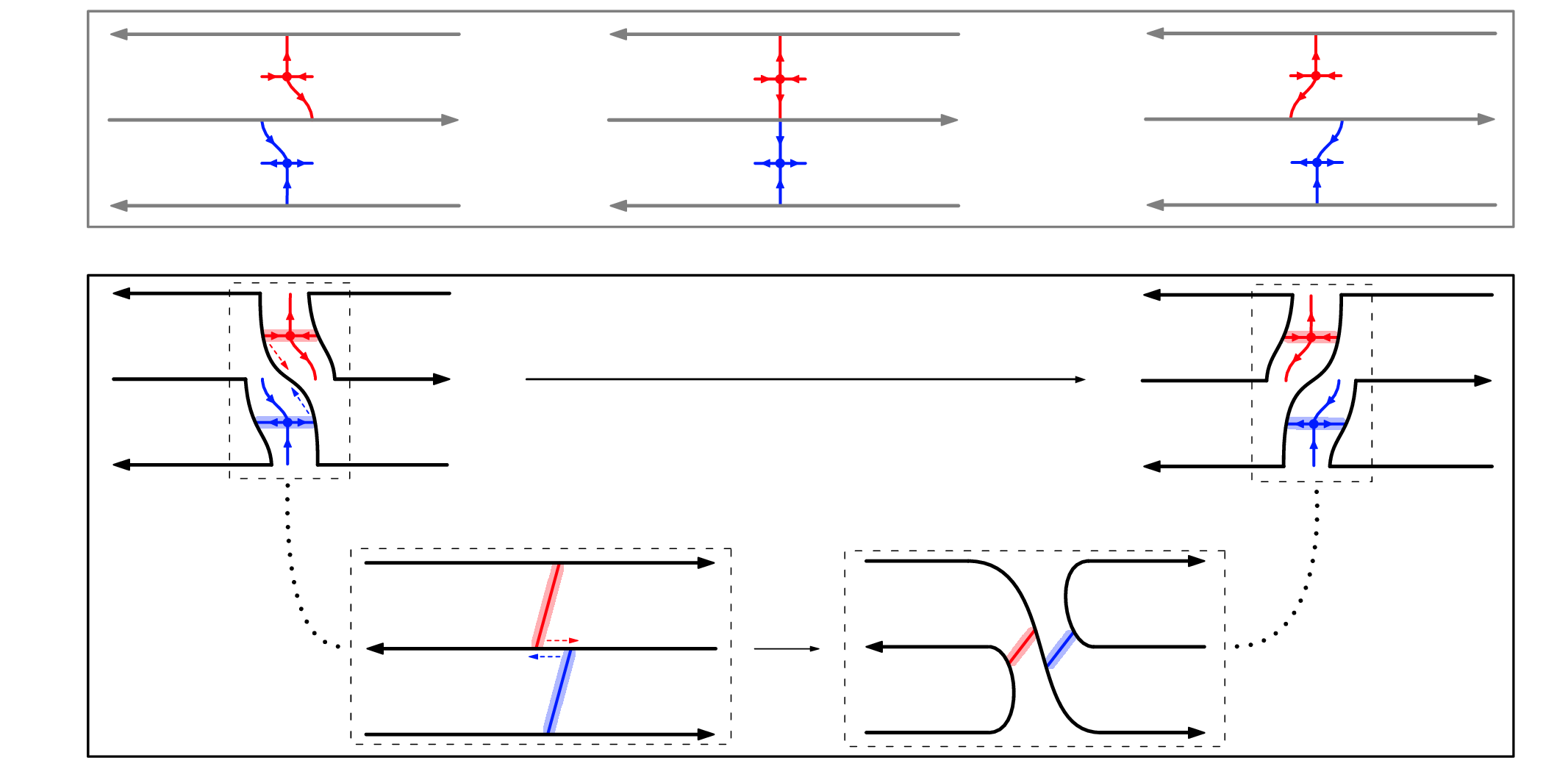}
	    \put(15,32.5){\small $t = -\delta$}
	    \put(48,32.5){\small $t = 0$}
	    \put(82,32.5){\small $t = \delta$}
	    \put(43,48.5){\small \textcolor{gray}{Folded Weinstein}}
	    \put(43,26){\small Bypass attachment}
	    \put(36.5,4){\tiny \textcolor{blue}{$K_+$}}
	    \put(31.5,10){\tiny {\color{red}  $K_-$}}
	    \put(46.75,-1.25){\small Convex}
	\end{overpic}
	\caption{The bypass-bifurcation correspondence as described by \autoref{prop:bypass_bifurcation_correspondence}. The top row is a family of folded Weinstein hypersurfaces where at $t=0$ there is a retrogradient $p_- \rightsquigarrow p_+$ of the characteristic foliation. Below is the corresponding situation for the convex hypersurfaces at times $t=\pm \delta$. The black arrows are the dividing sets. The bypass data is given by shifting the highlighted blue and red disks on the left in the indicated directions until they intersect.}
	\label{fig:bypass_bifurcation_correspondence}
\end{figure}
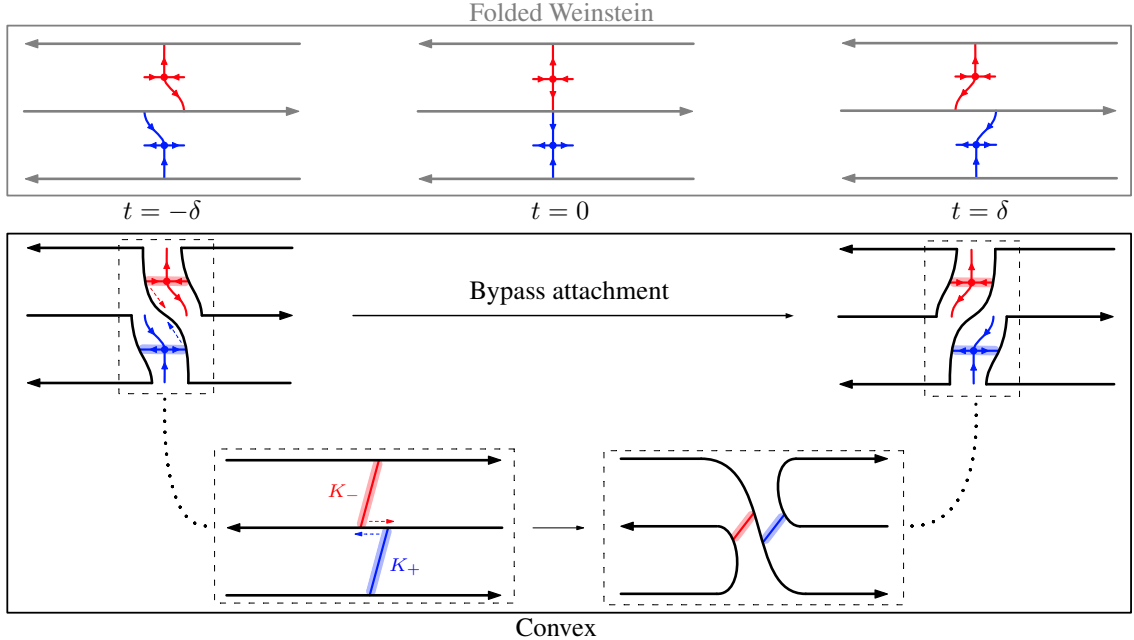

\section{Proof of stabilization equivalence}\label{section: outline of proof of stabilization equivalence} 

In this section we prove \autoref{theorem: stabilization equivalence}, modulo several theorems (Theorems~\ref{theorem: sigma times interval}, \ref{theorem: layer between N(L) and N(L')}, \ref{thm: layer between N and N' general case}, and \ref{theorem: birth-death}) that will be proven in subsequent sections. 

For the rest of the paper, unless otherwise stated, ``(P)OBD'' refers to a strongly Weinstein (P)OBD. Our starting point is a $1$-Morse family of smooth functions $f_t:M\to \R$, $t\in[0,1]$, where $f_0$ and $f_1$ are contact Morse functions corresponding to the two OBDs. The analysis of such a family $\{f_t\}_{t\in[0,1]}$ together with its gradient-like vector fields in smooth topology is classical and is due to Cerf~\cite{Cer70} and Hatcher-Wagoner~\cite{HW73}. Our goal is to make the family $f_t$ as contact as possible. In particular, we will use the smooth homotopy $f_t$ to build a family of $\theta$-decompositions, and then we will analyze the $\theta$-decompositions to conclude stable equivalence of the OBDs.

The $\theta$-decompositions are analyzed in the three steps described in the next subsection. The proof of stabilization equivalence is then given in \autoref{subsection:actual_proof_of_main_theorem}.

\s\n
{\em Idea of proof.} To clarify the role of each step below, we briefly outline the idea of the proof. For most times $t$, i.e., when $f_t$ is Morse, we can apply the technology of \cite{HH19} to produce a $\theta$-decomposition $\Theta_t$. If two such $\theta$-decompositions have buns that are smoothly isotopic, then after a flexibilization process of the buns (Step 2 below) we can assume they are contact isotopic. The $\theta$-decompositions are then related by a homotopy of patties, which allows us to analyze the OBDs (Step 1 below). It remains to consider moments in the homotopy $f_t$ that change the topological type of the buns of the $\theta$-decompositions. This only happens at birth-death type critical points of index $(n,n+1)$; note that birth-death points of strictly super- or sub-critical indices do not affect the relevant topologies. The analysis of the OBDs near such points is given by Step 3.

\subsection{Three main steps} Here we analyze special kinds of $\theta$-decompositions in three steps. 

\s\n
{\em Step 1: Homotopies of patties.}
\s

This step corresponds to \autoref{sec: sigma times interval}, where we will prove the following main technical theorem:  

\begin{theorem}\label{theorem: sigma times interval}
	Let $\Theta_t=H_0\cup (\Sigma\times[0,1])\cup H_1$, $t\in[0,1]$, be a $1$-parameter family of $\theta$-decompositions of $(M,\xi_t)$, where $\xi_t$ is independent of $t$ on $H_0$ and $H_1$. Then the OBDs corresponding to $\Theta_0$ and $\Theta_1$ are stably equivalent. 
\end{theorem}

\autoref{theorem: sigma times interval} is equivalent to the following statement, which generalizes a result of Bin Tian \cite{Ti18} in dimension three.

\begin{theorem}\label{theorem: sigma times interval, contact category version}
Any two sequences of bypass attachments starting with $\Sigma\times\{0\}$ that describe $(\Sigma\times  [0,1],\xi)$ can be related to each other by two types of moves: far commutativity and adding a trivial bypass. 
\end{theorem} 

This theorem guides our general philosophy, which is to push as much of the nontrivial contact topology into the family of patties as possible.

\s\n
{\em Step 2: Wrinkling Legendrians via bypasses.}
\s

In this step we show that we can flexibilize the buns of a $\theta$-decomposition. The precise statement is \autoref{thm: layer between N and N' general case} and the proof will be given in \autoref{section: proof of theorem layer}. As a consequence, we prove stable equivalence of the OBDs corresponding to homotopies with no birth-death points of middle index (\autoref{thm: no birth-death} below), which will be used in the proof of \autoref{theorem: stabilization equivalence}. To state the theorem cleanly, we introduce some language. 

\begin{defn}
Let $f:M \to \R$ be a self-indexing Morse function. Let $Y_0$ and $Y_1$ be the $n$-skeleta of $f$ and $-f$, respectively. We say that a $\theta$-decomposition $\Theta = H_0 \cup (\Sigma \times [0,1]) \cup H_1$ of $(M, \xi)$ is {\em bun-supported by $f$} if there is a diffeomorphism $\phi:M \to M$ isotopic to the identity such that $H_i$ is a neighborhood of $\phi(Y_i)$. 
\end{defn}

\begin{remark}
Given a self-indexing Morse function $f:M \to \R$, the proof of Corollary 1.3.1 in \cite{HH19} produces a $\theta$-decomposition of $(M, \xi)$ that is bun-supported by $f$, and moreover we may take the diffeomorphism $\phi$ to be $C^0$-close to the identity. 
\end{remark}

With this, we prove: 

\begin{theorem}\label{thm: no birth-death}
Let $\Theta_0$ and $\Theta_1$ be $\theta$-decompositions of $(M^{2n+1}, \xi)$ such that there exists a $1$-Morse family $f_t:M \to \R, t\in [0,1]$ satisfying:
\begin{itemize}
    \item for $t=0,1$, $\Theta_t$ is bun-supported by $f_t$, and 
    \item for $t\in [0,1]$, $f_t$ has no birth-death critical points of index $(n,n+1)$.
\end{itemize}
Then the OBDs corresponding to $\Theta_0$ and $\Theta_1$ are stably equivalent. In particular, any two $\theta$-decompositions bun-supported by the same Morse function give stably equivalent OBDs. 
\end{theorem}

\begin{remark}
Note that in this theorem we are not assuming that $f_0$ and $f_1$ are contact Morse functions, only that they are smooth Morse functions bun-supporting the corresponding $\theta$-decompositions. 
\end{remark}

In order to prove this theorem, we would like to ``freely'' move the skeleta of the contact handlebodies of the $\theta$-decompositions, i.e., find a $C^0$-close approximation of the smooth isotopy of the handles arising from the family $f_t$, $t\in [0,1]$. When $n=1$, one needs to apply a sequence of positive and negative stabilizations to the the Legendrian skeleta in order to realize the link isotopy by a Legendrian isotopy; this is well-known. From now on in this step we will assume $n>1$. In this case we replace the contact handlebodies $H_{i,t}$, $t=0,1$, of $\Theta_t$ by flexible contact handlebodies $H_{i,t}'$. This is analogous to the construction of flexible Weinstein manifolds due to Cieliebak-Eliashberg \cite{CE12}.

Before describing how to do this in general (\autoref{thm: layer between N and N' general case} below) we first consider a basic example. Let $L$ be a closed Legendrian submanifold diffeomorphic to $S^n$. Its standard neighborhood $N=N(L)$ is a contact handlebody and its boundary is a convex hypersurface which we denote by $\Sigma$. Let $L'\subset N$ be a wrinkled stabilization of $L$, i.e., those with an unfurled swallowtail as in Eliashberg-Mischachev~\cite{EM97}; see also Murphy~\cite[Section 3]{Mur12}. See \autoref{fig: Giroux2}. Its small standard neighborhood $N'=N(L')\subset N$ is also a contact handlebody with convex boundary $\Sigma'$. 

\begin{figure}[ht]
	\begin{overpic}[scale=.45]{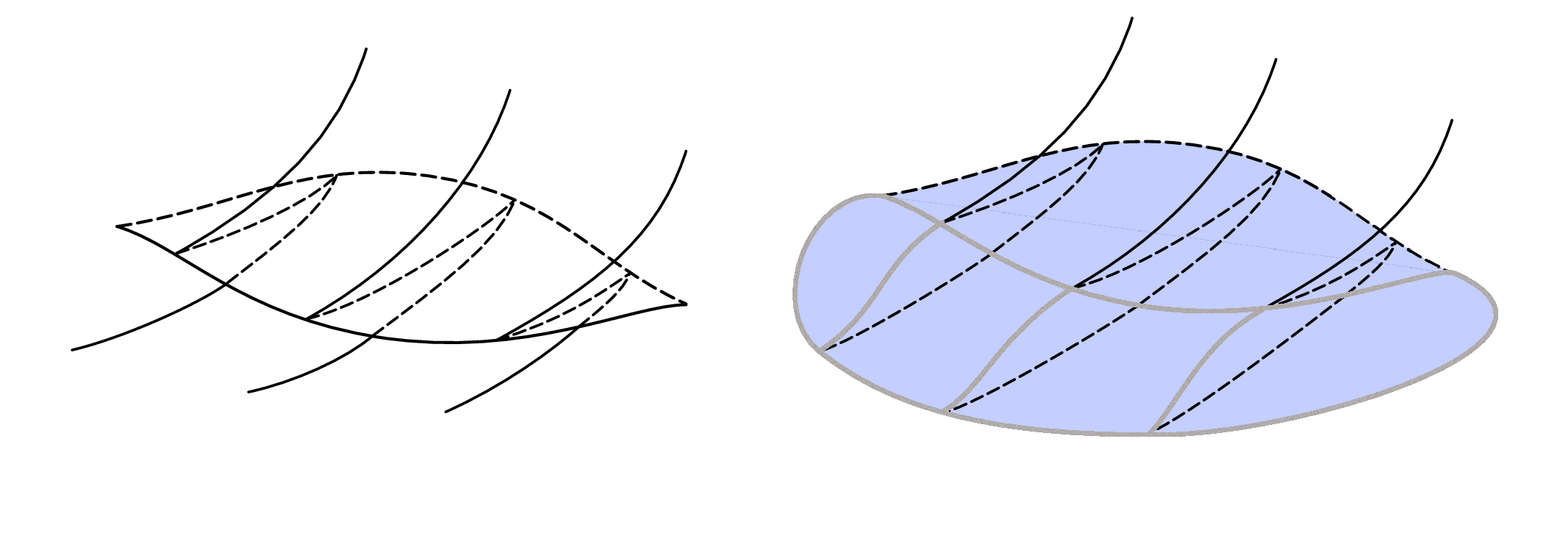}
	\end{overpic}
	\vskip-0.3in
	\caption{The front projection of an unfurled swallowtail (left) and a bypass attachment (right). On the right, the gray disk is a contact $n$-handle, and the blue interior of the ``air pocket" is the smoothly canceling contact $(n+1)$-handle of the destabilizing bypass.}
	\label{fig: Giroux2}
\end{figure}

At this point it is convenient to switch perspectives to POBDs. Note that a contact handlebody over a Weinstein domain $W$ admits a trivial POBD $(W,\varnothing)$, where $\varnothing$ is the trivial map $\varnothing\to W$. The layer $N\setminus \op{int}(N')$, i.e., the region bounded by $\Sigma$ and $\Sigma'$, is then described as follows:

\begin{theorem} \label{theorem: layer between N(L) and N(L')}
$\mbox{}$
\be
\item $N$ (resp.\ $N'$) is a contact handlebody over $W$ (resp.\ $W'$), obtained by a attaching a handle along the standard (resp.\ loose) Legendrian unknot on $\bdry D^{2n}$ where $D^{2n}$ is standard. More precisely, when $n=2$, a ``loose Legendrian" means a twice-stabilized Legendrian, with one positive and one negative stabilization as given by $\Lambda'$ in \autoref{fig:simple_legendrians}, and when $n>2$ the loose Legendrian is obtained from the $n=2$ one by spinning.
\item $N$ is obtained from $N'$ by attaching a bypass. Moreover, the resulting POBD obtained from attaching the bypass to $N'$ is a stabilization of $(W, \varnothing)$. See \autoref{fig:wrinkle2} for a schematic picture.
\ee
\end{theorem}

\begin{figure}[ht]
	\centering
	\begin{overpic}[scale=.44]{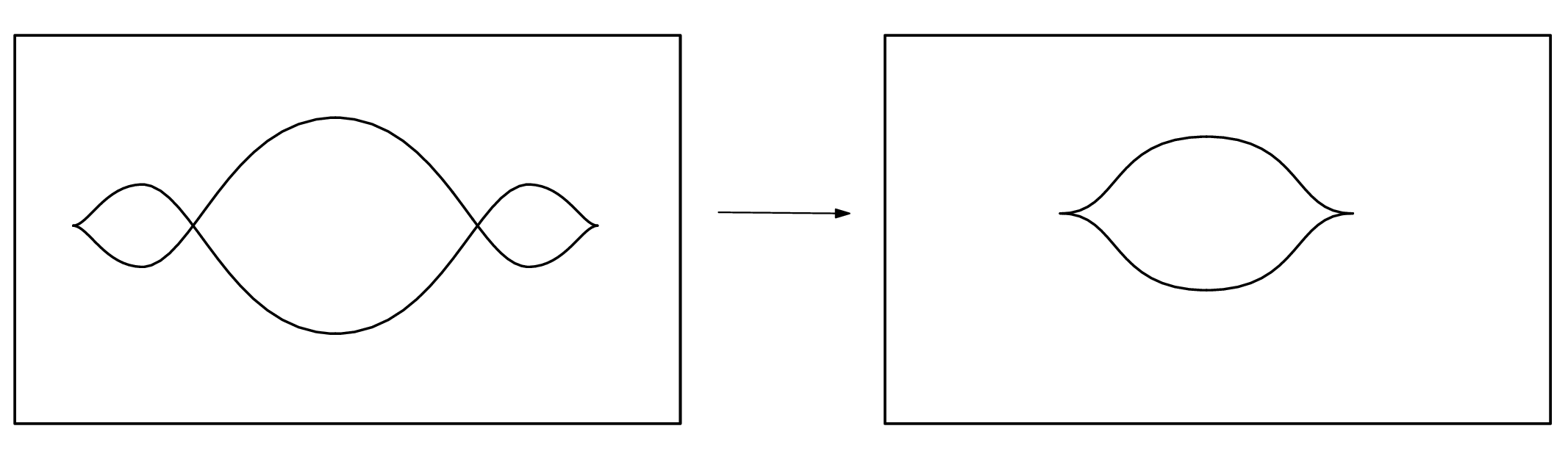}
	\put(46.75,13){\small bypass}
	\put(5,10){\small $\Lambda'$}
	\put(20,27.75){\small $N'$}
	\put(76,27.75){\small $N$}
	\put(10,0){\small contact handlebody over $W'$}
	\put(66,0){\small contact handlebody over $W$}
    \put(38.5,14){\small $(-1)$}
    \put(93.75,14){\small $(-1)$}
	\end{overpic}
	
	\caption{The statement of \autoref{theorem: layer between N(L) and N(L')}. On the left is a handlebody diagram for $W'$, and on the right is a handlebody diagram for $W$. The contact handlebody $N$ (over $W$) is obtained by attaching a bypass to $N'$.}
	\label{fig:simple_legendrians}
\end{figure}

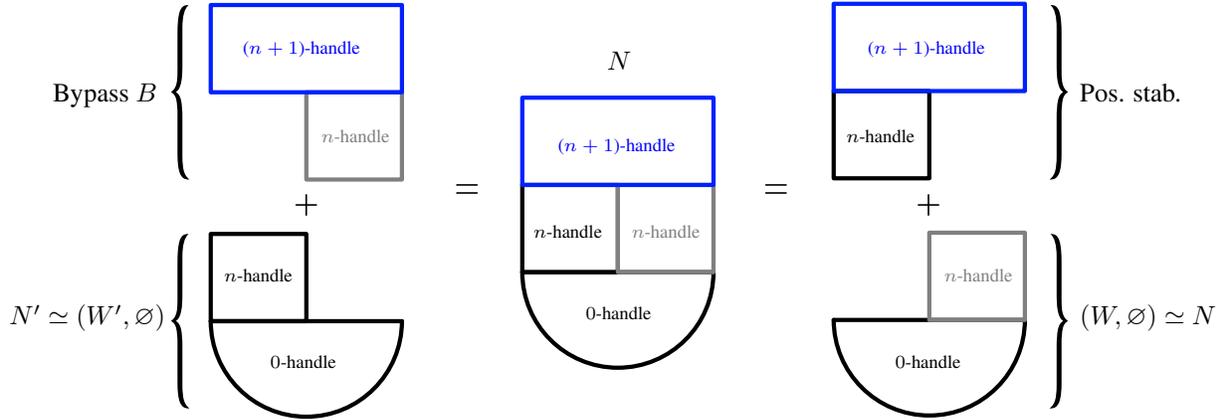
\begin{figure}[ht]
	\centering
	\begin{overpic}[scale=.4]{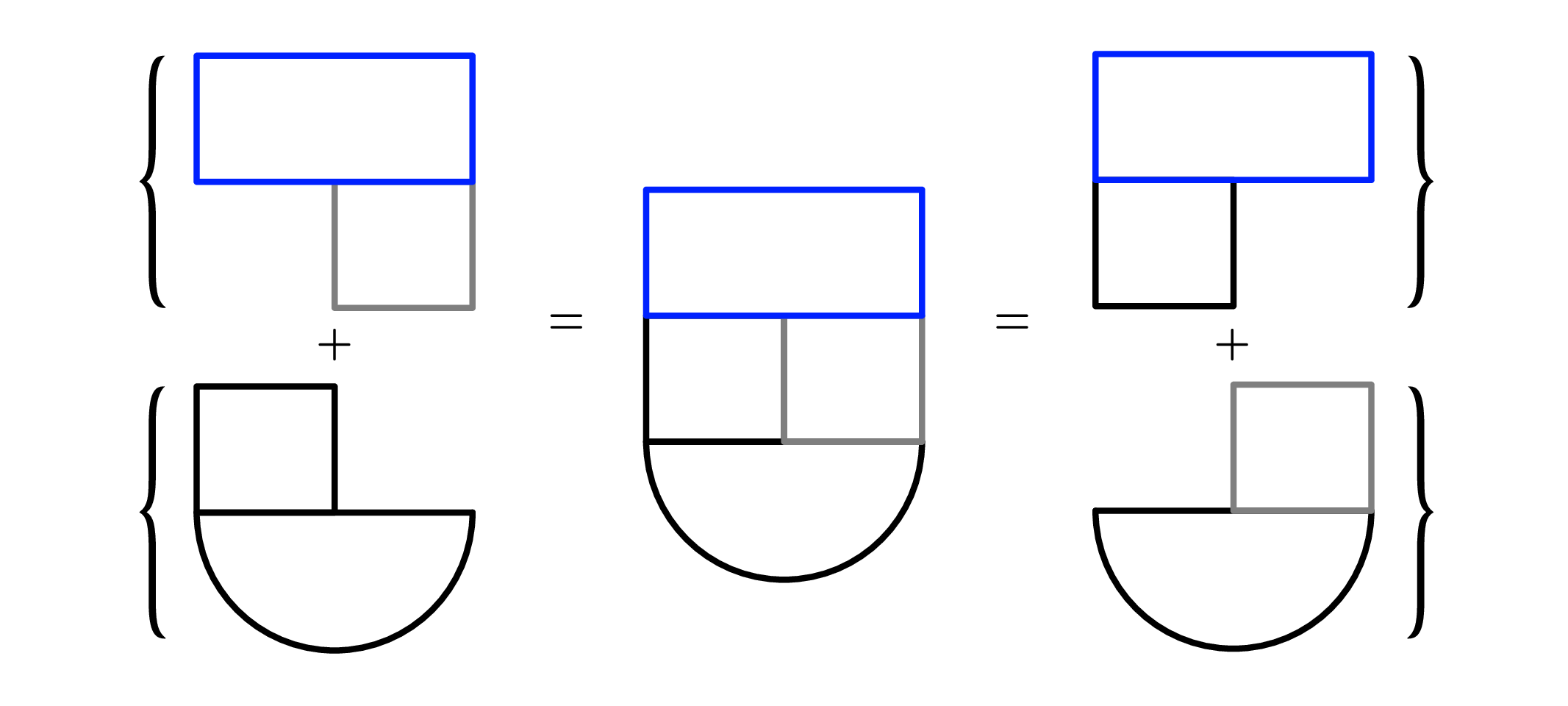}
	\put(-6, 12){\small $N'\simeq (W', \varnothing)$}
	\put(-2, 32.5){\small Bypass $B$}
	\put(92.5, 12){\small $(W, \varnothing)\simeq N$}
	\put(92.5, 32.5){\small Pos. stab.}
	\put(49, 35.5){\small $N$}
	\put(18,8){\tiny $0$-handle}
	\put(15.5,37){\tiny \textcolor{blue}{($n+1$)-handle}}
	\put(47,12.5){\tiny $0$-handle}
	\put(44.5,28){\tiny \textcolor{blue}{($n+1$)-handle}}
	\put(75.5,8){\tiny $0$-handle}
	\put(73,37){\tiny \textcolor{blue}{($n+1$)-handle}}
	\put(80,16){\tiny \textcolor{gray}{$n$-handle}}
	\put(13.75,16){\tiny $n$-handle}
	\put(42.25,20){\tiny $n$-handle}
	\put(51.25,20){\tiny \textcolor{gray}{$n$-handle}}
	\put(22.7,28.75){\tiny \textcolor{gray}{$n$-handle}}
	\put(71,28.75){\tiny $n$-handle}
	\end{overpic}
	\caption{A schematic description of the statement of \autoref{theorem: layer between N(L) and N(L')}(2). The colors are consistent with \autoref{fig: Legendrians}.
	}
	\label{fig:wrinkle2}
\end{figure}

See \autoref{fig: Giroux3} which describes the $1$- and $2$-handles of the bypass in the $n=1$ case and also the right-hand side \autoref{fig: Giroux2} which describes the $n$-handle and smoothly canceling $(n+1)$-handle. The explicit bypass data is given by \autoref{fig: Legendrians} in \autoref{section: proof of theorem layer}.
\begin{figure}[ht]
	\centering
	\begin{overpic}[scale=.33]{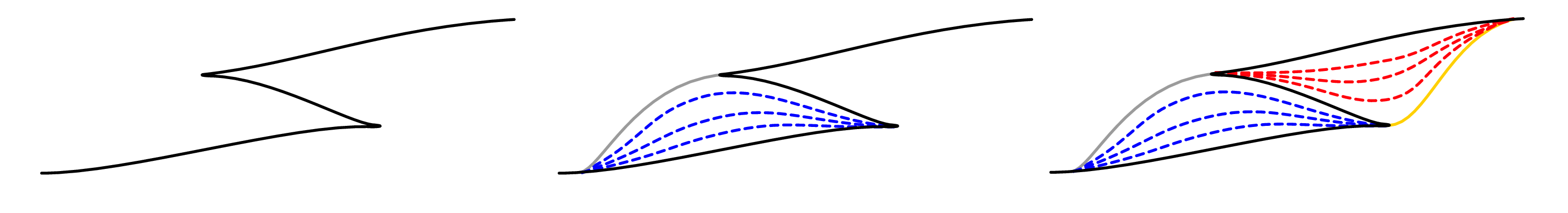}
	\end{overpic}
	\vskip-0.1in
	\caption{The middle depicts attaching a $1$-handle (solid gray arc) and a canceling $2$-handle (disk foliated by dotted blue arcs) of a bypass to a standard neighborhood of the once-stabilized Legendrian $L'$ on the left in the $n=1$ case. The right-hand side depicts attaching two bypasses.  Although it looks as though we need two bypasses to go from the standard neighborhood $N(L')$ to the neighborhood $N(L)$ of the original Legendrian, one bypass (as in the middle) suffices.}
	\label{fig: Giroux3}
\end{figure}

In general we have the following:

\begin{theorem}
\label{thm: layer between N and N' general case}
Let $N$ be a contact handlebody over a Weinstein domain $W$.
\be
\item If we replace each $n$-handle $h_i$, $i=1,\dots,\kappa$, of $W$ with core $L_i$ by its wrinkled stabilization $L'_i$, then we obtain a contact handlebody $N'$ over the flexible Weinstein domain $W'$ which is obtained from the neighborhood $W^{(n-1)}$ of the $(n-1)$-skeleton of $W$ by attaching $n$-handles $h_i'$, $i=1,\dots,\kappa$, along the stabilizations of $\bdry L_i$. Here the stabilizations are done along spheres $S^{n-1}$ bounding disks in $L_i$. 
\item $N$ is obtained from $N'$ by attaching bypasses $B_i$, $i=1,\dots,\kappa$, along $\bdry N'$, and moreover the resulting POBD for $N$ is a stabilization of $(W,\varnothing)$.
\ee
\end{theorem}

We will prove this theorem in \autoref{section: proof of theorem layer}. For now, we use it to prove \autoref{thm: no birth-death}.

\begin{proof}[Proof of \autoref{thm: no birth-death}.] 
For $t=0,1$, let $H_{i,t}$ be the  contact handlebody buns of $\Theta_t$. Let $H_{i,t}'\subset H_{i,t}$ be the ``wrinkled'' contact handlebodies described by \autoref{thm: layer between N and N' general case} and let $\Theta_{t}'$ be the $\theta$-decompositions obtained from $\Theta_t$ by replacing 
\[
H_{i,t}\to H'_{i,t}, \quad \Sigma_t\times[0,1]\to (\Sigma_t\times[0,1])\cup(H_{0,t}-H'_{0,t})\cup (H_{1,t}-H'_{1,t}).
\]
By \autoref{thm: layer between N and N' general case}(2), for each $t=0,1$ the OBDs corresponding to $\Theta_t$ and $\Theta_t'$ are stably equivalent. Thus, it suffices prove stable equivalence of the OBDs corresponding to $\Theta_0'$ and $\Theta_1'$. 

Because the smooth homotopy $f_t$ has no birth-death critical points of index $(n,n+1)$, the corresponding $n$-skeleta of $\pm f_{t}$ for $t=0,1$ are ambient isotopic. Since $\Theta_t$ is bun-supported by $f_t$ for $t=0,1$, the handlebodies $H_{i,0}$ and $H_{i,1}$ are smoothly isotopic. By the wrinkled stabilization of the Legendrian cores of the critical contact handles as described in \autoref{thm: layer between N and N' general case}(1), the wrinkled contact handlebodies $H_{i,0}'$ and $H_{i,1}'$ are then contact isotopic. Thus we can connect the $\theta$-decompositions $\Theta_0'$ and $\Theta_1'$ via these contact isotopies, and after normalizing the contact isotopy on the buns we are reduced to a family of $\theta$-decompositions as described by \autoref{theorem: sigma times interval} in Step 1. From that theorem it follows that the OBDs corresponding to $\Theta_0'$ and $\Theta_1'$, and thus $\Theta_0$ and $\Theta_1$, are stably equivalent.
\end{proof}

\s\n
{\em Step 3: Birth-death homotopies.} 
\s

This step corresponds to \autoref{section: proof of theorem: birth-death}, where we prove a theorem comparable to \autoref{thm: no birth-death} for a $1$-Morse family with a middle-index birth-death critical point. 

\begin{theorem}\label{theorem: birth-death}
Suppose that $f_t:(M^{2n+1}, \xi) \to \R$, $t\in [t_0 - \epsilon, t_0 + \epsilon]$, is a $1$-Morse family with $f_{t}$ Morse for $t\neq t_0$ and with a single birth-death critical point of index $(n,n+1)$ at $t=t_0$. Then there are $\theta$-decompositions $\Theta_{t_0\pm \epsilon}$ such that 
\begin{enumerate}
    \item $\Theta_{t_0\pm \epsilon}$ is bun-supported by $f_{t_0\pm \epsilon}$, and 
    \item the OBDs corresponding to $\Theta_{t_0 - \epsilon}$ and $\Theta_{t_0 + \epsilon}$ are stably equivalent. 
\end{enumerate}
\end{theorem}

\begin{remark}
Again, we are not assuming that $f_{t_0\pm \epsilon}$ are contact Morse functions.
\end{remark}

Here we describe the buns of the $\theta$-decompositions $\Theta_{t_0\pm \epsilon}$. In \autoref{section: proof of theorem: birth-death} we will describe the bypass decompositions of the patties and prove that the resulting OBDs are stably equivalent. First, assume that the birth-death point at $t=t_0$ is a birth point; the case of a death point is identical. 

\begin{itemize}
    \item {\em The $\theta$-decomposition $\Theta_{t_0-\epsilon}$.} 
    
    Here we simply apply the technology of \cite{HH19} to $f_{t_0-\epsilon}$ without need for specification. In particular, by the proof of Corollary 1.3.1 and in particular Claim 10.0.1 in \cite{HH19}, we may let $H_{0,t_0-\epsilon}$ be a contact handlebody that is $C^0$-close to a neighborhood of the $n$-skeleton of $f_{t_0-\epsilon}$, and let $H_{1,t_0-\epsilon}$ be the same for $-f_{t_0-\epsilon}$. The buns are then $H_{i,t_0-\epsilon}$ and the patty $\Sigma_{t_0-\epsilon} \times [0,1]$ is the region between $H_{0,t_0-\epsilon}$ and $H_{1,t_0-\epsilon}$. See the left side of \autoref{fig:buns}. By construction, $\Theta_{t_0-\epsilon}$ is bun-supported by $f_{t_0-\epsilon}$.  
    
    \item {\em The $\theta$-decomposition $\Theta_{t_0+\epsilon}$.} 
    
    We build a $\theta$-decomposition $\Theta_{t_0+\epsilon}$ starting from $\Theta_{t_0-\epsilon}$. Note that at $t=t_0+\epsilon$ there is a smoothly \textit{but not necessarily contact-topologically} canceling pair of $n$- and ($n+1$)-handles. We $C^0$-closely approximate the $n$-handle by a sufficiently stabilized\footnote{Here, ``sufficiently stabilized'' means sufficiently stabilized when $n=1$ and either loose or wrinkled with an unfurled swallowtail when $n>1$.} Legendrian $L_0$, subject to the condition that the attaching locus of $L_0$ on $\partial H_{0,t_0-\epsilon}$ is a sufficiently stabilized Legendrian in the dividing set of the latter. Similarly, we let $L_1$ be the analogous Legendrian approximation of the cocore of the ($n+1$)-handle. We set $H_{0,t_0+\epsilon} := H_{0,t_0-\epsilon} \cup N(L_0)$ and $H_{1,t_0+\epsilon} := H_{1,t_0-\epsilon} \cup N(L_1)$, where $N(L_0)$ and $N(L_1)$ are contact handles corresponding to $L_0$ and $L_1$. The buns are then $H_{i,t_0+\epsilon}$ and the patty $\Sigma_{t_0+\epsilon} \times [0,1]$ is the region between $H_{0,t_0-\epsilon}$ and $H_{1,t_0-\epsilon}$. See the right side of \autoref{fig:buns}. Again, by construction $\Theta_{t_0+\epsilon}$ is bun-supported by $f_{t_0+\epsilon}$. 
    
\end{itemize}

\begin{figure}[ht]
	\centering
	\begin{overpic}[scale=.32]{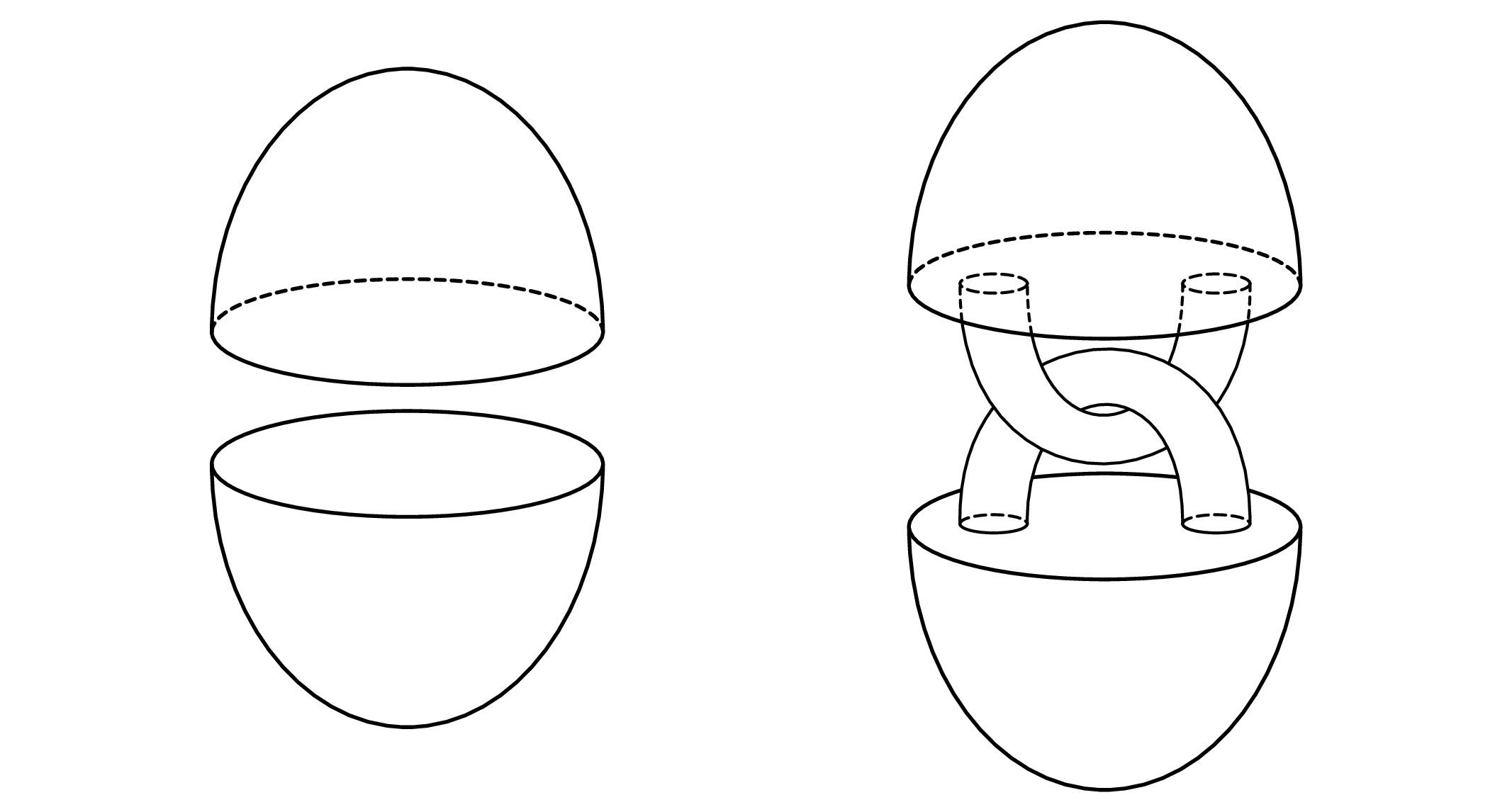}
	\put(23.5,2.25){\small $\Theta_{t_0-\epsilon}$}
	\put(22.25,13.5){\small $H_{0,t_0-\epsilon}$}
	\put(22.25,38.5){\small $H_{1,t_0-\epsilon}$}
	\put(69.5,-1.5){\small $\Theta_{t_0+\epsilon}$}
	\put(54.5,30){\small $N(L_1)$}
	\put(83,22){\small $N(L_0)$}
	\end{overpic}
	\caption{The buns of the $\theta$-decompositions $\Theta_{t_0 \pm \epsilon}$ on either side of a birth point of $f_t$.}
	\label{fig:buns}
\end{figure}
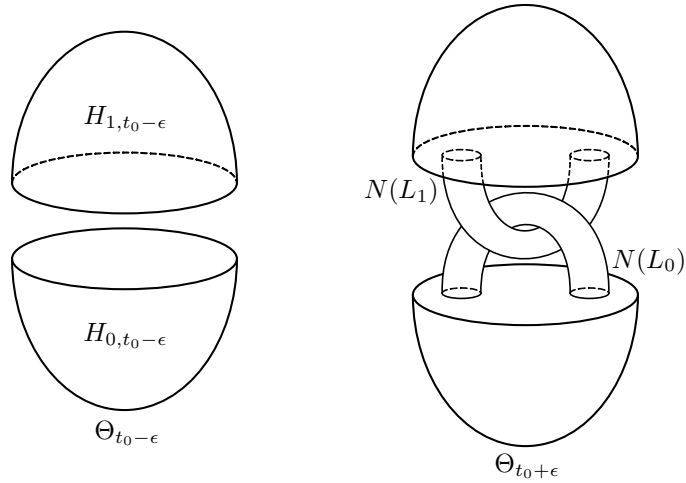

\subsection{Proof of \autoref{theorem: stabilization equivalence}.}\label{subsection:actual_proof_of_main_theorem} With the tools above we now prove our main theorem. 

\begin{proof}[Proof of \autoref{theorem: stabilization equivalence}.]
Given two supporting OBDs of $(M^{2n+1}, \xi)$, let $f_t:(M, \xi)\to \R$, $t\in [0,1]$ be a $1$-Morse family of smooth functions such that $f_0$ and $f_1$ are contact Morse functions corresponding to the OBDs. Let $\Theta_0$ and $\Theta_1$ be corresponding $\theta$-decompositions. 

By genericity there is a sequence $0 < t_1 < \cdots < t_K < 1$ such that $f_{t_k}$ has a single birth-death critical point of index $(n,n+1)$ for $k=1, \dots, K$, and $f_t$ has no such points for $t\neq t_1, \dots, t_K$. By \autoref{theorem: birth-death} in Step 3, for $\epsilon > 0$ sufficiently small there are $\theta$-decompositions $\Theta_{t_k\pm \epsilon}$, $k=1, \dots, K$, each bun-supported by $f_{t_k\pm \epsilon}$, such that:
\begin{itemize}
    \item For each $k=1, \dots, K$, the OBD corresponding to $\Theta_{t_k- \epsilon}$ is stably equivalent to that of $\Theta_{t_k+ \epsilon}$.
\end{itemize}
Next, since $f_t$ has no birth-death points of index $(n,n+1)$ for $t\neq t_1, \dots, t_K$, the following statements are applications of \autoref{thm: no birth-death} in Step 2: 
\begin{itemize}
    \item The OBD corresponding to $\Theta_0$ is stably equivalent to that of $\Theta_{t_1 - \epsilon}$.
    \item For $k= 1, \dots, K-1$ the OBD corresponding to $\Theta_{t_k+\epsilon}$ is stably equivalent to that of $\Theta_{t_{k+1}-\epsilon}$. 
    \item The OBD corresponding to $\Theta_{t_{K}+\epsilon}$ is stably equivalent to that of $\Theta_1$. 
\end{itemize}
It follows that the OBDs corresponding to $\Theta_0$ and $\Theta_1$ are stably equivalent.
\end{proof}

\section{Homotopies of patties} \label{sec: sigma times interval}

The goal of this section is to prove \autoref{theorem: sigma times interval}. The argument rests on the study of $2$-parameter families of characteristic foliations on a hypersurface $\Sigma$, and the interpretation of retrogradients in terms of bypass attachments. For convenience, we give an overview of the proof and organization of the section.

\s\n
{\em Overview of the proof.} Recall that in \autoref{theorem: sigma times interval}, $\xi_t, 0\leq t \leq 1$ is a family of contact structures on $\Sigma \times [0,1]_s$ such that $\xi_t$ is $t$-independent near the convex hypersurfaces $\Sigma \times \{i\}$ for $i=0,1$. We can assume that for $t=0,1$, the hypersurfaces $\Sigma \times \{s\} \subseteq (\Sigma \times [0,1], \xi_t)$ are convex except at a finite number of values of $s$, at which there is a retrogradient of the characteristic foliation. Each failure of convexity is described by a bypass attachment. Our goal is to prove that the sequences of bypasses for $(\Sigma \times [0,1], \xi_0)$ and $(\Sigma \times [0,1], \xi_1)$ are the same, up to far commutativity and addition (deletion) of trivial bypasses. Note that the latter corresponds to positive (de)stabilization of the OBD of $M= H_0 \cup (\Sigma \times [0,1]) \cup H_1$. 

The proof involves two steps. First, we relate the sequences of retrogradients in $(\Sigma \times [0,1], \xi_0)$ and $(\Sigma \times [0,1], \xi_1)$ Cerf-theoretically, which provides a finite list of $2$-parametric local models for retrogradient bifurcations. Second, we interpret each of these $2$-parametric local models in terms of bypass attachments, in particular showing that they amount to shuffling, adding, and deleting trivial bypasses. 

\s\n
{\em Organization.} The section is organized as follows. 

\begin{itemize}
    \item In \autoref{subsection:retrogradients_2_parameter}, we perform $2$-parametric Cerf theory on the retrogradients of the characteristic foliation of $\Sigma$ in order to generate a list of codimension-$2$ bifurcations that need to be analyzed in terms of bypass attachments.
    
    \item In \autoref{subsection:parallel_strands_bypass_lemma}, we prove \autoref{lemma:parallel_strands_bypass}, which describes how to perform a sequence of ``parallel'' bypasses in succession. 
    
    \item Finally, in \autoref{subsection:bypasses_2_parameter} we prove \autoref{theorem: sigma times interval}. In particular, we use \autoref{lemma:parallel_strands_bypass} from \autoref{subsection:parallel_strands_bypass_lemma} to analyze the codimension-$2$ retrogradient bifurcations from \autoref{subsection:retrogradients_2_parameter} in terms of bypass attachments. 
\end{itemize}

\subsection{Retrogradients in 2-parameter families}\label{subsection:retrogradients_2_parameter}

By \autoref{theorem:family genericity} a $1$-parameter family of hypersurfaces $\Sigma^{2n}$ in a contact manifold is $C^0$-generically convex except at a finite number of times. Each isolated failure of convexity is given by a retrogradient trajectory of the characteristic foliation from a nondegenerate negative index $n$ singular point to a nondegenerate positive index $n$ singular point. In this subsection we study the failure of convexity in the $2$-parameter case. In particular, let $\xi_t$, $t\in [0,1]$, be a family of contact structures on $\Sigma \times [0,1]_s$ which is independent of $t$ on $\Sigma\times\{0,1\}$. Let $\Sigma_{s,t}$ denote the characteristic foliation of $\Sigma_s:=\Sigma\times\{s\}$ with respect to $\xi_t$. We investigate all possibilities for failure of convexity of $\Sigma_s$ with respect to $\xi_t$ under suitably generic conditions. 

First, we recall some basic facts about the stable and unstable manifolds of nondegenerate, birth-death, and swallowtail singularities of smooth functions $f:\Sigma \to \R$ \cite{HW73} (here we are using the positive gradient flow of $f$).
\begin{itemize}
    \item Let $p$ be a nondegenerate critical point of index $k$. The stable manifold $W^{\mathrm{s}}(p)$ is a disk of dimension $k$ and the unstable manifold $W^{\mathrm{u}}(p)$ is a disk of dimension $2n-k$. In particular, there is a coordinate neighborhood of $p$ and a choice of auxiliary metric such that $p=0\in \R^{2n}$, 
    \[
    f(\mathbf{x}) = -x_1^2 + \cdots + -x_k^2 + x_{k+1}^2 + \cdots + x_{2n}^2,
    \]
    and 
    \begin{align*}
        W^{\mathrm{s}}(p) &= \R^k_{x_1, \dots, x_k} \times \{0\}_{x_{k+1},\dots, x_{2n}}, \\
        W^{\mathrm{u}}(p) &= \{0\}_{x_1, \dots, x_k}\times \R^{2n-k}_{x_{k+1}, \dots, x_{2n}}.
    \end{align*}
    
    \item Let $p$ be a birth-death singularity of index $(k,k+1)$, so that resulting nondegenerate critical points in the universal unfolding have index $k$ and $k+1$. The stable manifold $W^{\mathrm{s}}(p)$ of $p$ is a half-space of dimension $k+1$, and the unstable manifold $W^{\mathrm{u}}(p)$ is a half-space of dimension $2n-k$. In particular, there is a coordinate neighborhood of $p$ and a choice of auxiliary metric such that $p=0\in \R^{2n}$, 
    \[
    f(\mathbf{x}) = -x_1^2 + \cdots + -x_k^2 + x_{k+1}^2 + \cdots +x_{2n-1}^2 + x_{2n}^3,
    \]
    and 
    \begin{align*}
        W^{\mathrm{s}}(p) &= \R^k_{x_1, \dots, x_k} \times \{0\}_{x_{k+1},\dots, x_{2n-1}} \times \{x_{2n}\leq 0\}, \\
        W^{\mathrm{u}}(p) &= \{0\}_{x_1, \dots, x_k}\times \R^{2n-k-1}_{x_{k+1}, \dots, x_{2n-1}} \times \{x_{2n}\geq 0\}.
    \end{align*}
    
    \item Let $p$ be a swallowtail singularity of index $k$. The stable manifold $W^{\mathrm{s}}(p)$ is a disk of dimension $k$ and the unstable manifold $W^{\mathrm{u}}(p)$ is a disk of dimension $2n-k$. In particular, there is a coordinate neighborhood of $p$ and a choice of auxiliary metric such that $p=0\in \R^{2n}$, 
    \[
    f(\mathbf{x}) = -x_1^2 + \cdots + -x_{k-1}^2 + -x_k^4 + x_{k+1}^2 + \cdots  + x_{2n}^2,
    \]
    and 
    \begin{align*}
        W^{\mathrm{s}}(p) &= \R^k_{x_1, \dots, x_k} \times \{0\}_{x_{k+1},\dots, x_{2n}}, \\
        W^{\mathrm{u}}(p) &= \{0\}_{x_1, \dots, x_k}\times \R^{2n-k}_{x_{k+1}, \dots, x_{2n}}.
    \end{align*}
\end{itemize}

\begin{remark}
Assume that $\Sigma^{2n}$ is Weinstein convex. If $p$ is a positive birth-death singularity of index $(k,k+1)$, it must be the case that $0\leq k \leq n-1$, and if $p$ is a negative birth-death singularity of index $(k,k+1)$, it must be the case that $n \leq k \leq 2n-1$.
\end{remark}

\begin{defn}
A smooth function $f:\Sigma\to \R$ is {\em $0$-Morse} or simply {\em Morse} (resp.\ {\em $1$-Morse; $2$-Morse}) if each critical point of $f$ is nondegenerate (resp.\ nondegenerate or of birth-death type; nondegenerate, of birth-death type, or of swallowtail type). 
For $k=0,1,2$, a vector field $X$ on $\Sigma$ is {\em $k$-Morse} if $X$ is gradient-like for some $k$-Morse function $f:\Sigma \to \R$. More generally, we say that a $j$-parameter family of functions $f_{\mathbf{t}}:\Sigma \to \R$ with $\mathbf{t} \in \D^j$ is {\em $k$-Morse} if the function $f_{\mathbf{t}}$ is $k$-Morse for each $\mathbf{t} \in \D^j$. (Here $\D^j$ is the $j$-dimensional unit disk.)
\end{defn}

\begin{remark}
In general, a $k$-Morse vector field is the correct type of vector field that arises generically in a $k$-parameter family. Indeed, it is a classical result due to work of Cerf and Hatcher-Wagoner \cite{HW73} that a generic $2$-parameter family of vector fields $X_{s,t}$ for $(s,t)\in \D^2$ is $2$-Morse. 
\end{remark}

\begin{prop} \label{prop: extension of theorem: sigma times interval}
Let $\xi_t$, $t\in[0,1]$, be a $1$-parameter family of contact structures on $\Sigma\times[0,1]_s$ such that:
\begin{itemize}
\item $\Sigma\times\{0,1\}$ are Weinstein convex for all $t\in[0,1]$,
\item $\xi_t$ is independent of $t$ along $\Sigma\times\{0,1\}$ and \item \autoref{theorem:family genericity} holds for $t=0,1$.
\end{itemize}
Then:
\begin{itemize}
\item[(a)] the characteristic foliations $\Sigma_{s,t}$ can be made $2$-Morse after contact isotopies which leave $\partial (\Sigma\times[0,1])$ fixed and
\item[(b)] the stable and unstable submanifolds can be made to intersect as transversely as possible in a family.
\end{itemize}
\end{prop}

\begin{proof}[Idea of proof]
This is just a $2$-parametric version of the proof of \autoref{theorem:family genericity} from \cite[Section 9]{HH19} and we only provide a sketch. First we make all the singular points of $\Sigma_{s,t}$ $2$-Morse by $C^\infty$-small perturbations of $\xi_t$ rel boundary. Next, we can install/uninstall plugs locally (see the definition of a {\em $Y$-shaped plug with parameter $\epsilon>0$} in \cite[Definition 7.1.1]{HH19} and its construction with $\epsilon$ arbitrarily small when $Y$ is a standard Darboux ball in \cite[Theorem 7.5.1]{HH19}) so that all trajectories begin and end at singular points of $2$-Morse type and there are no cycles.  Once the above blocking condition is met, one can adjust the characteristic foliations on the nonsingular part by applying local perturbations of the form
$$U\times\{a\}\subset \Sigma\times\{a\}\rightsquigarrow\{s=f(x), x\in U\},$$ 
where $U\subset \Sigma$ is a small open set foliated by intervals of the characteristic foliation and $f(x)-a$ is $C^\infty$-close to zero with support in $U$, so that the stable and unstable submanifolds intersect as transversely as possible in a family.
\end{proof}

\begin{defn}
For $k=0,1,2$, a hypersurface $\Sigma \subset (M,\xi)$ is {\em $k$-Morse} if $\Sigma_{\xi}$ is $k$-Morse. We say that $\Sigma$ is {\em $k$-Morse$^+$} if, in addition, there are no trajectories from a negative singularity of $\Sigma_{\xi}$ to a positive singularity of $\Sigma_{\xi}$.
\end{defn}

Recall that a $1$-Morse$^+$ hypersurface is convex \cite{HH19}. The proof there easily generalizes to the following lemma:

\begin{lemma}\label{lemma:2_morse_convex}
Let $\Sigma \subset (M,\xi)$ be a $2$-Morse$^+$ hypersurface. Then $\Sigma$ is convex. 
\end{lemma}

Consequently, as in the $1$-parameter case, convexity fails $2$-parametrically only when there is a retrogradient trajectory, i.e., one from a negative singularity to a positive singularity. The following proposition describes all variants of this behavior. 

\begin{prop}\label{prop:codimension_2_bifurcations}
Let $X_{s,t}$ for $(s,t)\in [0,1]\times [0,1]$ be a generic $2$-parameter family of $2$-Morse characteristic foliations on $\Sigma^{2n}$, where $(\Sigma, X_{s,t})$ is $2$-Morse$^+$ and $t$-independent for $s=0,1$. Then there is an almost-everywhere immersed $1$-manifold-with-boundary $L\subset [0,1]\times [0,1]$ and a finite set $P \subset (0,1)\times (0,1)$ such that, for $(s,t) \notin L \cup P$, the hypersurface $(\Sigma, X_{s,t})$ is $2$-Morse$^+$ and hence convex. Moreover, failure of convexity along $L \cup P$ is described as follows:
\begin{enumerate}
    \item[(L)] For $(s,t) \in L\setminus P$, failure of convexity is given by a single transversely cut out retrogradient trajectory from a negative nondegenerate index $n$ singularity to a positive nondegenerate index $n$ singularity.
    
    \item[(P)] For each $(s,t) \in P$, failure of convexity is described by exactly one of the following:
    
    \begin{itemize}
        \item[(P1)] There are two separate retrogradients from negative nondegenerate index $n$ singularities to positive nondegenerate index $n$ singularities.
        
        \item[(P2)] There is a single retrogradient from a negative nondegenerate index $n$ singularity to a positive nondegenerate index $n$ singularity, but the respective unstable and stable manifolds are not transverse. 
        
        \item[(P3)] There is a single retrogradient from a negative nondegenerate point of index $n$ to a positive birth-death point of index $(n-1,n)$. 
        
        \item[(P3')] There is a single retrogradient from a negative birth-death point of index $(n,n+1)$ to a positive nondegenerate point of index $n$.
        
        \item[(P4)] There is a single retrogradient from a negative nondegenerate point of index $n$ to a positive nondegenerate point of index $n-1$.
        
        \item[(P4')] There is a single retrogradient from a negative nondegenerate point of index $n+1$ to a positive nondegenerate point of index $n$.

        \item[(P5)] There is a ``broken'' retrogradient of the form $p_1^- \rightsquigarrow p_2^- \rightsquigarrow p^+$, where $p_1^-$ and $p_2^-$ are negative nondegenerate points of index $n$ and $p^+$ is a positive nondegenerate point of index $n$.

        \item[(P5')] There is a ``broken'' retrogradient of the form $p^- \rightsquigarrow p_1^+ \rightsquigarrow p_2^+$, where $p^-$ is a negative nondegenerate point of index $n$ and $p_1^+$ and $p_2^+$ are positive nondegenerate points of index $n$.
        
    \end{itemize}
\end{enumerate}
\end{prop}

\begin{figure}[ht]
\vskip-.05in
	\begin{overpic}[scale=0.28]{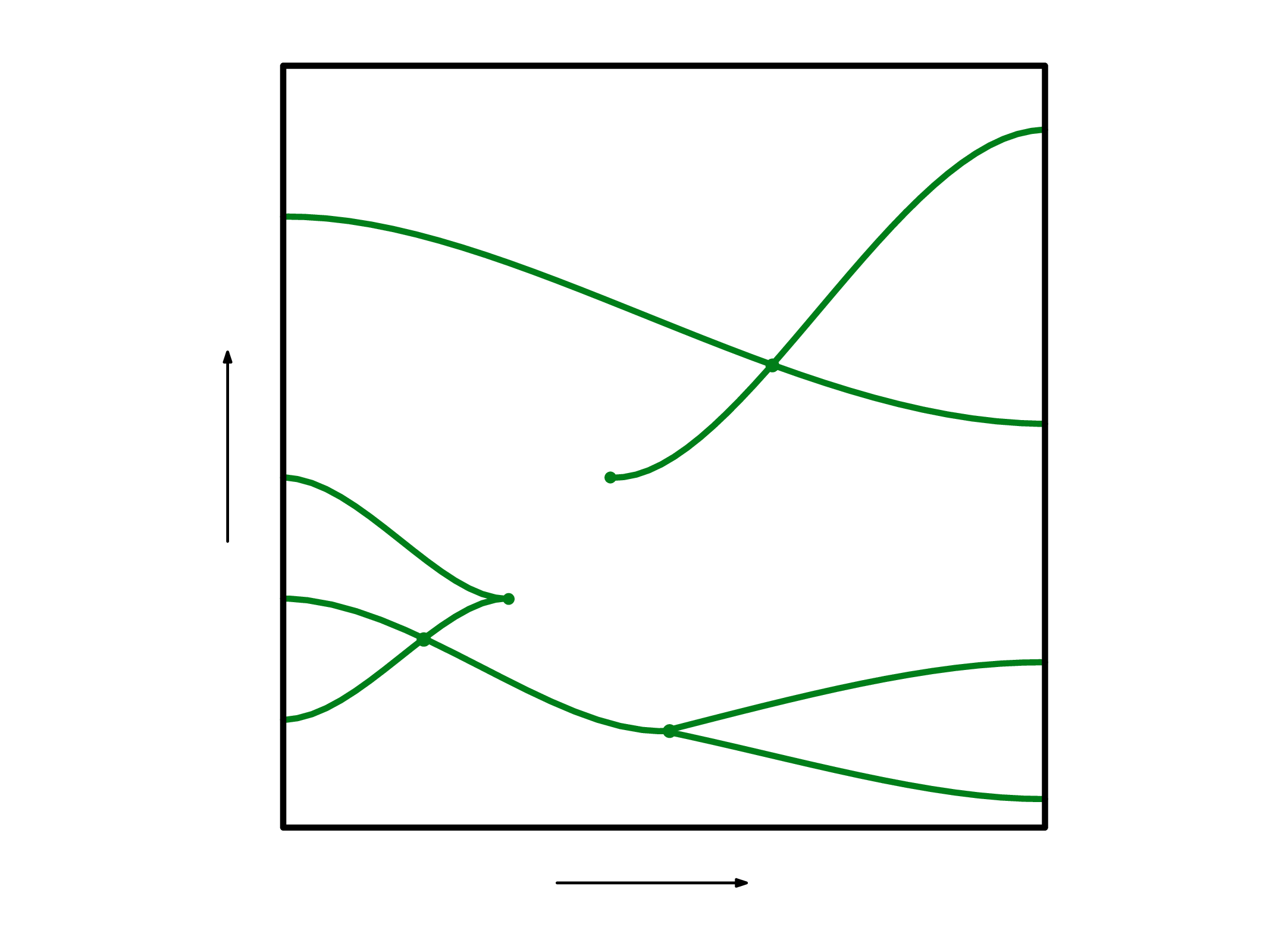}
	    \put(17.25,48.5){\small $s$}
	    \put(60.5,5){\small $t$}
	    \put(39,30){\tiny (P2)}
	    \put(32,20.75){\tiny (P1)}
	    \put(49,19.5){\tiny (P4)}
	    \put(57.5,49){\tiny (P1)}
	    \put(45,39.5){\tiny (P3)}
	\end{overpic}
	\vskip-.05in
	\caption{An example $2$-parameter space as described by \autoref{prop:codimension_2_bifurcations}. The contact manifolds $(\Sigma \times [0,1]_s, \xi_i)$ for $i=0,1$ are both described by a sequence of $4$ bypasses.}
	\label{fig:2parameter}
\end{figure}
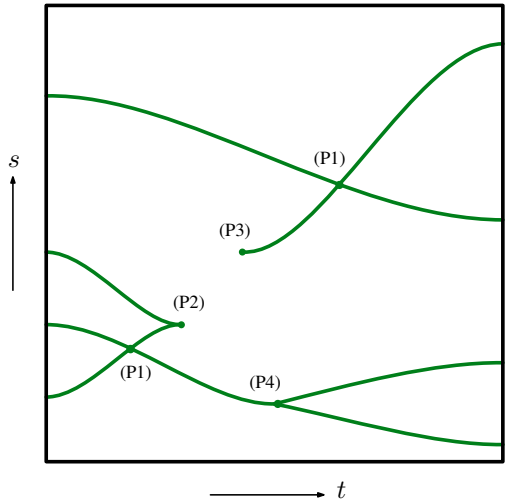

\begin{remark}
When $n=1$, the only possibilities in \autoref{prop:codimension_2_bifurcations} are (L), (P1), (P3), (P3'), (P5), and (P5').
\end{remark}

\begin{proof}
We analyze all potential retrogradients according to the type of critical points involved (nondegenerate, birth-death, and swallowtail). Throughout, the notation $p^{\pm} = p^{\pm}_{s,t}$  refers to a (family of) positive and negative singularities, respectively, and $W^{\mathrm{s}}(p)$ and $W^{\mathrm{u}}(p)$ are the stable and unstable manifolds of $p$. Recall that in a generic $2$-parameter family, swallowtail singularities are a codimension $2$ phenomenon (i.e., they arise at isolated parameter values), and birth-death singularities are a codimension $1$ phenomenon (i.e., they arise in $1$-dimensional families). Let $E_{\mathrm{bd}}, E_{\mathrm{sw}} \subset \D^2_{s,t}:=[0,1]_s \times [0,1]_t$ be the sets of parameter values where $X_{s,t}$ has birth-death and swallowtail singularities, respectively. By \autoref{prop: extension of theorem: sigma times interval}, the singularities behave as in a generic $2$-parameter family of Morse functions and the stable and unstable manifolds intersect as transversely as possible in a generic $2$-parameter family.  Hence we may assume that $E_{\mathrm{bd}}$ is an immersed $1$-manifold-with-boundary, $E_{\mathrm{sw}}$ is a finite set of points, and:
\begin{itemize}
    \item[(T1)] There is a finite number of isolated parameter values where $X_{s,t}$ has two birth-death singularities. In particular, $E_{\mathrm{bd}}$ has a finite number of double points, and is otherwise properly embedded in $\D^2_{s,t}$.
    
    \item[(T2)] There is a finite number of isolated parameter values where $X_{s,t}$ has a single swallowtail singularity, and at these moments $X_{s,t}$ has no birth-death singularities. In particular, $E_{\mathrm{bd}} \cap E_{\mathrm{sw}} = \varnothing$. 
\end{itemize}

\s\n
\noindent { \em Case I: Retrogradients involving exactly two nondegenerate points.}

\s 
First we analyze the possibilities of retrogradients appearing between exactly two nondegenerate critical points. We will do the transversality analysis carefully in this case. Let $f_{s,t}:\Sigma^{2n} \to \R$ be a $2$-parameter family of Lyapunov functions for $X_{s,t}$. Since we are only concerned with nondegenerate critical points, we may localize and assume that each $f_{s,t}$ is genuinely Morse. Let $p_{s,t}^{\pm}$ be (families of) nondegenerate critical points of index $k^{\pm}$, where 
\begin{equation} \label{equation: k plus k minus}
0 \leq k^+ \leq n \leq k^- \leq 2n.
\end{equation}
Because we are interested in existence of retrogradients, we may assume that there is a $c_0 \in \R$ such that, for all $s,t$, 
\[
f_{s,t}(p_{s,t}^-) < c_0 < f_{s,t}(p_{s,t}^+)
\]
and such that $c_0$ is regular for each $f_{s,t}$. Assume further that $f_{s,t}$ is $(s,t)$-independent along $N^{2n-1}:= f_{s,t}^{-1}(c_0) \subset \Sigma$. Finally we may assume that the families of nondegenerate critical points are not part of birth-death or swallowtail families, as these cases will be considered later. 

We are interested in the intersection of the unstable and stable manifolds of $p_{s,t}^{\pm}$ in $N$. To this end, we introduce some notation. Let 
\begin{gather*}
    W^{\mathrm{u}} := W^{\mathrm{u}}_{s,t}(p_{s,t}^-), \quad 
    W^{\mathrm{s}} := W^{\mathrm{s}}_{s,t}(p_{s,t}^+), \quad
    W^{\ast}_N := W^{\ast} \cap N, \\
    \tilde{W}^{\ast}_N := \set{(s,t,x)}{x \in W^{\ast}_N} \subseteq \D^2_{s,t} \times N.
\end{gather*}
Given $p^{\pm}$, our goal is to analyze
\[
L(p^-, p^+) := \set{(s,t)}{W^{\mathrm{u}}_N \cap W^{\mathrm{s}}_N \neq \varnothing} \subseteq \D^2_{s,t}
\]
and then ultimately the union $L$ of these sets over all pairs of positive and critical points. In words, $L(p^-, p^+)$ is the set of parameter values for which there exists a retrogradient flow line from $p^-$ to $p^+$. Define 
\[
\tilde{L}(p^-, p^+) := \set{(s,t,x)}{x \in W^{\mathrm{u}}_N \cap W^{\mathrm{s}}_N} \subseteq \D^2_{s,t} \times N
\]
and observe that $\pi(\tilde{L}) = L$, where $\pi:\D^2_{s,t}\times N \to \D^2_{s,t}$ is the projection onto the first factor. Also observe that $\tilde{L}(p^-,p^+) = \tilde{W}_N^{\mathrm{u}} \cap \tilde{W}_N^{\mathrm{s}}$.

For a fixed $(s,t)$-parameter value, we have $W_N^{\mathrm{u}} \cong S^{2n-k^--1}$ and $W_N^{\mathrm{s}} \cong S^{k^+-1}$. Therefore, 
\begin{gather*}
    \tilde{W}^{\mathrm{u}}_N = \op{Im} \left(\D^2_{s,t} \times S^{2n-k^--1} \hookrightarrow\D^2_{s,t} \times N\right), \\
    \tilde{W}^{\mathrm{s}}_N = \op{Im}\left(\D^2_{s,t} \times S^{k^+-1} \hookrightarrow \D^2_{s,t} \times N\right). 
\end{gather*}
By transversality, it follows that $\tilde{L}(p^-,p^+) \subset \D^2_{s,t} \times N^{2n-1}$ is a submanifold of dimension 
\begin{align*}
    (2n-k^--1 + 2) + (k^+ -1 + 2) - (2n-1 + 2) = k^+ - k^- + 1.
\end{align*}
By \eqref{equation: k plus k minus}, we have $k^+ - k^- + 1 \leq 1$, with three possibilities to consider:
\begin{enumerate}
    \item[(Ia)] ($k^+ = n$ and $k^- = n$). In this case $k^+ - k^- + 1 = 1$ and so $\tilde{L} \subseteq \D^2_{s,t} \times N$ is a smooth $1$-manifold with boundary contained in $[0,1]_s \times \{0,1\}_t \times N$. The projection $L=\pi(\tilde L)\subseteq \D^2_{s,t}$, which we may take to be generic, is an immersed curve except at a finite number of points, with the following codimension-$2$ phenomena:  
    \begin{itemize}
        \item Regular double points of $L$. At finitely many points of $L$, two transversely cut-out retrogradients exist simultaneously. This gives (P1).
        
        \item Cusp points of $L$.  This happens when the projection map $\pi:\tilde{L} \to L$ fails to be an immersion. In particular, at such a point $(s_0, t_0, x_0)$ the differential $d\pi = 0$ and so $\dim T_{x_0}W^{\mathrm{u}}_N \cap T_{x_0}W^{\mathrm{s}}_N = 1$. In particular, at a cusp of $L$, there is a single retrogradient between nondegenerate points of index $n$, but it is not cut out transversely in $\Sigma$. This gives (P2).
    \end{itemize}

    \item[(Ib)] ($k^+ = n$ and $k^- = n +1$). In this case $k^+ - k^- + 1 = 0$. Hence there is an isolated instance of a retrogradient from a negative index $n+1$ singularity to a positive index $n$ singularity. This describes situation (P4').
    
    \item[(Ic)] ($k^+ = n -1$ and $k^- = n$.) In this case $k^+ - k^- + 1 = 0$. Hence there is an isolated instance of a retrogradient from a negative index $n$ singularity to a positive index $n-1$ singularity. This describes situation (P4).
\end{enumerate}

\begin{remark}
Consider a retrogradient $p_{n}^- \rightsquigarrow p_{n-1}^+$ as described by (P4), occurring at an isolated moment $(s_0, t_0)$ in $2$-parameter space. Suppose that $p_{n,1}^+, \dots, p_{n,N}^+$ are positive nondegenerate index $n$ points with trajectories $p_{n-1}^+ \rightsquigarrow p_{n,j}^+$. By a gluing argument, near $(s_0,t_0)$ in $2$-parameter space there will be ``normal'' retrogradient trajectories $p_{n}^- \rightsquigarrow p_{n,j}^+$. This is observation not necessary for the statement of the present proposition, but is important in the proof of \autoref{theorem: sigma times interval}.
\end{remark}

\s\n
\noindent {\em Case II: Retrogradients involving birth-death points.}

\s According to (T1), there are three possibilities to consider. Here the transversality analysis is brief, as the details are similar to (but simpler than) Case I. 

\begin{enumerate}
    \item[(IIa)] {\em Both $p^-$ and $p^+$ are birth-death points.}  Suppose $p^-$ and $p^+$ are negative and positive birth-death singularities of index $(k^-,k^-+1)$ and $(k^+,k^+ + 1)$, respectively. By (T1) this situation occurs in isolated instances. We have 
    \begin{gather*}
        \dim W^{\mathrm{u}}(p^-) = 2n-k^-, \quad
        \dim W^{\mathrm{s}}(p^+) = k^+ + 1,
    \end{gather*}
    and so the expected dimension of their transverse intersection in a regular intermediate level set $N^{2n-1}$ is
\[
(k^+) + (2n-k^- - 1) - (2n-1) = k^+ - k^-. 
\]
Since $k^+ +1 \leq n$ and $k^- \geq n$, we have $k^+ - k^- < 0$ and so generically there is no intersection. Thus, there are no retrogradients between birth-death points. 

\item[(IIb)] {\em The point $p^-$ is of birth-death type and $p^+$ is nondegenerate.} Suppose $p^-$ is a negative birth-death singularity of index $(k^-,k^-+1)$ and $p^+$ is nondegenerate of index $k^+$. Then 
    \begin{gather*}
        \dim W^{\mathrm{u}}(p^-) = 2n-k^-, \quad
        \dim W^{\mathrm{s}}(p^+) = k^+.
    \end{gather*}
Since $p^-$ arises in a $1$-dimensional family, the expected dimension of the transverse intersection of $W^{\mathrm{u}}(p^-)$ and $W^{\mathrm{s}}(p^+)$ in a regular intermediate level set $N^{2n-1}$ is 
\[
(k^+ - 1) + (2n-k^- - 1) - (2n-1) + (1) = k^+ - k^-.
\]
Note that $k^+ \leq n$ and $k^- \geq n$. Thus, when $k^- = k^+ = n$, we have an isolated retrogradient from the negative index $(n,n+1)$ birth-death point $p^-$ to a positive index $n$ nondegenerate singularity $p^+$. This gives (P3').

\item[(IIc)] {\em The point $p^-$ is nondegenerate, and $p^+$ is of birth-death type.} The dimension computation here is slightly different to the previous case. Let $k^- = \ind(p^-)$ and let $p^+$ be of index $(k^+,k^++1)$. Then 
    \begin{gather*}
        \dim W^{\mathrm{u}}(p^-) = 2n-k^-, \quad
        \dim W^{\mathrm{s}}(p^+) = k^+ + 1.
    \end{gather*}
Again, this situation occurs in a $1$-dimensional family, so the relevant transverse intersection in a regular intermediate level set has dimension 
\[
(k^+) + (2n - k^- - 1) - (2n-1) + (1) = k^+ - k^- + 1.
\]
This time, we have the restriction $k^+ + 1 \leq n$ and $k^- \geq n$, so $k^+ - k^- + 1 \leq 0$ with equality occurring when $k^+ =n-1$ and $k^- = n$. This gives an isolated retrogradient trajectory from a negative index $n$ nondegenerate singularity to a positive birth-death singularity of index $(n-1,n)$, which is (P3).
\end{enumerate}

\s\n
\noindent {\em Case III: Retrogradients involving swallowtail points.}

\s Let $p^-$ be a negative swallowtail singularity of index $k^-$ occurring at the parameter $(s_0,t_0)$ and let and let $p^+$ be a positive singularity. By (T2), we may assume that $p^+$ is nondegenerate in a neighborhood of $(s_0,t_0)$. Let $k^+ := \ind(p^+)$.

We have 
\begin{align*}
    \dim W^{\mathrm{u}}(p^-) = 2n-k^-, \quad     \dim W^{\mathrm{s}}(p^+) = k^+.
\end{align*}
With $N^{2n-1}$ a regular intermediate level set as before, the expected dimension of the transverse intersection inside $N$ is
\[
(k^+ - 1) + (2n-k^- - 1) - (2n-1) = k^+ - k^- -1.
\]
Since $k^+ \leq k^-$, we have $k^+ - k^- -1 < 0$. Thus, generically there are no retrogradients from negative swallowtail points to positive nondegenerate points.

The case where $p^+$ is an isolated positive swallowtail point of index $k^+$ and $p^-$ is a (family of) negative nondegenerate point(s) of index $k^-$ is similar. The dimension of the relevant stable and unstable manifold intersection in a regular intermediate level set is 
 \[
    (k^+- 1) + (2n-k^- - 1) - (2n-1) = k^+ - k^- - 1.
    \]
    This dimension count is identical to prior case, and so generically there are no retrogradients from a nondegenerate point to a swallowtail point. Therefore, generically there are no retrogradients involving swallowtail singularities at all. 

\s\n
\noindent { \em  Case IV: Retrogradients involving more than two critical points.}

\s     

There is one last situation to consider which, informally, is a consequence of the following fact: transversely cut-out retrogradients between nondegenerate index $n$ points occur in $1$-dimensional families, and positive-positive and negative-negative index $n$ handleslides occur in $1$-dimensional families. In a $2$-para\-meter family, this allows for codimension-$2$ isolated instances of ``broken'' retrogradients involving either two positive points and one negative point, or two negative points and one positive point.

More precisely, let $p_i^{\pm}$, $i=1,2$, denote (families of) nondegenerate positive (resp.\ negative) critical points of index $n$. We proceed in the same way as in the beginning of Case I. Assume there are values $c_-, c_0, c_+ \in \R$ such that 
\[
f(p_1^-) < c_- < f(p_2^-) < c_0 < f(p_1^+) < c_+ < f(p_2^+). 
\]
Let $N_{\ast}^{2n-1} := f^{-1}(c_{\ast})$ denote the corresponding regular level sets in $\Sigma$ and let $W_{N_{\ast}}^{\ast}(p_i^{\ast}) = W^{\ast}(p_i^{\ast}) \cap N_{\ast}$ denote the intersection of the corresponding (un)stable manifolds with these level sets. Let 
\begin{align*}
    L_- &:= \set{(s,t)}{W^{\mathrm{u}}_{N_-}(p_1^-) \cap W^{\mathrm{s}}_{N_-}(p_2^-) \neq \varnothing} \subseteq \D^2_{s,t} \\
    L_0 &:= \set{(s,t)}{W^{\mathrm{u}}_{N_0}(p_2^-) \cap W^{\mathrm{s}}_{N_0}(p_1^+) \neq \varnothing} \subseteq \D^2_{s,t} \\
    L_+ &:= \set{(s,t)}{W^{\mathrm{u}}_{N_+}(p_1^+) \cap W^{\mathrm{s}}_{N_+}(p_2^+) \neq \varnothing} \subseteq \D^2_{s,t}.
\end{align*}
By the dimension counting analysis in Case I, in particular (Ia), $L_-$, $L_0$, and $L_+$ are all immersed curves in $\D_{s,t}^2$. The $0$-dimensional intersection $L_-\cap L_0$ then corresponds to broken retrogradients $p_1^- \rightsquigarrow p_2^- \rightsquigarrow p_1^+$, which gives (P5), and the $0$-dimensional intersection $L_0 \cap L_+$ likewise corresponds to broken retrogradients $p_2^- \rightsquigarrow p_1^+ \rightsquigarrow p_2^+$, which gives (P5').
\end{proof}

\subsection{A parallel strands bypass lemma}\label{subsection:parallel_strands_bypass_lemma}

In this subsection we state and prove the main lemma needed to interpret all codimension-$2$ retrogradient bifurcations in terms of bypass attachments. Informally, the lemma describes how to interpret a sequence of $k$ ``locally parallel'' retrogradients as a single bypass attachment and can be viewed as a higher-dimensional analog of a ``bypass rotation" in dimension three. See \autoref{fig:strandFW} for a heuristic picture and Figures~\ref{fig:strand_lemma} and \ref{fig:dual_strand_lemma} for a precise formulation; we will also use the notation $Y,G_\pm,\Gamma_0,\Gamma_1$ as indicated in the figures.

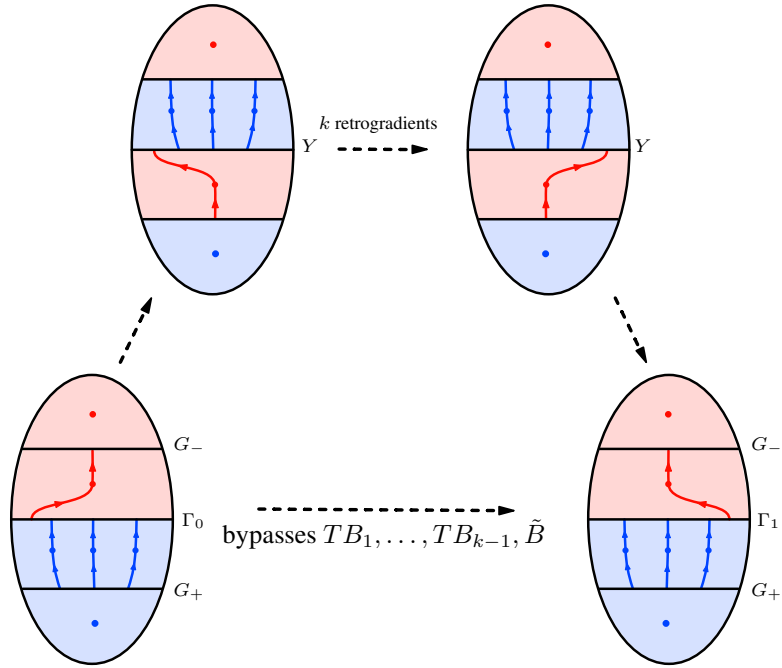
\begin{figure}[ht]
	\begin{overpic}[scale=0.43]{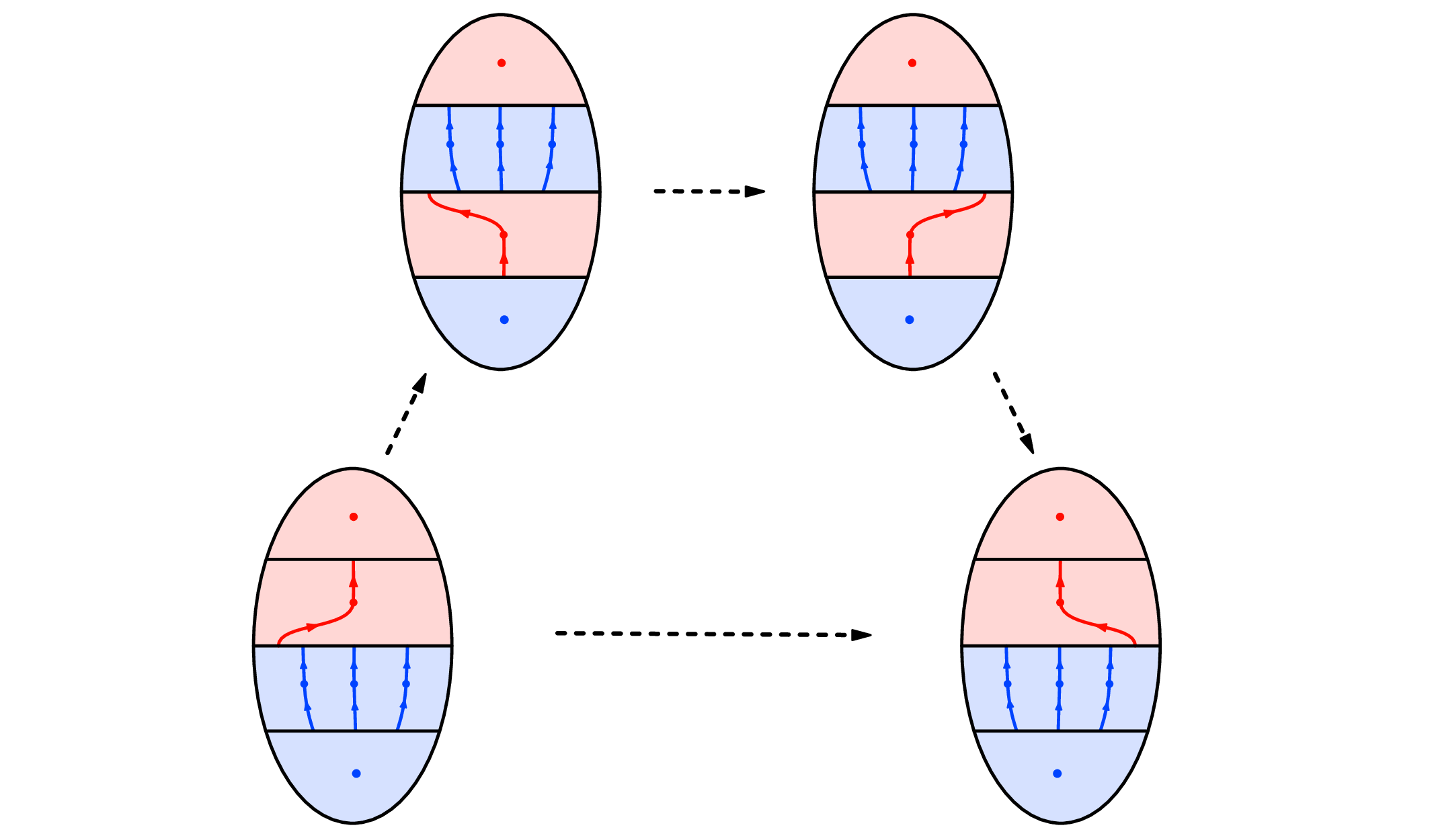}
	    \put(44,47){\tiny $k$ retrogradients}
	    \put(42.5,45){\tiny $Y$}
	    \put(71,45){\tiny $Y$}
	    \put(81,19.5){\tiny $G_-$}
	    \put(31.55,19.5){\tiny $G_-$}
	    \put(81,7){\tiny $G_+$}
	    \put(31.55,7){\tiny $G_+$}
	    \put(81.5,13){\tiny $\Gamma_1$}
	    \put(32.25,13){\tiny $\Gamma_0$}
		\put(35.75,11.5){\small bypasses $TB_1, \dots, TB_{k-1}, \tilde{B}$}
	\end{overpic}
	\caption{A heuristic depiction of \autoref{lemma:parallel_strands_bypass}. The second row contains the convex hypersurfaces $\Sigma \times \{i\}$ for $i=0,1$. The upper two folded Weinstein hypersurfaces depict a sequence of $k=3$ retrogradients, obtained after shuffling the central cobordisms. \autoref{lemma:parallel_strands_bypass} describes this as a single bypass attachment $\tilde{B}$ together with a sequence of trivial bypasses.}
	\label{fig:strandFW}
\end{figure}

\begin{lemma}[Parallel strands bypass lemma]\label{lemma:parallel_strands_bypass}
Suppose $(\Sigma^{2n} \times [0,1]_s, \xi), n\geq 2$ is a contact manifold such that there are $0 < s_1 < \cdots < s_k < 1$ satisfying the following properties:
\begin{enumerate}
    \item[(i)] For $s\neq s_j$, $\Sigma \times \{s\}$ is convex.
    \item[(ii)] For $s=s_j$, convexity fails because of a single retrogradient between nondegenerate critical points of index $n$, and the sequence of retrogradients is described by the local retrogradient diagram on the left side of \autoref{fig:strand_lemma}. In particular, each retrogradient is of the form $p^- \rightsquigarrow p_{j'}^+$ where $p^-$ is a single negative critical point and $\{p_1^+, \dots, p_{k'}^+\}$ is a set of positive critical points with $k'\leq k$. In particular $(\Sigma \times [0,1], \xi)$ is given by a sequence of corresponding bypass attachments $B_1, \dots, B_k$.
\end{enumerate}
Then, after isotopy of contact handles, the bypass sequence $B_1, \dots, B_k$ is equivalent to a sequence of bypasses $TB_1, \dots, TB_{k-1}, \tilde{B}$, where $TB_1, \dots, TB_{k-1}$ are trivial bypasses and $\tilde{B}$ is the bypass attached along $(\Lambda_{\pm}; D_{\pm})\subset \Sigma$ as indicated by the right side of \autoref{fig:strand_lemma}. Here $(\Lambda_-; D_-)$ is given by the stable disk of $p^-$ and $(\Lambda_+; D_+)$ is the standard disk filling of the indicated unknot surrounding all the strands.
\end{lemma}

\begin{figure}[ht]
	\begin{overpic}[scale=0.53]{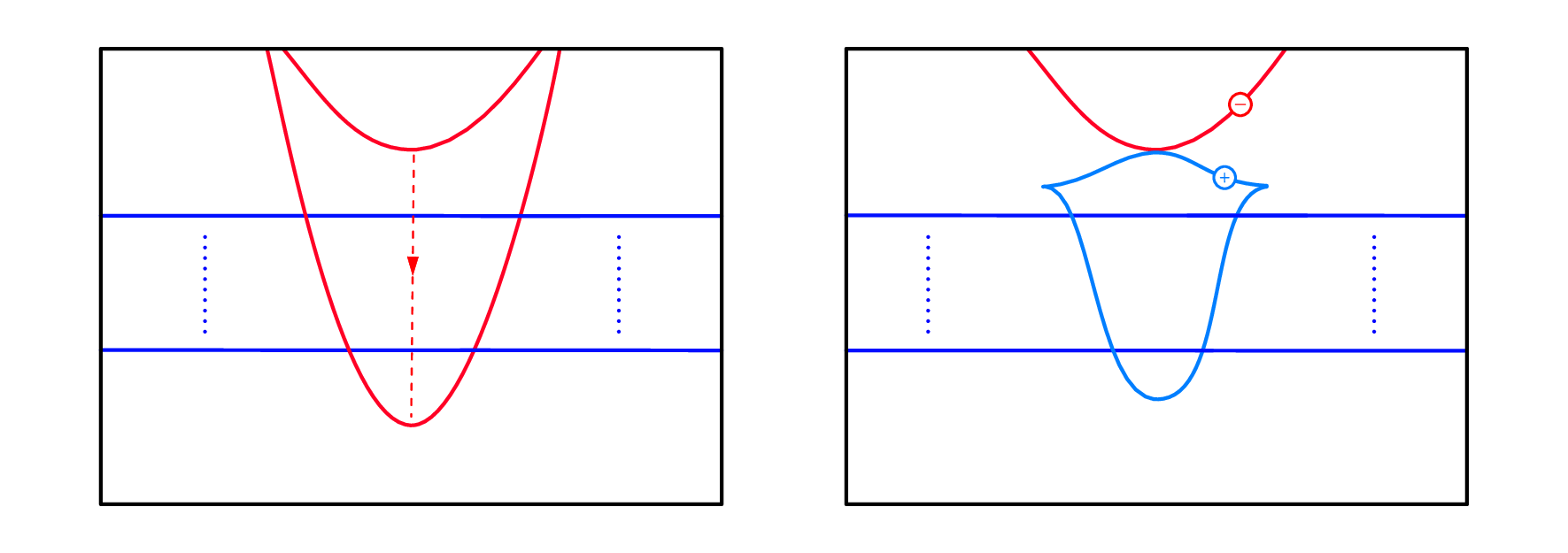}
	    \put(14,33){\small Local retrogradient diagram}
	    \put(26,1){\small $Y$}
	    \put(44,4){\small $Y$}
	    \put(64,33){\small Local bypass diagram}
	    \put(74,1){\small $\Gamma_0$}
	    \put(90,4){\small $G_+$}
	    \put(94,12.5){\tiny \textcolor{blue}{$(-1)$}}
	    \put(94,21){\tiny \textcolor{blue}{$(-1)$}}
	    \put(82,23){\small \textcolor{RoyalBlue}{$\Lambda_+$}}
	    \put(81,28){\small {\color{red}  $\Lambda_-$}}
	\end{overpic}
	\caption{The statement of \autoref{lemma:parallel_strands_bypass}. On the left is a local (parametric) retrogradient diagram on a folded Weinstein hypersurface that describes $k$ retrogradients, one for each $\xi$-transverse intersection of the parametric red strand with the $k$ blue strands. Note that the blue strands may or may not be part of distinct Legendrians. On the right is the nontrivial bypass diagram for $\tilde{B}$ provided by the lemma.}
	\label{fig:strand_lemma}
\end{figure}

%{\cbu  (TO DO: read through this version of the proof and try to improve it, make sure it still jives with the restated lemma)}

\begin{proof}[Proof of \autoref{lemma:parallel_strands_bypass}]
We induct on the number of strands. In the following, $p^-$ always refers to the nondegenerate negative critical point of index $n$ which is the source of each retrogradient, and we abuse notation and use $\ve, \ve', \ve''$ to denote sufficiently small parameters which often vary from step to step. The precise vertical placements of Legendrians are as indicated in the diagrams.

\s\n
{\em Base case $k=1$.} Suppose that the local retrogradient diagram has a single blue strand, corresponding to the intersection of the stable manifold of a positive critical point $p^+$ of index $n$ with a minimal folding locus $Y$. There is one bypass associated to the single retrogradient $p^- \rightsquigarrow p^+$. The bypass attachment data $(\Lambda_+,\Lambda_-; D_+,D_-)$ is such that $(\Lambda_\pm; D_\pm)$ is given by the cocore and belt sphere of the Weinstein $n$-handle in $R_\pm$ associated to $p^\pm$. Let $G_+\subseteq R_+ \subseteq \Sigma\times\{0\}$ be a regular contact level set such that the dividing set $\Gamma_0$ of $\Sigma\times\{0\}$ is obtained by $(-1)$-surgery along the attaching sphere of the handle associated to $p^+$. By \autoref{lemma:cocore_lemma}, we may identify $(\Lambda_+; D_+)$ as on the right side of \autoref{fig:strand_lemma} with one blue strand. This completes the base case.

\s\n
{\em Inductive step.} Suppose the statement of the lemma holds for some $k\geq 1$ and the local retrogradient diagram in a minimal folding locus $Y$ consists of $k+1$ parallel strands as on the left side of \autoref{fig:strand_lemma}, so that convexity of $\Sigma \times \{s\}$ fails for 
\[
0 < s_1 < \cdots < s_k < s_{k+1} < 1.
\] 
Recall that some of the parallel strands may or may not be associated to the same Legendrian, and that the proof does not require a distinction to be made. As before, let $G_+ \subseteq R_+ \subseteq \Sigma \times \{0\}$ be a regular contact level set such that the dividing set $\Gamma_0$ is obtained by $(-1)$-surgery along each parallel strand. By the inductive hypothesis, the contact manifold $(\Sigma \times [0,s_k + \ve], \xi)$ is contactomorphic to a single bypass attachment to $(\Sigma \times [0,\ve], \xi)$ along $(\Lambda_{\pm}; D_{\pm})\subset \Sigma\times\{\ve\}$ as indicated on the left side of \autoref{fig:inductive_step}. 

\begin{figure}[ht]
	\begin{overpic}[scale=0.53]{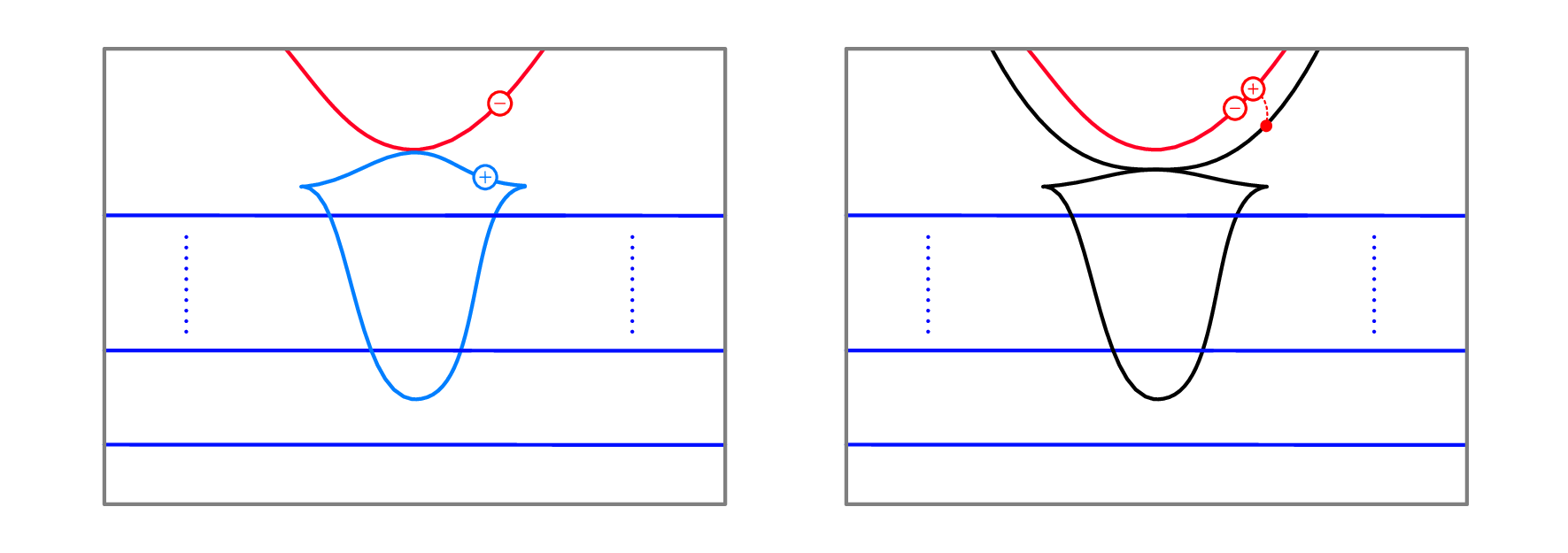}
	    \put(26,1){\small $\Gamma_0$}
	    \put(43,4){\small $G_+$}
	    \put(35,23){\small \textcolor{RoyalBlue}{$\Lambda_+$}}
	    \put(34,28){\small {\color{red}   $\Lambda_-$}}
	    \put(47,12.5){\tiny \textcolor{blue}{$(-1)$}}
	    \put(47,21){\tiny \textcolor{blue}{$(-1)$}}
	    \put(47,6.5){\tiny \textcolor{blue}{$(-1)$}}

	    \put(74,1){\small $\Gamma_{s_k+\ve}$}
	    \put(90,4){\small $G_+$}
	    \put(94,12.5){\tiny \textcolor{blue}{$(-1)$}}
	    \put(94,21){\tiny \textcolor{blue}{$(-1)$}}
	    \put(94,6.5){\tiny \textcolor{blue}{$(-1)$}}
	    \put(94,25){\tiny $(-1)$}
	    \put(94,28){\tiny {\color{red}   $(+1)$}}
	    \put(82,24){\small $\Lambda_- \uplus \Lambda_+$}
	    \put(84,28){\small {\color{red}   $\Lambda_-^{\ve}$}}
		
	\end{overpic}
	\caption{Applying the inductive hypothesis in the proof of \autoref{lemma:parallel_strands_bypass}. In both figures, there is a total of $k+1$ parallel blue strands. On the left is the bypass diagram furnished by the inductive step of the lemma applied to the first $k$ retrogradients. On the right is the result of attaching the bypass, interpreted as a pair of contact handles.}
	\label{fig:inductive_step}
\end{figure}
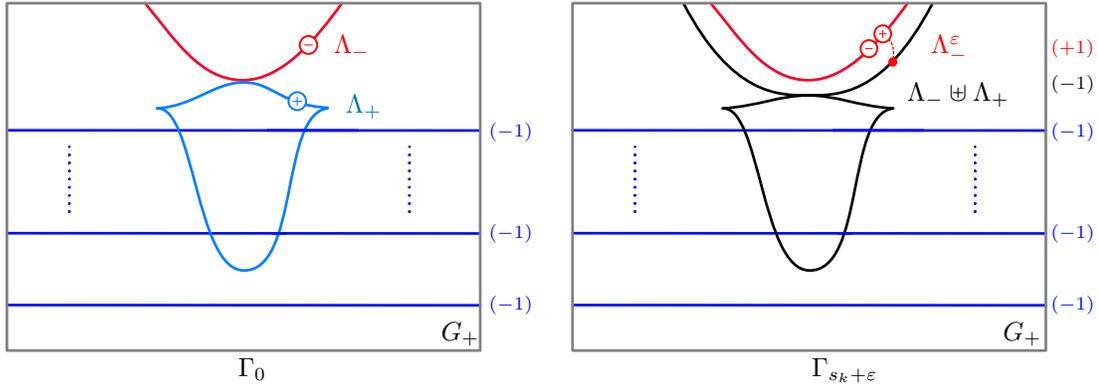
Attaching the $R_-$-centric bypass using \autoref{theorem:bypass_attachment} produces the diagram on the right side of \autoref{fig:inductive_step}. In particular, there is a contact $n$-handle attached along $\Lambda_- \uplus \Lambda_+$ and a contact ($n+1$)-handle with equator $\Lambda_-^{\ve}$. 

Next we identify the bypass attachment data $(\Lambda_{\pm}'; D_{\pm}')$ for the bypass corresponding to crossing the bottom-most strand: 
\begin{itemize}
    \item The negative data $(\Lambda_-'; D_-')$ is given by the cocore and belt sphere of the Weinstein $n$-handle attached to $R_-$ along $\Lambda_- \uplus \Lambda_+$. By \autoref{lemma:cocore_lemma}, we may diagrammatically witness this as an appropriately decorated small negative Reeb shift of $\Lambda_- \uplus \Lambda_+$ as on the left side of \autoref{fig:bypass2}. 
\begin{figure}[ht]
	\begin{overpic}[scale=0.53]{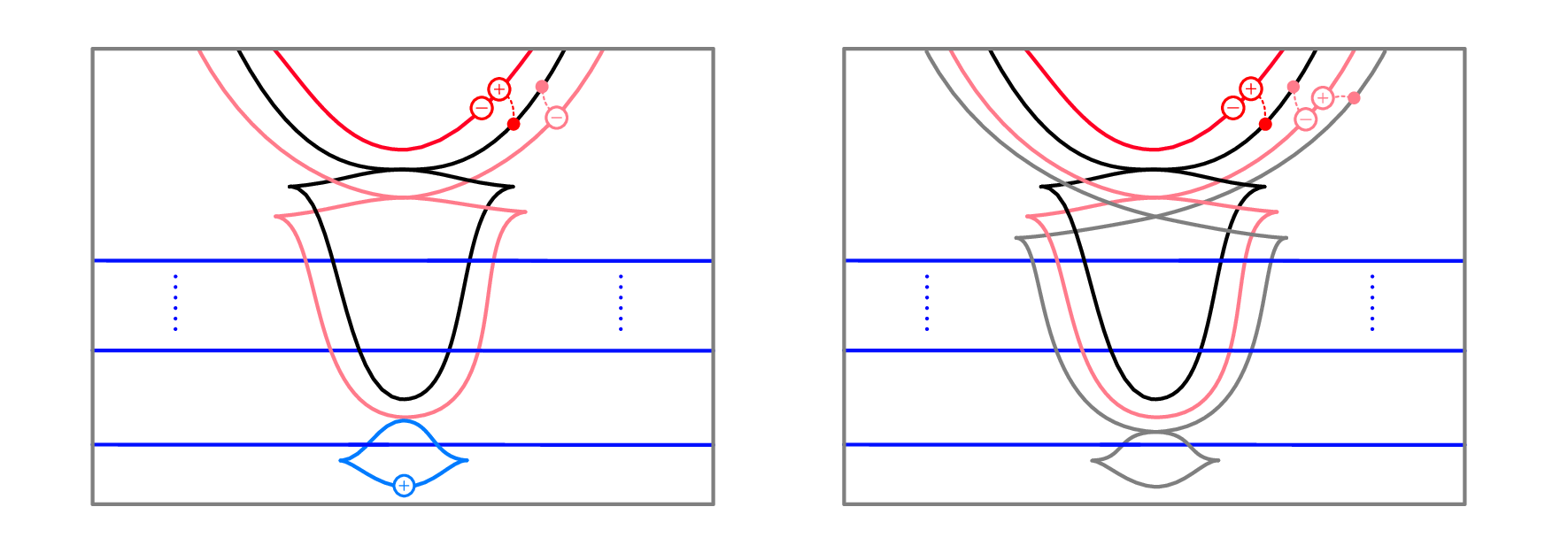}
	    \put(24,1){\small $\Gamma_{s_k+\ve}$}
	    \put(42,4){\small $G_+$}
	    \put(30,4){\small \textcolor{RoyalBlue}{$\Lambda_+'$}}
	    \put(36,26){\small \textcolor{Salmon}{$\Lambda_-'$}}
	    \put(46,12.5){\tiny \textcolor{blue}{$(-1)$}}
	    \put(46,18){\tiny \textcolor{blue}{$(-1)$}}
	    \put(46,6.5){\tiny \textcolor{blue}{$(-1)$}}
	    \put(46,26){\tiny $(-1)$}
	    \put(46,28){\tiny {\color{red}  $(+1)$}}

	    \put(73,1){\small $\Gamma_{1}$}
	    \put(90,4){\small $G_+$}
	    \put(94,12.5){\tiny \textcolor{blue}{$(-1)$}}
	    \put(94,18){\tiny \textcolor{blue}{$(-1)$}}
	    \put(94,6.5){\tiny \textcolor{blue}{$(-1)$}}
	    \put(94,26){\tiny $(-1)$}
	    \put(94,28){\tiny {\color{red}   $(+1)$}}
	    \put(80,9){\small \textcolor{gray}{$\Lambda_-' \uplus \Lambda_+'$}}
	    \put(94,22){\tiny \textcolor{gray}{$(-1)$}}
	    \put(94,24){\tiny \textcolor{Salmon}{$(+1)$}}
	    \put(86,25){\small \textcolor{Salmon}{$(\Lambda_-')^{\ve}$}}
		
	\end{overpic}
	\caption{The second bypass in the inductive step of the proof of \autoref{lemma:parallel_strands_bypass}. On the left is the bypass attachment data. On the right is the result of attaching the bypass as a pair of contact handles.}
	\label{fig:bypass2}
\end{figure}
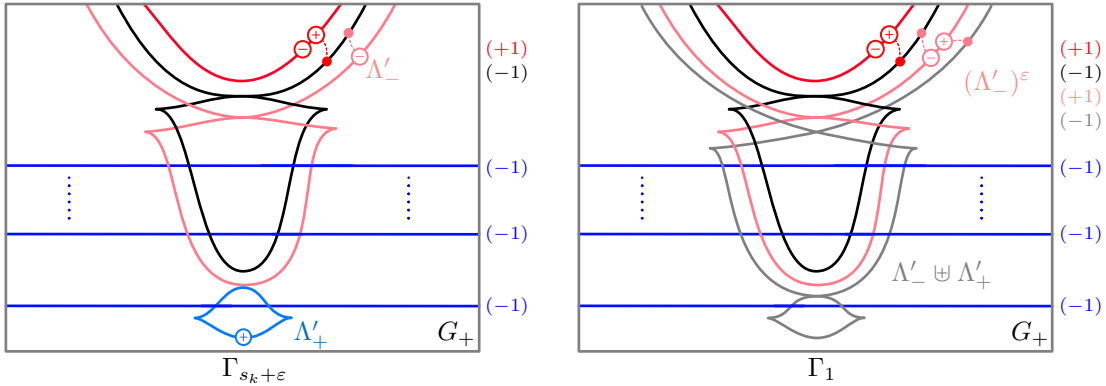
\item The positive data $(\Lambda_+'; D_+')$ is given by the cocore and belt sphere of the Weinstein $n$-handle in $R_+$ associated to the $(-1)$-surgery along the lowest strand, drawn in the diagram by \autoref{lemma:cocore_lemma} as a standard Legendrian knot linked once with the strand in question. In contrast to the base case, we position the cusps below the strand as on the left side of \autoref{fig:bypass2}. 
\end{itemize}
As before, attaching the $R_-$-centric bypass via a pair of contact handles produces the decorated diagram on the right side of \autoref{fig:bypass2}.

The rest of the proof proceeds by performing (decorated) Legendrian Kirby moves on the diagram on the right side of \autoref{fig:bypass2} (or equivalently on the left side of \autoref{fig:step1lemma}) to identify a trivial bypass and erase it from the diagram. In the process, the manipulations will result in the gray and pink Legendrians producing the desired final diagram, completing the induction.

\begin{remark}[Erasing surgeries vs.\ erasing decorated surgeries]
If decorations are ignored, then on the left side of \autoref{fig:step1lemma} we may cancel the $(+1)$-surgery along the black Legendrian $\Lambda_- \uplus \Lambda_+$ and the $(+1)$-surgery along the pink Legendrian $\Lambda_-' = (\Lambda_- \uplus \Lambda_+)^{-\ve}$. This would leave us with an undecorated surgery diagram that matches the conclusion of \autoref{fig:strand_lemma}. 
\begin{figure}[ht]
	\begin{overpic}[scale=0.53]{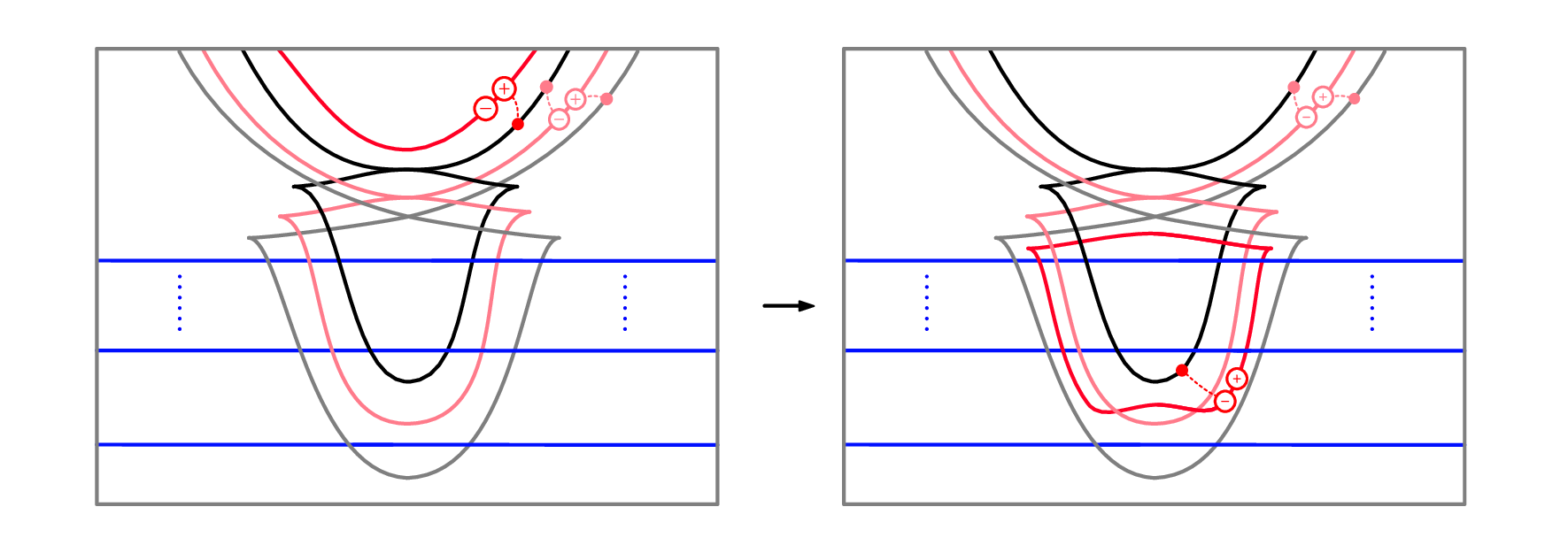}
	    \put(25,1){\small $\Gamma_{1}$}
	    \put(42,4){\small $G_+$}
	    \put(46.5,12.5){\tiny \textcolor{blue}{$(-1)$}}
	    \put(46.5,18){\tiny \textcolor{blue}{$(-1)$}}
	    \put(46.5,6.5){\tiny \textcolor{blue}{$(-1)$}}
	    \put(46.5,26){\tiny $(-1)$}
	    \put(46.5,28){\tiny {\color{red}  $(+1)$}}
	    \put(46.5,22){\tiny \textcolor{gray}{$(-1)$}}
	    \put(46.5,24){\tiny \textcolor{Salmon}{$(+1)$}}

	    \put(73,1){\small $\Gamma_{1}$}
	    \put(90,4){\small $G_+$}
	    \put(94,12.5){\tiny \textcolor{blue}{$(-1)$}}
	    \put(94,18){\tiny \textcolor{blue}{$(-1)$}}
	    \put(94,6.5){\tiny \textcolor{blue}{$(-1)$}}
	    \put(94,26){\tiny $(-1)$}
	    \put(94,28){\tiny {\color{red}  $(+1)$}}
	    \put(94,22){\tiny \textcolor{gray}{$(-1)$}}
	    \put(94,24){\tiny \textcolor{Salmon}{$(+1)$}}

	\end{overpic}
	\caption{Step 1 of the inductive step in the proof of \autoref{lemma:parallel_strands_bypass}: Sliding the red down over the black.}
	\label{fig:step1lemma}
\end{figure}
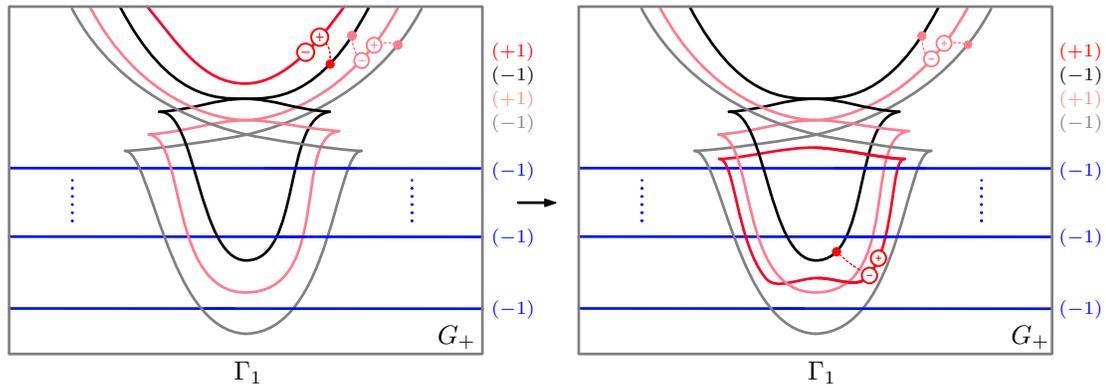
In fact, by \autoref{lemma:trivial_bypass_lemma}, the contact $n$-handle associated to $\Lambda_- \uplus \Lambda_+$ and the contact ($n+1$)-handle with equator $\Lambda_-'$ do form a trivial bypass. However, these two handles cannot be erased from the diagram because the ($n+1$)-handle with red equator $\Lambda_-^{\ve}$ intersects the black handle, as indicated by the decorations. This is the point of the rest of the proof. Instead of erasing the black and pink, we will perform Legendrian Kirby moves to separate the interaction of the two bypasses, ultimately erasing the black and \textit{red} (the bypass coming from the inductive hypothesis) as a trivial bypass. 
\end{remark}

We describe the relevant decorated Legendrian Kirby moves in four steps. 

\s\n
{\em Step 1: Slide red down over black.}

\s
The sequence of Kirby moves begins with the diagram on the left side of \autoref{fig:step1lemma}, which is equivalent to the right side of \autoref{fig:bypass2} by two Kirby moves. We then use \autoref{lemma:legendrian_handleslide_n_handle} to slide the red Legendrian $\Lambda_-^{\ve}$ down over the contact $n$-handle attached along the black Legendrian $\Lambda_- \uplus \Lambda_+$. In effect, this step simply changes the bypass in \autoref{fig:inductive_step} from an $R_-$-centric model to an $R_+$-centric model, though we need to exercise caution with the other Legendrians and Lagrangians in the diagram. 

First we analyze the effect on the Legendrian. After a sequence of omitted Kirby moves, we obtain the red Legendrian on the right side of \autoref{fig:step1lemma}. (As a sanity check, were it not for the presence of the pink Legendrian we could simply compute that sliding $\Lambda_-^{\ve}$ down over $\Lambda_- \uplus \Lambda_+$ produces 
\[
\left((\Lambda_- \uplus \Lambda_+) \uplus \Lambda_-\right)^{-\ve} \cong \Lambda_+^{-\ve}.
\]
To understand the any potential linking behavior with the pink Legendrian, the Kirby moves are necessary.)

Next we analyze the effect on the Lagrangian disks on the right side of \autoref{fig:step1lemma}. By \autoref{theorem:bypass_attachment}, the new negative red disk is isotopic to any of the following:
\be
\item $(K \uplus_b D_-)^{-\ve}$, where $K$ is the negative core of the Weinstein $n$-handle attached to $R_-$ along the black Legendrian;
\item $(K')^{-\ve'} \uplus_b D_-^{-\ve}$, where $K'$ is the cocore of the same Weinstein $n$-handle and $\ve'>0$ is chosen so that $(K')^{-\ve'}$ and $D_-^{-\ve}$ intersect ``contact transversely" at a point; 
\item $K^\parallel \uplus_b D_-^{+\ve}$, where $\parallel$ indicates a non-intersecting pushoff. 
\ee By \autoref{theorem:bypass_attachment} and the inductive hypothesis, the positive red disk is the disk $D_+$ provided by the statement of the lemma, i.e.,
\[
D_+ \cong K_1 \uplus_b \cdots \uplus_b K_k,   \]
where $K_j$ is isotopic to the unstable manifold of the positive critical point $p_{j'}^+$ associated to the $j$th strand, counted from top to bottom. 

The $n=1$ case of Step 1 is given by the left and middle diagrams in \autoref{fig:step2lemma_low}.

\begin{figure}[ht]
\vskip-.3in
	\begin{overpic}[scale=0.55]{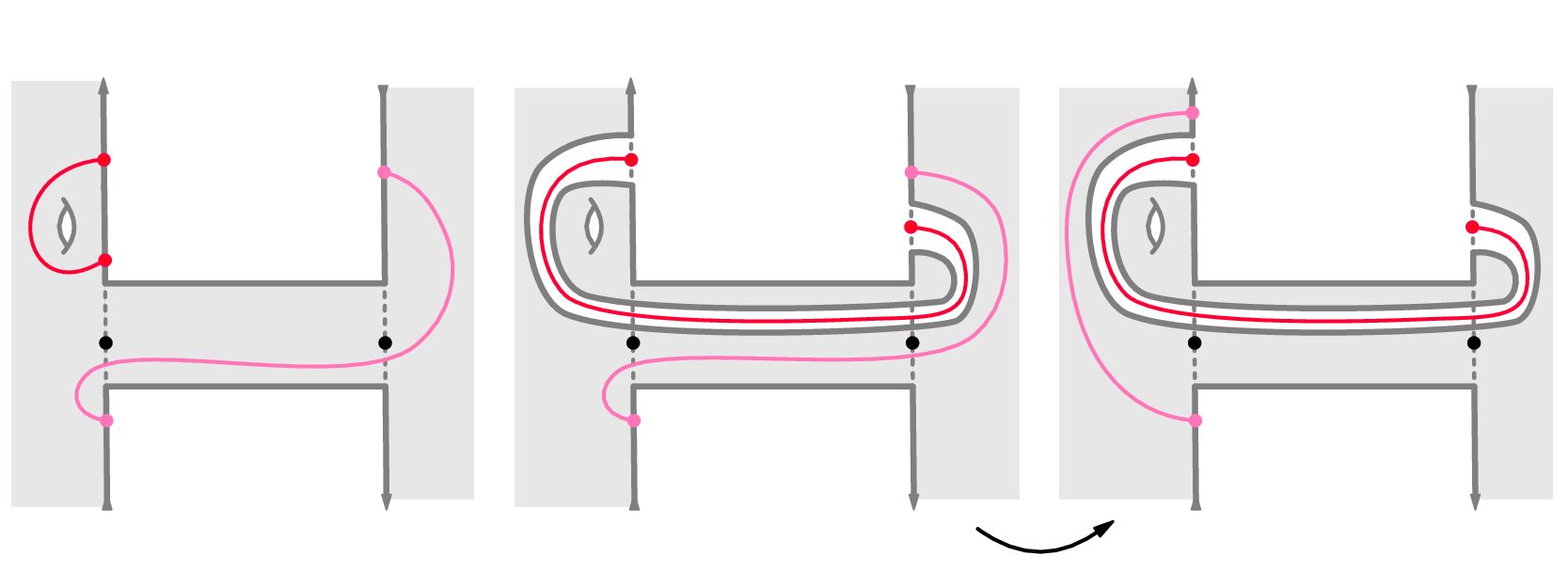}
	    
	\end{overpic}
	\vskip-.15in
	\caption{The $n=1$ cases of Step 1, given by the left and middle frames, and Step 2, given by the middle and right frames. In the left frame, a Weinstein $1$-handle is attached to $R_-$, with the boundary orientation indicated by the arrows, along the black Legendrian dots in the original dividing set. The pink arc is a Reeb down-shift of the cocore of the handle and the red arc is a Legendrian-Lagrangian pair. In the middle frame, the red pair is slid down across the handle and then a neighborhood of the resulting arc is removed; this corresponds to the right side of \autoref{fig:step1lemma}. In going from the middle frame to the right, the pink pair is slid up across the handle removal along the red pair.}
	\label{fig:step2lemma_low}
\end{figure}

\s\n
{\em Step 2: Slide pink up over red.}

\s
Next we slide the pink Legendrian up over the red, this time using \autoref{lemma:legendrian_handleslide_n+1_handle} (via the bigon located between strands $k$ and $k+1$). Another sequence of omitted Kirby moves produces the new pink Legendrian on the right of \autoref{fig:step2lemma}.

\begin{figure}[ht]
	\begin{overpic}[scale=0.53]{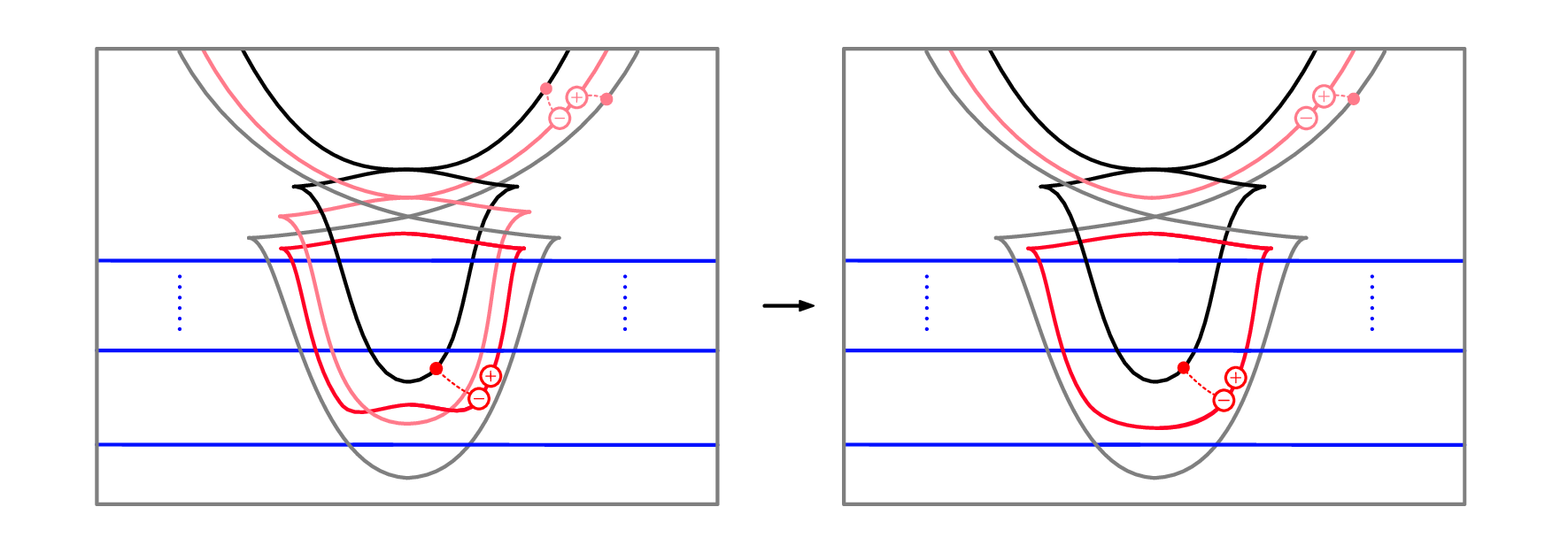}
	    \put(25,1){\small $\Gamma_{1}$}
	    \put(42,4){\small $G_+$}
	    \put(46.5,12.5){\tiny \textcolor{blue}{$(-1)$}}
	    \put(46.5,18){\tiny \textcolor{blue}{$(-1)$}}
	    \put(46.5,6.5){\tiny \textcolor{blue}{$(-1)$}}
	    \put(46.5,26){\tiny $(-1)$}
	    \put(46.5,28){\tiny {\color{red} $(+1)$}}
	    \put(46.5,22){\tiny \textcolor{gray}{$(-1)$}}
	    \put(46.5,24){\tiny \textcolor{Salmon}{$(+1)$}}

	    \put(73,1){\small $\Gamma_{1}$}
	    \put(90,4){\small $G_+$}
	    \put(94,12.5){\tiny \textcolor{blue}{$(-1)$}}
	    \put(94,18){\tiny \textcolor{blue}{$(-1)$}}
	    \put(94,6.5){\tiny \textcolor{blue}{$(-1)$}}
	    \put(94,26){\tiny $(-1)$}
	    \put(94,28){\tiny {\color{red} $(+1)$}}
	    \put(94,22){\tiny \textcolor{gray}{$(-1)$}}
	    \put(94,24){\tiny \textcolor{Salmon}{$(+1)$}}

	\end{overpic}
	\caption{Step 2 of the inductive step in the proof of \autoref{lemma:parallel_strands_bypass}: Sliding the pink up over the red.}
	\label{fig:step2lemma}
\end{figure}
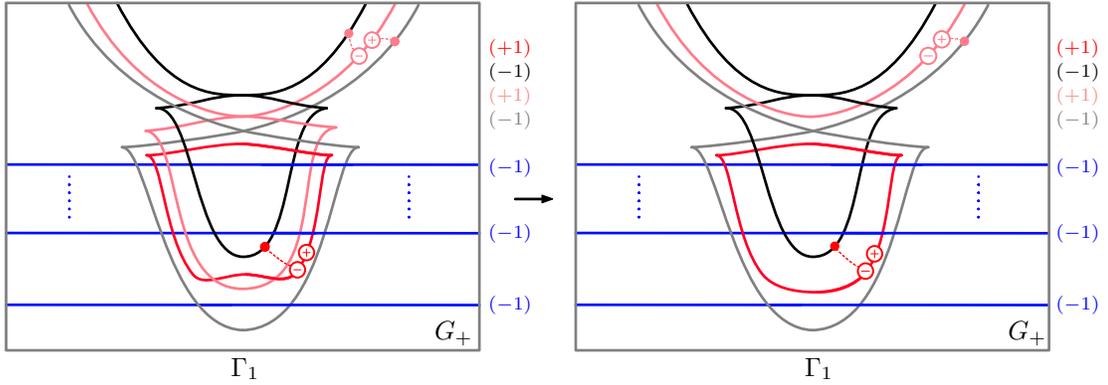

We analyze the effect of the handleslide on the Lagrangian disks. Before the slide, the negative pink disk is a Reeb down-shift $(K')^{-\ve''}$  of the cocore $K'$. The negative red disk on the left side of \autoref{fig:step2lemma} is isotopic to $K^\parallel \uplus_b D_-^{+\ve}$ by Step 1. After the slide, by \autoref{lemma:legendrian_handleslide_n+1_handle}, the negative pink disk is isotopic to
\[
(K^\parallel \uplus_b D_-^{+\ve})^\parallel \uplus_b (K')^{-\ve''}\simeq (D_-^{+\ve})^\parallel\uplus_b(K^\parallel \uplus_b (K')^{-\ve''}).
\]
Since $(K^\parallel \uplus_b (K')^{-\ve''})\simeq (K')^\parallel$ we have 
\[
(K^\parallel \uplus_b D_-^{+\ve})^\parallel \uplus_b (K')^{-\ve''}\simeq (D_-^{+\ve})^\parallel \uplus_b (K')^\parallel,
\]
which is a non-intersecting pushoff of $(D_-^{+\ve})^\parallel$. Note that the new negative pink disk is isotopic to $D_-$ if we delete the handle corresponding to the black Legendrian. See \autoref{fig:step2lemma_low} for the $n=1$ case of Step 2.

The positive pink disk, before the slide, arises from sliding the positive Weinstein cocore of the $n$-handle associated to the $(-1)$-surgery along the lowest blue strand up over the contact $n$-handle attached along the gray Legendrian; indeed, the disk comes from the bypass data $D_+' = K_{k+1}$ in \autoref{fig:bypass2}. By \autoref{lemma:legendrian_handleslide_n+1_handle}, sliding over the positive red disk $D_+$ gives a new positive pink disk isotopic to
\[
D_+ \uplus_b K_{k+1} \cong K_1 \uplus_b \cdots \uplus_b K_k \uplus_b K_{k+1},
\]
followed by a slide up over the gray handle.

\s\n
{\em Step 3: Slide black down over pink.}

\s
Now we slide the black down over the pink, again using \autoref{lemma:legendrian_handleslide_n+1_handle}. At the Legendrian level, the result --- after Kirby moves --- is the black Legendrian on the right side of \autoref{fig:step3lemma}.

\begin{figure}[ht]
	\begin{overpic}[scale=0.53]{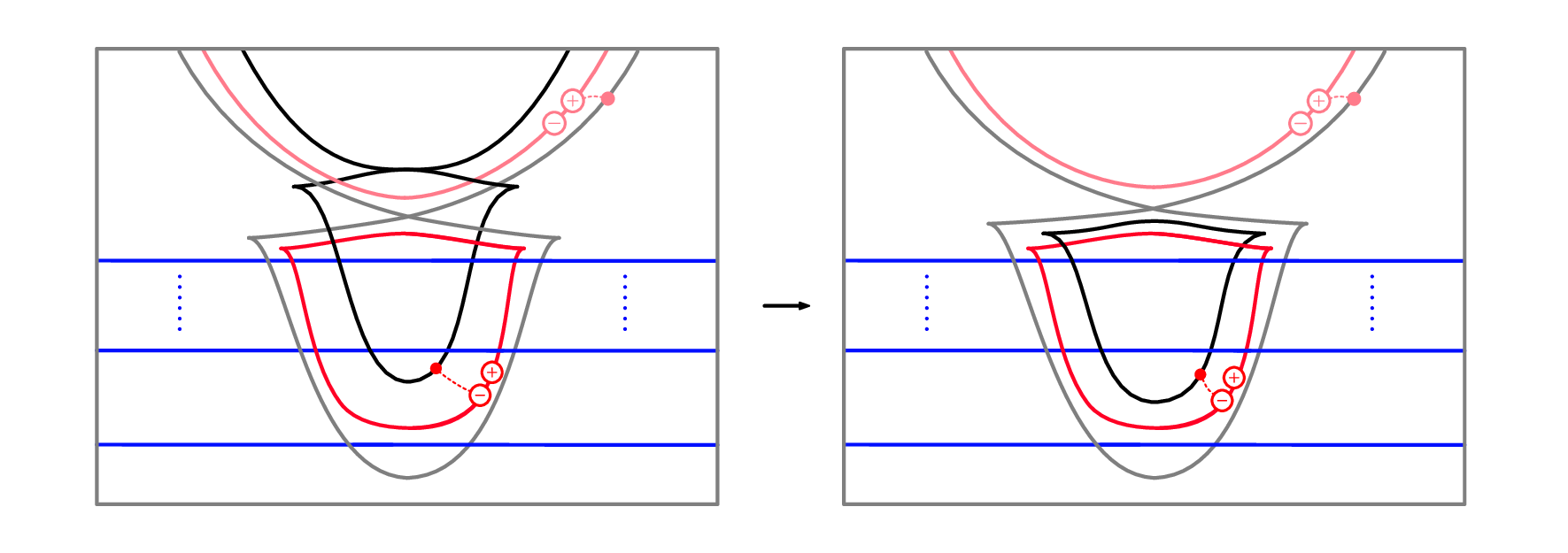}
	    \put(25,1){\small $\Gamma_{1}$}
	    \put(42,4){\small $G_+$}
	    \put(46.5,12.5){\tiny \textcolor{blue}{$(-1)$}}
	    \put(46.5,18){\tiny \textcolor{blue}{$(-1)$}}
	    \put(46.5,6.5){\tiny \textcolor{blue}{$(-1)$}}
	    \put(46.5,26){\tiny $(-1)$}
	    \put(46.5,28){\tiny {\color{red}  $(+1)$}}
	    \put(46.5,22){\tiny \textcolor{gray}{$(-1)$}}
	    \put(46.5,24){\tiny \textcolor{Salmon}{$(+1)$}}

	    \put(73,1){\small $\Gamma_{1}$}
	    \put(90,4){\small $G_+$}
	    \put(94,12.5){\tiny \textcolor{blue}{$(-1)$}}
	    \put(94,18){\tiny \textcolor{blue}{$(-1)$}}
	    \put(94,6.5){\tiny \textcolor{blue}{$(-1)$}}
	    \put(94,26){\tiny $(-1)$}
	    \put(94,28){\tiny {\color{red}  $(+1)$}}
	    \put(94,22){\tiny \textcolor{gray}{$(-1)$}}
	    \put(94,24){\tiny \textcolor{Salmon}{$(+1)$}}

	\end{overpic}
	\caption{Step 3 of the inductive step in the proof of \autoref{lemma:parallel_strands_bypass}: Sliding the black down over the pink.}
	\label{fig:step3lemma}
\end{figure}
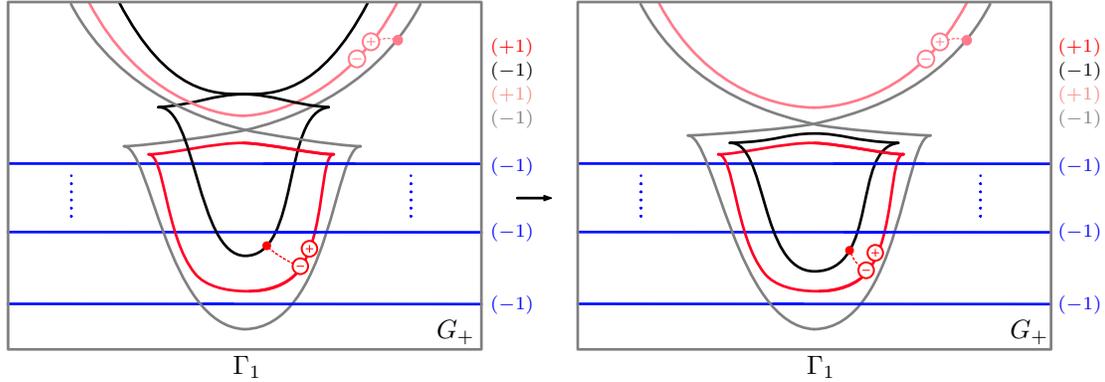

The only Lagrangian disk affected by this slide is the negative red disk, which nontrivially intersects the black handle and  is isotopic to $(K')^{-\ve'} \uplus_b D_-^{-\ve}$ by Step 1, where $K'$ is the cocore of the negative Weinstein $n$-handle attached along the black Legendrian. We claim that the effect of this step is that the new negative red disk is a Reeb down-shift $(K'')^{-\ve''}$ of the cocore $K''$ of the handle attached along the new black Legendrian: Indeed, the new negative red disk is obtained by sliding $(K')^{-\ve'}$ and $D_-^{-\ve}$ separately and then taking their boundary connected sum.  This gives $(K'')^{-\ve''} \uplus_b (D_-^{-\ve}\uplus_b (D_-^{-\ve})^\parallel)$, once we recall from Step 2 that the negative pink disk is isotopic to $D_-^{-\ve}$. This then simplifies to 
\[
(K'')^{-\ve''}\uplus_b D'\simeq (K'')^{-\ve''},
\]
where $D'$ is a standard Lagrangian disk filling of a Legendrian unknot. See \autoref{fig:step3lemma_low} for the $n=1$ case of Step 3.

\begin{figure}[ht]
\vskip-.1in
	\begin{overpic}[scale=0.55]{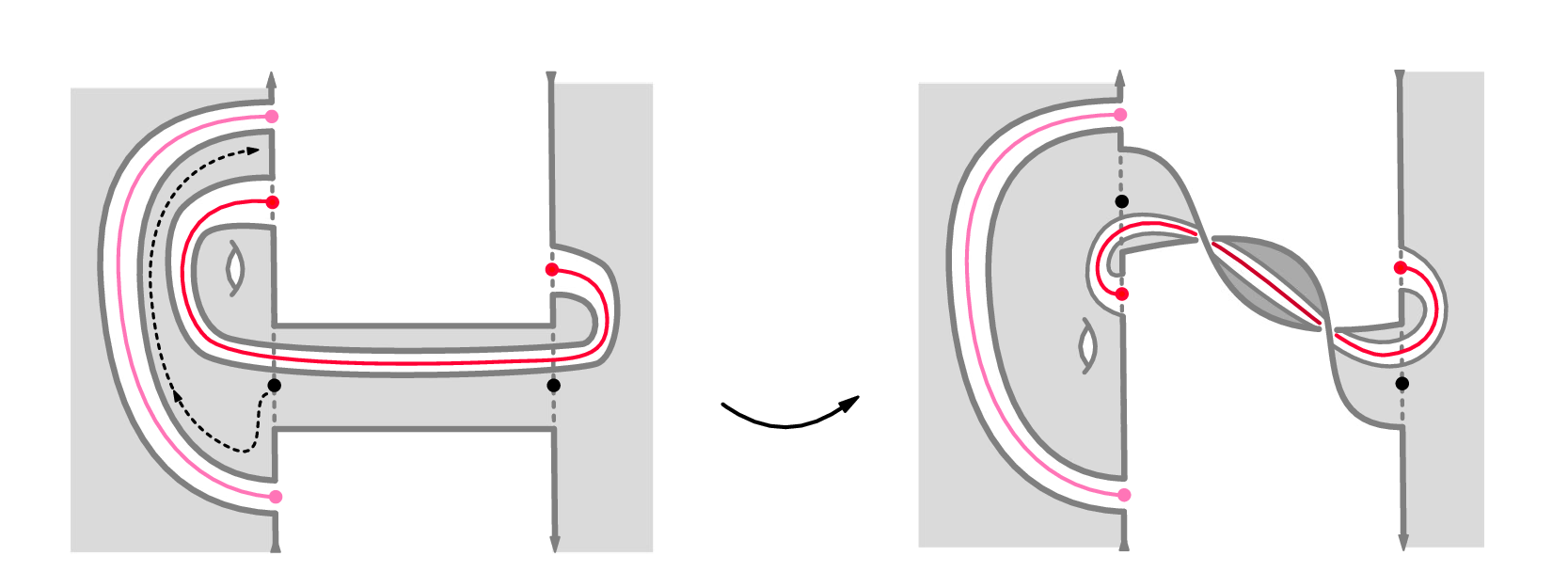}
	  
	\end{overpic}
	\vskip-.05in
	\caption{The $n=1$ version of the maneuver in Step 3, a continuation of \autoref{fig:step2lemma_low}. A neighborhood of the pink arc (a Lagrangian disk in $R_-$) has been removed, corresponding to the $(+1)$-surgery of the dividing curve along the pink Legendrian $S^0$ in \autoref{fig:step3lemma}. We slide the handle attached along the black Legendrian $S^0$ down over this $(+1)$-surgery, the motion being indicated by the dashed arrow on the left. After the slide, the red arc (a Lagrangian disk) is isotopic to the cocore of the (slid) black handle. In particular, the black handle attachment is canceled by the red arc removal.} 
	\label{fig:step3lemma_low}
\end{figure}

\s\n
{\em Step 4: Identify the black and red as a trivial bypass and erase.}

\s
By Step 3 the negative red disk is a downshift of the negative cocore of the black handle. Thus, by \autoref{lemma:trivial_bypass_lemma}, the black and red decorated Legendrians form a trivial bypass and can be erased from the diagram. We are left with the gray and pink decorated Legendrians, which describe a contact $n$-handle and ($n+1$)-handle respectively. Note that these handles correspond to an $R_-$-centric bypass with attaching data described by the conclusion of the lemma in question; see \autoref{fig:strand_lemma}. Indeed, the negative pink disk is isotopic to $D_-$, and the positive pink disk is isotopic to $K_1 \uplus_b \cdots \uplus_b K_{k+1}$ after a slide across the gray handle. This completes the induction and the proof of the lemma. 
\end{proof}

\begin{remark}
There is a dual version of \autoref{lemma:parallel_strands_bypass} that describes a positive strand moving up across a sequence of locally parallel negative strands as a single bypass attachment; see \autoref{fig:dual_strand_lemma}. The details of the statement and proof, omitted here, are completely analogous to those of \autoref{lemma:parallel_strands_bypass}.
\end{remark}

\begin{figure}[ht]
	\begin{overpic}[scale=0.53]{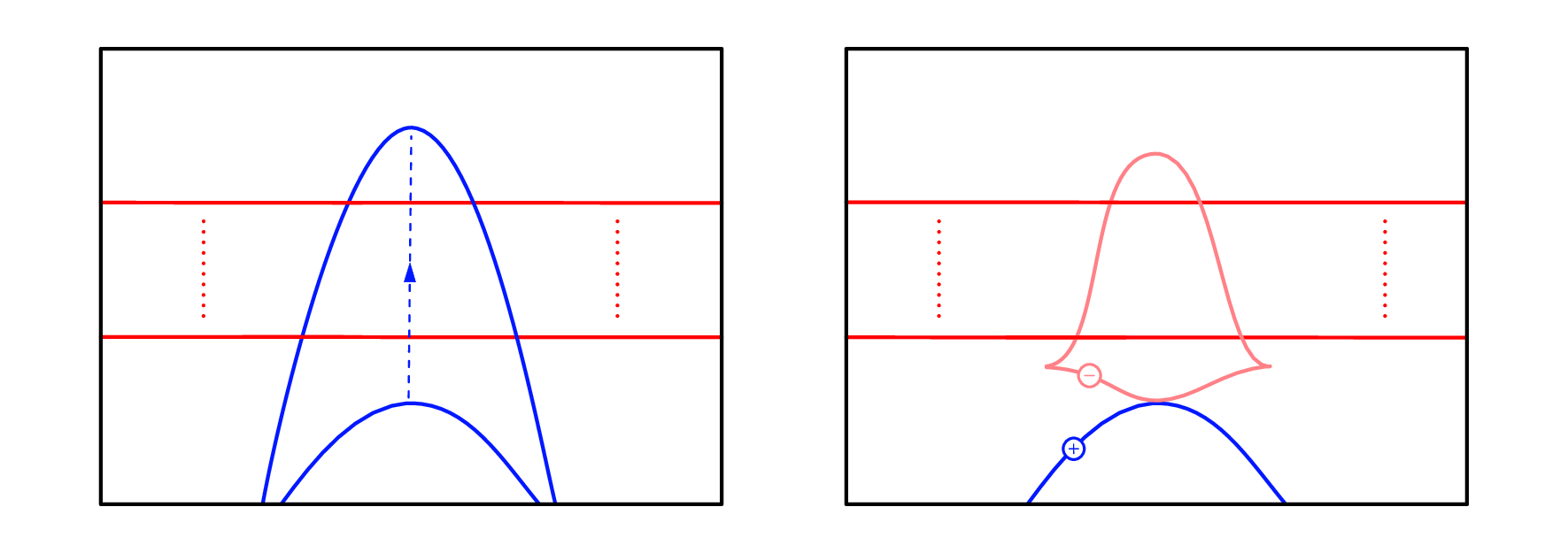}
	    \put(14,33){\small Local retrogradient diagram}
	    \put(26,1){\small $Y$}
	    \put(44,4){\small $Y$}
	    \put(64,33){\small Local bypass diagram}
	    \put(74,1){\small $\Gamma_0$}
	    \put(90,4){\small $G_-$}
	    \put(94,13.5){\tiny {\color{red}  $(-1)$}}
	    \put(94,22){\tiny {\color{red}  $(-1)$}}
	    \put(82,5){\small \textcolor{blue}{$\Lambda_+$}}
	    \put(81,9){\small \textcolor{Salmon}{$\Lambda_-$}}
	\end{overpic}
	\caption{The dual version of \autoref{lemma:parallel_strands_bypass}.}
	\label{fig:dual_strand_lemma}
\end{figure}

\subsection{Bypasses in 2-parameter families}\label{subsection:bypasses_2_parameter}

Armed with \autoref{lemma:parallel_strands_bypass}, we prove \autoref{theorem: sigma times interval}. 

\begin{proof}[Proof of \autoref{theorem: sigma times interval}.]

It suffices to consider each local codimension-$2$ retrogradient bifurcation described by \autoref{prop:codimension_2_bifurcations} in terms of bypasses. In particular, for each (P$x$) in the proposition, we assume that $(\Sigma \times [0,1]_s, \xi_t)$ is a contact manifold which locally models (P$x$), and we prove that the sequences of bypasses for $(\Sigma \times [0,1], \xi_i)$ with $i=0,1$ are the same up to shuffling, adding, and deleting trivial bypasses. We remind the reader that the only situations to analyze when $n=1$ are (P1), (P3),(P3'), (P5), and (P5'). 

\s\n
{\em (P1).} This situation concerns the moment when there are two separate retrogradients occurring simultaneously, both between nondegenerate critical points of index $n$. In the $(s,t)$-parameter space, this corresponds to two regular strands of the $1$-manifold $L$ crossing; see the top right of \autoref{fig:P1}.

If the two simultaneous retrogradients connect two distinct sets of positive and negative critical points, there is nothing to prove; this reverses the order of two unrelated bypass attachments and corresponds to far commutativity. There is more to analyze when the simultaneous retrogradients involve either the same positive and/or negative critical point. 

\begin{figure}[ht]
	\begin{overpic}[scale=0.3]{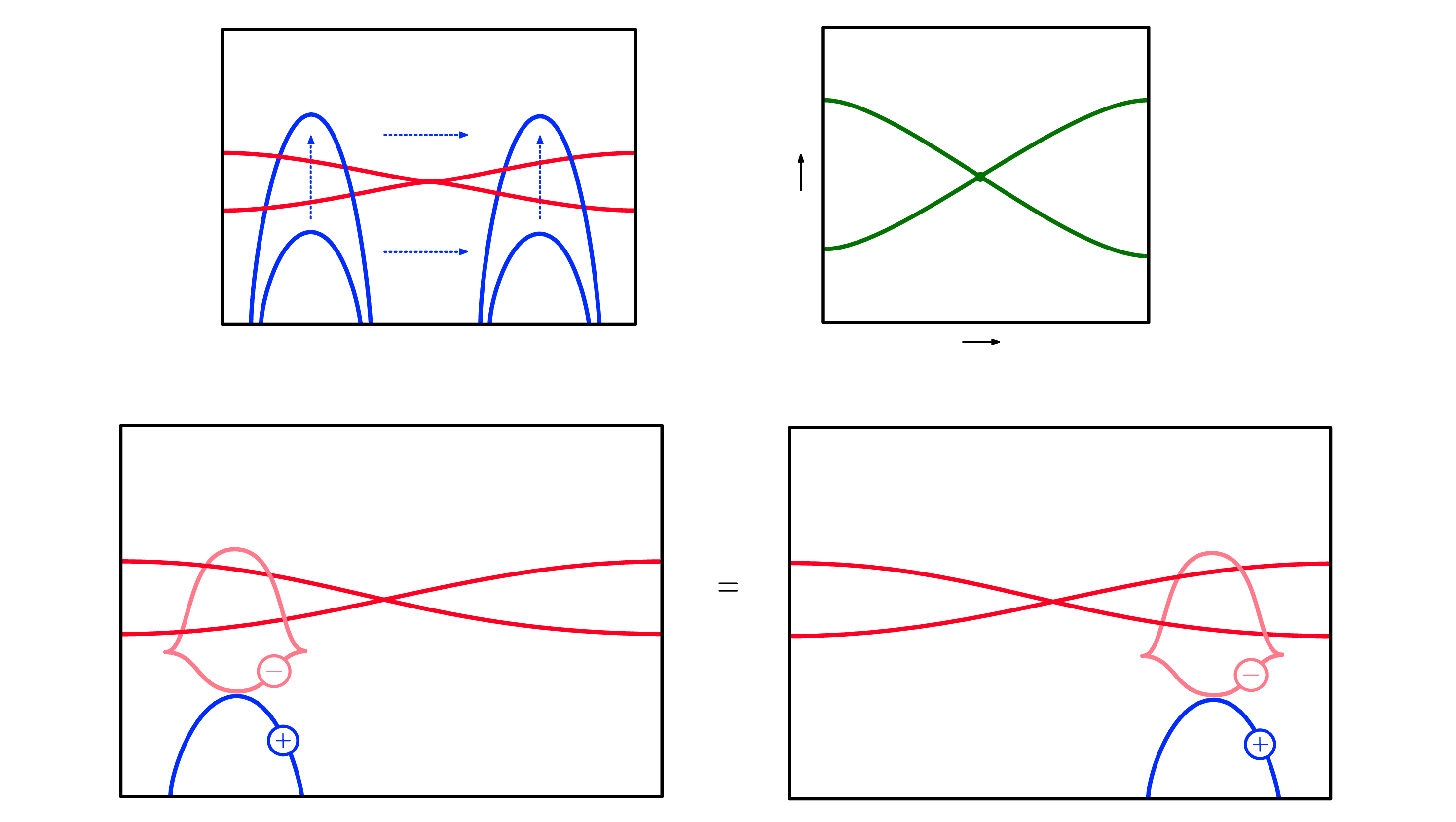}
	    \put(71.5,0){\tiny $\Gamma$}
	    \put(88.25,2.5){\tiny $G_-$}
	    \put(92,12){\tiny {\color{red}  $(-1)$}}
	    \put(92,18){\tiny {\color{red}  $(-1)$}}
	    \put(27,0){\tiny $\Gamma$}
	    \put(42,2.5){\tiny $G_-$}
	    \put(45.75,12){\tiny {\color{red}  $(-1)$}}
	    \put(45.75,18){\tiny {\color{red}  $(-1)$}}
	    \put(29,32.5){\tiny $Y$}
	    \put(41.75,35){\tiny $Y$}
	    \put(69.25,32.5){\tiny $t$}
	    \put(54.5,46.5){\tiny $s$}
	    \put(38.5,28.5){\small (local bypass diagrams)}
	    \put(16,55.5){\small (local retrogradient diagram)}
	    \put(59,55.5){\small ($2$-parameter space)}
	    \put(11.25, 21){\small \textcolor{Salmon}{$(\Lambda_-; D_-)$}}
	    \put(78.5, 21){\small \textcolor{Salmon}{$(\Lambda_-'; D_-')$}}
	\end{overpic}
	\caption{The analysis of (P1) in terms of bypasses. The top left is the local retrogradient diagram and to its right is the corresponding $(s,t)$-parameter space identifying the existence of retrogradients. The regular double point in green precisely (P1). On the bottom are the corresponding bypass diagrams corresponding to the left and right sides of the parameter space, after an application of \autoref{lemma:parallel_strands_bypass}.}
	\label{fig:P1}
\end{figure}
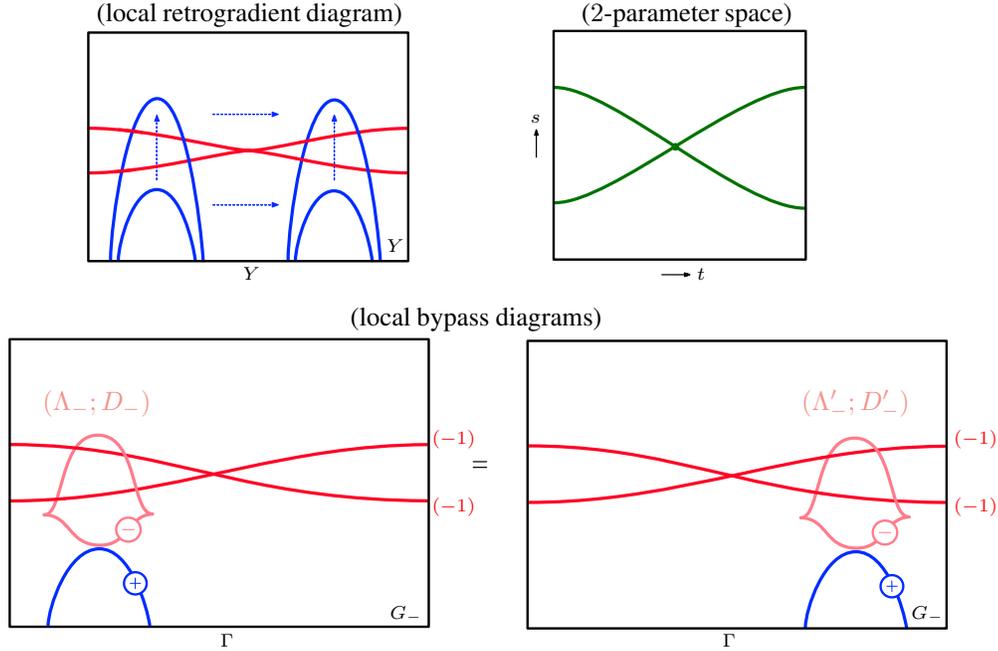

{\em The case $n>1$.} With a specific choice of local model, we can use \autoref{lemma:parallel_strands_bypass}. Assume that the two retrogradients involve the same positive point $p^+$; we allow for either one negative point or two distinct negative points. The remaining case is identical. We may normalize the situation so that it is given by the local retrogradient model given at the top left of \autoref{fig:P1}.  This is possible because the neighborhood of the stable disk of $p^+$, together with the two retrogradients to $p^+$ can be normalized so that the retrogradients are close together. In the top left frame, the vertical arrows correspond to increasing $s$ and the horizontal arrows correspond to increasing $t$. Here we make no assumption about whether the two red Legendrian strands are associated with the same negative point or distinct points. We apply the dual version of \autoref{lemma:parallel_strands_bypass} for $t=0,1$ to witness each sequence of bypasses as a single bypass with the indicated negative data, namely $(\Lambda_-; D_-)$ for $t=0$ and $(\Lambda_-'; D_-')$ for $t=1$. By the lemma, $(\Lambda_-; D_-)$ and $(\Lambda_-'; D_-')$ are both standard fillings of isotopic knots, and hence the resulting bypass data is Hamiltonian isotopic.

{\em The case $n=1$.} In low dimensions we must consider two situations. Note that in the $n>1$ case, our local model deals with the following two cases at the same time because it does not matter from which critical point(s) the red strands originate.

\begin{itemize}
    \item {\em One positive point and two negative points.} Assume that there are two retrogradients involving one positive point and two distinct negative points. We can model this with the far left folded Weinstein surface of \autoref{fig:P1_low_data}. We may then identify the data for the corresponding bypass attachments according to the bypass-bifurcation correspondence given by \autoref{prop:bypass_bifurcation_correspondence}. This gives the local model for bypass attachment data as described by the far right diagram of \autoref{fig:P1_low_data}. In particular, two possible bypass arcs are produced by shifting the blue arc in either (but not both) of the indicated directions until it intersects with either the top or bottom red arc.
    
    \begin{figure}[ht]
	\begin{overpic}[scale=0.45]{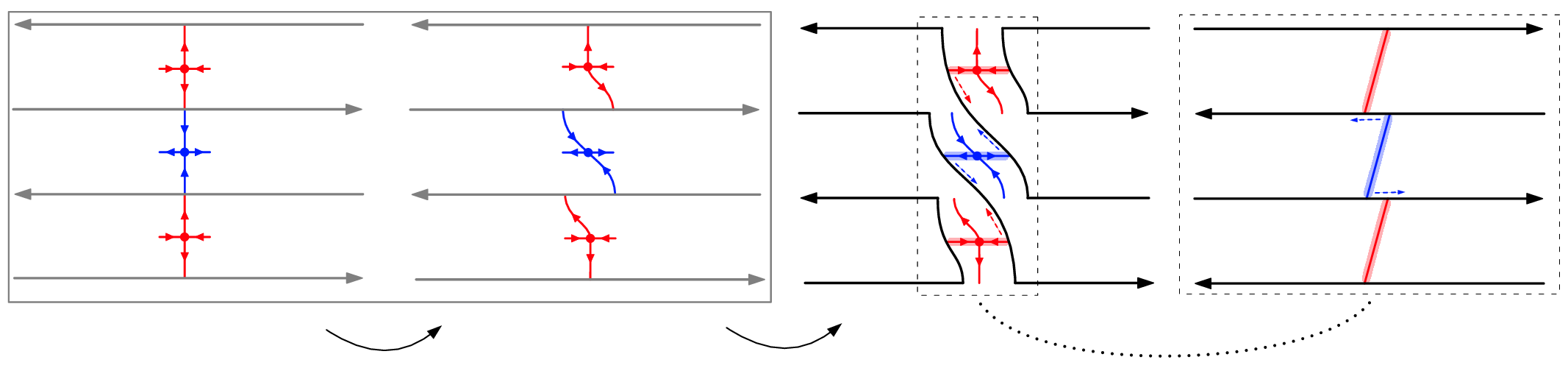}
	    \put(19,-0.5){\tiny before the bypasses}
	    \put(46.5,-0.5){\tiny identify data}
	    \put(19,23.5){\tiny \textcolor{gray}{Folded Weinstein}}
	    \put(72,23.5){\tiny Convex}
	\end{overpic}
	\caption{Identifying bypass data in the $n=1$ version of (P1) when there is one positive point and two negative points.}
	\label{fig:P1_low_data}
    \end{figure}
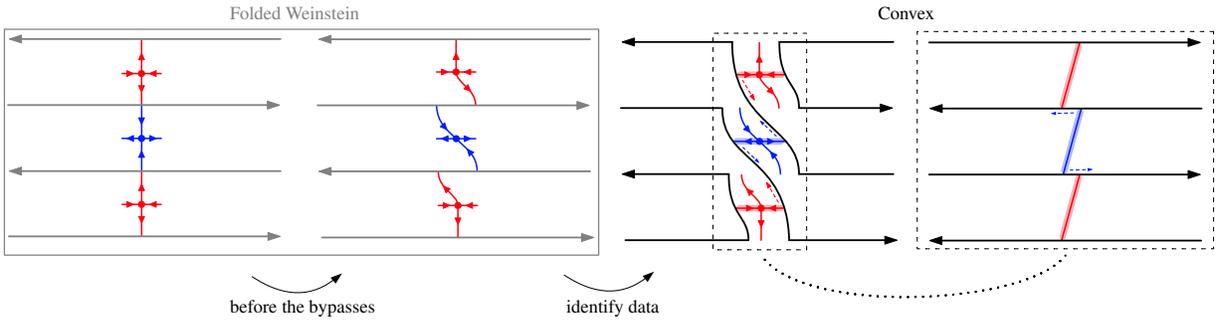
    
    We may then verify that the two bypass attachments commute. This computation is given by \autoref{fig:P1_low}.

    \begin{figure}[ht]
	\begin{overpic}[scale=0.4]{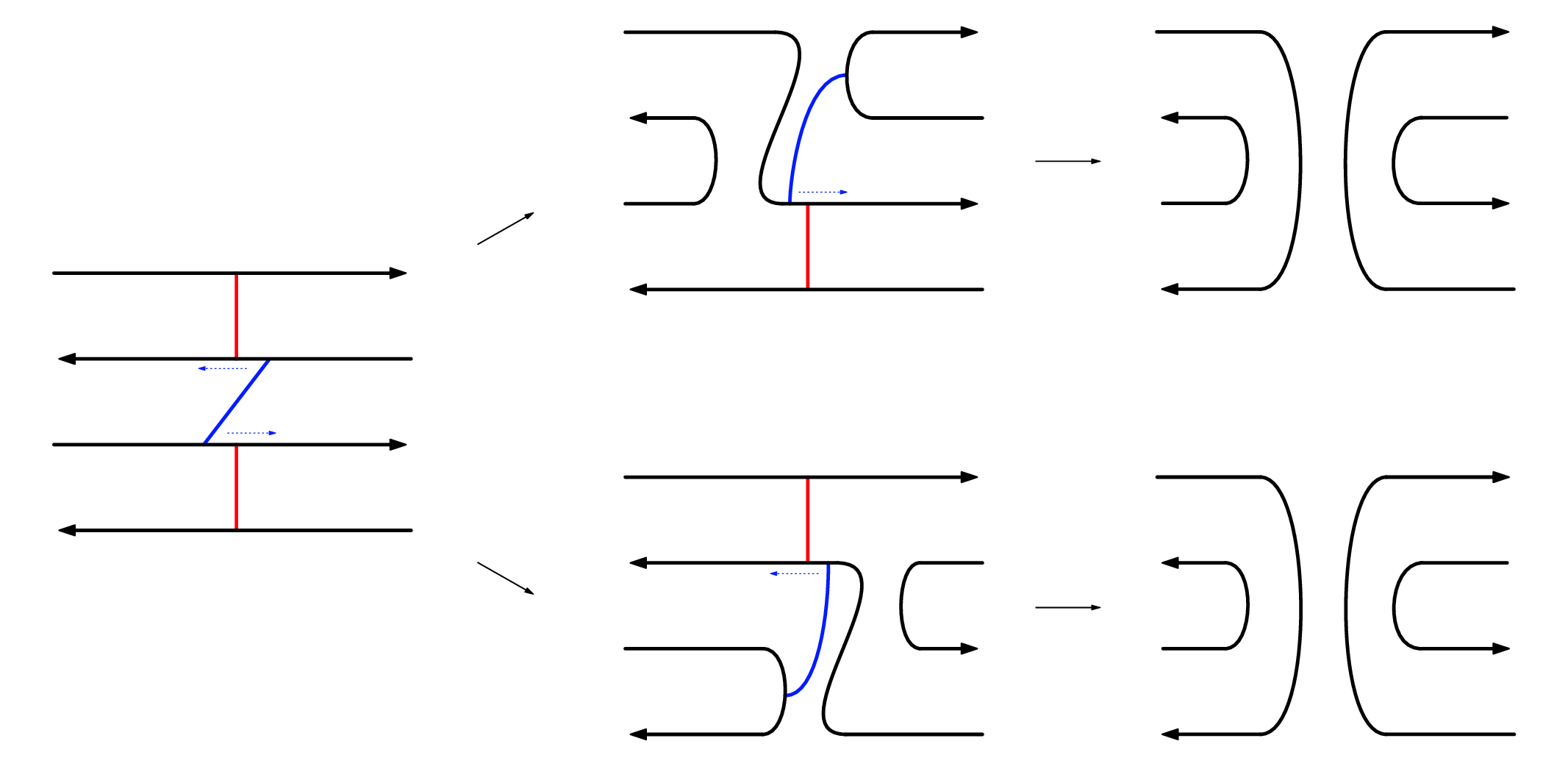}
	    \put(27, 37.5){\small top bypass}
	    \put(25, 7.75){\small bottom bypass}
	    \put(64.75, 39.5){\small bottom}
	    \put(65, 36.75){\small bypass}
	    \put(66.5, 11.25){\small top}
	    \put(65, 8.25){\small bypass}
	    
	\end{overpic}
	\caption{The $n=1$ version of (P1) when there is one positive point and two negative points. The data for the two commuting bypasses is identified on the left and comes from \autoref{fig:P1_low_data}.}
	\label{fig:P1_low}
    \end{figure}
    
    The case of one negative point and two positive points is identical. 
    
    \begin{figure}[ht]
	\begin{overpic}[scale=0.45]{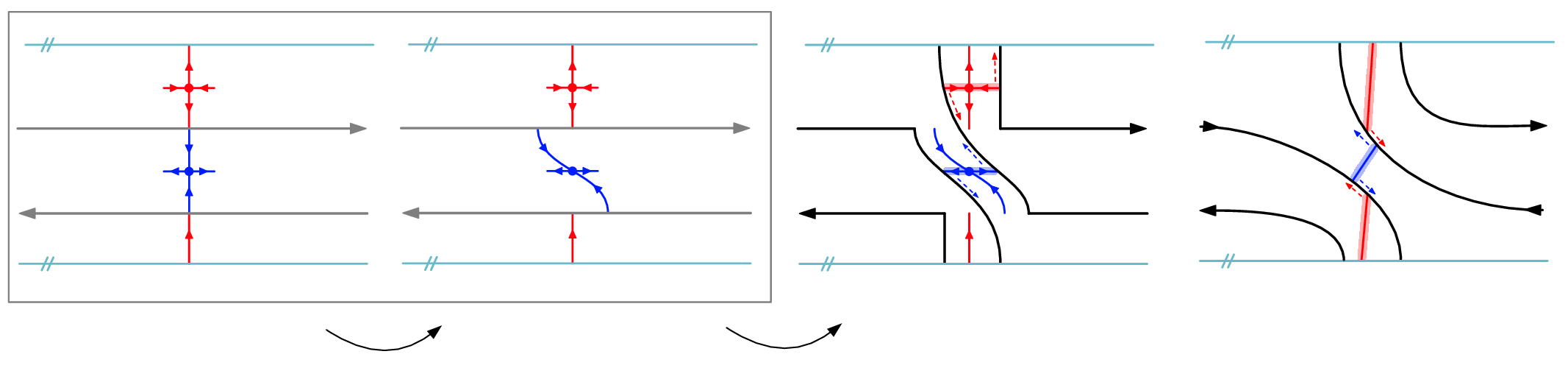}
	    \put(19,-0.5){\tiny before the bypasses}
	    \put(46.5,-0.5){\tiny identify data}
	    \put(19,23.5){\tiny \textcolor{gray}{Folded Weinstein}}
	    \put(72,23.5){\tiny Convex}
	\end{overpic}
	\caption{Identifying bypass data in the $n=1$ version of (P1) when there is one positive point and one negative point. In all diagrams the light blue lines are identified.}
	\label{fig:P1_low_data2}
    \end{figure}
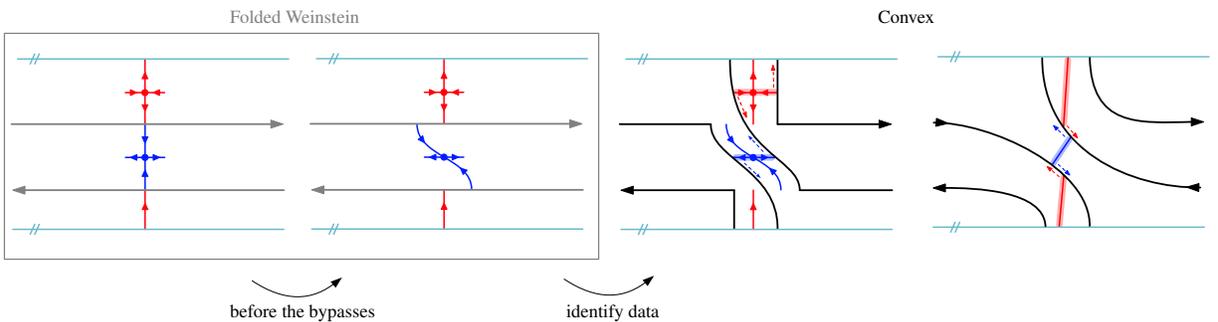
    
    \item {\em One positive point and one negative point.} Here we assume there are two retrogradients, both connecting the same negative point to the same positive point. The analysis is similar to the previous case. In \autoref{fig:P1_low_data2} we model this double retrogradient with the folded Weinstein surface on the far left. The light blue lines are identified. We again use \autoref{prop:bypass_bifurcation_correspondence} to identify the bypass data corresponding to each retrogradient. After isotopy, this gives us a local model for the two bypass attachments as described by the far right diagram of \autoref{fig:P1_low_data2}.
    
    We then verify that the two bypasses commute in \autoref{fig:P1_low2}. 
    
    \begin{figure}[ht]
	\begin{overpic}[scale=0.45]{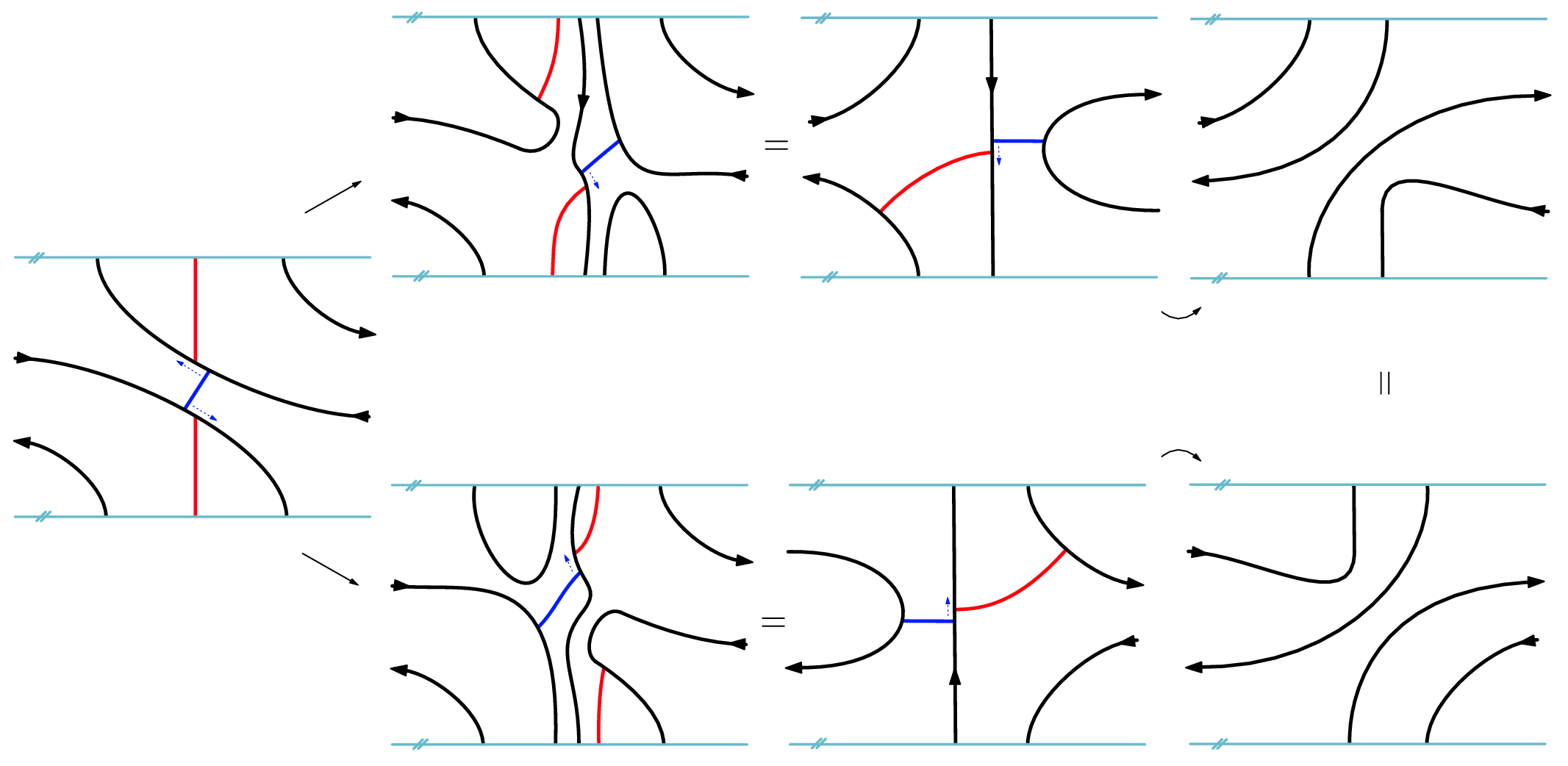}
	    \put(12, 37){\small top bypass}
	    \put(8, 11){\small bottom bypass}
	    \put(69, 26.25){\small bottom bypass}
	    \put(71, 21.5){\small top bypass}
	\end{overpic}
	\caption{The $n=1$ version of (P1) when there is one positive point and one negative point. The data for the two commuting bypasses is identified on the left and comes from \autoref{fig:P1_low_data2}. The second column to to the third column is an isotopy of the dividing set and the red arc for the sake of simplifying the picture.}
	\label{fig:P1_low2}
    \end{figure}
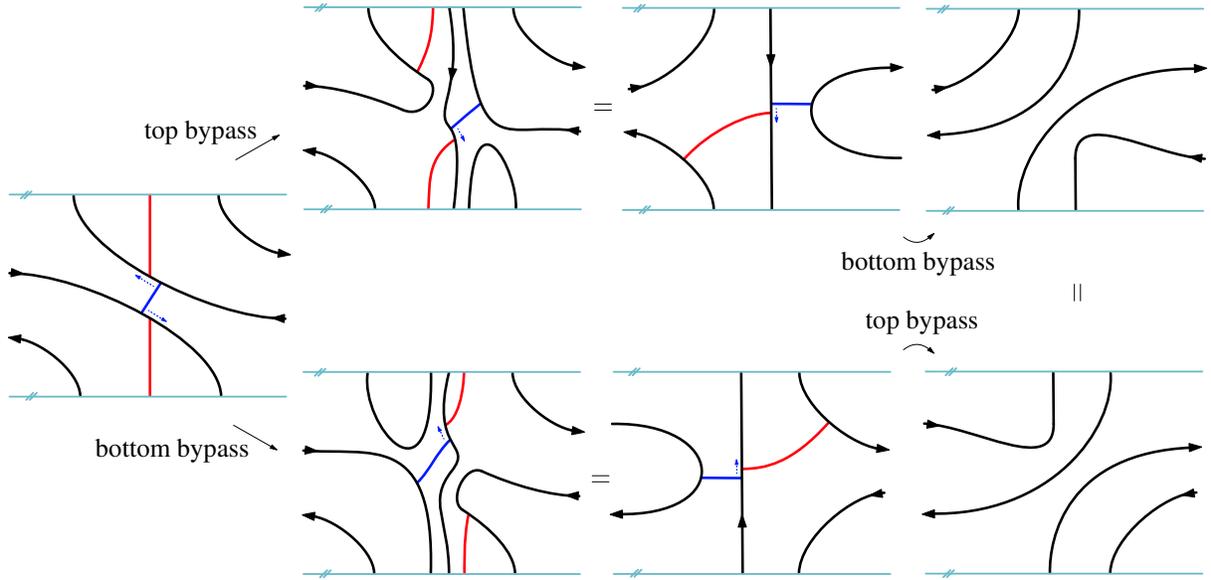
\end{itemize}

\s\n
{\em (P2).} We model (P2) with the local retrogradient diagram in \autoref{fig:P2}. The retrogradient occurring at the cusp of the red Legendrian in the local retrogradient diagram is precisely the moment when the stable and unstable manifolds of the critical points are not transverse; this is easily seen in the Lagrangian projection. 

For $t=1$, there are no retrogradients and hence no bypasses. In particular, $(\Sigma\times[0,1], \xi_1)$ is a vertically invariant contact structure. We therefore need to show that the two retrogradients for $t=0$ correspond to bypass attachments that yield a vertically invariant contact structure. By \autoref{lemma:parallel_strands_bypass}, the sequence of bypass attachments is equivalent to a single bypass with data given by the lower left diagram of \autoref{fig:P2}. After isotopy, the negative data is a standard Lagrangian disk filling of the standard unknot positioned above the positive data, and hence the resulting bypass is trivial.

\begin{figure}[ht]
	\begin{overpic}[scale=0.3]{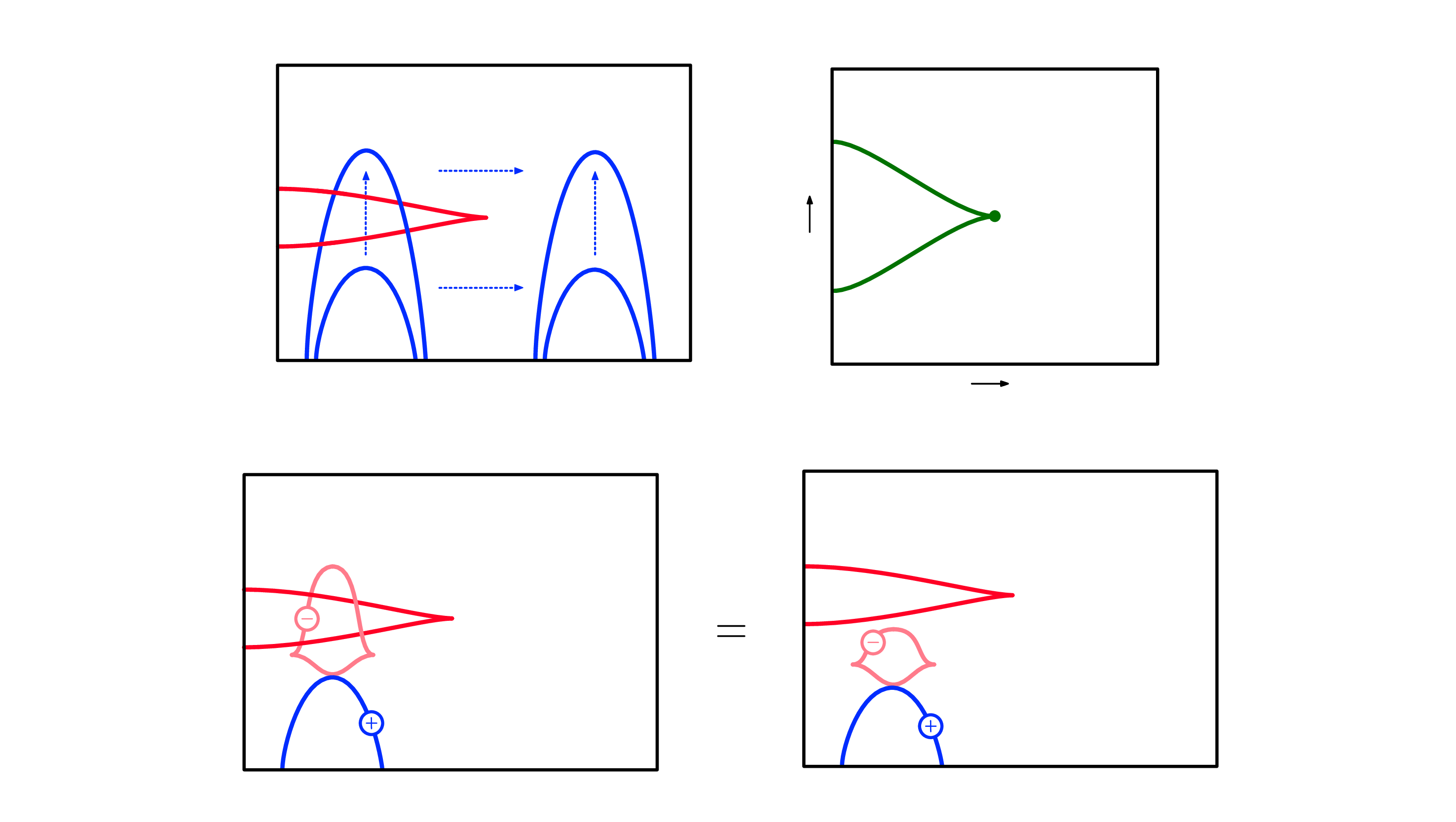}
	    \put(69.5,1.5){\tiny $\Gamma$}
	    \put(80.25,4.5){\tiny $G_-$}
	    \put(79,13){\tiny {\color{red}  $(-1)$}}
        \put(30.25,1.5){\tiny $\Gamma$}
	    \put(42,4.5){\tiny $G_-$}
	    \put(40.75,13){\tiny {\color{red}  $(-1)$}}
	    \put(32,30){\tiny $Y$}
	    \put(45.75,32.5){\tiny $Y$}
	    \put(70,29.75){\tiny $t$}
	    \put(55.25,43.75){\tiny $s$}
	    \put(38.5,25.5){\small (local bypass diagrams)}
	    \put(19.75,53){\small (local retrogradient diagram)}
	    \put(59,53){\small ($2$-parameter space)}
	\end{overpic}
	\vskip-.05in
	\caption{The analysis of (P2) in terms of bypasses. \autoref{lemma:parallel_strands_bypass} provides the bottom left figure and the bottom right is obtained by isotoping the negative data.}
	\label{fig:P2}
\end{figure}
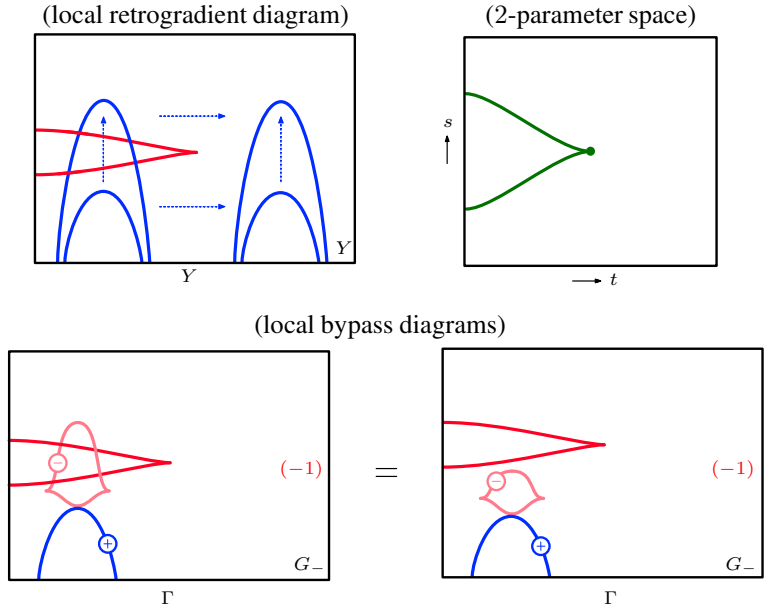

There are other occurrences of (P2) that correspond to reflecting the retrogradient diagram in \autoref{fig:P2} across the vertical axis, and also by reflecting across the horizontal axis together with a change in the role of positive and negative data. The analysis for each of these is completely analogous with appropriate applications of \autoref{lemma:parallel_strands_bypass} and its dual. 

\s\n
{\em (P3) and (P3').} We analyze (P3) explicitly; (P3') is completely analogous. Suppose there is a single retrogradient from a negative nondegenerate point $p^-$ of index $n$ to a positive birth-death point of index $(n-1, n)$. Because of the birth-death point we may assume that for $t=0$ there are no retrogradients, and hence $(\Sigma \times [0,1], \xi_0)$ is $s$-invariant and that for $t=1$ there is a single retrogradient between $p^-$ and the birthed positive index $n$ point $p^+$. We prove that this retrogradient corresponds to a trivial bypass. 

{\em The case $n>1$.} The idea is the same as the proof of \autoref{lemma:parallel_strands_bypass}. The positive data of the corresponding bypass attachment is the belt sphere and cocore $(\Lambda_+; D_+)$ of the positive Weinstein $n$-handle associated to $p^+$. As in \autoref{lemma:cocore_lemma}, we isotop this pair to disconnect it from the handle. Since $p^+$ and the birthed positive point of index $n-1$ cancel, their corresponding handles can be canceled, leaving a standard Legendrian unknot with a standard Lagrangian disk as the positive bypass data. We can further assume after isotopy that it is positioned below the negative data. Thus, the bypass is trivial, as desired. See \autoref{fig:P3}. 

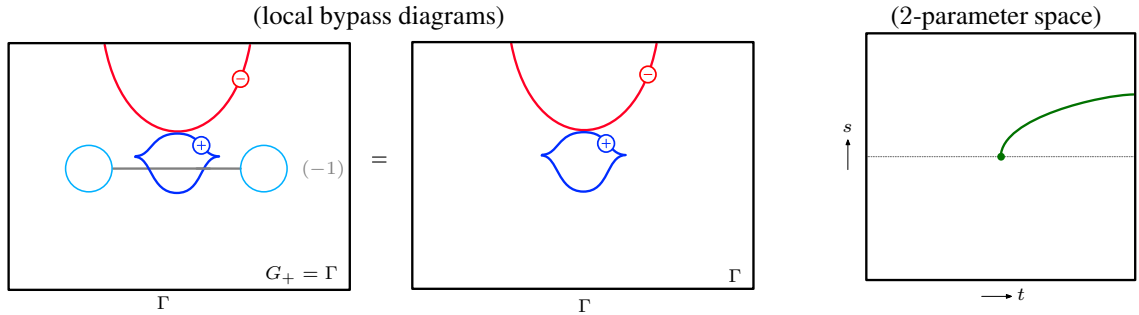
\begin{figure}[ht]
	\begin{overpic}[scale=0.32]{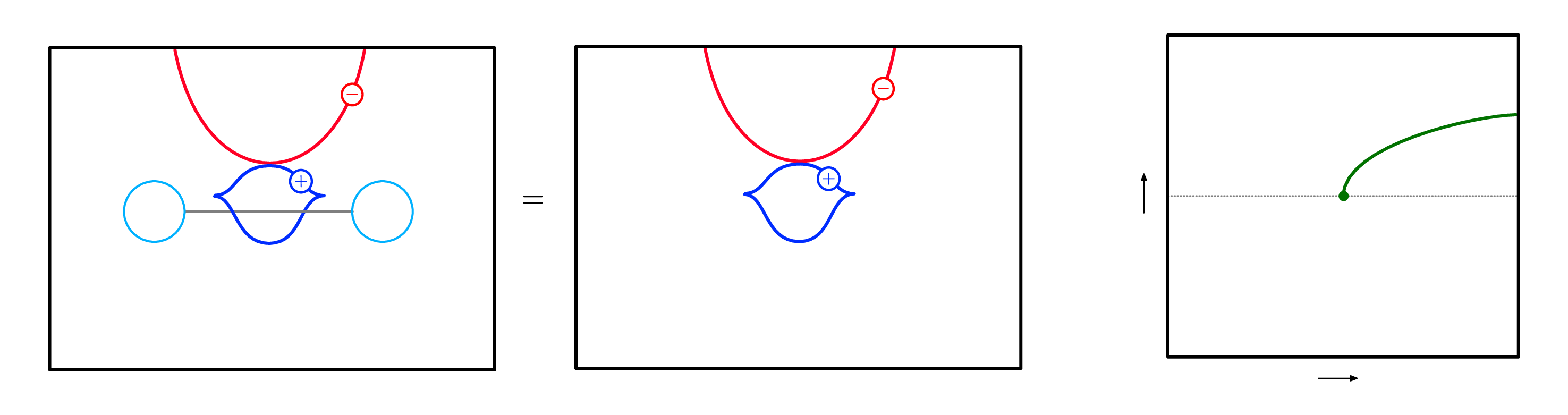}
	    \put(50.5,1){\tiny $\Gamma$}
	    \put(63,3.5){\tiny $\Gamma$}
	    \put(15.5,1.25){\tiny $\Gamma$}
	    \put(24.5,3.75){\tiny $G_+ = \Gamma$}
	    \put(27.5,12.5){\tiny \textcolor{gray}{$(-1)$}}
	    \put(87,2){\tiny $t$}
	    \put(72.5,16){\tiny $s$}
	    \put(23.5,25){\small (local bypass diagrams)}
	    \put(76.5,25){\small ($2$-parameter space)}
	\end{overpic}
	\vskip-.05in
	\caption{The analysis of (P3) in terms of bypasses. On the left, the $(-1)$-surgered gray arc and the light blue circles represent a canceling pair of $n$- and ($n-1$)-handles, and the decorated blue Legendrian represents the cocore of the gray handle. On the right is the $2$-parameter space. The dashed gray line indicates the presence of a birth-death point. Above represents the region where the canceling points live, and below is the region where they die.}
	\label{fig:P3}
\end{figure}

{\em The case $n=1$.} The low-dimensional case is described by \autoref{fig:P3_low}. The $2$-parameter family of folded Weinstein diagrams in the gray box depicts the positive birth-death point in the second column and the pair of canceling critical points in the third column. As in Figures \ref{fig:P1_low_data} and \ref{fig:P1_low_data2}, using \autoref{prop:bypass_bifurcation_correspondence} we use the lower right folded Weinstein diagram to identify the bypass data corresponding to the bypass described by moving up the third column. Observe that the resulting positive blue arc is a standard arc that sits below the red arc, and hence the bypass is trivial.  

\begin{figure}[ht]
\vskip.1in
	\begin{overpic}[scale=0.58]{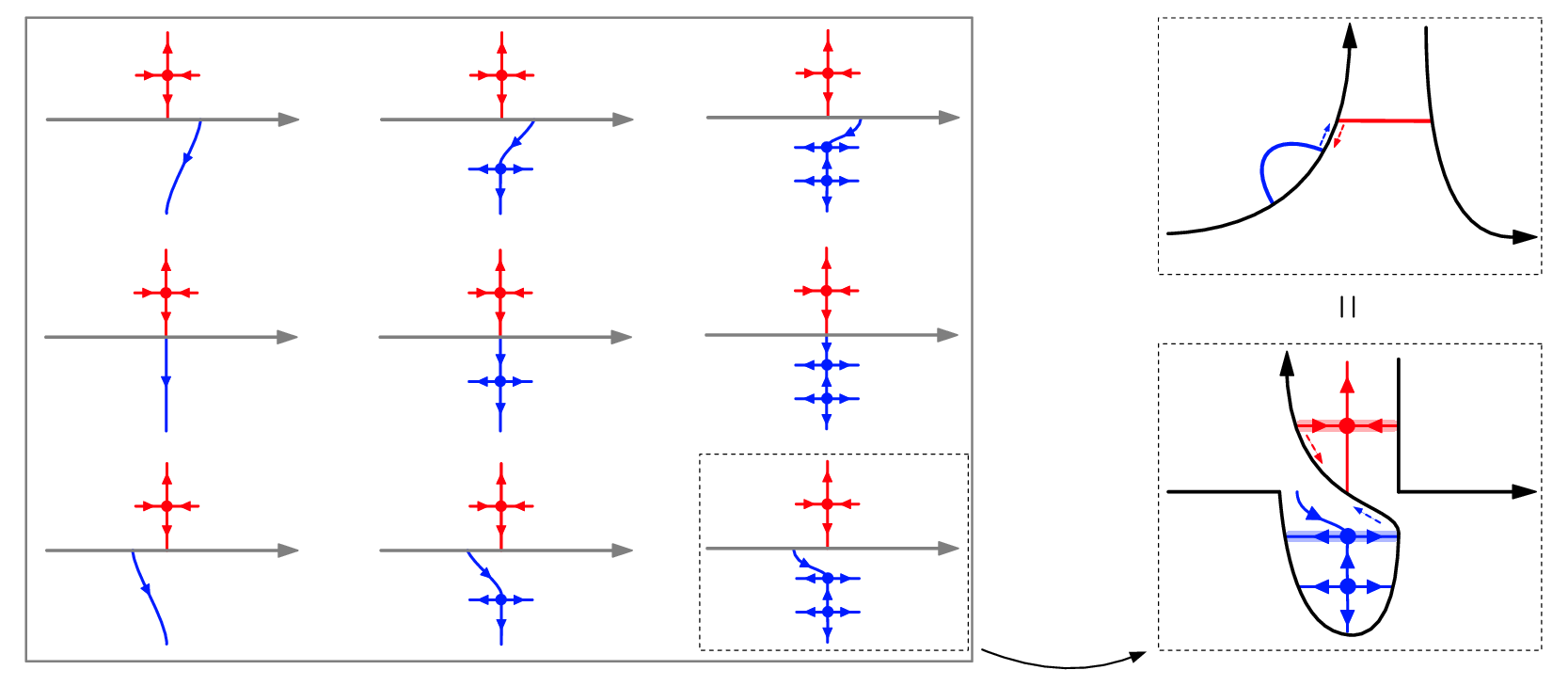}
	    \put(65.5,5){\tiny identify}
	    \put(65.75,3.5){\tiny bypass}
	    \put(66.5,2){\tiny data}
	    \put(26.5,43.5){\tiny \textcolor{gray}{Folded Weinstein}}
	    \put(84,43.5){\tiny Convex}
	\end{overpic}
	\caption{The $n=1$ version of (P3). The resulting bypass arc described by the top right figure gives a trivial bypass.}
	\label{fig:P3_low}
\end{figure}
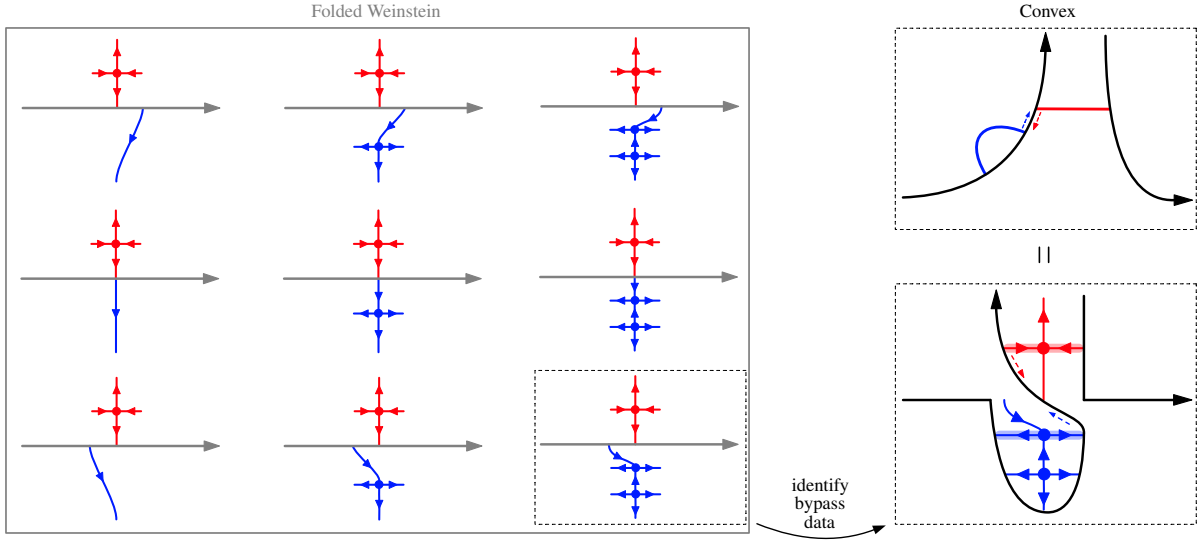

\s\n
{\em (P4) and (P4').} We analyze (P4) explicitly; (P4') is completely analogous. Assume that there is a single retrogradient from a negative nondegenerate point $p_n^-$ of index $n$ to a positive nondegenerate point $p_{n-1}^+$ of index $n-1$. On either side of this retrogradient in $2$-parameter space, there may be nearby retrogradients from $p_n^-$ to other positive critical points of index $n$, and this is the behavior that needs to be analyzed in terms of bypass attachments. 

\begin{figure}[ht]
	\begin{overpic}[scale=0.33]{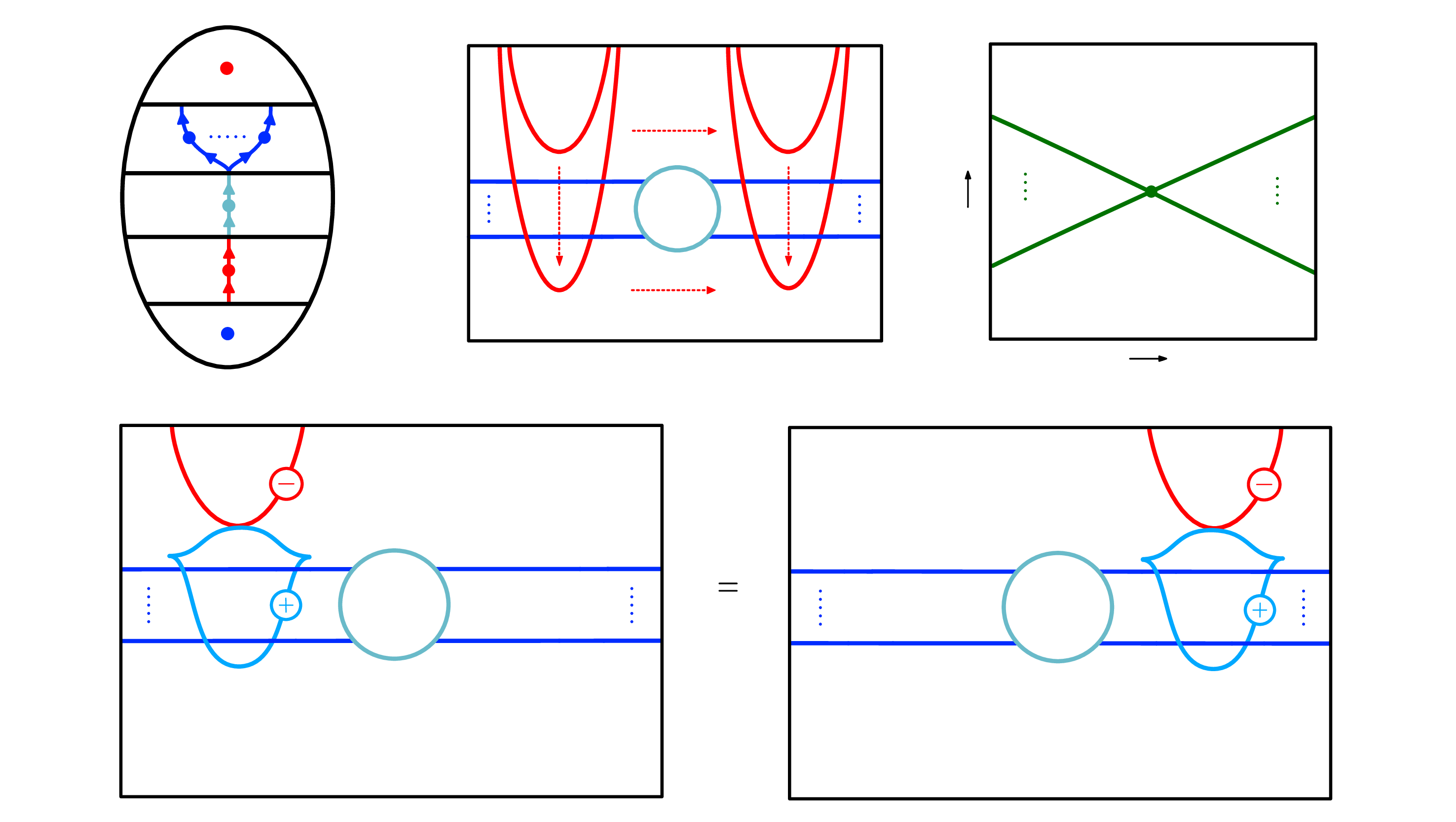}
	    \put(71.5,0){\tiny $\Gamma$}
	    \put(88.25,2.5){\tiny $G_+$}
	    \put(92,12){\tiny \textcolor{blue}{$(-1)$}}
	    \put(92,17.5){\tiny \textcolor{blue}{$(-1)$}}
	    \put(27,0){\tiny $\Gamma$}
	    \put(43,2.5){\tiny $G_+$}
	    \put(45.75,12){\tiny \textcolor{blue}{$(-1)$}}
	    \put(45.75,17.5){\tiny \textcolor{blue}{$(-1)$}}
	    \put(46,31.25){\tiny $Y'$}
	    \put(58.25,33.75){\tiny $Y$}
	    \put(80.5,31.5){\tiny $t$}
	    \put(66,45.25){\tiny $s$}
	    \put(39.5,28.5){\small (local bypass diagrams)}
	    \put(7.5,56.25){\small $\Sigma$ (folded Weinstein)}
	    \put(34.5,55){\small (local retrogradient diagram)}
	    \put(70.5,55){\small ($2$-parameter space)}
	    \put(82, 49){\tiny \textcolor{Green}{$k_1$ strands}}
	    \put(69.5, 49){\tiny \textcolor{Green}{$k_0$ strands}}
	    \put(9.5,46.5){\tiny \textcolor{blue}{$p_{n,1}^+$}}
	    \put(18.75,46.25){\tiny \textcolor{blue}{$p_{n,K}^+$}}
	    \put(17,42){\tiny \textcolor{Cerulean}{$p_{n-1}^+$}}
	    \put(17,37.5){\tiny {\cbu  $p_{n}^-$}}
	    \put(23.5,39.5){\tiny $Y$}
        \put(23.5,44){\tiny $Y'$}
        \put(22.25,35){\tiny $G_+$}
	\end{overpic}
	\caption{The analysis of (P4) in terms of bypasses. In the top middle frame, the (downward) vertical arrows correspond to increasing $s$ and the horizontal arrows correspond to increasing $t$.} 
	\label{fig:P4}
\end{figure}
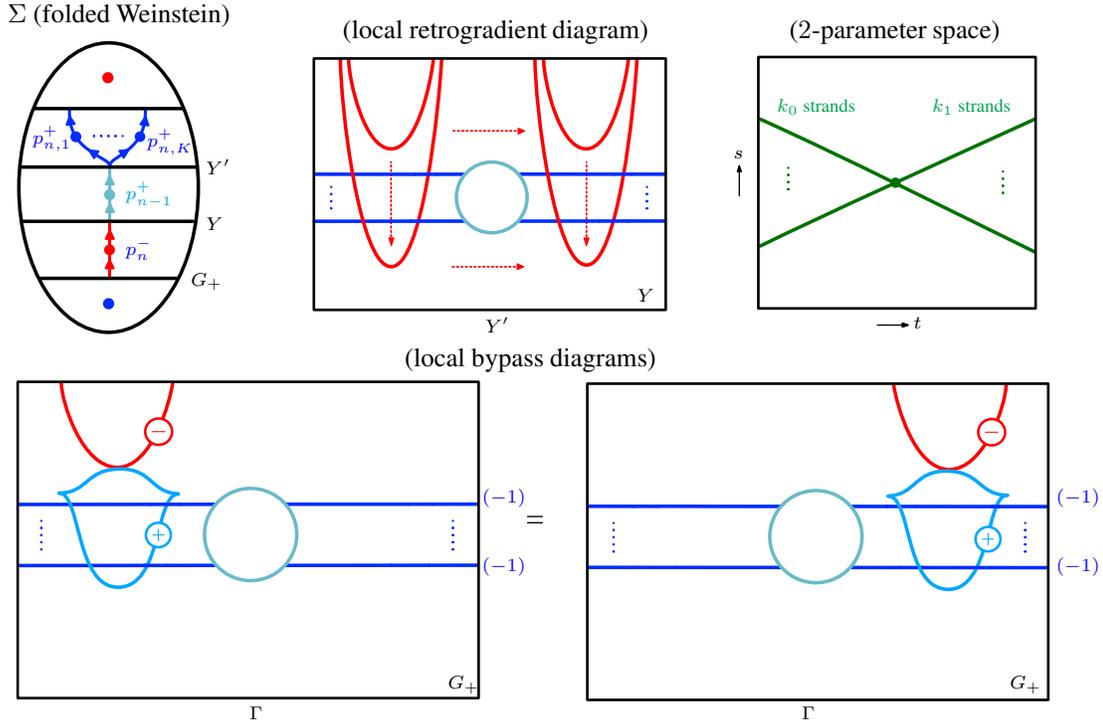

Indeed, let $p_{n,1}^+, \dots, p_{n, K}^+$ be the positive critical points of index $n$ such there are (possibly multiple) trajectories $p_{n-1}^+ \rightsquigarrow p_{n,k}^+$ for each $1\leq k\leq K$. Let $G_+\subseteq R_+$ be a regular contact level set below $p_{n-1}^+$ and $p_{n,1}^+, \dots, p_{n, K}^+$. We can then view and model this assumption in a Legendrian surgery diagram as follows: The attaching region of the Weinstein ($n-1$)-handle associated to $p_{n-1}^+$ is given by the light blue disk in the diagrams in \autoref{fig:P4}, again with the spinning convention of \cite{casals2019legendrianfronts}.
The attaching spheres of the Weinstein $n$-handles associated to $p_{n,1}^+, \dots, p_{n, K}^+$ are given by the $\tilde{K}$ blue strands. The assumption that there are trajectories $p_{n-1}^+ \rightsquigarrow p_{n,k}^+$ means that the blue strands intersect and pass through the attaching region of the ($n-1$)-handle. We allow $\tilde{K}\geq K$ to account for the possibility that there are multiple trajectories $p_{n-1}^+ \rightsquigarrow p_{n,k}^+$ for a fixed $k$, and after localizing we may assume there are $K_0$ strands on the left and $K_1$ on the right, so that $\tilde{K} = K_0 + K_1$. 

By a gluing argument, the (P4) retrogradient $p_{n-1}^+ \rightsquigarrow p_{n,k}^+$ forms broken flow lines $p_{n-1}^+ \rightsquigarrow p_{n,k}^+ \rightsquigarrow p_{n,k}^+$ that may be perturbed into genuine flow lines. This is witnessed by the local retrogradient diagram in \autoref{fig:P4}. On either side of the (P4) retrogradient $p_{n-1}^+ \rightsquigarrow p_{n,k}^+$, the sequences of bypasses is given as a single bypass by \autoref{lemma:parallel_strands_bypass}, yielding the two local bypass diagrams in the bottom row of \autoref{fig:P4} with isotopic data. This shows that both sequences of bypasses for $t=0,1$ are equivalent, completing the analysis of (P4) and the proof of \autoref{theorem: sigma times interval}.

\s\n
{\em  (P5) and (P5').} Finally we analyze (P5) and (P5'), which consider broken retrogradients of the form $p_1^- \rightsquigarrow p_2^- \rightsquigarrow p^+$ and $p^- \rightsquigarrow p_1^+ \rightsquigarrow p_2^+$, where all critical points are of index $n$.

{\em The case $n>1$.} For the high-dimensional case we consider (P5'). As usual, the case of (P5) is identical. We thus assume there is a broken retrogradient $p^- \rightsquigarrow p_1^+ \rightsquigarrow p_2^+$ as in the top left diagram of \autoref{fig:P5}. A generic perturbation of this broken retrogradient is depicted by the top middle diagram in the figure, which is based in $Y$ and presents $Y''$. There are two vertical axes of local spin symmetry in the front, which coincide with the maximum and minimum height of the light blue Legendrian. The retrogradient $p^-\rightsquigarrow p_1^+$ occurs when the Legendrian $\Lambda_-$, the ascending sphere of $p^-$ in $Y$, intersects $\Lambda_{+,1}$, the attaching sphere of $p_1^+$ in $Y$. The trajectory $p_1^+ \rightsquigarrow p_2^+$ corresponds to the handleslide of $\Lambda_{+,2}$, the attaching sphere of $p_{2}^+$, up over the surgered $\Lambda_{+,1}$. In the $2$-parametric family, the arrow indicating the motion of $\Lambda_-$ occurring in $Y$ corresponds to increasing $s$, and the arrow indicating the handleslide, corresponds to increasing $t$. Consequently, for $t=0$, before the handleslide, there is one retrogradient and hence one bypass, while for $t=1$ there are two retrogradients and hence two bypasses.

The corresponding bypass diagrams in the bottom row are furnished by \autoref{lemma:parallel_strands_bypass}, as usual. After sliding $\Lambda_{+,2}$ over $\Lambda_{+,1}$ on the left, the bypass diagrams become equal, hence the bypass sequences are the same up to an additional trivial bypass on the left.

\begin{figure}[ht]
	\begin{overpic}[scale=0.33]{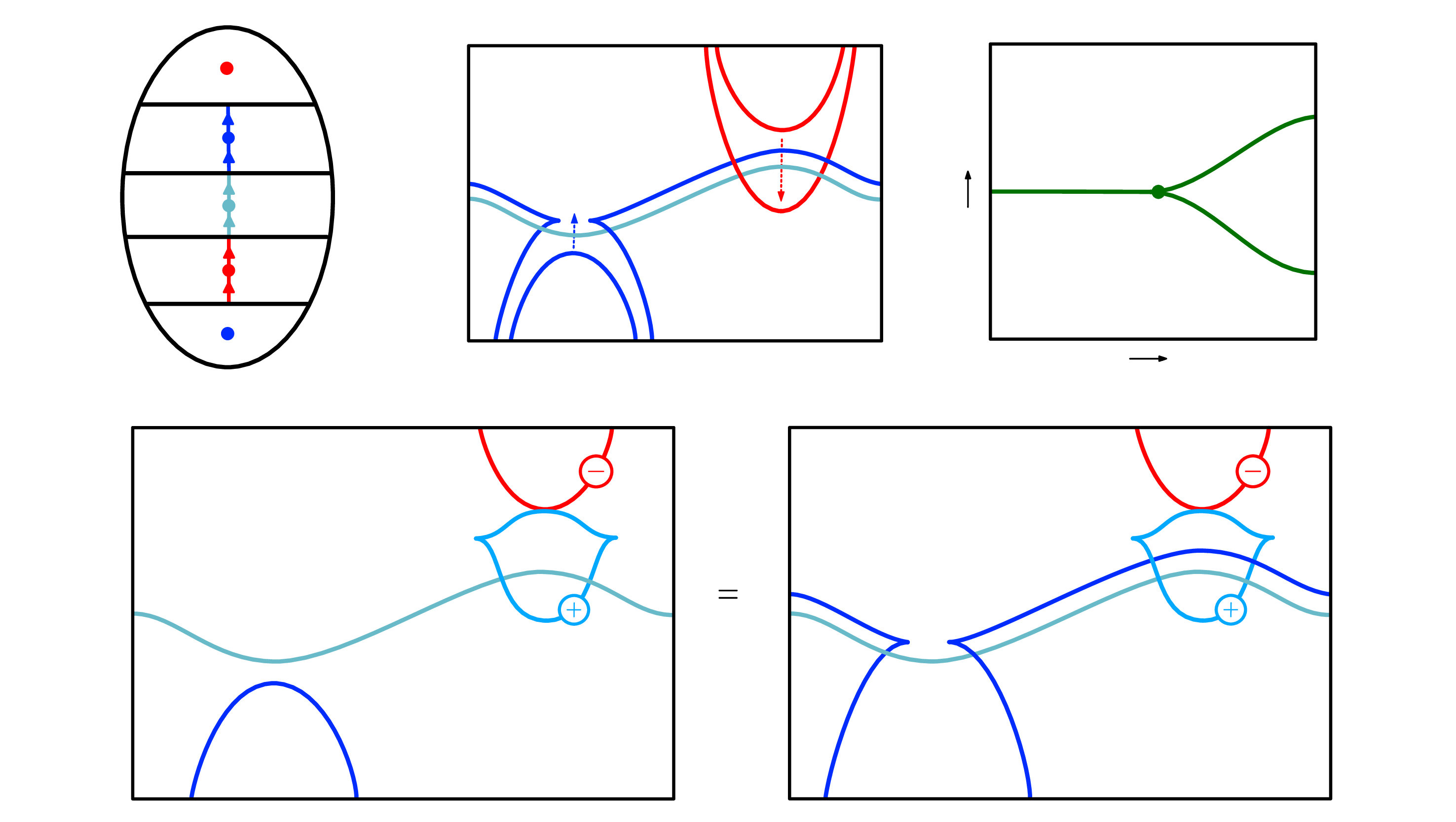}
	    \put(71.5,0){\tiny $\Gamma$}
	    \put(88.25,2.5){\tiny $G_+$}
	    \put(92,13.25){\tiny \textcolor{Cerulean}{$(-1)$}}
	    \put(92,15.75){\tiny \textcolor{blue}{$(-1)$}}
	    \put(27,0){\tiny $\Gamma$}
	    \put(43,2.5){\tiny $G_+$}
	    \put(46.75,13.25){\tiny \textcolor{Cerulean}{$(-1)$}}
	    \put(46.75,11){\tiny \textcolor{blue}{$(-1)$}}
	    \put(46,31.25){\tiny $Y''$}
	    \put(58.25,33.75){\tiny $Y$}
	    \put(80.5,31.5){\tiny $t$}
	    \put(66,45.25){\tiny $s$}
	    \put(39.5,28.5){\small (local bypass diagrams)}
	    \put(7.5,56.25){\small $\Sigma$ (folded Weinstein)}
	    \put(34.5,55){\small (local retrogradient diagram)}
	    \put(70.5,55){\small ($2$-parameter space)}
	    \put(17,46.75){\tiny \textcolor{blue}{$p_{2}^+$}}
	    \put(17,42.5){\tiny \textcolor{Cerulean}{$p_{1}^+$}}
	    \put(17,37.5){\tiny {\cbu  $p_{n}^-$}}
	    \put(23.5,39.5){\tiny $Y$}
        \put(23.5,44){\tiny $Y'$}
        \put(22.5,48.75){\tiny $Y''$}
        \put(22.25,35){\tiny $G_+$}
        \put(61,42){\tiny \textcolor{Cerulean}{$(-1)$}}
        \put(61,44){\tiny \textcolor{blue}{$(-1)$}}
        \put(52.5,50){\small \textcolor{red}{$\Lambda_-$}}
        \put(37.5,35){\small \textcolor{blue}{$\Lambda_{+,2}$}}
        \put(45.5,40){\small \textcolor{Cerulean}{$\Lambda_{+,1}$}}
	\end{overpic}
	\caption{The analysis of (P5') in terms of bypasses. In the top middle frame, the red arrow corresponds to increasing $s$ and the blue arrow corresponds to increasing $t$.} 
	\label{fig:P5}
\end{figure}
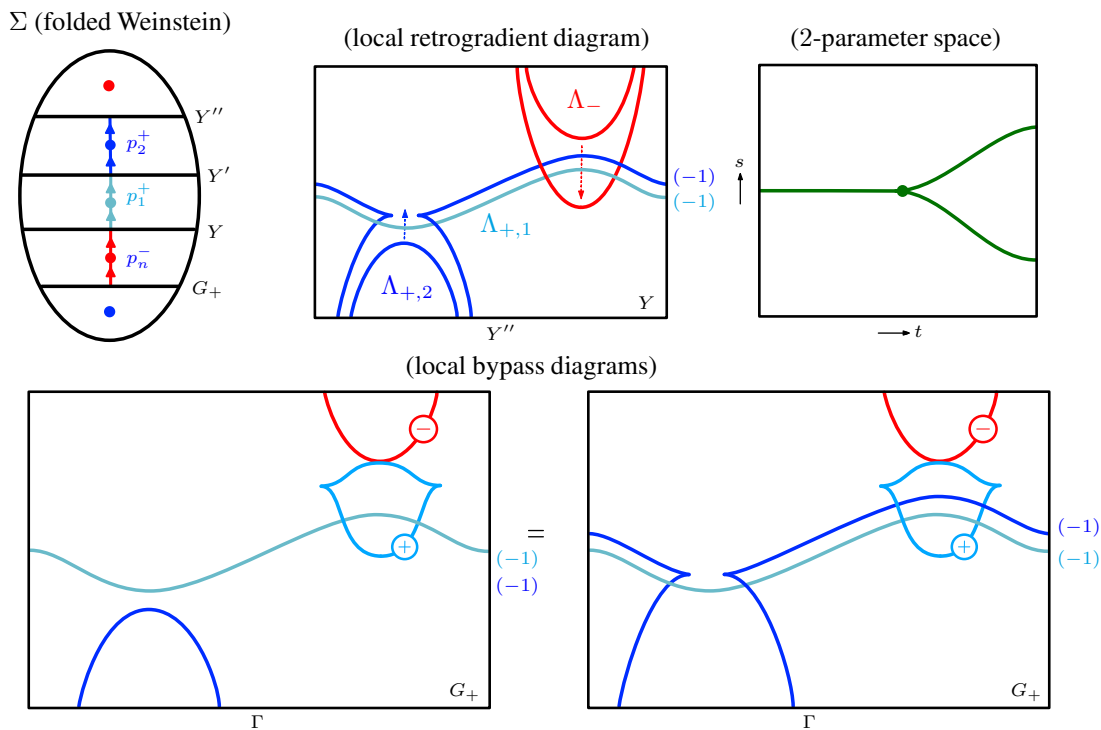

{\em The case $n=1$.} For the low-dimensional case we instead analyze (P5). The situation is summarized by \autoref{fig:P5_low}, which is more complicated but analogous to, for instance, \autoref{fig:P3_low}. The gray box in the center of \autoref{fig:P5_low} depicts the generic local model of a (P5) broken retrogradient in the folded Weinstein setting; this is given by the diagram surrounded by the light blue box. Surrounding this light blue box is the generic $2$-parameter perturbation, where from left to right we perturb the $p_1^- \rightsquigarrow p_2^-$ negative trajectory and from bottom to top we perturb the $p_2^- \rightsquigarrow p^+$ retrogradient trajectory. Note that, as a result, the left side of the family exhibits one retrogradient from bottom to top, and the right side exhibits two. As in \autoref{fig:P3_low} we identify the initial bypass data for each family in the lower row and compute the correspond bypass attachments. It follows that the left and right bypass sequences are the same, up to the addition of a trivial bypass on the left. Note that this can be viewed as an instance of bypass rotation.
\begin{figure}[ht]
\vskip.1in
	\begin{overpic}[scale=0.58]{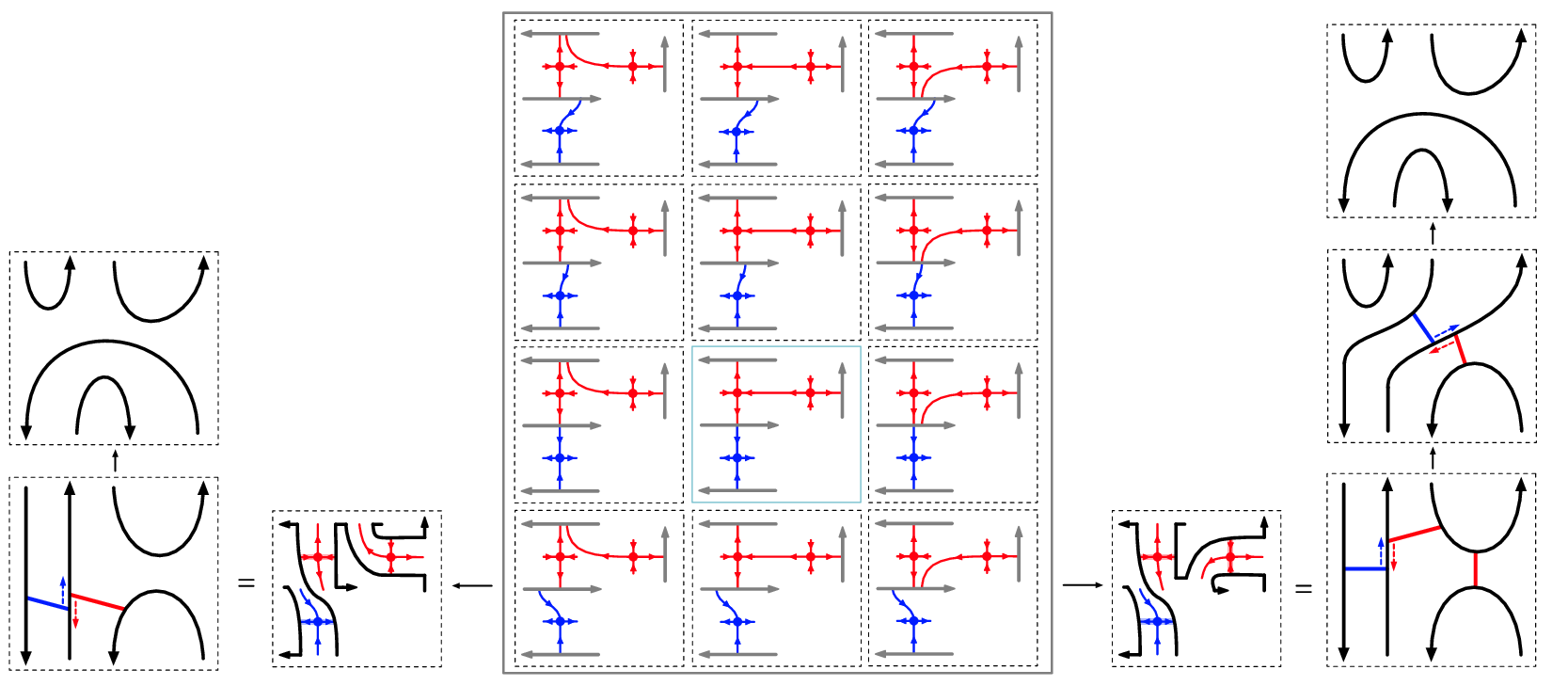}
	    \put(70.5,0){\tiny identify bypass data}
	    \put(17,0){\tiny identify bypass data}
	    \put(44,43.5){\tiny \textcolor{gray}{Folded Weinstein}}
	    \put(89,43){\tiny Convex}
     \put(5,28.5){\tiny Convex}
	\end{overpic}
	\caption{The $n=1$ version of (P5).}
	\label{fig:P5_low}
\end{figure}
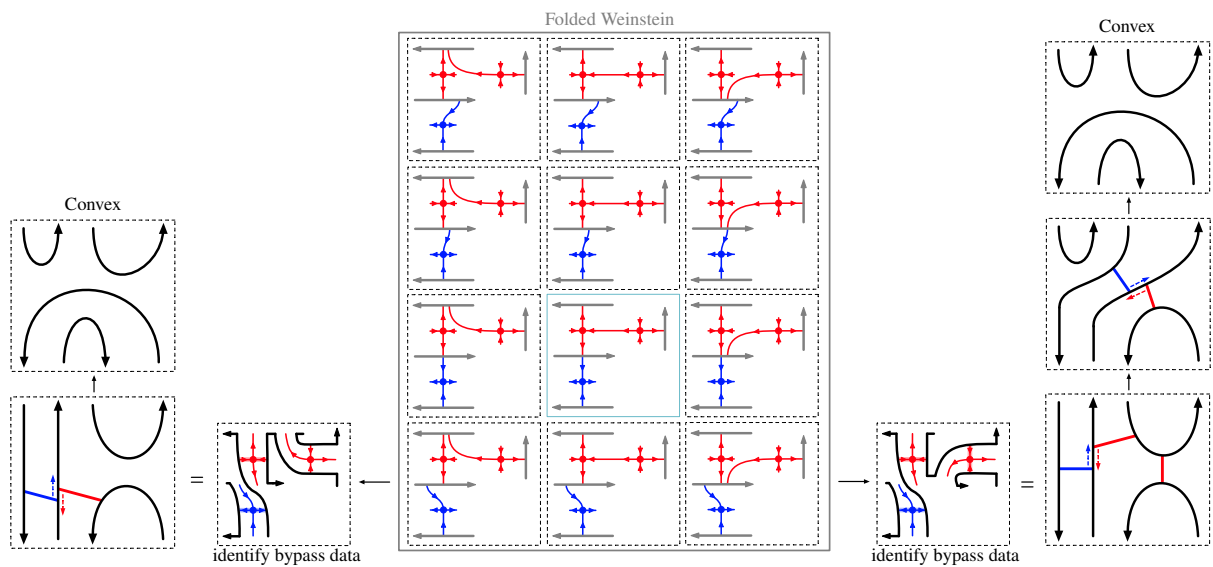
\end{proof}

\section{Wrinkling Legendrians via bypasses}
\label{section: proof of theorem layer}

In this section we prove \autoref{theorem: layer between N(L) and N(L')} and \autoref{thm: layer between N and N' general case}.

\subsection{Wrinkled Legendrian spheres}  \label{subsection: description of N(L')}

First we prove \autoref{theorem: layer between N(L) and N(L')}. We remind the reader that $L$ is a standard Legendrian sphere, $L'$ is a wrinkled Legendrian sphere, and that we have standard neighborhoods $N(L') \subset N(L)$.

\begin{proof}[Proof of \autoref{theorem: layer between N(L) and N(L')}(1).]
To prove (1) we describe $N(L)$ and $N(L')$ as contact handlebodies. The description of $N(L)$ as a contact handlebody over $W$ as in \autoref{theorem: layer between N(L) and N(L')}(1) is standard, so the rest of the proof is devoted to describing $N(L')$ as a contact handlebody over $W'$ as on the left side of \autoref{fig:simple_legendrians}. 

Note that $L'$ is the union of two (topological) $n$-disks, one being the middle sheet of the wrinkle bounded by the piecewise smooth sphere $S^{n-1}$ with the cusp edge $S^{n-2}$, and the other its complement. So our starting point is to first describe $N(L')$ as $N(D_1) \cup N(D_2)$ where: 
\begin{itemize}
    \item The $n$-disk $D_1$ is a smooth, small extension of the middle sheet depicted by the red disk on the left side of \autoref{fig:wrinkledisks}.
    \item The $n$-disk $D_2$ is a slight retraction of the complement of the middle sheet, whose boundary is the green sphere on the right side of \autoref{fig:wrinkledisks}. 
\end{itemize}
We can then interpret $N(D_2)$ as a contact $n$-handle attached to the Darboux ball $N(D_1)$ along the green $S^{n-1} = \partial D_2$, hence giving a contact handlebody description of $N(L')$. The rest of the proof proceeds by giving a careful local description of $N(D_1) \cup N(D_2)$, and then using this local description to identify $\partial D_2 \subset \partial N(D_1)$ as the Legendrian on the left side of \autoref{fig:simple_legendrians}.

\begin{figure}[ht]
\vskip-.1in
	\begin{overpic}[scale=.3]{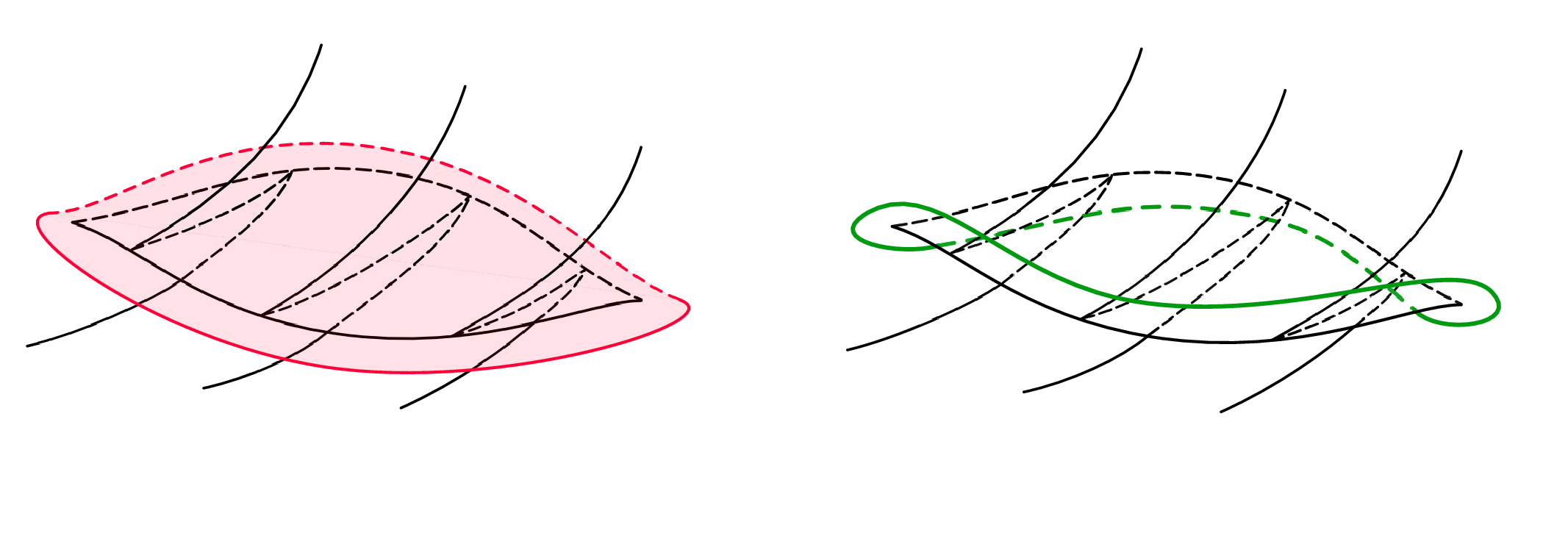}
    \put(10,26){\small {\color{red}  $D_1$}}
    \put(56,24){\small \textcolor{ForestGreen}{$\partial D_2$}}
	\end{overpic}
    \vskip-0.3in
	\caption{The disks $D_1$ and $D_2$.}
	\label{fig:wrinkledisks}
\end{figure}

\s\n 
{\em Description of $N(D_1)$.}
\s

Let $D_1$ be a Legendrian $n$-disk. For $\epsilon > 0$ small, there is a contact handlebody neighborhood 
\[
\left(N(D_1) \simeq \D^*D_1 \times [-\epsilon, \epsilon]_t, \, \alpha = dt + \beta\right)
\]
where $\beta$ is a Liouville form on $\D^*D_1$ that will be described shortly. After smoothing, $N(D_1)$ has convex boundary with dividing set
\[
\Gamma:=\Gamma_{\bdry N(D_1)}=\bdry (\mathbb{D}^*D_1)\times\{0\}=(\mathbb{S}^*D_1\cup  \mathbb{D}^*D_1|_{\bdry D_1})\times\{0\}.
\]
Here $\mathbb{S}^*$ (resp.\ $\mathbb{D}^*$) denotes the unit sphere (resp.\ unit disk) cotangent bundle.  We refer to $\mathbb{S}^*D_1$ as the {\em horizontal boundary} and $\mathbb{D}^*D_1|_{\bdry D_1}$ as the {\em vertical boundary} of $\mathbb{D}^*D_1$.

Next we describe the Liouville form $\beta$ near $\partial \D^*D_1$. Let $x_1,\dots, x_{n-1}$ be local coordinates for $\bdry D_1$ and $x'$ be a strictly increasing function of the radial coordinate for $D_1$ such that $\bdry D_1=\{x'= 0\}$.  Let $y_1,\dots, y_{n-1},y'$ be dual coordinates of $x_1,\dots,x_{n-1},x'$ in $T^*D_1$ so that $x_1,\dots,x_{n-1},y_1,\dots,y_{n-1}$ are coordinates on $T^*(\bdry D_1)$. Choose $\epsilon_0>0$ small and a nondecreasing smooth function $\phi:(-\infty,0]\to [0,\epsilon_0)$ such that $|\phi'|\leq 2\epsilon_0$, $\phi(-\sqrt{\epsilon_0})=0$, and $\phi(0)=\epsilon_0$. We then choose the Liouville form on $\mathbb{D}^*D_1$ to be 
\[
\beta:=-\textstyle\sum_{i=1}^{n-1}y_i\, dx_i- y'\, dx'+\phi(x')\, dy'.
\]
Note that the Liouville vector field 
\[
X_{\beta} = \textstyle\sum_{i=1}^{n-1} y_i\, \partial_{y_i} + \frac{ y'\, \partial_{y'} + \phi(x')\bdry_{x'}}{1+\phi'(x')}\, 
\]
does not point transversely out of the vertical boundary $\mathbb{D}^*D_1|_{\bdry D_1}=\mathbb{D}^*(\bdry D_1)\times [-1,1]_{y'}$ unless we add a term such as $\phi(x')\, dy'$.

\s\n
{\em Description of the attaching sphere of $D_2$.}
\s

Next we describe the sphere $\partial D_2 \subset \Gamma$ along which we will glue $D_2$. First, we introduce some auxiliary data $\Lambda_{2,+}$ and $\Lambda_{2,-}$ in $\partial N(D_1)$ but slightly off of $\Gamma$, and then we will flow this data to $\Gamma$ to describe $\partial D_2$. Namely, we let $\Lambda_{2,\pm} \simeq D^{n-1}$ be Legendrian disks in $\mathbb{S}^*D_1 \times \{\pm \epsilon/2\}$, such that, under the projection map $\pi: T^*D_1 \times [-\epsilon, \epsilon] \to D_1$, their images are given by \autoref{fig: attaching-Legendrian'}. 

\begin{figure}[ht]
\vskip.15in
	\begin{overpic}[scale=.65]{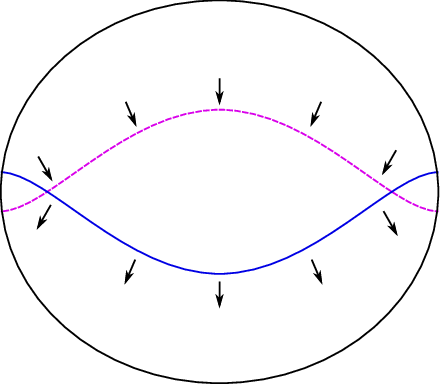}
    \put(47,30){\tiny $\Lambda_{2,+}$} \put(47,52){\tiny $\Lambda_{2,-}$}
	\end{overpic}
	\caption{The solid (resp.\ dotted) line indicates that $D_2$ approaches $D_1$ from the $t>0$ side (resp.\ from the $t<0$ side). The arrows indicate the normal orientations of $T\Lambda_{2,+}$ and $T\Lambda_{2,-}$ or, equivalently, the Reeb vector fields along $T\Lambda_{2,+}$ and $T\Lambda_{2,-}$.}
	\label{fig: attaching-Legendrian'}
\end{figure}

More precisely, the projections are $\pi(\Lambda_{2,\pm}) = D_{\pm}'$, where $\partial D_1 = D_+ \cup D_-$ is the decomposition of $\partial D_1$ into the upper and lower hemispheres and $D_{\pm}'$ are pushoffs of $D_{\pm}$ into $D_1$ so that $D_{\pm} \cap D_{\pm}' = \varnothing$, $\partial D_{\pm}' \subset \partial D_1$, and $\partial_{x'} \in TD_{\pm}'|_{\partial D_{\pm}'}$. By this last requirement, $\partial \Lambda_{2,+}$ and $\partial \Lambda_{2,-}$ lie on $\{y'=0\}$. 

Next, flow $\Lambda_{2,\pm}$ to $\Gamma$ using the characteristic foliation of $\Gamma\times[-\epsilon,\epsilon]$; this is directed by $\bdry_t -R_{\Gamma}$, where $R_{\Gamma}$ is the Reeb vector field of $\beta|_{\Gamma}$. Note that the result of the flow, which we denote by $(\Lambda_{2,\pm})^\sharp$, lie in the horizontal component $\mathbb{S}^*D_1 \times \{0\}$ of $\Gamma$.

We then take $D_2$ to have Legendrian boundary sphere
\[
\Lambda_2 := (\Lambda_{2,+})^{\sharp} \cup (\Lambda_{2,0})^{\sharp} \cup (\Lambda_{2,-})^{\sharp} \subset \Gamma,
\]
where $(\Lambda_{2,0})^{\sharp} \simeq S^{n-2}\times[-1,1]$ is a cylindrical component in the vertical boundary component $\D^*D_1|_{\partial D_1} \times \{0\}$. In order to identify the wrinkled sphere $L'$ and the resulting handlebody structure, we need to (more) carefully describe each of the three pieces of $\partial D_2$, as well as the disk $D_2$ itself near $\partial D_2$.

\s\n
{\em Local model of $D_2$ near interior points of $\Lambda_{2,\pm}$.}

Let $p$ be an interior point of $\Lambda_{2,+}$. We can choose a neighborhood $N(p)\simeq \mathbb{D}^*D^{n-1}\times \mathbb{D}^*D^1 \times [-\epsilon,\epsilon]$ of $p$ in $N(D_1)$ with slightly adjusted coordinates
\[
(x_1,\dots,x_{n-1},y_1,\dots,y_{n-1}),(x'',y''),t
\]
such that:
\be
\item[(a)] the contact form is $\alpha = dt-\sum_{i=1}^{n-1}y_i\, dx_i -y''\, dx''$; 
\item[(b)] $N(p)\cap D_1=\{y_1=\dots=y_{n-1}=y''=t=0\}$; 
\item[(c)] $N(p)\cap \Lambda_{2,+} = \{y_1=\dots=y_{n-1}=x''=0, y''=1, t=\epsilon/2\}$; and
\item[(d)] we are viewing the sphere cotangent bundle $\mathbb{S}^* D^n$ as 
$(\mathbb{S}^*D^{n-1}\times \mathbb{D}^* D^1) \cup (\mathbb{D}^* D^{n-1}\times \mathbb{S}^*D^1)$,
and hence $N(p)\cap \Lambda_{2,+}\subset \mathbb{S}^*D^n\times\{\epsilon/2\}$.
\ee
Note that the front projection of $N(p)$ is given by the coordinates $(x_1, \dots, x_{n-1},x'', t)$ and that the dual $y$-coordinates describe the slopes of the contact structure with respect to the $x$-coordinates. 

Letting $Op(\cdot)$ be an unspecified open enlargement as usual, we then take $D_2\cap Op(N(p))$ --- or more precisely, the version of $D_2$ before isotoping $\Lambda_{2,+}$ to $(\Lambda_{2,+})^\sharp$ --- to be 
$$\{(x_1,\dots,x_{n-1},y_1,\dots,y_{n-1})~|~ y_1=\dots= y_{n-1}=0\}\times C_+,$$
where $C_+\subset Op(\mathbb{D}^*D^1 \times [-\epsilon,\epsilon])$ is the curve such that:
\be
\item[(i)] $(dt-y''\, dx'')|_{C_+}=0$; 
\item[(ii)] $C_+$ is a subset of $\{x''\geq 0\}$ with an endpoint at $(x'',y'',t)=(0,1,\epsilon/2)$; and  
\item[(iii)] $y''|_{C_+}$ is a smooth function of $x''$ with $y''(0)=1$ and positive derivative at $x''=0$.
\ee
See \autoref{fig: profile'}.
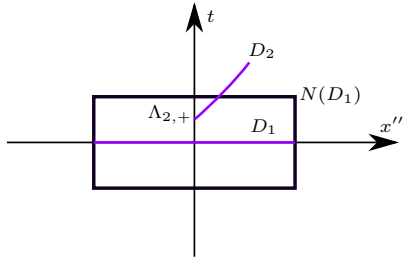
\begin{figure}[ht]
\vskip.15in
	\begin{overpic}[scale=.23]{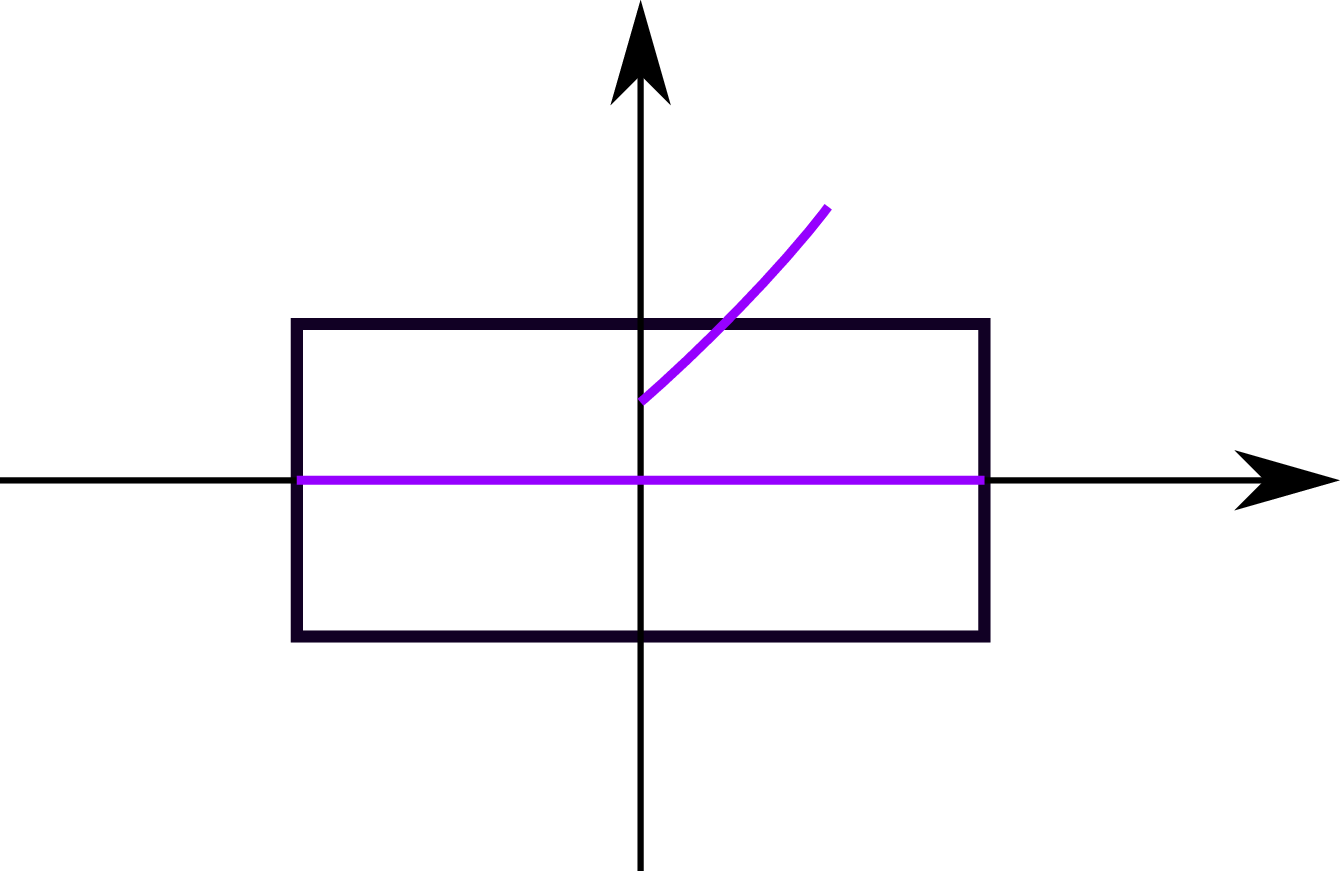}
    \put(95,32){\tiny $x''$} \put(74.5,40){\tiny $N(D_1)$}
    \put(62,32){\tiny $D_1$} \put(61.5,51.5) {\tiny $D_2$} \put(51,60){\tiny $t$} \put(36,35.5){\tiny $\Lambda_{2,+}$}
	\end{overpic}
	\caption{Local pictures of projections of $N(D_1)$, $D_1$, and $D_2$ (before isotoping $\Lambda_{2,+}$ to $(\Lambda_{2,+})^\sharp$) to the $tx''$-plane. Here $D_1$ and $D_2$ are products of $\{(x,y)~|~ y=0\}$ and the purple arcs shown.}
	\label{fig: profile'}
\end{figure}

The situation for an interior point of $\Lambda_{2,-}$ is analogous, with (c), (ii), (iii) replaced by:
\be
\item[(c')] $N(p)\cap \Lambda_{2,-} = \{y_1=\dots=y_{n-1}=x''=0, y''=1, t=-\epsilon/2\}$;
\item[(ii')] $C_-$ is a subset of $\{x''\leq 0\}$ with an endpoint at $(x'',y'',t)=(0,1,-\epsilon/2)$; and 
\item[(iii')] $y''|_{C_-}$ is a smooth function of $x''$ with $y''(0)=1$ and negative derivative at $x''=0$.
\ee

\s\n
{\em Local model near $\bdry \Lambda_{2,\pm}$.} 
\s

We consider a neighborhood $N(\bdry \Lambda_{2,+})\simeq \mathbb{D}^*S^{n-2}\times \mathbb{D}^*D^2 \times [-\epsilon,\epsilon]$
of $\bdry \Lambda_{2,+}\simeq S^{n-2}$ in $N(D_1)$ with coordinates 
\[
(x_1,\dots,x_{n-2},y_1,\dots,y_{n-2}),(x',x'',y',y''),t.
\]
Here $\bdry_{x'}$ points out of $D_1$, $\bdry_{x''}$ is transverse to $\Lambda_{2,+}$ in $D_1$ (and is consistent with $\bdry_{x''}$ used above), 
\[
\bdry\Lambda_{2,+} = S^{n-2}_0\times \{(x',y',x'',y'',t)=(0,0,0,1,\epsilon/2)\},
\]
where $S^{n-2}_0$ is the zero section $S^{n-2}_0$ of $\mathbb{D}^*S^{n-2}$, and $\Gamma$ restricts to $\mathbb{D}^*S^{n-2}\times \{(x',x'',y',y'')~|~ x'=0\} \times \{0\}$. The contact form on $N(\bdry\Lambda_{2,+})$ is
\[
dt-\left(\textstyle\sum_{i=1}^{n-2}y_i\, dx_i\right) -y'\, dx'+\phi (x') \, dy'-y''\, dx'',
\]
where $\phi(x')$ is as before. Assuming that $\bdry \Lambda_{2,+}$ and $\bdry \Lambda_{2,-}$ are sufficiently close, we may take 
\[
\bdry \Lambda_{2,-} = S^{n-1}_0 \times \{(x',y',x'',y'',t)=(0,0,-\epsilon_1,1,-\epsilon/2)\}
\]
where $\epsilon_1>0$ is small but satisfies $\epsilon_1\gg \epsilon$; this is possible since we can reduce the thickness of the contact handlebody $N(D_1)$.

In these coordinates, the characteristic foliation along $\Gamma\times [-\epsilon,\epsilon]$ is directed by $\bdry_t - R_\beta= \bdry_t - \tfrac{1}{\epsilon_0} \bdry_{y'}$. Thus, flowing $\Lambda_{2,\pm}$ to $\Gamma$ to get $(\Lambda_{2,\pm})^{\sharp}$ gives 
\begin{align*}
    \bdry (\Lambda_{2,+})^{\sharp} &= S^{n-1}_0 \times (0,\tfrac{\epsilon}{2\epsilon_0},0, 1,0) \\
    \bdry (\Lambda_{2,-})^{\sharp} &= S^{n-1}_0 \times (0,-\tfrac{\epsilon}{2\epsilon_0},-\epsilon_1,1,0)
\end{align*}
Here we can further shrink $\epsilon$ so that  $\tfrac{\epsilon}{2\epsilon_0}\ll \epsilon_1$ (basically we may assume that $\tfrac{\epsilon}{2\epsilon_0}=0$).  Finally, the Legendrian cylinder $(\Lambda_{2,0})^{\sharp}$ which connects $\bdry (\Lambda_{2,+})^{\sharp}$ to $\bdry (\Lambda_{2,-})^{\sharp}$ is given by 
\[
(\Lambda_{2,0})^{\sharp} = S^{n-1}_0 \times C'
\]
where $C'$ is a ``minimally twisting" Legendrian arc with respect to the contact form $\phi(0) \, dy' -y''\, dx''$ connecting the points $(y',x'',y'')=(\tfrac{\epsilon}{2\epsilon_0},0,1)$ and $(-\tfrac{\epsilon}{2\epsilon_0},-\epsilon_1,1)$ and isotopic to the straight line between them. 

\s\n
{\em Identification of $\partial D_2$.}
\s

We now identify the resulting Legendrian $\bdry D_2$. The $n=2$ and $n>2$ cases are slightly different:

\s\n
{\em Case $n=2$.} Each of $\bdry(\Lambda_{2,+})^\sharp$ and $\bdry (\Lambda_{2,-})^\sharp$ has two components and $\bdry D_2$ is the $\op{tb}=-3$, $\op{rot}=0$ unknot obtained by isotoping each corresponding pair of endpoints of $\bdry \Lambda_{2,+}$ and $\bdry \Lambda_{2,-}$ and converting it into a cusp. 

\s\n
{\em Case $n>2$.} There is only one component and $\bdry D_2$ is a loose Legendrian unknot since $S^{n-1}\times C'$ is a standard model for a stabilization.

\s
The above description of $N(L')$ as a contact handlebody completes the proof of \autoref{theorem: layer between N(L) and N(L')}(1).
\end{proof}

\begin{proof}[Proof of \autoref{theorem: layer between N(L) and N(L')}(2).]
Next we show that $N(L)$ is obtained from a bypass attachment $B$ to $N(L')$ and that the resulting POBD is a stabilization of $(W, \varnothing)$. The bypass data is given at the top left of \autoref{fig: Legendrians}. Namely, $(\Lambda_+; D_+)$ is a positive Reeb shift of the cocore of the positive Weinstein $n$-handle attached along the loose Legendrian and $(\Lambda_-; D_-)$ is a standard Legendrian unknot/Lagrangian disk pair. At the top left of \autoref{fig: Legendrians} we attach the resulting $R_+$-centric bypass. The diagram below it is obtained by an isotopy of the gray Legendrian. Then, by \autoref{lemma:trivial_bypass_lemma}, the gray and blue handles form a trivial bypass, i.e., a positive stabilization, and can be erased. This yields a contact handlebody over $W$ as desired.
\begin{figure}[ht]
	\centering
	\begin{overpic}[scale=.44]{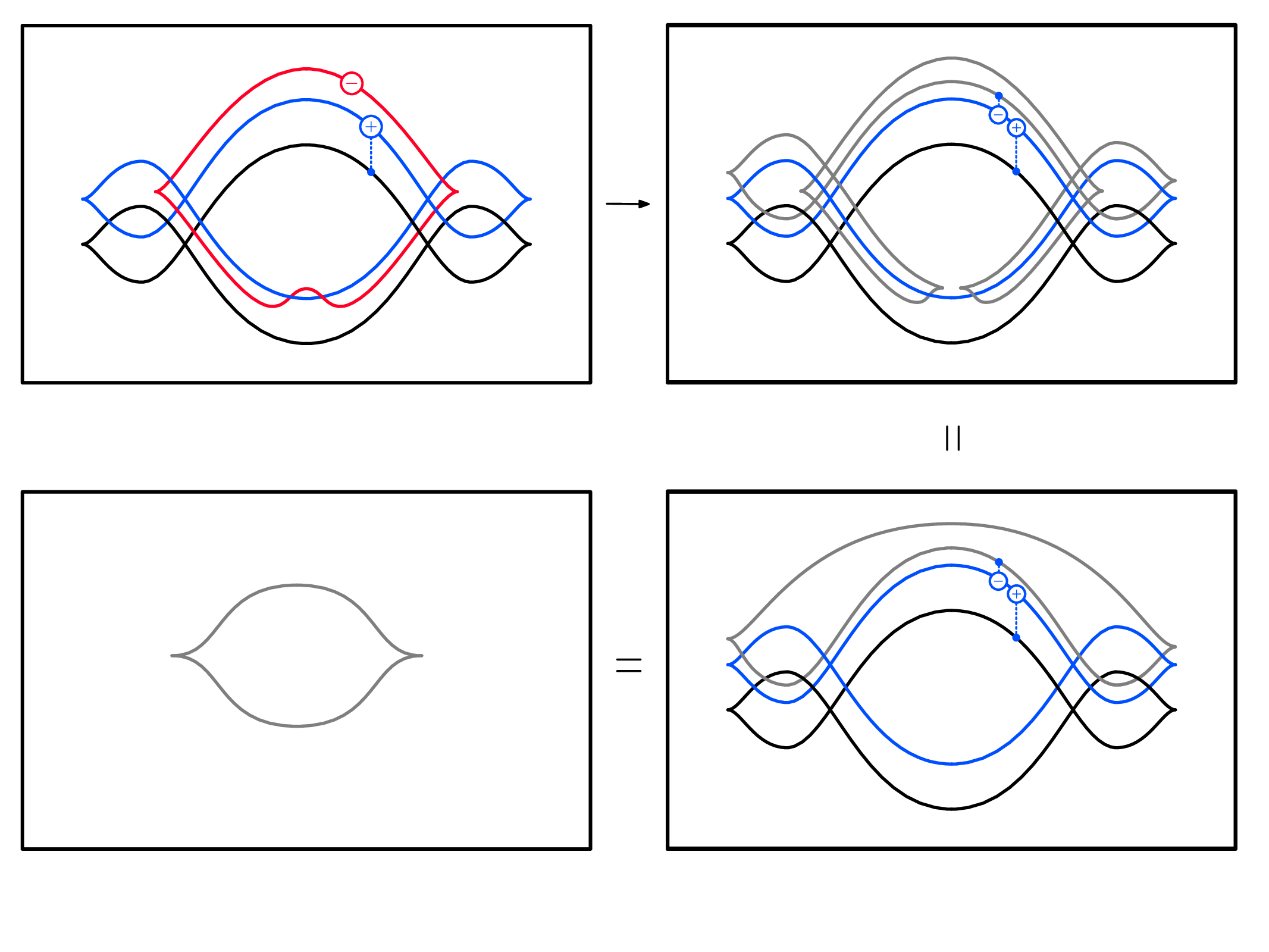}
	\put(16,68){\color{red}\small $\Lambda_-$}
	\put(58,66){\small \textcolor{gray}{$\Lambda_-\uplus \Lambda_+$}}
	\put(10,63.5){\color{blue}\small $\Lambda_+$}
	\put(93.25,55){\tiny $(-1)$}
	\put(93.25,61){\tiny \textcolor{gray}{$(-1)$}}
	\put(93.25,59){\tiny \textcolor{blue}{$(+1)$}}
	\put(93.25,18){\tiny $(-1)$}
	\put(93.25,24){\tiny \textcolor{gray}{$(-1)$}}
	\put(93.25,22){\tiny \textcolor{blue}{$(+1)$}}
	\put(42.25,24){\tiny \textcolor{gray}{$(-1)$}}
	\put(42.25,55){\tiny $(-1)$}
	\put(16, 74){\small (Bypass data $B$)}
	\put(70, 74){\small (Attach $B$)}
	\put(23.5,42.5){\small $\Gamma'$}
	\put(74.5,42.5){\small $\Gamma$}
	\put(23,6){\small $\Gamma$}
	\put(74.5,6){\small $\Gamma$}
	\end{overpic}
	\vskip-.4in
	\caption{The proof of \autoref{theorem: layer between N(L) and N(L')}(2).}
	\label{fig: Legendrians}
\end{figure}    
\end{proof}

\subsection{The general case}

Now we consider more generally how to wrinkle the core disks of $n$-handles. 

\begin{proof}[Proof of \autoref{thm: layer between N and N' general case}.]
    
The argument is similar to the case of a wrinkled sphere with only a few modifications. As in the statement of the theorem, let $L_i$ denote the core disk of the Weinstein $n$-handle $h_i$ and $L_i'$ the wrinkled stabilization of $L_i$; see the diagram on the left side of \autoref{fig:wrinkled_core}. Note that $\partial L_i' = \partial L_i$.

To prove \autoref{thm: layer between N and N' general case}(1), we will decompose each wrinkled core disk $L_i'$ into three pieces, doing so in two steps. First, in analogy with \autoref{theorem: layer between N(L) and N(L')}(1), we identify two pieces: 
\begin{itemize}
    \item a Lagrangian disk $D_{1,i} \simeq D^n$, highlighted in red in the middle diagram of \autoref{fig:wrinkled_core}, and 
    \item a Lagrangian annulus $A_i \simeq S^{n-1}\times [0,1]$, where one boundary component $S^{n-1}\times \{0\} = \partial L_i'$ is attached along a neighborhood $W^{(n-1)}$ of the ($n-1$)-skeleton of $W$ and the other boundary component $S^{n-1}\times \{1\}$ is attached to $\D^*D_{1,i}$ along the stabilized Legendrian, which we denote by $\Lambda_i$, from the proof of \autoref{theorem: layer between N(L) and N(L')}(1). 
\end{itemize}
Next we decompose the annulus $A_i$ into two pieces:
\begin{itemize}
    \item a neighborhood of an arc $c_i = \{p\} \times [0,1] \subset S^{n-1} \times [0,1]$, drawn in purple in the diagram on the right side of \autoref{fig:wrinkled_core}, and 
    \item a Lagrangian disk $D_{2,i}$ whose boundary is drawn in light green in the same diagram. 
\end{itemize}

\begin{figure}[ht]
	\centering
	\begin{overpic}[scale=.45]{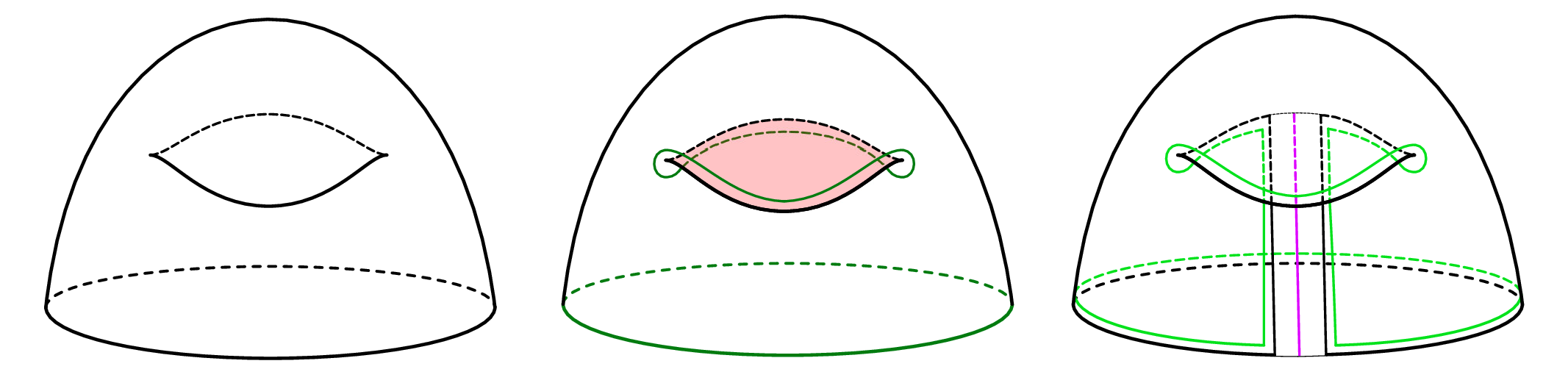}
    \put(9.5,-0.75){\small wrinkled core $L_i'$}
    \put(74.5,4){\small\cg $\partial D_{2,i}$}
    \put(82,0){\small\textcolor{Mulberry}{$c_i$}}
    \put(46.5,3){\tiny\textcolor{ForestGreen} {$S^{n-1}\times \{0\}$}}
    \put(44,17.5){\tiny\textcolor{ForestGreen} {$\Lambda_i = S^{n-1}\times \{1\}$}}
    \put(48.25,13.5){\small\textcolor{red} {$D_{1,i}$}}
	\end{overpic}
	\caption{In the middle diagram the wrinkled disk $L_i'$ is decomposed into two pieces: the red disk $D_{1,i}$ and the annulus $A_i$ with dark green boundary components. In the right diagram, we decompose the annulus $A_i$ into two more pieces: a neighborhood of the arc $c_i$ and a disk $D_{2,i}$ with the indicated light green boundary.}
	\label{fig:wrinkled_core}
\end{figure}
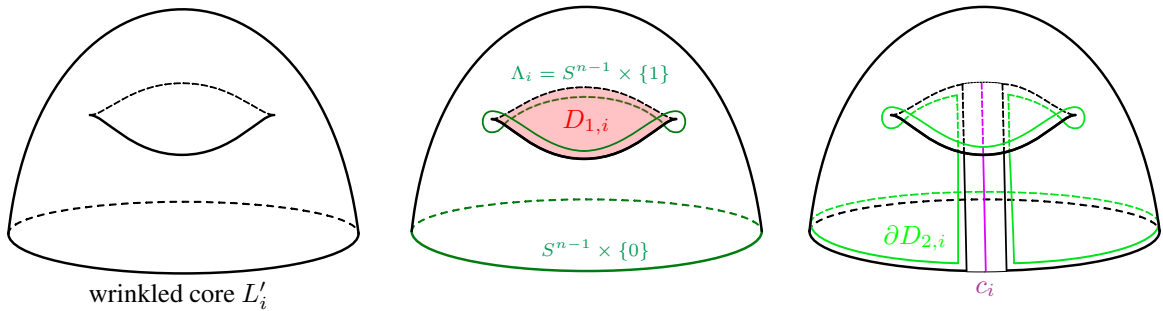

We can then describe the attachment of the wrinkled disk $L_i'$ as follows. First, the attachment of $c_i$ is a contact $1$-handle operation that glues the standard Weinstein ball $\D^*D_{1,i}$ to a neighborhood of $W^{(n-1)}$. Note that this can be viewed as birthing a canceling pair of Weinstein $0$- and $1$-handles. Finally, attaching the rest of $L_i'$ amounts to attaching $D_{2,i}$ along the connected sum $\Lambda_i \# \bdry L_i$, which is contact isotopic to the stabilization of $\bdry L_i$. Thus, the handles $h_i'$ in the statement of \autoref{thm: layer between N and N' general case}(1) are exactly neighborhoods of the disks $D_{2,i}$. This completes the proof of (1). 

The proof of \autoref{thm: layer between N and N' general case}(2) is identical to that of \autoref{theorem: layer between N(L) and N(L')}(2) and is given by \autoref{fig:wrinkled_core_bypass}. Up to the presence of the $1$-handle the bypass computation is the same as in \autoref{fig: Legendrians}.
\begin{figure}[ht]
	\centering
	\begin{overpic}[scale=.34]{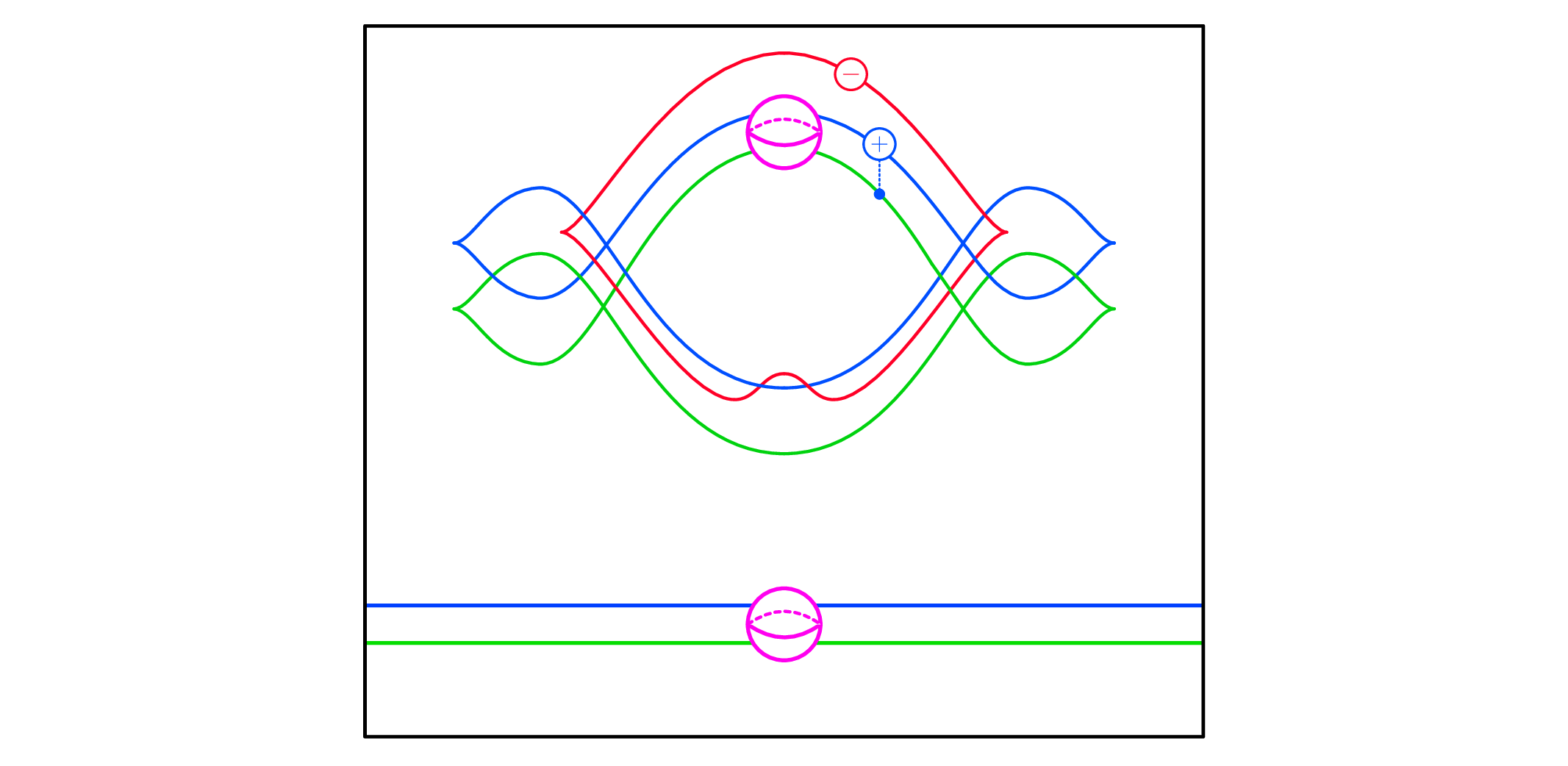}
    \put(71,20){\tiny \cg $(-1)$}
	\end{overpic}
	\vskip-0.15in
	\caption{The bypass data (blue and red) for \autoref{thm: layer between N and N' general case}(2); compare with \autoref{fig: Legendrians}. Note that the pair of purple balls represents a $1$-handle, not an ($n-1$)-handle.}
	\label{fig:wrinkled_core_bypass}
\end{figure}
\end{proof}

\section{Birth-death homotopies}\label{section: proof of theorem: birth-death}

In this section we prove \autoref{theorem: birth-death}. We remind the reader that we have identified buns $H_{i, t_0\pm \epsilon}$ for eventual $\theta$-decompositions $\Theta_{t_0\pm \epsilon}$ such that 
\[
H_{i,t_0 +\epsilon} = H_{i,t_0-\epsilon} \cup N(L_i)
\]
where $L_0$ is a Legendrian approximation of the core of the smooth $n$-handle and $L_1$ is a Legendrian approximation of the cocore of the canceling smooth ($n+1$)-handle; see again \autoref{fig:buns}. To prove \autoref{theorem: birth-death} it remains to exhibit bypass decompositions of the corresponding patties so that the resulting OBDs at $t=t_0 \pm \epsilon$ are stably equivalent.  

Writing $\Sigma_{\pm} := \Sigma_{t_0 \pm \epsilon} = \partial H_{0,t_{0}\pm \epsilon}$, we have the following lemma:

\begin{lemma}\label{lemma: before and after'}
There exist bypass decompositions $\Delta_{t_0\pm \epsilon}$ of the patties $\Sigma_{\pm} \times [0,1]$ such that: 
\begin{enumerate}
    \item[(1)] The contact handlebodies $\tilde{H}_{0,t_0 - \epsilon}$ and $\tilde{H}_{0,t_0 + \epsilon}$ obtained by attaching all of the $n$-handles of the bypasses from $\Delta_{t_0\pm \epsilon}$ to $H_{0,t_0\pm \epsilon}$ are the same. 
    \item[(2)] The contact handlebodies $\tilde{H}_{1,t_0 - \epsilon}$ and $\tilde{H}_{1,t_0 + \epsilon}$ obtained by attaching all of the ($n+1$)-handles of the bypasses (viewed as upside-down $n$-handles) from $\Delta_{t_0\pm \epsilon}$ to $H_{1,t_0\pm \epsilon}$ are the same. 
\end{enumerate}
\end{lemma}

\autoref{theorem: birth-death} follows immediately from the lemma, together with \autoref{theorem: sigma times interval} from Step 3 of \autoref{section: outline of proof of stabilization equivalence}. Indeed, using the lemma we can build $\theta$-decompositions (different from $\Theta_{t_0\pm \epsilon}$) on either side of $t=t_0$ with identical buns; the OBDs are then related by a homotopy of patties and are thus stably equivalent. The rest of the section is dedicated to proving \autoref{lemma: before and after'}. In \autoref{subsection:construction_of_bypass_decompositions} we construct the bypass decompositions $\Delta_{t_0 - \epsilon}$ and $\Delta_{t_0 + \epsilon}$, and then in \autoref{subsection:proof_of_before_and_after} we prove the lemma.  

\subsection{Construction of the bypass decompositions}\label{subsection:construction_of_bypass_decompositions} To construct the bypass decompositions $\Delta_{t_0 \pm \epsilon}$ we introduce a new cast of characters. Let $D_0$ be a smooth $(n+1)$-dimensional half-disk corresponding to the canceling $(n+1)$-handle such that $\bdry D_0= L_0 \cup L_0'$, where $L_0'\subset \bdry H_{0,t_0-\epsilon}$ is not necessarily Legendrian, and let $D_1$ be a smooth $(n+1)$-dimensional half-disk corresponding to the upside down $(n+1)$-handle such that $\bdry D_1= L_1\cup L_1'$, where $L_1'\subset \bdry H_{1,t_0-\epsilon}$ is also not necessarily Legendrian. After isotopy we may take the intersection $D_0\cap D_1$ to be an embedded arc $\gamma: [0,1]\to M$ such that $\gamma(0)\in L_1$, $\gamma(1)\in L_0$, $\gamma$ is isotropic, and $\gamma$ is a sufficiently stabilized Legendrian when $n=1$. See \autoref{fig: birth-death'}.

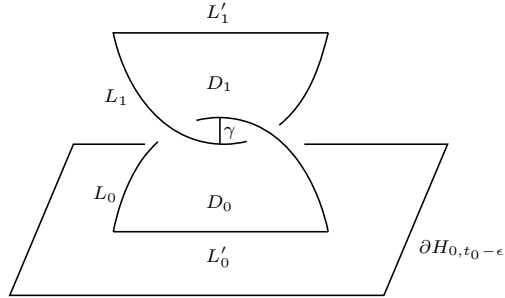
\begin{figure}[ht]
\vskip.15in
	\begin{overpic}[scale=.80]{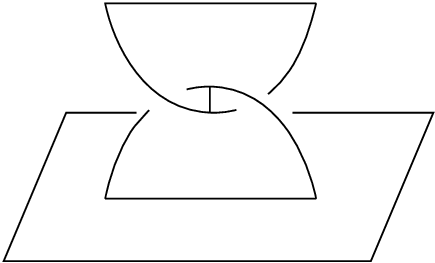}
		\put(19.5,22){\tiny $L_0$}
		\put(45,20){\tiny $D_0$}
		\put(22,44){\tiny $L_1$}
		\put(45,47){\tiny $D_1$}
		\put(49,36.5){\tiny $\gamma$}
		\put(93,10){\tiny $\bdry H_{0,t_0-\epsilon}$}
		\put(45,8){\tiny $L_0'$}
		\put(45,63){\tiny $L_1'$}
	\end{overpic}
	\caption{The ``linked'' Legendrian disks $L_0$ and $L_1$.  Here $\bdry H_{1,t_0-\epsilon}$ is not shown.}
	\label{fig: birth-death'}
\end{figure}

The point is that when constructing the bypass decompositions $\Delta_{t_0- \epsilon}$ and $\Delta_{t_0+ \epsilon}$ we can localize all of the differences between the two to a standard neighborhood of $\gamma$. 

\s\n 
{\em The bypass decomposition of $\Sigma_- \times [0,1]$.} 
\s

We describe a sequence of disjoint embedded hypersurfaces 
\[
\Sigma_{-,0},\dots,\Sigma_{-,6} \subset \Sigma_-\times[0,1]
\]
with $\Sigma_{-,0}=\Sigma_-\times \{0\}$ and $\Sigma_{-,6}=\Sigma_-\times\{1\}$.

For $i=0,1$, let $N(L_i)$ be a standard neighborhood of $L_i$, let $\mathring{L}_i$ be $L_i$ with a small disk centered at $L_i\cap \gamma$ removed, and let $N(\mathring{L}_i)= \pi^{-1}(\mathring{L}_i)$ where $\pi: N(L_i)\to L_i$ is the standard projection. Let $N(\gamma)$ be a standard neighborhood of $\gamma$ such that $\bdry N(\mathring{L}_i)$ and $\bdry N(\gamma)$ intersect only along their boundary and $N(\mathring{L}_i)\cup N(\gamma)$ forms a standard neighborhood of $L_i\cup \gamma$. Finally let $N(D_i)$ be a tubular neighborhood of $D_i$ whose radius is one order of magnitude smaller than those of $N(\gamma)$ and $N(L_i)$. Then:
\be
\item $\Sigma_{-,1}$ is the boundary of $H_{0,t_0-\epsilon}\cup N(\mathring{L}_0)$,
\item $\Sigma_{-,2}$ is the boundary of $H_{0,t_0-\epsilon}\cup N(\mathring{L}_0)\cup (N(D_0)\setminus N(\gamma))$, 
\item $\Sigma_{-,3}$ is the boundary of $H_{0,t_0-\epsilon}\cup N(L_0)\cup N(D_0)$,
\item $\Sigma_{-,4}$ is the (negative) boundary of $H_{1,t_0-\epsilon}\cup  N(\mathring{L}_1)\cup (N(D_1)\setminus N(\gamma))$,
\item $\Sigma_{-,5}$ is the (negative) boundary of $H_{1,t_0-\epsilon}\cup N(\mathring{L}_1)$,
\ee
which we assume have been made disjoint after pushing off in a straightforward manner; see the left side of \autoref{fig: Giroux1'}. 

\begin{figure}[ht]
    \vskip-.05in
	\begin{overpic}[scale=.42]{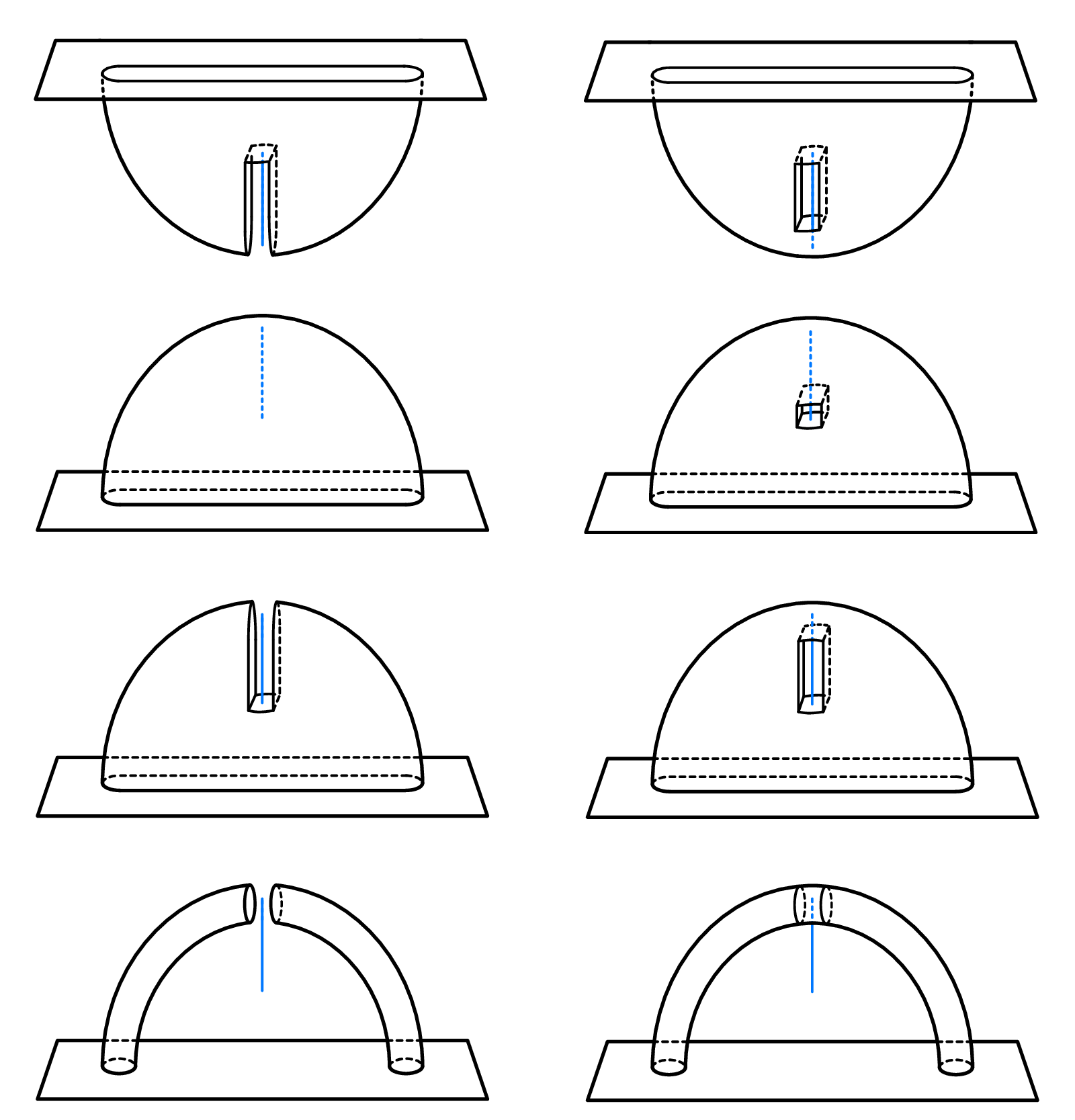}
	\put(-2,8){\small $\Sigma_{-,1}$}
	\put(-2,33.5){\small $\Sigma_{-,2}$}
	\put(-2,59){\small $\Sigma_{-,3}$}
	\put(-2,88.5){\small $\Sigma_{-,4}$}
	\put(91.5,8){\small $\Sigma_{+,1}$}
	\put(91.5,33.5){\small $\Sigma_{+,2}$}
	\put(91.5,59){\small $\Sigma_{+,3}$}
	\put(92.5,88.5){\small $\Sigma_{+,4}$}
	\end{overpic}
	\caption{Some of the hypersurfaces that sweep through $\Sigma_{-}\times [0,1]$ at $t=t_0 - \epsilon$ on the left and the corresponding hypersurfaces for $\Sigma_{+}\times [0,1]$ at $t=t_0 + \epsilon$ on the right. In all figures the light blue arc is $\gamma$.}
	\label{fig: Giroux1'}
\end{figure}
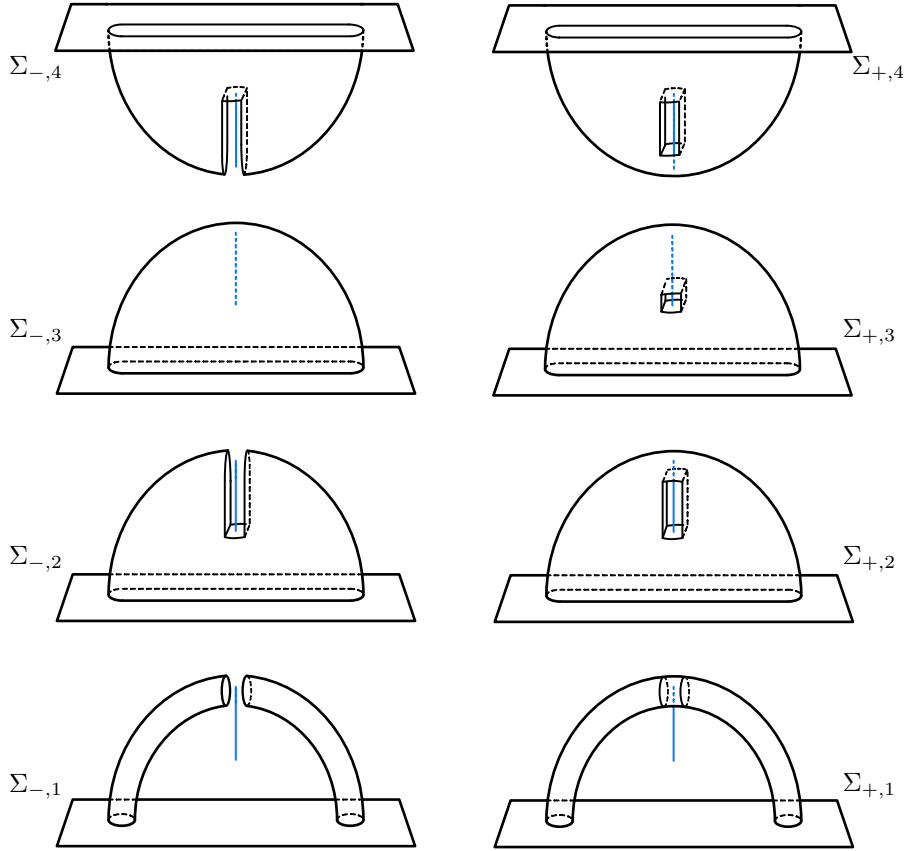

Then, construct a codimension one foliation of $\Sigma_-\times[0,1]$ by hypersurfaces isotopic to $\Sigma_-\times\{i\}$ such that the $\Sigma_{0,j}$, $j=0,\dots,6$, are leaves and consider the corresponding bypass decomposition $\Delta_{t_0-\epsilon}$. Some preliminary observations about this bypass decomposition: 
\begin{enumerate}
    \item[(I)] We may take the leaves of the foliation between $\Sigma_{-,0}$ and $\Sigma_{-,1}$ to all be convex, so that this layer contains no bypass attachments. Indeed, note that $N(\mathring{L}_0)$ can be interpreted as a contact $n$-handle with a small neighborhood of the cocore deleted. Thus, there is a contact vector field that extends in a nonvanishing way from $H_{0,t_0-\epsilon}$ to $H_{0,t_0-\epsilon}\cup N(\mathring{L}_0)$. 
    
    The same can be said of the region between $\Sigma_{-,5}$ and $\Sigma_{-,6}$.
    
    \item[(II)] The region between $\Sigma_{-,1}$ and $\Sigma_{-,2}$ (and likewise between $\Sigma_{-,4}$ and $\Sigma_{-,5}$) may involve bypass attachments, but we do not need to describe them explicitly. 
    
    \item[(III)] The region between $\Sigma_{-,2}$ and $\Sigma_{-,3}$, i.e., the neighborhood of $\gamma$, is given by a sequence of bypass attachments that we will describe explicitly in the next subsection. 
\end{enumerate}

\s\n 
{\em The bypass decomposition of $\Sigma_+ \times [0,1]$.} 
\s

We construct analogs $\Sigma_{+,j}$, $j=1,\dots,5$, of $\Sigma_{-,j}$, where $\Sigma_{+,1}:= \Sigma_+\times\{0\}$ and $\Sigma_{+,5}:=\Sigma_+\times\{1\}$, as follows. Let $N(L_i), N(D_i)$, and $N(\gamma)$ be as before.
\be
\item[(2)] $\Sigma_{+,2}$ is the boundary of $H_{0,t_0+\epsilon}\cup (N(D_0)\setminus N(\gamma))$,
\item[(3)] $\Sigma_{+,3}$ is the boundary of $H_{0,t_0+\epsilon}\cup (N(D_0)\setminus N(L_1))$,%previously was $N(R_m)$ 
\item[(4)] $\Sigma_{+,4}$ is the (negative) boundary of $H_{1,t_0+\epsilon}\cup (N(D_1)\setminus N(\gamma))$.
\ee
See the right side of \autoref{fig: Giroux1'}. 

\subsection{Proof of \autoref{lemma: before and after'}.}\label{subsection:proof_of_before_and_after} The hypersurfaces $\Sigma_{\pm} \times \{i\}$ described above localize everything to a neighborhood of the arc $\gamma$. Thus, we first identify a normal form for the neighborhood of $\gamma$. Here there is a difference in the low-dimensional ($n=1$) and high-dimensional ($n>1$) cases, since $\gamma$ is Legendrian in the former and subcritically isotropic in the latter. 

The low-dimensional case is standard. 

\begin{lemma}[Neighborhood of $\gamma$ when $n=1$.]\label{lemma:normal_form_low}
We may assume that $N(\gamma)$ is a standard neighborhood of $\gamma$ such that $N'(\gamma) := N(\gamma) \cap D_0$ is a convex strip with Legendrian boundary subdivided by Legendrian arcs $\delta_1, \dots, \delta_{m-1}$ into convex disks $R_1, \dots, R_m$ with $\mathrm{tb}(\partial R_i) = -1$, arranged in order so that $R_1$ is adjacent to $L_0$ and $R_m$ contains $L_1 \cap D_0$. See \autoref{fig:newstrip}.  Moreover, $\Sigma_{-,2}$ can be adjusted so that $\Sigma_{-,2}\cap N'(\gamma)=\bdry N'(\gamma)\setminus \op{int} L_0$ and the intersection is transverse.
\end{lemma}

\begin{figure}[ht]
	\begin{overpic}[scale=0.38]{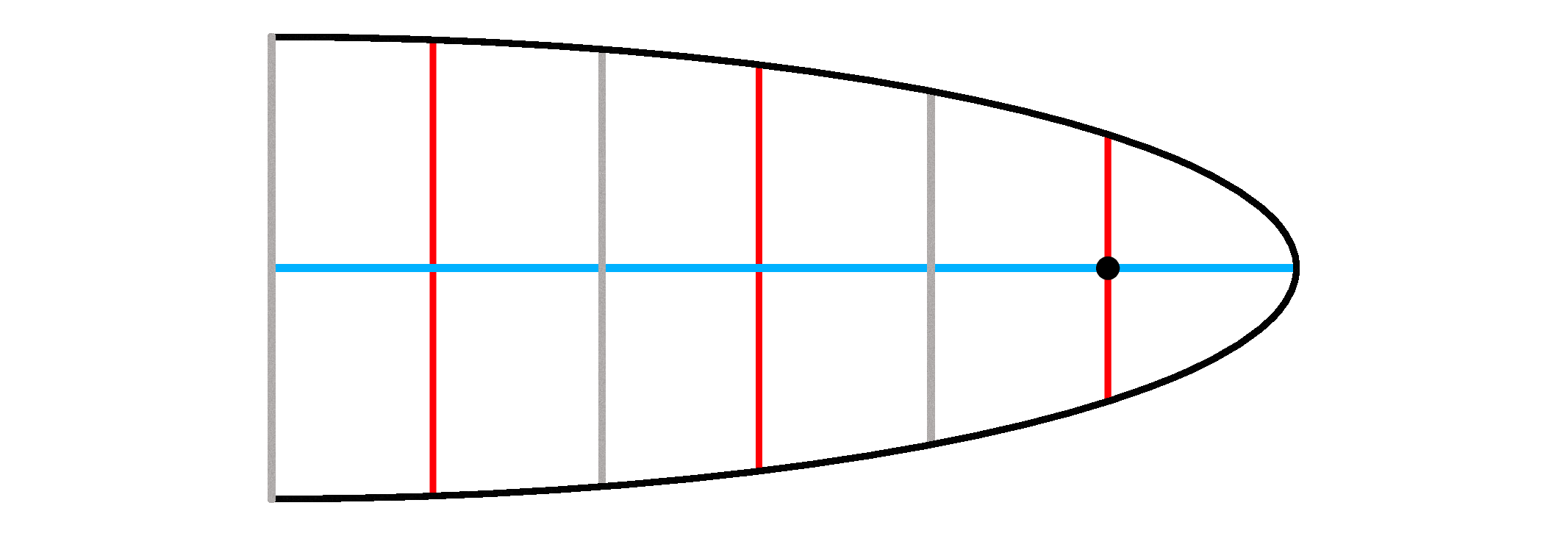}
		\put(30,19){\small \textcolor{Cerulean}{$\gamma$}}
		\put(13.5,25){\small \textcolor{gray}{$L_0$}}
		\put(72,14){\small $L_1$}
		\put(26,33){\small $R_1$}
		\put(46.5,31.5){\small $R_2$}
		\put(69.5,27){\small $R_3$}
		\put(35.5,8){\small \textcolor{gray}{$\delta_1$}}
		\put(56,10){\small \textcolor{gray}{$\delta_2$}}
	\end{overpic}
	\caption{The intersection $N'(\gamma)=D_0\cap N(\gamma)$ in the $3$-dimensional case with $m=3$. The red arcs are the dividing curves, the black dot is the intersection $L_1\cap D_0$, and the black arc is $\Sigma_{-,2}\cap N'(\gamma)$.}
	\label{fig:newstrip}
\end{figure}
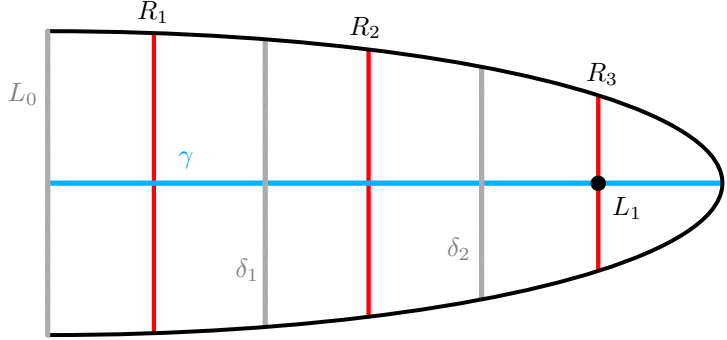

To state the higher-dimensional version we introduce some notation. We use sub- and super-scripts  $\mathrm{dim}\, 3$, for example as in $N_{\mathrm{dim} \,3}(\gamma)$, $N'_{\mathrm{dim}\, 3}(\gamma)$, and $L_1^{\mathrm{dim} \,3}$, to indicate the low-dimensional objects described in the $n=1$ case by \autoref{lemma:normal_form_low}. Recall that the isotropic arc $\gamma:[0,1] \to M$ parametrizing $D_0\cap D_1$ satisfies $\gamma(0) \in L_1$ and $\gamma(1)\in L_0$. 

\begin{lemma}[Neighborhood of $\gamma$ when $n>1$.]\label{lemma: normalization for n greater than 1'}
After an isotopy of $D_0$ relative to $\bdry D_0$, there is a neighborhood $N(\gamma)\subset (M,\xi)$ of the form $N_{\mathrm{dim}\, 3}(\gamma)\times U \subset N_{\mathrm{dim}\, 3}(\gamma)\times T^*\R^{n-1}$,
where $U$ is a tubular neighborhood of the zero section $\R^{n-1}_0$ in $T^*\R^{n-1}$,
such that:
\be
\item[(i)] the contact structure on $N_{\mathrm{dim} \,3}(\gamma)\times T^*\R^{n-1}$ is $\op{ker}(\alpha+\beta)$, where $\alpha$ is a contact form on $N_{\mathrm{dim}\, 3}(\gamma)$ and $\beta$ is the standard Liouville form on $T^*\R^{n-1}$;
\item[(ii)] $N'(\gamma):=N(\gamma)\cap D_0=N'_{\mathrm{dim}\, 3}(\gamma)\times \R^{n-1}_0$;
\item[(iii)] $T_{\gamma(0)}L_1$ is $T_{\gamma(0)}L_1^{\mathrm{dim}\, 3}$ times a fiber of $T^*\R^{n-1}$;
 and 
\item[(iv)] $\op{tb}(\bdry N'_{\mathrm{dim} \,3}(\gamma))=-1$.
\ee
\end{lemma}

\begin{proof}
We explain the $n=2$ case, leaving the $n>2$ case to the reader. We abuse notation and isotop $D_0$ and $\gamma$ relative to their boundaries in stages without changing their names. Let $v_t:= \dot\gamma(t)$. Also let $\perp$ denote the symplectic orthogonal complement with respect to the symplectic structure $\omega_\xi$ on $\xi$ and $J$ an almost complex structure on $\xi$ compatible with $\omega_\xi$. We also write $R$ for any tangent vector on $M$ transverse to $\xi$ (not just the Reeb vector field). 

Since every point in $(M,\xi)$ has a standard neighborhood by the Pfaff-Darboux theorem, there exist an isotopy of $\gamma$ rel boundary and a neighborhood of the resulting $\gamma([1-\epsilon,1])$ with $\epsilon>0$ small on which (i) and (ii) hold. 

Next we isotop $\gamma$ rel boundary so that $T_{\gamma(0)}L_1=\R\langle Jv_0, Jw_0\rangle$, where $v_0=\dot\gamma(0)$ and $w_0\in \R\langle v_0, Jv_0\rangle^\perp$. We further isotop $D_0$ rel boundary so that $T_{\gamma(0)}D_0=\R\langle R,v_0,w_0\rangle$; this is possible since the original $T_{\gamma(0)}D_0$ and $\R\langle R,v_0,w_0\rangle$ can be connected by a $1$-parameter family of $3$-planes transverse to $T_{\gamma(0)}L_1$ and containing $v_0$. We then take a neighborhood of $\gamma([0,\epsilon])$ on which (i)--(iii) hold; here $R,v_0,Jv_0$ span the $\R^3$-direction and $w_0,Jw_0$ span the $T^*\R^1$-direction.

Finally we consider the neighborhood of $\gamma([\epsilon,1-\epsilon])$. It is easy to perturb $D_0$ such that, for all $t\in [\epsilon, 1-\epsilon]$, $T_{\gamma(t)}D_0=\R\langle R,v_t, w_t\rangle$, $\xi(\gamma(t))=\R\langle v_t,Jv_t\rangle \oplus \R\langle v_t,Jv_t\rangle^\perp$, and $w_t\in \R\langle Jv_t\rangle \oplus \R\langle v_t,Jv_t\rangle^\perp$. We then homotop $w_t$, $t\in[\epsilon, 1-\epsilon]$, inside $(\R\langle Jv_t\rangle \oplus \R\langle v_t,Jv_t\rangle^\perp) \setminus \{0\}$ (which is homotopic to $S^2$) relative to its endpoints so that $w_t$ agrees with the desired model.  Once we have aligned $N(\gamma)$ with the desired model on the linear level, the lemma follows.
\end{proof}

When $n=1$, note that each convex disk $R_1, \dots, R_m$ in \autoref{lemma:normal_form_low} is a bypass half-disk. When $n>1$, (iv) in \autoref{lemma: normalization for n greater than 1'} means that we may treat the latter as an analogue of the former but always with $m=1$. In particular, we have: 

\begin{claim}\label{claim:geometric_bypass}
When $n>1$, $N(\gamma)$ can be interpreted as a single bypass.
\end{claim}

We can see this as follows. First, consider the $n=1$ case with $m=1$, so that $N_{\op{dim}\, 3}'(\gamma) = R_m$. Observe that:
\be
\item $\bdry N_{\op{dim}\, 3}'(\gamma)$ is the union of Legendrian arcs $L':= L_0^{\op{dim}\, 3}\cap N_{\op{dim}\, 3}'(\gamma)$ and $L''\subset \bdry \Sigma_{-,2}^{\op{dim} \, 3}$, in gray and black respectively in \autoref{fig:geometric_bypass}, that intersect at two points, and $L''$ in turn is the union of two Legendrian arcs $a_{+1}$ and $a_{-1}$ that meet at a point $p$;
\item $\op{int}(N_{\op{dim}\, 3}'(\gamma))$ is foliated by a $1$-parameter family of Legendrian arcs $a_t$, $t\in[-1,+1]$, from $a_{-1}$ to $a_{+1}$ that start at $p$ and end on $L'$ such that $a_{0}$ intersects $L_1^{\op{dim}\, 3}$. 
\ee
We thus have that $N_{\op{dim}\, 3}'(\gamma)$ is a bypass half-disk with attaching arc $L'' = a_{-1}\cup a_{+1}$. The arcs $a_{\pm 1}$ are the positive and negative attaching data, respectively, and the gray arc $L'$ is the core of the ($n=1$)-handle for the bypass. A standard model for $N_{\op{dim}\, 3}'(\gamma)$ in the front projection is a standard $1$-dimensional Legendrian unknot with two cusps $p,q$, where:
\be
\item the cusp $q$ is at the center of the Legendrian arc $L'$ and 
\item the cusp $p$ is the initial point of the family $a_t$, $t\in[-1,+1]$.
\ee 
See the left side of \autoref{fig:geometric_bypass}. 

\begin{figure}[ht]
\vskip-0.45in
	\begin{overpic}[scale=0.43]{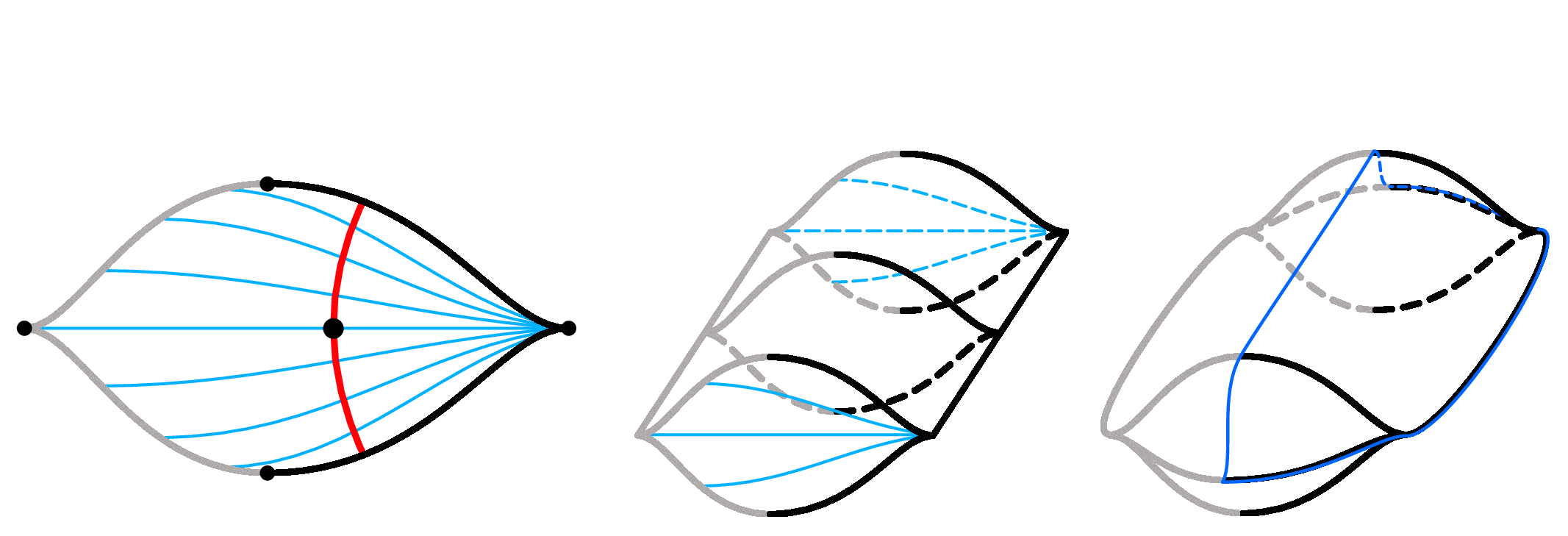}
		\put(-1,16){\small \textcolor{gray}{$L'$}}
		\put(98,14){\small \textcolor{blue}{$D_+$}}
		\put(35.5,11){\small $p$}
        \put(1,11){\small $q$}
        \put(27,20.5){\small $a_{+1}$}
        \put(27,5){\small $a_{-1}$}
        \put(22,14){\tiny $L_1^{\op{dim}\ 3}$}
        \put(13,24.75){\small $N_{\op{dim}\, 3}'(\gamma)$}
        \put(55.25,26.25){\small $N'(\gamma)$}
        \put(85.25,26.25){\small $\tilde{N}'(\gamma)$}
	\end{overpic}
	\caption{The left disk is a front projection model of a bypass half-disk when $n=1$; the colors agree with \autoref{fig:newstrip}. The middle and right diagrams are front projections when $n=2$. On the far right, the boundary of the positive data $D_+$ of the bypass is outlined in blue.}
	\label{fig:geometric_bypass}
\end{figure}
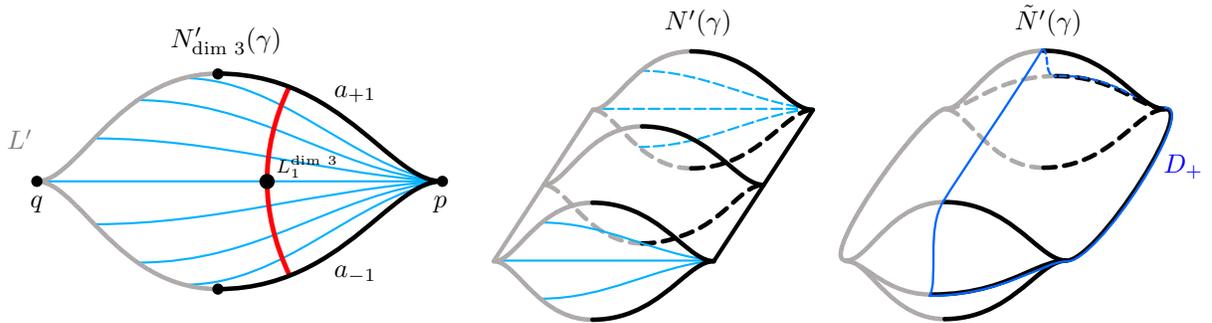

By 
\autoref{lemma: normalization for n greater than 1'}, when $n>1$, in the front projection $N'(\gamma) = N(\gamma)\cap D_0$ has the form 
\[
N_{\op{dim}\, 3}'(\gamma) \times D^{n-1}.
\]
See the middle of \autoref{fig:geometric_bypass}. The claim now is that there is an extension of $N'(\gamma)$ to $\tilde{N}'(\gamma)$ (i.e., a capping off) so that in the front projection $\tilde{N}'(\gamma)$ is the standard $n$-dimensional Legendrian unknot with one cusp edge $\simeq S^{n-1}$, together with a foliated (on $\op{int}(\tilde{N}'(\gamma))$) $1$-parameter family of Legendrian $n$-disks with boundary $S^{n-1}$ and going from a retraction $D_-$ of the bottom sheet to a retraction $D_+$ of the top sheet; see the right side of \autoref{fig:geometric_bypass}. This is not completely trivial as there is a ``size issue": we need enough space in the cotangent direction of $T^* D^{n-1}$ in order to damp out $N'(\gamma)$ to $\tilde{N}'(\gamma)$. However, $N_{\op{dim}\, 3}'(\gamma)$ can be made small by retracting to the disk $\cup_{t\in[-\epsilon,+\epsilon]} a_t$, while preserving the size in the cotangent direction of $T^*D^{n-1}$.  We can then interpret the retracted $\tilde{N}'(\gamma)$ as a bypass with positive and negative data $D_{\pm}$. Note that this involves initially choosing $\Sigma_{-,2}$ with some more care.

Now we prove our main lemma. 

\begin{proof}[Proof of \autoref{lemma: before and after'}.]
We will prove (1). By the symmetry of the bypass decompositions $\Delta_{t_0\pm \epsilon}$ of $\Sigma_{\pm}\times [0,1]$, (2) is given by the exact same (upside-down) argument. We write $\tilde{H}_{\pm, i}$ for the contact handlebodies obtained by attaching all of the $n$-handles of the bypasses below $\Sigma_{\pm, i}$ to $H_{0,t_0\pm \epsilon}$. Thus our goal is to show that the contact handlebodies $\tilde{H}_{-, 5}$ and $\tilde{H}_{+, 5}$ are the same. 

\s\n
{\em Step 1: The layers between $\Sigma_{\pm, 1}$ and $\Sigma_{\pm, 2}$.} 
\s

Note that $\Sigma_{-,1}$ and $\Sigma_{+,1}$ only differ in a small neighborhood $N(L_0) \cap N(\gamma)$. So given a sequence of bypasses inductively attached to $\Sigma_{-,1}$ (by which we mean that the first bypass $B_1$ is attached to $\Sigma_{-,1}$ but the next one $B_2$ is attached to the boundary of a thickening of $\Sigma_{-,1}\cup B_1$, not necessarily to $\Sigma_{-,1}$) in the layer from $\Sigma_{-,1}$ to $\Sigma_{-,2}$, we can inductively attach the exact same bypasses to $\Sigma_{+,1}$ to obtain $\Sigma_{+,2}$. 

\s\n
{\em Step 2: The layers between $\Sigma_{\pm, 2}$ and $\Sigma_{\pm, 3}$.}
\s 

In the case $n=1$ we refer to \autoref{lemma:normal_form_low}. As remarked above, each of the convex disks $R_1, \dots, R_m$ are bypass half-disks. Therefore the layer between $\Sigma_{-,2}$ and $\Sigma_{-,3}$ is given by attaching the bypasses corresponding to $R_m, R_{m-1}, \dots, R_1$, in that order. To the corresponding $\Sigma_{+, 2}$ we attach $R_1, \dots, R_{m-1}$ in that (reversed) order. The arcs $\delta_1,\dots,\delta_{m-1}$ appear in both $\tilde{H}_{-,3}$ and $\tilde{H}_{+,3}$ even though the order of attachment of the bypasses was reversed. 

The case $n>1$ is simpler because of \autoref{lemma: normalization for n greater than 1'} and \autoref{claim:geometric_bypass}. To go from $\Sigma_{-,2}$ to $\Sigma_{-,3}$, the claim tells us that we fill in the rest of $N(\gamma)$ with a single bypass attachment, which we also refer to as $R_m$ (with $m=1$) for consistency. In particular, the $n$-handle of the bypass attachment is a small standard neighborhood of $L_0$ near $\gamma(1)$ and the ($n+1$)-handle of the bypass attachment is a small standard neighborhood of $L_1$ near $\gamma(0)$. Likewise, the corresponding layer between $\Sigma_{+,2}$ and $\Sigma_{+,3}$ does not involve any bypasses because the aforementioned $n$- and ($n+1$)-handles are, by construction, already distributed onto the upper and lower handlebodies. 

\s\n
{\em Step 3: The handlebodies $\tilde{H}_{\pm,3}$.}
\s 

Before considering the rest of the layers we make an observation. At this point, the handlebodies $\tilde{H}_{-,3}$ and $\tilde{H}_{+,3}$, i.e., those obtained by attaching all of the $n$-handles from Step 1 and Step 2 above to $H_{0,t_0-\epsilon}$ and $H_{0,t_0+\epsilon}$, are the same. Indeed, all of the bypasses up through $\Sigma_{\pm, 2}$ and all of the bypasses $R_1, \dots, R_{m-1}$ are the same. Finally, the $n$-handle from the bypass $R_m$ in $\tilde{H}_{-,3}$, the only one as of yet not accounted for, is present in $\tilde{H}_{+,3}$ from the beginning by construction of the Legendrian arch that forms the lower bun of $\Theta_{t_0 + \epsilon}$. 

\s\n
{\em Step 4: The rest of the layers.} 
\s

By Step 3 it suffices to prove that the rest of the bypasses inductively attached to $\Sigma_{-,3}$ to reach $\Sigma_{-,5}$ can likewise be inductively attached to $\Sigma_{+,3}$ to reach $\Sigma_{+,5}$. These are just upside-down versions of Step 1 and Step 2. 

This completes the proof of the lemma. 
\end{proof}

\section{Weinstein Lefschetz Fibrations}\label{section: proofs of Lefschetz}

In this section we prove Corollaries~\ref{cor: proof of GP} and \ref{cor: conj of GP}, which we restate below as \autoref{theorem:new_WLF'} for the convenience of the reader. First we recall the relevant language from \autoref{sec:intro}. An abstract WLF, henceforth abbreviated AWLF, is the data 
\[
((W^{\flat}, \lambda^{\flat}, \phi^{\flat});\, \mathcal{L} = (L_1, \dots, L_m)),
\]
abbreviated $(W^{\flat}; \mathcal{L})$, where $(W^{\flat}, \lambda^{\flat}, \phi^{\flat})$ is a ($2n-2$)-dimensional Weinstein domain and $L_1, \dots, L_m \subset W^{\flat}$ are exact parametrized Lagrangian spheres. We say that two AWLFs are {\em move equivalent} if they are related to a common AWLF by sequences of the four moves (deformation, cyclic permutation, Hurwitz move, and stabilization) described in \autoref{sec:intro}. A consecutive sequence of stabilization moves will be called a {\em multistabilization}. We will work strictly with AWLFs and refer to \cite{GP17} for more details on how to pass between AWLFs and WLFs in the sense of \autoref{def: WLF}. 

Corollaries~\ref{cor: proof of GP} and \ref{cor: conj of GP} can be rephrased as follows:

\begin{theorem}[Weinstein-AWLF correspondence]\label{theorem:new_WLF'}
$\mbox{}$
    \begin{enumerate}
        \item Let $(W, \lambda, \phi)$ be a Weinstein domain. Then there is an AWLF $(W^{\flat}; \mathcal{L})$ such that the total space $|W^{\flat}; \mathcal{L}|$ is $1$-Weinstein homotopic to $(W, \lambda, \phi)$.
        
        \item Let $(W^{\flat}_i; \mathcal{L}_i)$, $i=0,1$, be AWLFs. Then the total spaces $|W^{\flat}_i; \mathcal{L}_i|$, $i=0,1$, are $1$-Weinstein homotopic if and only if $(W^{\flat}_i; \mathcal{L}_i)$, $i=0,1$, are move equivalent. 
    \end{enumerate}
    \end{theorem}

The notion of the total space of an AWLF, i.e., the association of a Weinstein structure to an AWLF, will be reviewed in \autoref{subsec:AWLFtoWeinstein'}. In \autoref{subsec:WeinsteintoAWLF'} we discuss the other direction of associating an AWLF to a Weinstein structure. In particular, we prove \autoref{theorem:new_WLF'}(1), i.e., \autoref{cor: proof of GP}, and discuss some subtleties. Finally in \autoref{subsec:move_eq'} we prove \autoref{theorem:new_WLF'}(2), i.e., \autoref{cor: conj of GP}.

\subsection{From AWLF to Weinstein}\label{subsec:AWLFtoWeinstein'}

Given an AWLF $(W^{\flat}; \mathcal{L})$, its {\em total space}, denoted $|W^{\flat}; \mathcal{L}|$, is a Weinstein domain obtained by:
\begin{itemize}
    \item[(1)] first forming the $2n$-dimensional Weinstein structure 
    \[
    (W^{\flat} \times D^2, \lambda^{\flat} + \lambda_{\mathrm{st}}, \phi^{\flat} + \phi_{\mathrm{st}}),
    \]
    where $(D^2, \lambda_{\mathrm{st}}, \phi_{\mathrm{st}})$ is the standard Weinstein structure on the disk with radial Liouville vector field, and then
    \item[(2)] attaching critical Weinstein handles along 
    \[
    \tilde{\mathcal{L}}: = \tilde{L}_1 \cup\cdots\cup \tilde{L}_m \subset W^{\flat} \times \bdry D^2,
    \]
    where $\tilde{L}_j$ is a Legendrian lift of $L_j$ that projects to a small interval centered at $\frac{2\pi j}{m}\in S^1 = \bdry D^2$.
\end{itemize}
By rescaling the Weinstein structure on $D^2$ if necessary we further require each of the small intervals in (2) to be pairwise disjoint.  Also, the precise locations $\tfrac{2\pi j}{m}\in S^1$ are not important, just their ordering around $S^1$.  We refer to \cite[Definition 6.3]{GP17} for a more detailed description of the total space of an AWLF. 

For convenience we adopt the following language:

\begin{defn}\label{def:weinstein_homotopies'}
    Let $(W^{2n}, \lambda_t, \phi_t)$, $t\in [0,1]$, be a $1$-Weinstein homotopy. We say that the homotopy is 
    \begin{itemize}
        \item {\em critically strict} if there are no handleslides of index $n$ handles and no birth-death type points of index $(n-1,n)$;
        
        \item a {\em critical handleslide homotopy} if at $t=\frac{1}{2}$ there is a single handleslide of $n$-handles, and is otherwise critically strict; and 
        \item a {\em critical birth-death homotopy} at $t=\frac{1}{2}$ if there is a single birth-death type point of index $(n-1,n)$, and is otherwise critically strict.
    \end{itemize}
    In this section we drop the ``critical'' adjective and will refer to homotopies of the above type as \textit{strict, handleslide}, and \textit{birth-death} homotopies, respectively. 
\end{defn}

\begin{remark}
Let $(W^{\flat}; \mathcal{L})$ be an AWLF. The total space $|W^{\flat}; \mathcal{L}|$ is well-defined up to strict Weinstein homotopy supported away from a neighborhood of the ($n-1$)-skeleton. That is, the total space is well-defined up to isotoping (but not handlesliding) the attaching spheres that lift the Lagrangian vanishing cycles.
\end{remark}

\subsection{From Weinstein to AWLF}\label{subsec:WeinsteintoAWLF'}

Here we prove \autoref{theorem:new_WLF'}(1), i.e., \autoref{cor: proof of GP}. Afterwards, we introduce a new definition (\autoref{def:CAWLF}) motivated by the proof and discuss some of its subtleties. We remind the reader that ``OBD'' always means ``strongly Weinstein OBD.''

\begin{proof}[Proof of \autoref{theorem:new_WLF'}(1).]
Let $(W^{2n}, \lambda, \phi)$ be a Weinstein domain. By \cite{Cie02}, there is a neighborhood $W^{(n-1)}$ of the ($n-1$)-skeleton of $(W, \lambda, \phi)$ such that $(W^{(n-1)}, \lambda, \phi)$ is $1$-Weinstein homotopic to 
\[
(W^{\flat} \times D^2, \lambda^{\flat} + \lambda^{\mathrm{st}}, \phi^{\flat} + \phi^{\mathrm{st}}),
\]
where $(W^{\flat}, \lambda^{\flat}, \phi^{\flat})$ is a Weinstein domain of dimension $2n-2$. Note that this $1$-Weinstein homotopy will in general not be strict when $n=3$; see \cite[p.\ 292]{CE12}. 

Next we identify three OBDs of $\bdry(W^{\flat} \times D^2)$.
\begin{itemize}
    \item First, $\bdry(W^{\flat} \times D^2)$ admits the trivial compatible OBD $(B, \pi)$ with page $W^{\flat}$ and identity monodromy. This is the OBD induced by the AWLF $(W^{\flat}; \varnothing)$.

    \item Secondly, let $\Lambda=\Lambda_1\cup\dots\cup \Lambda_m\subset \bdry W^{(n-1)}$ be the collection of attaching Legendrian $(n-1)$-spheres with respect to $\phi$ and let $\tilde{\Lambda} = \tilde{\Lambda}_1\cup \cdots \cup \tilde{\Lambda}_m \subset \bdry(W^{\flat} \times D^2)$ denote its image under the $1$-Weinstein homotopy furnished by \cite{Cie02}. By \cite[Corollary 1.3.3]{HH19}, there exists an OBD $(B',\pi')$ of $\bdry(W^{\flat} \times D^2)$ such that $\tilde{\Lambda}$ is contained in a page.

    \item Finally, by \autoref{theorem: stabilization equivalence}, there exist sequences of stabilizations applied to $(B,\pi)$ and $(B',\pi')$ such that they become strongly Weinstein homotopic to a common $(B'',\pi'')$. 
\end{itemize}
Observe that each OBD stabilization of the trivial OBD $(B,\pi)$ of $\bdry(W^{\flat} \times D^2)$ is induced by an AWLF stabilization of $(W^{\flat}; \varnothing)$. Thus there is an AWLF $(W^{\natural}; \mathcal{L}')$ such that 
\begin{itemize}
    \item $(W^{\natural}; \mathcal{L}')$ is an AWLF multistabilization of $(W^{\flat}; \varnothing)$ and hence $|W^{\natural}; \mathcal{L}'|$ is $1$-Weinstein homotopic to $W^{\flat} \times D^2$, and 
    \item the OBD of $\bdry(W^{\flat} \times D^2)$ induced by $(W^{\natural}; \mathcal{L}')$ is strongly Weinstein homotopic to $(B'', \pi'')$. 
\end{itemize}
Using this strong Weinstein homotopy of OBDs we may Legendrian isotop $\tilde{\Lambda}$ to
$$\tilde{\mathcal{L}}'' = \tilde{L}_1'' \cup \cdots \cup \tilde{L}_m''\subset W^{\flat}\times \bdry D^2$$ 
such that $\tilde{\mathcal{L}}''$ lies on a page of the OBD induced by $(W^{\natural}; \mathcal{L}')$. Then $\tilde{\mathcal{L}}''$ projects to a set of Lagrangian vanishing cycles $\mathcal{L}'' = (L_1'', \dots, L_m'') \subset W^{\natural}$ and the AWLF
\[
(W^{\natural}; \mathcal{L} := \mathcal{L}' \cup \mathcal{L}''),
\]
where $\mathcal{L}' \cup \mathcal{L}''$ is the ordered union of all vanishing cycles in $\mathcal{L}'$ and $\mathcal{L}''$, has a total space which is $1$-Weinstein homotopic to $(W, \lambda, \phi)$.
\end{proof}

As described in \autoref{subsec:AWLFtoWeinstein'}, given an AWLF one builds a Weinstein domain which is well-defined up to strict Weinstein homotopy. In the opposite direction, the association of an AWLF to a given Weinstein structure $(W, \lambda, \phi)$ is more delicate because the total space of the AWLF $(W^{\natural}; \mathcal{L})$ produced by the proof of \autoref{theorem:new_WLF'}(1) is in general not strictly Weinstein homotopic to $(W, \lambda, \phi)$. In order to navigate the proof of \autoref{theorem:new_WLF'}(2) carefully, we need to be more precise about the type of AWLF produced by the proof of \autoref{theorem:new_WLF'}(1) as it relates to the input Weinstein structure. This motivates \autoref{def:CAWLF}.

\begin{defn}\label{def:CAWLF}
Let $(W^{2n}, \lambda, \phi)$ be a Weinstein domain. An AWLF $(W^{\natural}; \mathcal{L})$ is {\em compatible with $(W, \lambda, \phi)$} if there is an ordered decomposition $\mathcal{L} = \mathcal{L}' \cup \mathcal{L}''$ such that:
\begin{enumerate}
    \item[(a)] the total space $|W^{\natural}; \mathcal{L}'|$ is $1$-Weinstein homotopic to $(W^{(n-1)}, \lambda, \phi)$, where $W^{(n-1)}$ is a neighborhood of the ($n-1$)-skeleton of $(W, \lambda, \phi)$; 
    \item[(b)] the $1$-Weinstein homotopy in (a) extends via a strict Weinstein homotopy to a $1$-Weinstein homotopy from $|W^{\natural}; \mathcal{L}|$ to $(W, \lambda, \phi)$; and 
    \item[(c)] under the homotopy in (b), the Legendrian lifts $\tilde{\mathcal{L}}'' \subset W^{\flat} \times \bdry D^2$ of the vanishing cycles $\mathcal{L}''$ are identified with the attaching spheres of the $n$-handles of $\phi$ on $\partial W^{(n-1)}$.  
\end{enumerate}
We use the notation $(W^{\natural}; \mathcal{L}', \mathcal{L}'')$ to indicate a preferred ordered decomposition $(W^{\natural}; \mathcal{L} = \mathcal{L}' \cup \mathcal{L}'')$, and we say that $(W^{\natural}; \mathcal{L}', \mathcal{L}'')$ is {\em compatible with $(W,\lambda,\phi)$} if $(W^{\natural}; \mathcal{L})$ is compatible with $(W,\lambda,\phi)$ via the indicated preferred decomposition. 
\end{defn}

Less precisely, an AWLF $(W^{\natural}; \mathcal{L}', \mathcal{L}'')$ is compatible with $(W, \lambda, \phi)$ if $|W^{\natural}; \mathcal{L}'|$ gives a (stabilized) neighborhood of the ($n-1$)-skeleton and if $\mathcal{L}''$ lifts to the $n$-handles of $\phi$. See \autoref{fig:CAWLF}.

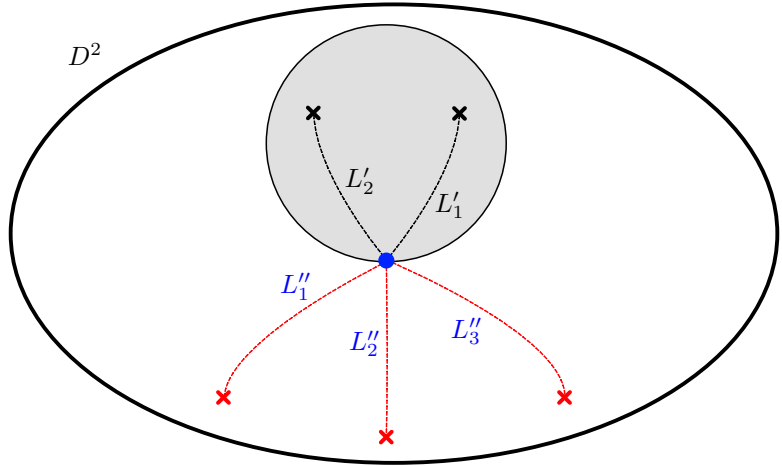
\begin{figure}[ht]
	\begin{overpic}[scale=.3]{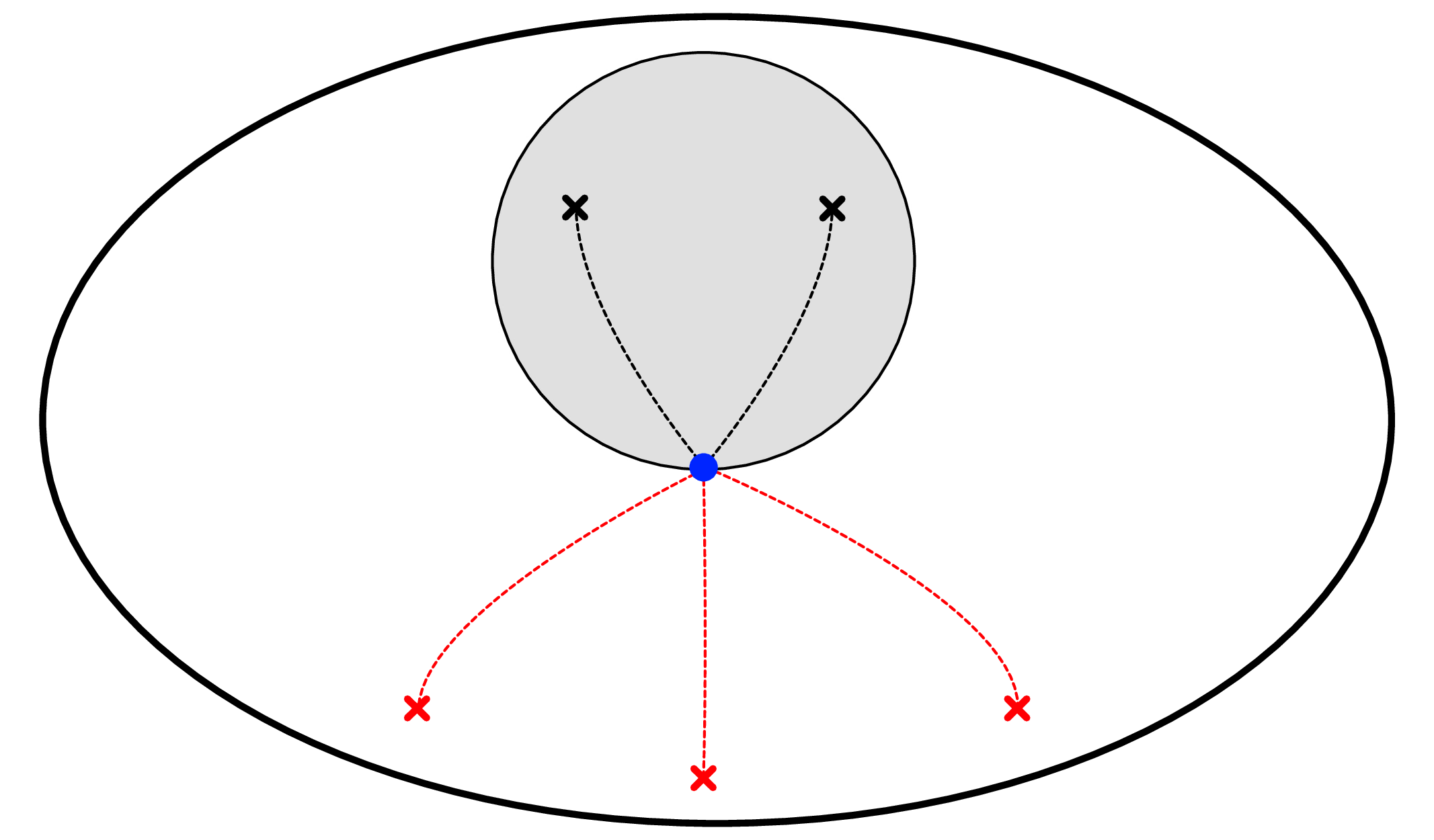}
	\put(55,32){\small $L_1'$}
    \put(44,35){\small $L_2'$}
    \put(36,22){\small {\cbu  $L_1''$}}
    \put(44.5,15){\small {\cbu  $L_2''$}}
    \put(57,16.5){\small {\cbu  $L_3''$}}
    \put(10,50){\small $D^2$}
	\end{overpic}
	\caption{A heuristic picture of an AWLF $\left(W^{\natural};\, \mathcal{L}' = (L_1', L_2'),\, \mathcal{L}'' = (L_1'', L_2'', L_3'')\right)$ compatible with a Weinstein structure $(W, \lambda, \phi)$ where $\phi$ has three $n$-handles. More precisely, the figure depicts a WLF $p:W \to D^2$ induced by the AWLF. The blue dot is the image of the distinguished fiber $W^{\natural}$, the black singular values and Lagrangian thimbles in the smaller gray disk correspond to the bi-stabilization $(W^{\natural}; \mathcal{L}')$, and the red singular values and thimbles correspond to the cycles $(L_1'', L_2'', L_3'')$ that lift to the attaching spheres of the three $n$-handles of $\phi$.}
	\label{fig:CAWLF}
\end{figure}

\begin{remark}\label{remark:CAWLFs}
Here we collect some observations about compatibility. 
\begin{enumerate}
    \item Let $(W^{\flat}; \mathcal{L})$ be an AWLF. Then $(W^{\flat}; \varnothing, \mathcal{L})$ is compatible with $|W^{\flat}; \mathcal{L}|$. 

    \item Let $(W^{\natural}; \mathcal{L})$ be an AWLF that admits an ordered decomposition $\mathcal{L} = \mathcal{L}' \cup \mathcal{L}''$ such that $(W^{\natural}; \mathcal{L}')$ is a multistabilization of $(W^{\flat}; \varnothing)$. If $\mathcal{L}'$ is nonempty, then $(W^{\natural}; \mathcal{L}', \mathcal{L}'')$ is \textit{not} compatible with $|W^{\natural}; \mathcal{L}|$, even though $(W^{\natural}; \mathcal{L}')$ is a multistabilization. Indeed, the $n$-handles of the total space $|W^{\natural}; \mathcal{L}|$, by definition, are attached along Legendrian lifts of all of the vanishing cycles in $\mathcal{L} = \mathcal{L}' \cup\mathcal{L}''$. However, compatibility with a prescribed Weinstein structure requires the $n$-handles of the latter to correspond to $\mathcal{L}''$. 
    
    \item The notion of compatibility via a fixed decomposition $\mathcal{L} = \mathcal{L}' \cup \mathcal{L}''$ is not necessarily preserved under $1$-Weinstein homotopy. Indeed, let $(W^{\natural}; \mathcal{L}', \mathcal{L}'')$ be compatible with $(W, \lambda, \phi)$. Let $(W, \lambda_t, \phi_t)$, $t\in [0,1]$, be either a handleslide or birth-death homotopy as in \autoref{def:weinstein_homotopies'} with $(\lambda_0, \phi_0) = (\lambda, \phi)$. Then $(W^{\natural}; \mathcal{L}', \mathcal{L}'')$ is \textit{not} compatible with  $(W, \lambda_1, \phi_1)$. However, the compatibility of $(W^{\natural}; \mathcal{L})$ may be preserved with a different choice of ordered decomposition of $\mathcal{L}$; see the proof of \autoref{lemma:move_eq_bd_hslide'} below. 
\end{enumerate}
\end{remark}

Given the Weinstein domain $(W, \lambda, \phi)$, the AWLF produced by the proof of \autoref{theorem:new_WLF'}(1) is compatible with $(W, \lambda, \phi)$ --- this is the key point of the definition of compatibility. In fact, the result of the proof of \autoref{theorem:new_WLF'}(1) may be rephrased and strengthened as follows: 
    
\begin{lemma}[Existence of compatible AWLFs]\label{lemma:CAWLF_existence}
Let $(W, \lambda, \phi)$ be a Weinstein domain and let $\Lambda \subset \partial W$ be a possibly disconnected Legendrian submanifold. Then there exists an AWLF $(W^{\natural}; \mathcal{L}= \mathcal{L}'\cup \mathcal{L}'')$ and a Lagrangian submanifold $L \subset W^{\natural}$ such that 
    \begin{itemize}
        \item $(W^{\natural}; \mathcal{L}', \mathcal{L}'')$ is compatible with $(W, \lambda, \phi)$, and 
        \item there is a Legendrian lift of $L$ in $W^{\natural}\times \partial D^2$ that lies on a single page and is identified with $\Lambda$ under the $1$-Weinstein homotopy from $|W^{\natural}; \mathcal{L}|$ to $(W, \lambda, \phi)$. 
    \end{itemize}
\end{lemma}

Informally, the lemma says that given a Weinstein domain $W$ we may produce a compatible AWLF and moreover given a Legendrian $\Lambda \subset \partial W$ we may further require $\Lambda$ to sit on a single page of the OBD induced by the AWLF, even if $\Lambda$ has multiple components.
   
\subsection{Proof of move equivalence}\label{subsec:move_eq'}

Here we prove \autoref{theorem:new_WLF'}(2). The ``if'' direction is clear and was already observed by \cite{GP17}, so it remains to prove the ``only if'' direction, i.e., \autoref{cor: conj of GP}. The strategy is identical to the proof of \autoref{theorem: stabilization equivalence} in that we slice up a generic $1$-Weinstein homotopy into the three types from \autoref{def:weinstein_homotopies'} and analyze the compatible AWLFs on each end. 

The key part of the proof consists of the following two lemmas: Lemmas~\ref{lemma:move_eq_strict'} and \ref{lemma:move_eq_bd_hslide'}.

\begin{lemma}[Move equivalence over strict homotopies]\label{lemma:move_eq_strict'}
Let $(W, \lambda_t, \phi_t)$, $t\in [0,1]$, be a strict Weinstein homotopy and suppose that $(W^{\natural}_i; \mathcal{L}_i', \mathcal{L}_i'')$, $i=0,1$, are AWLFs that are compatible with $(W, \lambda_i, \phi_i)$, $i=0,1$, respectively. Then $(W^{\natural}_0; \mathcal{L}_0', \mathcal{L}_0'')$ and $(W^{\natural}_1; \mathcal{L}_1', \mathcal{L}_1'')$ are move equivalent. In particular, any two AWLFs compatible with the same Weinstein structure are move equivalent. 
\end{lemma}

\begin{proof}
For each $i=0,1$, since $(W^{\natural}_i; \mathcal{L}_i', \mathcal{L}_i'')$ is compatible with $(W, \lambda_i, \phi_i)$ we have that $(W^{\natural}_i; \mathcal{L}_i')$ is a multistabilization of $(W^{\flat}_i, \varnothing)$ for some Weinstein domain $W^{\flat}_i$. Moreover, $|W^{\flat}_i; \varnothing| = W^{\flat}_i \times D^2$ is $1$-Weinstein homotopic to a neighborhood $W^{(n-1)}_i$ of the ($n-1$)-skeleton of $(W, \lambda_i, \phi_i)$. Since the homotopy $(W, \lambda_t, \phi_t)$, $t\in [0,1]$, is strict, there is a deformation move that takes $W_0^{\flat}$ to $W^{\flat}_1$. Thus, after this move we can assume that $(W^{\natural}_0; \mathcal{L}_0')$ and $(W^{\natural}_1; \mathcal{L}_1')$ are both multistabilizations of a common $(W^{\flat}, \varnothing)$. 

We then have the following claim:

\begin{claim}[Move equivalence of multistabilizations]\label{claim:move_eq_multistab'}
Let $(W_i^{\natural}; \mathcal{L}_i)$, $i=0,1$, be multistabilizations of $(W^{\flat}; \varnothing)$. Then $(W_i^{\natural}; \mathcal{L}_i)$, $i=0,1$, are move equivalent. 
\end{claim}

\begin{proof}[Proof of \autoref{claim:move_eq_multistab'}.]
    Write
    \[
    W_i^{\natural} = W^{\flat} \cup h_{i,1} \cup \cdots \cup h_{i,k_i},
    \]
    where $k_i$ is the number of vanishing cycles in $\mathcal{L}_i$ and $h_{i,j}$ are the ($n-1$)-handles used in the multistabilizations. Up to Weinstein homotopy of each $W_i^{\natural}$, we may assume that the attaching loci of both sets of handles together
    \[
    h_{0,1},\dots, h_{0,k_0}, h_{1,1}, \dots, h_{1,k_1}
    \]
    are pairwise disjoint and lie in $\partial W^{\flat}$. This means that the Weinstein domains
    \begin{align*}
        &W_0^{\natural} \cup h_{1,1} \cup \cdots \cup h_{1,k_1} \\
        &W_1^{\natural} \cup h_{0,1} \cup \cdots \cup h_{0,k_0}
    \end{align*}
    are the same Weinstein domain, which we name $W^{\sharp}$. 
    
    Then $(W^{\sharp}; \mathcal{L}_1, \mathcal{L}_0)$ is a multistabilization (after the initial deformation move) of $(W^{\flat}_0; \mathcal{L}_0)$, and likewise $(W^{\sharp}; \mathcal{L}_0, \mathcal{L}_1)$ is a multistabilization (and deformation) of $(W^{\flat}_1; \mathcal{L}_1)$. After permutation, it follows that both of the initial AWLFs are move equivalent to $(W^{\sharp}; \mathcal{L}_0, \mathcal{L}_1)$.
\end{proof}

We now finish the proof of \autoref{lemma:move_eq_strict'}. By \autoref{claim:move_eq_multistab'}, there is a common multistabilization (and deformation)
    \[
    (W^{\sharp}; \mathcal{L}_0', \mathcal{L}_1')
    \]
    of both $(W^{\flat}_0; \mathcal{L}_0')$ and $(W^{\flat}_1; \mathcal{L}_1')$. 
    
    Consider the following AWLF moves, where we use $\sim$ to denote move equivalence. First, applied to $(W^{\flat}_1; \mathcal{L}_1', \mathcal{L}_1'')$ we have 
    \[
    (W^{\flat}_1; \mathcal{L}_1', \mathcal{L}_1'') \sim (W^{\sharp}; \mathcal{L}_0', \mathcal{L}_1', \mathcal{L}_1'') 
    \]
    by the above multistabilization (and deformation). Next, applied to $(W^{\flat}_0; \mathcal{L}_0', \mathcal{L}_0'')$ we have 
    \begin{align*}
        (W^{\flat}_0; \mathcal{L}_0', \mathcal{L}_0'') &\sim (W^{\flat}_0; \mathcal{L}_0'', \mathcal{L}_0') \\
        &\sim (W^{\sharp}; \mathcal{L}_1', \mathcal{L}_0'', \mathcal{L}_0') \\
        &\sim (W^{\sharp};  \mathcal{L}_0', \mathcal{L}_1', \mathcal{L}_0'').
    \end{align*}
    The first and third lines are permutation moves, while the second line is the aforementioned multistabilization. Finally, since both initial AWLFs are compatible with strictly homotopic Weinstein structures, we can relate $\mathcal{L}_i''$, $i=0,1$, by a further deformation. This completes the proof.
\end{proof}

\begin{lemma}[Move equivalence over non-strict homotopies]\label{lemma:move_eq_bd_hslide'}
Let $(W, \lambda_t, \phi_t)$, $t\in [0,1]$, be either
\begin{enumerate}
    \item[(1)] a handleslide homotopy, or  
    \item[(2)] a birth-death homotopy.
\end{enumerate}
Suppose that $(W^{\natural}_i; \mathcal{L}_i', \mathcal{L}_i'')$, $i=0,1$, are AWLFs that are compatible with $(W, \lambda_i, \phi_i)$, $i=0,1$, respectively. Then $(W^{\natural}_0; \mathcal{L}_0', \mathcal{L}_0'')$ and $(W^{\natural}_1; \mathcal{L}_1', \mathcal{L}_1'')$ are move equivalent.
\end{lemma}

\begin{proof}
Since all AWLFs compatible with the same structure are move equivalent by \autoref{lemma:move_eq_strict'}, we may discard the given AWLFs in the statement of the lemma and instead construct AWLFs, which we may also name $(W^{\flat}_i; \mathcal{L}_i', \mathcal{L}_i'')$, $i=0,1$, that are compatible with $(W, \lambda_i, \phi_i)$, $i=0,1$, respectively, and show that they are move equivalent. The statement of the lemma then follows. 

\s\n 
{\em Case (1): A handleslide homotopy.}
\s  

First we construct a compatible AWLF at $t=0$. Let $W^{(n-1)}$ denote a neighborhood of the ($n-1$)-skeleton of $(W, \lambda_0, \phi_0)$. We may assume without loss of generality that there are Legendrian spheres $\Lambda_-, \Lambda_+ \subset \partial W^{(n-1)}$ such that: 
    \begin{itemize}
        \item $\Lambda_-$ and $\Lambda_+$ intersect $\xi$-transversally at a single point, where $\xi$ is the induced contact structure on $\partial W^{(n-1)}$; 
        \item $\phi_0$ has two $n$-handles attached along Legendrians $\Lambda_-^{-\ve}$ and $\Lambda_+$, where $\Lambda_-^{-\ve}$ is a time $-\ve$ pushoff in the Reeb direction; 
        \item $\phi_1$ has two $n$-handles attached along $(\Lambda_- \uplus \Lambda_+)^{\ve}$ and $\Lambda_+$, i.e., we have slid the $n$-handle attached along $\Lambda_-^{-\ve}$ up over the $n$-handle attached along $\Lambda_+$.  
    \end{itemize}
By \autoref{lemma:CAWLF_existence}, there is a compatible AWLF $(W^{\natural}; \mathcal{L}', \varnothing)$ of $W^{(n-1)}$ and cycles $L_-, L_+, L_{\uplus}\subset W^{\natural}$ such that
\begin{itemize}
    \item $L_{-}$ lifts to $\Lambda_-^{-\ve}\subset \partial W^{(n-1)}$,
    \item $L_{+}$ lifts to $\Lambda_+\subset \partial W^{(n-1)}$,  
    \item $L_{\uplus}$ lifts to $(\Lambda_- \uplus \Lambda_+)^{\ve} \subset \partial W^{(n-1)}$,
    \item $L_-$ and $L_+$ intersect transversally in $W^{\natural}$ at a single point and $L_{\uplus} = \tau_{L_+}L_-$.
\end{itemize}
Then $(W^{\natural}; \mathcal{L}', (L_-, L_+))$ and $(W^{\natural}; \mathcal{L}', (L_+, \tau_{L_+}L_-))$ are AWLFs compatible with $(W, \lambda_i, \phi_i)$, $i=0,1$, respectively. They are move equivalent by a Hurwitz move. 

\s\n 
{\em Case (2): A birth-death homotopy.} 
\s 

Assume that the birth-death type point is a birth point (the case of a death point is identical), so that $(W, \lambda_1, \phi_1)$ is obtained from $(W, \lambda_0, \phi_0)$ by attaching a canceling pair of ($n-1$)- and $n$-handles. Let $(W^{\flat}_0; \mathcal{L})$ be compatible with $(W, \lambda_0, \phi_0)$. We will construct an AWLF of the form 
\[
\left(\tilde{W}_1^{\sharp};\, \tilde{\mathcal{L}} \cup \mathcal{L} \cup L\right)
\]
that is compatible with $(W, \lambda_1, \phi_1)$ and is move equivalent to $(W^{\flat}_0; \mathcal{L})$. This will be done in a number of stages and we refer to \autoref{fig:bdWLF} to keep track of the intermediary Weinstein domains. 

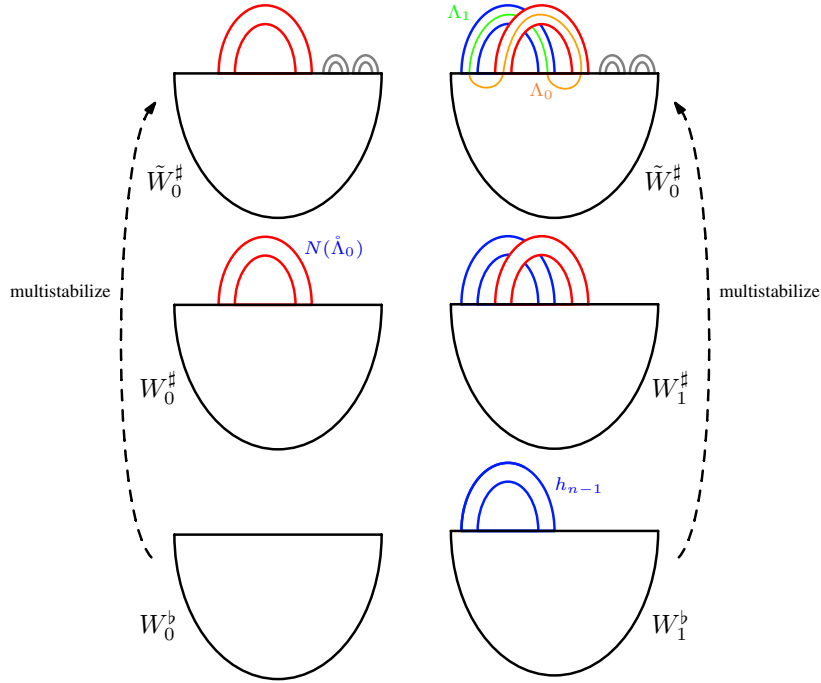
\begin{figure}[ht]
    \vskip-.05in
	\begin{overpic}[scale=.3]{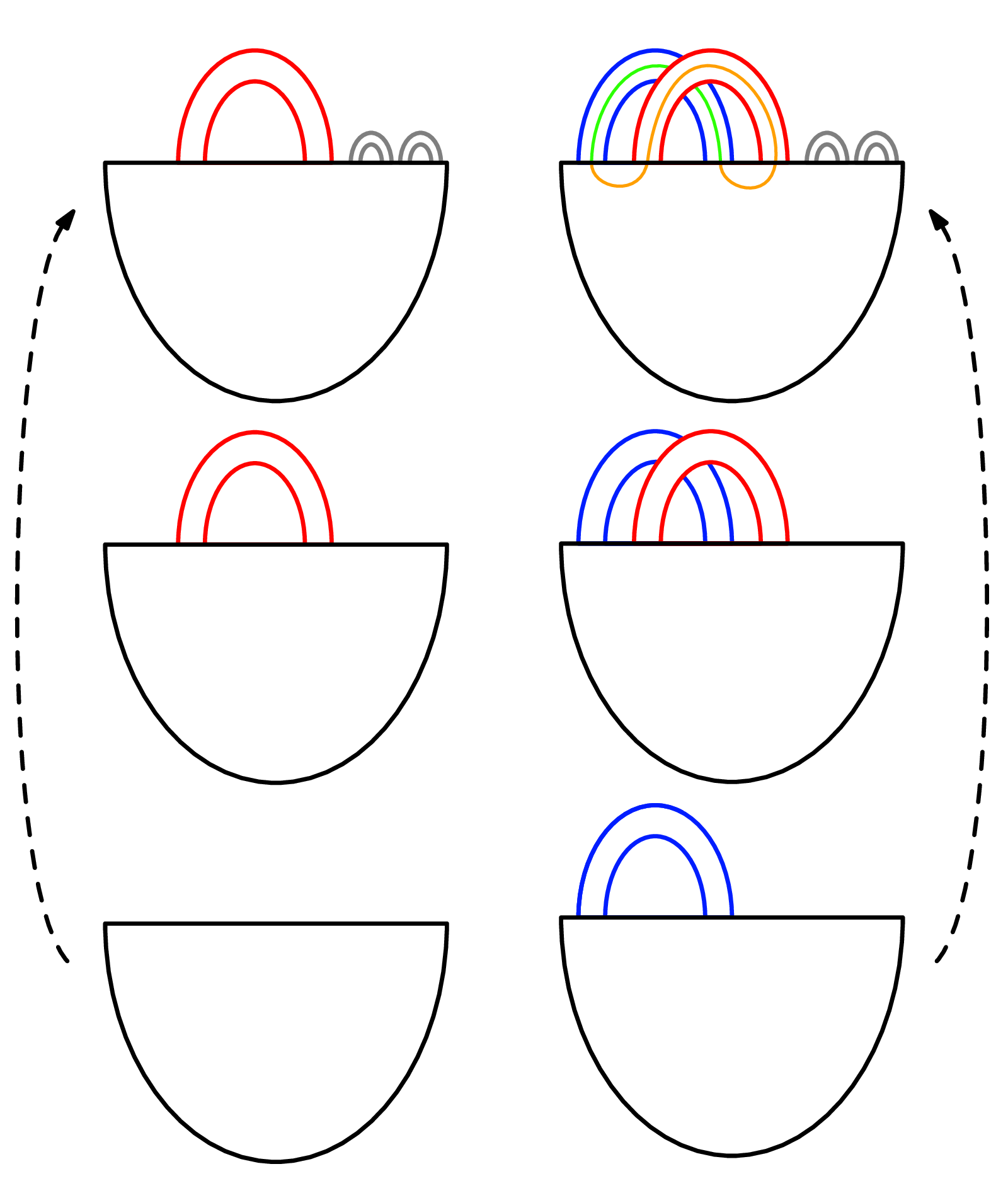}
	\put(4,8){\small $W_0^{\flat}$}
    \put(4,41){\small $W_0^{\sharp}$}
    \put(5,70){\small $\tilde{W}_0^{\sharp}$}
	\put(75.5,8){\small $W_1^{\flat}$}
    \put(75.5,41){\small $W_1^{\sharp}$}
    \put(74.5,70){\small $\tilde{W}_0^{\sharp}$}
    \put(27,61){\tiny {\cbu  $N(\mathring{\Lambda}_0)$}}
    \put(62,28){\tiny \textcolor{blue}{$h_{n-1}$}}
    \put(47,94){\tiny \textcolor{green}{$\Lambda_1$}}
    \put(58.5,83){\tiny \textcolor{Orange}{$\Lambda_0$}}
    \put(-14,55){\tiny multistabilize}
    \put(85,55){\tiny multistabilize}
	\end{overpic}
	\caption{The various Weinstein domains involved in the birth-death case of \autoref{lemma:move_eq_bd_hslide'}. The small gray handles represent stabilizations.}
	\label{fig:bdWLF}
\end{figure}

First, after strict Weinstein homotopy of $(W, \lambda_1, \phi_1)$, we may assume that 
\[
(W, \lambda_1, \phi_1) = |W_1^{\flat}:=W_0^{\flat} \cup h_{n-1};\, \mathcal{L}| \,\cup\, h_n
\]
where $h_{n-1}$ is a ($2n-2$)-dimensional ($n-1$)-handle and $h_n$ is a $2n$-dimensional $n$-handle attached along $\Lambda \subset \bdry (W_1^{\flat}\times D^2$). We may further assume that $\Lambda = \Lambda_0 \cup \Lambda_1$ with $\partial \Lambda_0 = \partial \Lambda_1$, where $\Lambda_1$ is the core of $h_{n-1}$ on the fiber $W_1^{\flat} \times \{\theta = 0\}$. See \autoref{fig:bdWLF}. 

Next, we wish to attach a neighborhood $N(\Lambda_0)$ of $\Lambda_0$ to $W_0^{\flat} \times \{\theta=0\}$ as a Weinstein handle, but $\op{int}(\Lambda_0)$ may intersect nontrivially with this fiber in a way that a generic perturbation cannot undo. However, it is possible to isotop intersections away as follows. Note that by transversality and a dimension count, via a generic isotopy rel $\bdry \Lambda_0$ we may assume that $\op{int}(\Lambda_0) \cap \op{Skel}(W_0^{\flat} \times \{\theta = 0\}) = \varnothing$. We then use the locally defined contact vector field $\theta\, \partial_{\theta} + X_{W_0^{\flat}}$ near $\{\theta = 0\}$, where $X_{W_0^{\flat}}$ is the Liouville vector field of $W_0^{\flat}$, to push the intersection $\Lambda_0 \cap (W_0^{\flat} \times \{\theta = 0\})$ out through $\partial W_0^{\flat} \times \{\theta = 0\}$. Then we can further adjust $W_0 \times \{\theta=0\}$ near $\partial \Lambda_0$ so that $\Lambda_0 = \tilde{\Lambda}_0 \cup \mathring{\Lambda}_0$, where $\mathring{\Lambda}_0$ is disjoint from $W_0^{\flat} \times \{\theta=0\}$ except along $\bdry \mathring{\Lambda}_0$,  $\partial \mathring{\Lambda}_0 \subset \bdry  W_0^{\flat} \times \{\theta=0\}$, and $\tilde{\Lambda}_0\simeq S^{n-2}\times [-1,1]$ is a cylindrical component that lies in $W_0^{\flat} \times \{\theta=0\}$ between $\partial \Lambda_1$ and $\partial \mathring{\Lambda}_0$; see \autoref{fig:lambda0}.

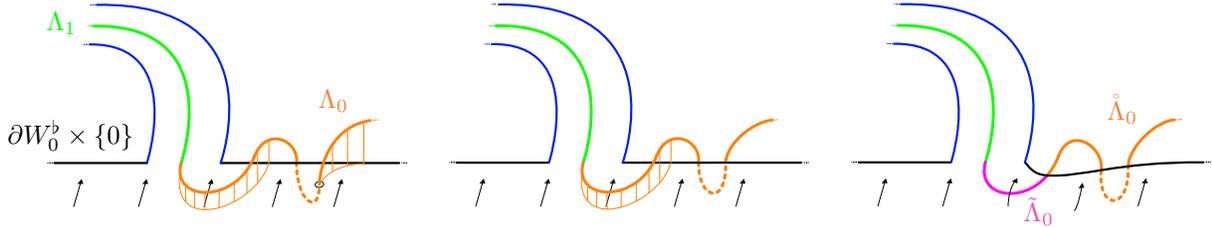
\begin{figure}[ht]
	\begin{overpic}[scale=.44]{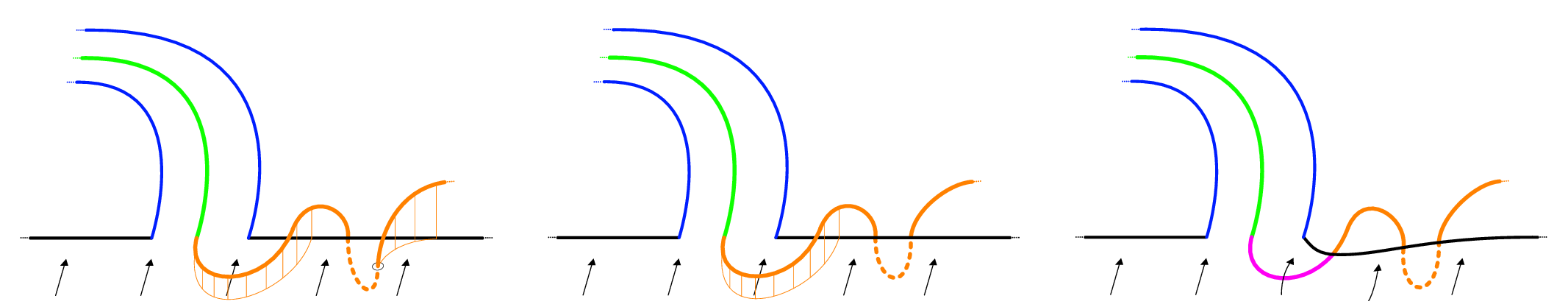}
	\put(-2,6){\small $\partial W_0^{\flat} \times \{0\}$}
    \put(1.25,15.5){\small\textcolor{green}{$\Lambda_1$}}
    \put(24,9){\small \textcolor{Orange}{$\Lambda_0$}}
    \put(90,8){\small \textcolor{Orange}{$\mathring{\Lambda}_0$}}
    \put(83,-0.5){\small \textcolor{Rhodamine}{$\tilde{\Lambda}_0$}}
	\end{overpic}
	\caption{Adjusting $\Lambda_0$ with respect to $W_0^{\flat} \times \{\theta = 0\}$. The black vector field is $X_{W_0^{\flat}}$. From the first to the second frame we push out the intersection $\Lambda_0 \cap (W_0^{\flat} \times \{\theta = 0\})$ and from the second to the third frame we adjust the Weinstein hypersurface so that it contains a cylindrical component (in pink) of $\Lambda_0$.}
	\label{fig:lambda0}
\end{figure}

Now, we define two embedded Weinstein domains in $\bdry (W_1^{\flat}\times D^2)$: 
\begin{align*}
    W_0^{\sharp} &:= (W_0^{\flat} \times \{\theta = 0\}) \cup N(\mathring{\Lambda}_0),   \\
    W_1^{\sharp} &:= (W_1^{\flat} \times \{\theta = 0\}) \cup N(\mathring{\Lambda}_0).
\end{align*}
See the second row of \autoref{fig:bdWLF}. 

By the proof of \cite[Corollary 1.3.3]{HH19}, there is a supporting OBD of $\bdry (W_0^{\flat} \times D^2)$ such that the Weinstein hypersurface $W_0^{\sharp}$ constructed above is a sublevel set of a page, the latter of which we denote $\tilde{W}_0^{\sharp}$. By the same argument using \autoref{theorem: stabilization equivalence} as in the proof of \autoref{lemma:CAWLF_existence}, we can witness this OBD by an AWLF multistabilization $(\tilde{W}_0^{\sharp}; \tilde{\mathcal{L}})$ of $(W_0^{\flat}; \varnothing)$. Finally, we let $\tilde{W}_1^{\sharp}:= \tilde{W}_0^{\sharp} \cup h_{n-1}$. See the top row of \autoref{fig:bdWLF}.

Now, let $L\subset \tilde{W}_1^{\sharp}$ be the Lagrangian corresponding to $\Lambda = \Lambda_0 \cup \Lambda_1$. We claim that the AWLF $(\tilde{W}_1^{\sharp};\, \tilde{\mathcal{L}} \cup \mathcal{L} \cup L)$ is compatible with $(W, \lambda_1, \phi_1)$ and is move equivalent to $(W^{\flat}; \mathcal{L})$. 

\begin{itemize}
    \item {\em Compatibility.} Let $\mathcal{L} = \mathcal{L}' \cup \mathcal{L}''$ be the decomposition so that $(W^{\flat}_0; \mathcal{L}', \mathcal{L}'')$ is compatible with $(W, \lambda_0, \phi_0)$. Then $(\tilde{W}_1^{\sharp}; \tilde{\mathcal{L}} \cup \mathcal{L}')$ is a multistabilization of $(W_1^{\flat}; \mathcal{L}')$ and hence gives a neighborhood of the ($n-1$)-skeleton of $(W, \lambda_1, \phi_1)$. The cycles $\mathcal{L}'' \cup L$ give the $n$-handles of $(W,\lambda_1, \phi_1)$. Thus, 
    \[
    \left(\tilde{W}_1^{\sharp}; \,\tilde{\mathcal{L}} \cup \mathcal{L}', \, \mathcal{L}'' \cup L \right) = \left(\tilde{W}_1^{\sharp};\, \tilde{\mathcal{L}} \cup \mathcal{L} \cup L\right)
    \]
    is compatible with $(W,\lambda_1, \phi_1)$.

    \item {\em Move equivalence.} Applying a permutation move gives $(\tilde{W}_1^{\sharp};\, L \cup \tilde{\mathcal{L}} \cup \mathcal{L})$. This is a stabilization of $(\tilde{W}_0^{\sharp};\, \tilde{\mathcal{L}} \cup \mathcal{L})$, which is a multistabilization of $(W^{\flat}; \mathcal{L})$. 
\end{itemize}
This completes the proof of \autoref{lemma:move_eq_bd_hslide'}. 
\end{proof}

Finally, we prove \autoref{theorem:new_WLF'}(2), which is more or less automatic from Lemmas~\ref{lemma:move_eq_strict'} and \ref{lemma:move_eq_bd_hslide'}.

\begin{proof}[Proof of \autoref{theorem:new_WLF'}(2).]
Let $(W^{\flat}_i; \mathcal{L}_i)$, $i=0,1$, be AWLFs such that the total spaces $|W^{\flat}_i; \mathcal{L}_i|$, $i=0,1$, are $1$-Weinstein homotopic. By \autoref{remark:CAWLFs}, $(W^{\flat}_i; \varnothing, \mathcal{L}_i)$, $i=0,1$, are compatible with $|W^{\flat}_i; \mathcal{L}_i|$, $i=0,1$, respectively.

By genericity, the $1$-Weinstein homotopy connecting $|W^{\flat}_i; \mathcal{L}_i|$, $i=0,1$, can be sliced into a concatenation of (rescaled) strict, handleslide, and birth-death homotopies. By \autoref{lemma:CAWLF_existence}, there are compatible AWLFs at each slicing point, and by Lemmas~\ref{lemma:move_eq_strict'} and \ref{lemma:move_eq_bd_hslide'} each AWLF is move equivalent to the next. Hence the original AWLFs are move equivalent, as desired. 
\end{proof}

\bibliography{mybib}
\bibliographystyle{amsalpha}

\end{document}